\documentclass[reqno]{amsart}

\usepackage{amsmath, amssymb, amsthm, amsfonts, color, mathrsfs}

\usepackage[dvipsnames]{xcolor}

\usepackage[margin = 1.1in]{geometry}
\usepackage{adjustbox}

\input xy
\xyoption{all}

\usepackage{hyperref}

\usepackage{tikz-cd}

\numberwithin{equation}{subsection}

\newtheorem{theorem}[subsubsection]{Theorem}
\newtheorem{lemma}[subsubsection]{Lemma}
\newtheorem{prop}[subsubsection]{Proposition}
\newtheorem{cor}[subsubsection]{Corollary}

\theoremstyle{definition}
\newtheorem{defn}[subsubsection]{Definition}

\newtheorem{remark}[subsubsection]{Remark}
\newtheorem{exam}[subsubsection]{Example}

\newtheorem{assumption}[subsubsection]{Assumption}

\usepackage[utf8]{inputenc}
\usepackage[T1]{fontenc}
\usepackage{xcolor} 
\usepackage[most]{tcolorbox} 

\definecolor{humanbubble}{RGB}{235, 245, 255} 
\definecolor{aibubble}{RGB}{245, 245, 245}    
\definecolor{loggray}{RGB}{230, 230, 230}
\definecolor{indexbg}{RGB}{248, 248, 248}

\newtcolorbox{interactionlog}[2][]{
  enhanced,
  arc=0pt, outer arc=0pt,
  colback=white, colframe=black!60,
  boxrule=0.8pt,
  fonttitle=\bfseries\sffamily, coltitle=black, colbacktitle=loggray,
  title={Human-AI Interaction Card \if\relax\detokenize{#1}\relax\else for #1\fi},
  halign title=center, attach title to upper,
  after title={\vspace{4pt}\hrule\vspace{10pt}},
  lower separated=true,
  segmentation style={solid, black!60, line width=0.8pt},
  colbacklower=indexbg,
  after upper={\par\vfill
    \begin{tcolorbox}[
      enhanced, colback=indexbg, colframe=white, boxrule=0pt,
      top=0pt, bottom=0pt, fontupper=\footnotesize\sffamily,
      title=Model Index, coltitle=black!70, attach title to upper,
      after title={:\enskip}, sharp corners
    ]
    #2 
    \end{tcolorbox}
  }
}

\def\AA{\mathbb{A}}
\def\BB{\mathbb{B}}
\def\CC{\mathbf{C}}
\def\DD{\mathbb{D}}
\def\EE{\mathbb{E}}
\def\FF{\mathbb{F}}
\def\GG{\mathbb{G}}

\def\LL{\mathbb{L}}
\def\MM{\mathbb{M}}

\def\PP{\mathbb{P}}

\def\RR{\mathbf{R}}

\def\TT{\mathbb{T}}

\def\VV{\mathbb{V}}

\def\ZZ{\mathbb{Z}}

\newcommand\cC{\mathcal{C}}
\newcommand\cD{\mathcal{D}}
\newcommand\cE{\mathcal{E}}
\newcommand\cF{\mathcal{F}}

\newcommand\cH{\mathcal{H}}

\newcommand\cK{\mathcal{K}}
\newcommand\cL{\mathcal{L}}

\newcommand\cO{\mathcal{O}}
\newcommand\cP{\mathcal{P}}

\newcommand\cY{\mathcal{Y}}

\newcommand{\A}{\mathbf{A}}

\def\bD{\mathbf{D}}

\newcommand\frB{\mathfrak{B}}

\newcommand\frL{\mathfrak{L}}

\newcommand\frN{\mathfrak{N}}

\newcommand\frP{\mathfrak{P}}

\newcommand\frR{\mathfrak{R}}

\newcommand\frd{\mathfrak{d}}

\newcommand{\sL}{{\mathscr{L}}}

\newcommand\alg{\textup{alg}}

\newcommand\Av{\textup{Av}}

\newcommand{\Bun}{\textup{Bun}}

\newcommand{\ch}{\textup{char}}

\newcommand\ev{\textup{ev}}
\newcommand{\even}{\textup{even}}

\newcommand\Fr{\textup{Fr}}

\newcommand\Frob{\textup{Frob}}

\newcommand\Gal{\textup{Gal}}

\newcommand{\Gr}{\textup{Gr}}

\newcommand\Hk{\mathrm{Hk}}

\newcommand\id{\textup{id}}

\newcommand\Jac{\textup{Jac}}

\newcommand\Out{\textup{Out}}

\newcommand{\Pic}{\textup{Pic}}
\newcommand\pr{\textup{pr}}

\newcommand\pt{\textup{pt}}

\newcommand\Rep{\textup{Rep}}
\newcommand{\Res}{\textup{Res}}

\newcommand\sgn{\textup{sgn}}
\newcommand\Sht{\textup{Sht}}

\newcommand\Span{\textup{Span}}
\newcommand\Spec{\textup{Spec}\ }

\newcommand\Sym{\textup{Sym}}

\newcommand{\Tr}{\textup{Tr}}

\newcommand{\univ}{\textup{univ}}

\newcommand{\vol}{\textup{vol}}

\newcommand\Aut{\textup{Aut}}
\newcommand\Isom{\textup{Isom}}
\newcommand\Hom{\textup{Hom}}
\newcommand\End{\textup{End}}

\DeclareMathOperator\Pf{Pf}

\newcommand\GL{\textup{GL}}

\newcommand\PGL{\textup{PGL}}

\newcommand\SL{\textup{SL}}

\newcommand\SO{\textup{SO}}

\DeclareMathOperator\Spin{Spin}

\DeclareMathOperator{\U}{U}

\newcommand{\Gm}{\GG_m}

\newcommand{\ad}{\textup{ad}}
\newcommand{\Ad}{\textup{Ad}}
\renewcommand\sc{\textup{sc}}
\newcommand{\der}{\textup{der}}

\newcommand\xch{\mathbb{X}^*}
\newcommand\xcoch{\mathbb{X}_*}

\newcommand\upH{\textup{H}}

\newcommand{\Ql}{\Q_{\ell}}
\newcommand{\Qlbar}{\overline{\Q}_\ell}

\renewcommand{\j}[1]{\langle{#1}\rangle}
\newcommand{\wt}[1]{\widetilde{#1}}
\newcommand{\wh}[1]{\widehat{#1}}
\newcommand\quash[1]{}
\newcommand\un{\underline}
\newcommand{\bu}{\bullet}
\newcommand{\ov}{\overline}
\newcommand{\bs}{\backslash}

\newcommand{\nablaop}{\nabla^{\eta}_{\mu}}
\newcommand{\pderiv}[2]{\frac{\partial #1}{\partial #2}}
\newcommand\op{\oplus}
\newcommand\ot{\otimes}

\renewcommand\c\circ

\newcommand{\homog}[2]{\textup{H}_{#1}({#2})}  
\newcommand{\cohog}[2]{\textup{H}^{#1}({#2})}     
\newcommand{\cohoc}[2]{\textup{H}_{c}^{#1}({#2})}     

\newcommand\mt{\mapsto}
\newcommand{\incl}{\hookrightarrow}
\newcommand{\isom}{\stackrel{\sim}{\to}}

\newcommand{\surj}{\twoheadrightarrow}

\newcommand\xr{\xrightarrow}

\newcommand{\ari}{\ar@{^{(}->}} 

\renewcommand\a\alpha
\renewcommand\b\beta
\newcommand\g\gamma
\newcommand\G\Gamma
\renewcommand\d\delta
\newcommand\D\Delta
\newcommand{\e}{\epsilon}
\newcommand{\io}{\iota}
\renewcommand{\th}{\theta}

\newcommand{\ph}{\varphi}
\renewcommand\r\rho
\newcommand{\s}{\sigma}
\renewcommand{\t}{\tau}

\newcommand{\y}{\eta}
\newcommand{\z}{\zeta}
\newcommand{\ep}{\epsilon}

\renewcommand{\l}{\lambda}
\renewcommand{\L}{\Lambda}
\newcommand{\om}{\omega}

\newcommand{\Sig}{\Sigma}

\newcommand\na{\natural}
\newcommand\sh{\sharp}

\renewcommand\v{\vee}

\newcommand\aug{\textup{aug}}

\newcommand\PD{\textup{PD}}
\newcommand\AB{\textup{AB}}

\newcommand\bl{\blacktriangleright}

\newcommand\nb{\nabla}
\newcommand\pl{\partial}
\newcommand\qs{\textup{qs}}
\newcommand\frev{\mathfrak{ev}}

\usepackage{color}

\newcommand\mnote[1]{\marginpar{\tiny #1}}

\newcommand{\inj}{\hookrightarrow}

\newcommand{\tw}[1]{\langle #1 \rangle}

\newcommand{\F}{\mathbf{F}}
\newcommand{\Z}{\mathbf{Z}}
\newcommand{\Q}{\mathbf{Q}}
\newcommand{\Qll}[1]{\Q_{\ell, #1}}
\newcommand{\ul}[1]{\underline{#1}}
\newcommand{\ol}[1]{\overline{#1}}
\newcommand{\mrm}[1]{\mathrm{#1}}
\newcommand{\co}{\colon}

\newcommand{\RGamma}{\rR \Gamma}

\DeclareMathOperator{\Id}{Id}

\RequirePackage{xspace}

\newcommand{\rH}{\ensuremath{\mathrm{H}}\xspace}

\newcommand{\rR}{\ensuremath{\mathrm{R}}\xspace}

\title[Arithmetic volumes of moduli stacks of Shtukas]{Arithmetic volumes of moduli stacks of Shtukas}

\dedicatory{Dedicated to the memory of Dick Gross}
\author{Tony Feng}
\address{Department of Mathematics, University of California Berkeley, University Drive, Berkeley, CA 94720}
\email{fengt@berkeley.edu}
\author{Zhiwei Yun}
\address{Department of Mathematics, Massachusetts Institute of Technology, 77 Massachusetts Ave, Cambridge, MA 02139}
\email{zyun@mit.edu}
\author{Wei Zhang}
\address{Department of Mathematics, Massachusetts Institute of Technology, 77 Massachusetts Ave, Cambridge, MA 02139}
\email{weizhang@mit.edu}
\date{}

\begin{document}

\begin{abstract}
    We define and study ``tautological classes'' in the cohomology of moduli stacks of shtukas, pursuing two directions of applications. First, we prove a formula relating the ``arithmetic volume'' of tautological classes to higher derivatives of Artin $L$-functions, which can be viewed as an arithmetic analog of Hirzebruch's Proportionality principle. Second, we define and analyze the structure of the ``phantom tautological ring'',  using a general relation between Hecke correspondences and Vinberg's degeneration, and give applications to a function field analog of Colmez's Conjecture. 
\end{abstract}

\maketitle

\tableofcontents

\section{Introduction}

\subsection{Volume and Hirzebruch proportionality for locally symmetric spaces}
Hirzebruch's Proportionality Theorem \cite{Hir} (in the compact case), extended by Mumford to the non-compact case \cite{Mum}, asserts that the integral (also known as the ``volume") of any Chern class polynomial of automorphic vector bundles on a locally symmetric space is proportional to the integral of the corresponding Chern classes on a partial flag variety. Moreover, the proportionality constant is essentially the special value of a certain product of (shifts of) zeta or Hecke $L$-functions. 

More precisely, let $(G,D)$ be a Shimura datum, where $G$ is a reductive group over $\Q$ and $D\simeq G(\RR)/K$ is a Hermitian symmetric domain with $K$ a maximal compact subgroup of $G(\RR)$. Let $G_c\subset G(\CC)$ be a compact form containing $K$, and  $D^\vee=G_c/K$ the compact dual of $D$.
Note that $D^\v$ can be realized as the $\CC$-points of a partial flag variety $G(\CC)/P_{\mu}(\CC)$, where $P_\mu$ is a parabolic subgroup associated to the Shimura cocharacter $\mu$. We have an analytic open embedding $D\incl D^\vee$. 
For any discrete subgroup $\G\subset G(\RR)$, we have a  natural map, which we call the \emph{Hodge morphism}\footnote{Terminology based on \cite[Introduction]{WZ} where the authors call the target the \emph{Hodge stack}.}, 
\begin{equation}\label{eq: Hdg}
\xymatrix{  \ev:\G\bs D\ar[r] & G(\CC)\bs D^\vee\cong \BB P_\mu(\CC).}
\end{equation}
Automorphic vector bundles on the source are by definition the pull-back of vector bundles  on the target. 
Then the map $ \ev$ induces a homomorphism on the cohomology rings
\begin{equation}\label{eq:hodge-pullback}
    \xymatrix{\ev^*:\cohog{*}{\BB P_\mu(\CC),\Q}=R^{W_\mu}_\Q \ar[r]& \cohog{*}{\G\bs D,\Q}}.
\end{equation}
Here $R_{\Q}=\cohog{*}{\BB T(\CC),\Q}$, where $T$ is a maximal torus of $G$, and $W_{\mu}$ is the Weyl group of the Levi subgroup associated to $\mu$, acting on $T$ in the natural way. Chern--Weil theory implies that the cohomology classes coming from $\cohog{>0}{\BB G(\CC),\Q}=R^W_{\Q,+}$ (here $W$ is the Weyl group of $G$) are mapped to zero under $\ev^*$. Thus the ring homomorphism $\ev^*$ factors through a homomorphism
\begin{equation}\label{Herm rho}
    \xymatrix{\r:\cohog{*}{D^\vee,\Q}=R^{W_\mu}_\Q/(R^W_{\Q,+})\ar[r]& \cohog{*}{\G\bs D,\Q}}.
\end{equation}

Now suppose $\G\bs D$ is compact and let $N=\dim D=\dim D^\vee$. The map $\r$ then restricts to an isomorphism on the top cohomology
\begin{equation*}
    \r^{2N}: \cohog{2N}{D^\vee,\Q}\isom \cohog{2N}{\G\bs D,\Q}.
\end{equation*}
Now both sides are one-dimensional vector spaces, and they are each equipped with canonical trivialization (up to a Tate twist) given by pairing with the fundamental cycles of $D^\vee$ and of $\G\bs D$ respectively. Therefore, there is a unique constant $c\in \Q^\times$ (cf. \cite[Corollary 7.21]{WZ}) making the following diagram commutative
\begin{equation}\label{Hirz prop}
    \xymatrix{\cohog{2N}{D^\vee,\Q}\ar[d]_{\int_{D^\v}}\ar[r]^-{\r^{2N}}_-{\sim} & \cohog{2N}{\G\bs D,\Q}\ar[d]^{\int_{\G\bs D}}\\
    \Q(-N) \ar[r]^{c}  &\Q(-N)}
\end{equation}
This is the famous \emph{Hirzebruch Proportionality principle} \cite{Hir}. We may suggestively write the proportionality principle as
\begin{equation}\label{eq: H prop}
\int_{\G\bs  D} \ev^*\eta=c \int_{D^\vee}\wt\eta
\end{equation}
for any class  $\eta\in\cohog{2N}{\BB P_\mu(\CC),\Q}$, where $\wt\eta$ is the pull back of $\eta$ along the map $D^\vee\to G(\CC)\bs D^\vee\cong \BB P_\mu(\CC)$.  When $\G\bs D$ is non-compact, Mumford proved in \cite{Mum} that the same assertion goes through after replacing it by any smooth toroidal compactification.

 Moreover, the proportionality constant $c$ in \eqref{eq: H prop} can be interpreted as the special value of an $L$-function attached to the ``motive of $G$'' in the sense of Gross \cite{Gro97} (up to an index factor depending on $\G$).

\subsection{Arithmetic volumes of Shimura varieties}
When there are natural integral models for  Shimura varieties and their automorphic vector bundles, there is an ``arithmetic'' analogue of the volume of Chern classes, defined via the arithmetic intersection theory, which we call the \emph{arithmetic volume}. In this case, the arithmetic volume has only been computed for specific families of examples, such as in the case where $G$ is the isometry group of a quadratic space \cite{Hor} over $\Q$ or a hermitian space \cite{BH, Guo} with respect to an imaginary quadratic field extension of a totally real field  with signature $(1,n-1)$ at one place and definite in all other places. 
In these works, the arithmetic volumes are shown to be essentially a special value of the first derivative of the L-function of $G$.

\subsection{Arithmetic volumes of moduli of Shtukas}

In this paper we investigate the analogous question over function fields, namely the arithmetic volumes of ``automorphic vector bundles" on moduli spaces of shtukas for a general reductive group. Our initial motivation was to extract the intersection number out of the constant term of the arithmetic theta series in  the top degree case in our previous work \cite{FYZ2}. The main difficulty is that these moduli spaces are almost never proper and their compactifications (when the number of legs is more than one) are more complicated than the toroidal compactification of Shimura varieties. As one of the key steps of this paper, we will define a ``regularized" volume, via the trace of the endomorphism on the cohomology of $\Bun_G$ induced by certain natural correspondences.

Now let $X$ be a smooth, projective, geometrically connected curve over $\F_q$. For simplicity, in the Introduction we consider only split reductive groups $G$ over $\F_q$; the main text treats arbitrary quasisplit reductive group schemes over $X$ (and we explain in Remark \ref{rem:beyond-quasisplit} that we can reduce arbitrary reductive group schemes over $X$ to the quasisplit case, for the purpose of computing arithmetic volumes). Let $\Bun_G$ be the moduli stack of $G$-bundles on $X$. 
 
\subsubsection{Moduli of shtukas} 
Let $\mu=(\mu_{1},\cdots ,\mu_{r})$ be an admissible (see \eqref{eq:modification-condition}) $r$-tuple of modification types for $G$. There is an associated moduli stack of shtukas $\Sht_{G}^{\mu}$. It is neither proper nor of finite type, hence the Chern numbers of $\Sht^{\mu}_{G}$ are not a priori defined. We now outline a regularized definition of Chern numbers, and relate them to higher derivatives of the $L$-function attached to the motive of $G$ (in the sense of Gross \cite{Gro97}) over the curve $X$.

The general setup is as follows. We have a Hecke stack 
\begin{equation*}
\xymatrix{ & \Hk^{\mu}_{G}\ar[dl]_{h_{0}}\ar[dr]^{h_{r}} \\
\Bun_{G} && \Bun_{G}}
\end{equation*}
defined using the same modification data $\mu$ as $\Sht^{\mu}_{G}$. Then $\Sht^{\mu}_{G}$ is defined as the fibered product 
\begin{equation}\label{eq:intro-shtukas}
\begin{tikzcd}[column sep = huge]
\Sht^{\mu}_G \ar[r] \ar[d] & \Hk_G^{\mu} \ar[d, "{(h_0, h_r)}"] \\
\Bun_G \ar[r, "{(\Id, \Frob)}"] & \Bun_G \times \Bun_G    
\end{tikzcd}
\end{equation}
In other words, $\Sht^\mu_G$ is the ``Frobenius-twisted fixed points'' of a correspondence on $\Bun_G$. Therefore we regularize Chern numbers of $\Sht_G^\mu$
as the trace of a suitable endomorphism on the cohomology of $\Bun_G$.


\subsubsection{Endomorphisms of cohomology} We have $\dim \Sht^{\mu}_{G}=r+\sum_{i=1}^{r}\j{2\r,\mu_{j}}$. We write $D_{\mu_j} := \j{2\r,\mu_{j}}$ and $D_\mu := \sum_{j=1}^r D_{\mu_j}$.

Henceforth, $\rH^*(-)$ always means \emph{geometric} cohomology of $\ol \Q_\ell$ (i.e., cohomology over the algebraic closure of the ground field), where $\ell$ is invertible in $\F_q$. Let  $\y_j \in \cohog{2(D_{\mu_j}+1)}{\Hk^{\mu_j}_{G}}$. The numerology is such that $\y_j \cap [\Hk^{\mu_j}_{G}]$ is a Borel-Moore homology class in $\Hk^{\mu_j}_{G}$ of degree $2\dim \Bun_{G}$. For the natural correspondence 
\[
\Bun_G \xleftarrow{h_{j-1}} \Hk_G^{\mu_j} \xrightarrow{h_{j}} \Bun_G
\]
we may view $\y\cap [\Hk^{\mu_j}_{G}] $ as a cohomological correspondence for the constant sheaf on $\Bun_{G}$ supported on $\Hk^{\mu_j}_{G}$. Using that  $h_{j}$ is proper, it induces a degree-preserving map on cohomology (of the geometric fiber with $\Qlbar$-coefficients)
\begin{eqnarray*}
\G^{\mu_j}_{\y_j}&:& \cohog{*}{\Bun_{G}}\to \cohog{*}{\Bun_{G}}\\
&&\th\mt h_{j-1,*}(h_{j}^{*}(\th)\cup \y_j)
\end{eqnarray*}
Similarly, since $h_j$ is proper it induces a degree-preserving map on compactly supported cohomology
\begin{eqnarray*}
{}_{c}\G^{\mu_j}_{\y_j}&:& \cohoc{*}{\Bun_{G}}\to \cohoc{*}{\Bun_{G}}\\
&&\th\mt h_{j*}(h_{j-1}^{*}(\th)\cup \y_j)
\end{eqnarray*}
(The reversal of the roles of $h_{j-1}$ and $h_{j}$ in these definitions is due to the adjoint relationship of $\G^{\mu}_{\y}$ and ${}_{c}\G^{\mu}_{\y}$. This will be explained further in \S \ref{sss:Gamma c}.)

\subsubsection{Tautological classes}
In analogy with the Shimura variety case, there is a natural way to supply such $\y_j \in \cohog{2(D_{\mu_j}+1)}{\Hk^{\mu_j}_{G}}$, by taking Chern classes of automorphic vector bundles. We show in \S\ref{sssec: canonical parabolic reduction} that for each $i=1, \ldots, r$ there is a canonical morphism 
\begin{equation*}
\xymatrix{ \ev_{j}: \Hk^{\mu_j}_{G}\ar[r] & \BB P_{\mu_j}} 
\end{equation*}
 where $P_{\mu_j} \subset G$ is the parabolic subgroup defined as the attracting locus of $\mu_j:\GG_m\to G$. We also have the leg map $p_j:\Hk^{\mu_j}_{G}\to X$. We will call classes pulled back along the {\em enhanced Hodge morphism}
\begin{equation}\label{eq:enhanced-hodge}
(\ev_{j},p_j)^\ast:  \cohog{2(D_{\mu_j}+1)}{\BB P_{\mu_j}\times X} \to \cohog{2(D_{\mu_j}+1)}{\Hk^{\mu_j}_{G}}
\end{equation}
{\em tautological classes} on $\Hk^{\mu_j}_{G}$.

\subsubsection{Arithmetic Volume} We may now define the arithmetic volume of $\Sht_G^\mu$ with respect to tautological classes. For the sake of simplicity, we present a simplified definition in the Introduction for the case where $G$ is semisimple.
 
\begin{defn} Assume $G$ is semisimple. Let $\y_j \in  \cohog{2(D_{\mu_j}+1)}{\BB P_{\mu_j}\times X}$ for $j=1,\cdots, r$. We define the volume of $\Sht^{\mu}_{G}$ with respect to  $\y :=(\y_1,\cdots,\y_r) $ to be
\begin{equation}\label{def vol Sht}
\vol(\Sht^{\mu}_{G}, \eta):= \Tr(
{}_{c}\G^{\mu_r}_{\y_r}\c\cdots\c  {}_{c}\G^{\mu_1}_{\y_1}\c\Frob^*   \mid \rH_c^*(\Bun_G)).
\end{equation}
\end{defn}

As $\rH_c^*(\Bun_G)$ is infinite-dimensional, the convergence of the trace in the definition of $\vol(\Sht^{\mu}_{G}, \eta)$ needs to be justified, which we do in greater generality in Proposition \ref{p:trace conv}.

\begin{remark}For an $r$-tuple $\mu=(\mu_{1},\cdots ,\mu_{r})$, the $\ev_{j}$ assemble to a morphism 
 $$ \xymatrix{ \ev_{\mu}: \Hk^{\mu}_{G}\ar[r] & \prod_{j=1}^r \BB P_{\mu_j}} 
 $$
  The pre-composition of $\ev_{\mu}$ with $\Sht^{\mu}_{G}\to \Hk^{\mu}_{G}$ defines a morphism  $\Sht^{\mu}_{G}\to \prod_{j} \BB P_{\mu_j}$, which may be viewed as an analog of the Hodge morphism \eqref{eq: Hdg}. Vector bundles on $\Sht^{\mu}_{G}$ built from pullbacks of vector bundles on $\BB P_{\mu_j}$ are the analogue of ``automorphic vector bundles" on Shimura varieties. Therefore, our arithmetic volume may be viewed as the analog of the arithmetic volumes of automorphic vector bundles on Shimura varieties.
\end{remark}

\subsubsection{Main results}
Our main result is a formula of the arithmetic volume in terms of (mixed) higher derivative of $L$-functions attached to the Gross motive of a reductive group \cite{Gro97}. We present the formula in the simplified setting where $G$ is split semisimple, although in the main text we treat general quasisplit $G\rightarrow X$, which witnesses more interesting $L$-functions.

Choose a  maximal torus and Borel subgroup $T\subset B\subset G$ and denote the Weyl group $W(G,T)$ by $W$.
Denote $R=\cohog{*}{\BB T}$, which is canonically isomorphic to the graded ring $\Sym(\xch(T)_{\Qlbar}(-1))$, with the grading such that $\xch(T)_{\Qlbar}(-1)$ is concentrated in degree $2$. The {\em Gross motive of $G$} (or rather, its $\ell$-adic realization) from \S\ref{sss:motive} is the graded $\Ql$-vector space with $\Frob$-action
\begin{equation}
    \VV_G\simeq R^{W}_+/(R^{W}_+)^2
\end{equation}
where $R^{W}_+ \subset R^W$ is the augmentation ideal. Let $(d_1, \ldots, d_n)$ be the multiset of degrees of a homogeneous basis $f_1,\cdots, f_n$ of $\VV_G$. When $G$ is almost simple, the multiset $(d_1, \ldots, d_n)$ coincides with the multiset $\{e_{i}+1\}$, where $\{e_{i}\}$ are the exponents of $G$. We define the associated multivariate $L$-function to be\footnote{In the non-semisimple case the analogous definition needs to be slightly renormalized -- see \eqref{eq: L with pole}.}
\begin{equation*}
\sL_{X, G}(s_{1},\cdots, s_{n}) :=\prod_{i=1}^n \z_{X}(s_{i}+d_{i}).
\end{equation*}


For $j \in \{1, 2, \ldots, r\}$, let $W_{\mu_j}$ be the Weyl group of the Levi $L_{\mu_j}$ of $P_{\mu_j}$. Via the pull-backs from the maps $\BB T\to\BB P_{\mu_j} \to\BB G$, we have
 $$
 R^{W} \simeq \cohog{*}{\BB G} \subset \cohog{*}{\BB P_{\mu_j}}\simeq R^{W_{\mu_j}}.
$$
For each $\y_j\in R^{W_{\mu_j}}$ of degree $2(D_{\mu_j}+1)$, we define in \eqref{eq: nabla lambda eta}
an endomorphism $\nb^{\y_j}_{\mu_j}: R^W\to R^W$. It is a derivation that carries $R^{W}_{+}$ to $R^{W}_{+}$, and hence induces a graded  linear endomorphism of $\VV_G$:
\begin{equation*}
\ov\nb^{\y_j}_{\mu_j}\in \End^{gr}(\VV_G).
\end{equation*}

Now let $\mu  = (\mu_1, \ldots, \mu_r)$ and $\eta = (\eta_1, \ldots, \eta_r)$ be $r$-tuples as above.
We will assume (Commutativity Assumption \ref{assump:split-operators-commute}) that the operators $\ov\nb^{\y_j}_{ \mu_j}$ pairwise commute for $j=1, \ldots, r$.
The assumption holds automatically if the $d_i$'s are all distinct, which holds for all simple groups except those of type $D_{2n}$, and also holds for $\GL_n$. We may choose the homogeneous basis $f_1,\cdots, f_n$ of $\VV_G$ to consist of (generalized) eigenvectors of $\ov\nb^{\y_j}_{\mu_j}$ for each $j\in \{1,\cdots,r\}$. Let $\e_i(\y_j, \mu_j)$ be the generalized eigenvalue of $\ov\nb^{\y_j}_{\mu_j}$ on $f_i$. 
Consider the differential operator
\begin{equation*}
\frd_{j}:=-(\log q)^{-1}\sum_{i=1}^{n}\e_{i}(\y_{j},  \mu_j)\pl_{s_{i}}.
\end{equation*}

Some special cases of our main result for split semisimple $G$ are as follows. First we consider tautological classes which come wholly from $\cohog{2(D_{\mu_j}+1)}{\BB P_{\mu_j}}$. 
 
\begin{theorem}\label{th: main}
Let $\mu=(\mu_{1},\cdots, \mu_{r})$ be an admissible sequence of minuscule dominant coweights of $G$. Let $\eta = (\eta_1, \ldots, \eta_r)$, where $\y_{j}\in \cohog{2(D_{\mu_j}+1)}{\BB P_{\mu_j}}$ satisfying the Commutativity Assumption \ref{assump:split-operators-commute}. Then, with the notations introduced above, we have 
\begin{equation}\label{eq:vol eta}
\vol(\Sht^{\mu}_{G}, \eta) = \# \pi_1(G) \cdot q^{\dim \Bun_{G}}\left(\prod_{j=1}^{r}\frd_{j}\right)\sL_{X,G}(s_{1},\cdots, s_{n})\Big|_{s_{1}=s_{2}=\cdots=s_{n}=0}.
\end{equation}
\end{theorem}

This theorem computes an analog of arithmetic volumes for (integral models of) Shimura varieties mentioned earlier. It relates the arithmetic volume of $\Sht_G^{\mu}$ to higher derivatives of $L$-functions. We also prove a formula that may be viewed as an analog of the Hirzebruch Proportionality principle \eqref{eq: H prop} for locally symmetric spaces. It is obtained by looking at tautological classes with ``maximal contribution from the curve direction''. 

\begin{theorem}\label{thm:intro-2}Let $\mu=(\mu_{1},\cdots, \mu_{r})$ be an admissible sequence of minuscule dominant coweights of $G$. Let $\y'_{j}\in \cohog{2D_{\mu_j}}{\BB P_{\mu_j}}$, and $\xi \in \cohog{2}{X}$ be the class of a point. Let $\eta = (\y'_1 \xi, \ldots, \y'_r \xi)$. Then with the notations above, we have 
\begin{equation}\label{eq:vol eta'}
\vol(\Sht^{\mu}_{G} ,\eta) = \# \pi_1(G) \cdot q^{\dim \Bun_{G}}\sL_{X,G}(0,\cdots, 0)\left( \prod_{j=1}^r\int_{G/P_{\mu_j}} \eta_j'\right).
\end{equation}
\end{theorem}

Theorems \ref{th: main} and \ref{thm:intro-2} represent two extremes of tautological classes: those with ``minimal'' and ``maximal'' contribution from the cohomology of $X$. In general, we calculate $
\vol(\Sht^{\mu}_{G} , \eta)$ for \emph{all} tautological classes in the relevant degree (Theorem \ref{th:vol gen} for split $G$, and Theorem \ref{thm:non-split-volume} for quasisplit $G/X$). We give a concrete illustration of Theorem \ref{thm:non-split-volume} without spelling out all the definitions, deferring the full formulation to \S \ref{ssec:unitary-groups}. 

\begin{theorem}\label{th:intro-vol-U}
Let $\Sht^r_{\U(n)}$ be the moduli stack of Hermitian shtukas with respect to a finite \'etale double cover $X'/X$, and let $\cP_1, \ldots, \cP_r$ be the tautological line bundles on $\Sht^r_{\U(n)}$. Let $\cE$ be a vector bundle of rank $n$ on $X'$. With respect to these choices,
\begin{equation*}
\vol(\Sht_{\U(n)}^r, \prod_{j=1}^{r}c_{n}(p_{j}^{*}\cE^{*}\ot\cP_j)) =2\frac{q^{n^{2}(g-1)}}{(\log q)^{r}}\left(\frac{d}{ds}\right)^{r}\Big|_{s=0}(q^{-s\deg \cE}L_{X,\U(n)}(2s)).
\end{equation*}
\end{theorem}

\begin{remark}
    One initial motivation for this work was to investigate the \emph{singular} terms in a conjectural Higher Siegel--Weil formula, generalizing \cite{FYZ} which proved such a formula for the ``non-singular terms''. On the other end of the spectrum, the ``most singular term'', contributing to the constant Fourier coefficient, is the arithmetic volume of certain tautological Chern classes, and Theorem \ref{th:intro-vol-U} should provide the comparison to the corresponding piece of the higher derivative of the Eisenstein series. In fact, our method was already used in the paper \cite{FHM25} which extends the higher Siegel--Weil formula to the corank-1 case, the ``least degenerate'' of the singular terms. Concretely, \cite[\S 7]{FHM25} can be viewed as specializing our proof of Theorem \ref{th: vol U} to the case $n=1$. 
\end{remark}

\subsubsection{The case $r=0$}
When $r = 0$, $\mu$ is empty and $\Sht_G^\mu$ identifies with the groupoid $\Bun_G(\F_q)$ viewed as a discrete stack. In this case, our main result specializes to the Tamagawa number formula of Gaitsgory--Lurie \cite{GL14}\footnote{They assume a simply connected hypothesis, but it is easy to deduce the general semisimple case from this.}, generalizing earlier work such as \cite{BD09}. Our proof of this special case is the same as that of \cite{GL14}: both are based on the Atiyah--Bott description for the cohomology of $\Bun_G$. We prove this description for general reductive group schemes over $X$, using the simply connected case proved in \cite{GL14} as input.

\subsubsection{Calculation of eigenweights}

To apply Theorem \ref{th: main} in practice, we need to calculate the eigenweights $\epsilon_i(\eta_j, \mu_j)$. 
In \S\ref{sec:more-examples}, we work out the eigenweights in many examples for $G$ of Type $A,B,D$. More precisely, we obtain the eigenweights when $G=\GL_n$ or the special orthogonal 
groups, and $\mu_i$ are all special minuscule coweights corresponding to the standard representation  of the Langlands dual group of $G$. For $G=\GL_n$ we are also able to calculate the eigenweights for $\mu=(1,1,0,\cdots,0)$, for which the answer takes a much more complicated shape, see Proposition \ref{prop: GL len=2}.\footnote{This result was first found by a brutal calculation of the authors, and then proved by a more elegant argument found by Gemini Deep Think (\emph{IMO Gold version}); we will present Gemini's argument instead of our original one.} The determination of the remaining minuscule coweights in Type A, as well as the spin coweights in Types C and D  will be explained in the separate paper \cite{FAI}, using different methods (also devised by AI). Finally, a work-in-progress of Zeyu Wang will generalize the calculation of both the arithmetic volume and eigenweights to all (not necessarily minuscule) coweights in all types.

\subsection{Tautological rings}
It is natural to consider the cohomology classes on $\Sht^\mu_G$ obtained by pulling back along the enhanced Hodge morphism \eqref{eq:enhanced-hodge}
\begin{equation*}
  (\ev_{j},  p_j): \Sht^\mu_G\to \BB P_{\mu_j} \times X
\end{equation*}
 for $1\le j\le r$. From this we obtain a ring homomorphism
\begin{equation*}
    \wt\rho:\wt C^\mu_G:=\bigotimes_{j=1}^r(\cohog{*}{X}\ot R^{W_{\mu_j}})\to \cohog{*}{\Sht^\mu_G}.
\end{equation*}
The {\em tautological ring} for $\Sht^\mu_G$ is the image of $\wt\rho$. When $G$ is a constant reductive group scheme over $X$, we define an ideal $I_G^\mu$ with explicit generators given by \eqref{rel in C}, and we prove in Theorem \ref{th:taut hom quot} that the pull-back map $\wt\rho$ factors through the quotient ring $C^\mu_G:=\wt C^\mu_G/I_G^\mu$, which we call the {\em phantom} tautological ring.   

\subsubsection{Structure of the phantom tautological ring}
The phantom tautological ring has nice properties. In Proposition \ref{p:taut ring free over HX} we show that the ring $C^\mu_G$ is a flat deformation of $\otimes_{j=1}^r \cohog{*}{G/P_{\mu_j}}$ over the Artinian ring $\cohog{*}{X^r}$. As a corollary, the top degree of $C^\mu_G$ is canonically isomorphic to the top degree cohomology of $\prod_{j=1}^r (X\times G/P_{\mu_j})$,
 $$
(C^\mu_G)_{2N}\isom \bigotimes_{j=1}^r(\cohog{2}{X}\ot \cohog{2D_{\mu_j}}{G/P_{\mu_j}}),
$$
where $N=\sum_{j=1}^{r}(D_{\mu_j}+1)=\dim \Sht^\mu_G$.
With this isomorphism, the pairing with the fundamental class of $\prod_{j=1}^r (X\times G/P_{\mu_j})$ defines a volume functional on $(C^\mu_G)_{2N}$, denoted by $\int_{\prod_{j=1}^r (X\times G/P_{\mu_j})}$. 
On the other hand, we also have the volume functional  $\vol(\Sht_G^\mu, -)$ defined earlier by \eqref{def vol Sht} (applicable to all elements in $(\wt C^\mu_G)_{2N}$). Proposition \ref{p:vol factor thru C} shows that  $\vol(\Sht_G^\mu, -)$  factors through the quotient $(C^\mu_G)_{2N}$ of $(\wt C^\mu_G)_{2N}$.
Then a special case of our main result (i.e., \eqref{eq:vol eta'}) compares these two linear functionals on $(C^\mu_G)_{2N}$ (see Proposition \ref{p:vol on taut}) 
    \begin{equation*}
        \vol(\Sht_G^\mu, -) =\# \pi_1(G) q^{\dim \Bun_G}\prod_{i=1}^n\z_X(d_i)\cdot \int_{\prod_{j=1}^r (X\times G/P_{\mu_j})}(-).
    \end{equation*}
     This may be viewed as the analogue of Hirzebruch Proportionality for the moduli stack of shtukas. 
     
\begin{remark}\label{rem hk=1}
Due to the non-properness of $\Sht^\mu_G$, the induced map $\rho:C^\mu_G  \to \cohog{*}{\Sht^\mu_G}$  is not injective (e.g., the top degree map vanishes). Hence the tautological ring of $\Sht^\mu_G$ is a further quotient of $C^\mu_G$. However, we have reasons to believe that $C^\mu_G$ should be viewed as the ``correct'' tautological ring for $\Sht^\mu_G$, that ``amends" the non-properness of $\Sht^\mu_G$. In fact, we show that $C^\mu_G$ satisfies Poincar\'e duality under the volume form and hence carries the structure of a Frobenius algebra (Proposition \ref{p:taut duality}).
We expect that the image of $C^\mu_G$ is large enough to account for the Hecke (generalized) eigenspace for the most non-tempered representation, namely the trivial representation ${\bf 1}$, in the cohomology:
$$
C^\mu_G /\ker(\rho) \simeq \cohog{*}{\Sht^\mu_G}[{\bf 1}].
$$
\end{remark}

\subsubsection{The tautological ring of Shimura varieties} Our results should also give a conceptual framework for the tautological ring of Shimura varieties (i.e., the subring of the Chow ring generated by the Chern classes of all automorphic bundles), a topic which has seen much study, for example in \cite{vdG,EH,WZ}. In \cite{EH}, Esnault--Harris proved that the positive degree $\ell$-adic Chern classes of flat automorphic bundles  vanish and conjectured that the same vanishing result should hold in Chow groups. We prove an analog of this result as a Corollary \ref{c:taut hom gen} to  Theorem \ref{th:taut hom quot}, where the $\ell$-adic Chern classes of ``flat automorphic bundle" in \cite{EH} correspond to the pull-back from $\ev_i$ of classes in $R^W=\cohog{*}{\BB G}$. It is amusing to note that in \cite{EH} the proof uses the Hecke eigen-property of the Chern classes of flat automorphic bundles, while our proof of the function field case avoids the expected Hecke eigen-property in view of Remark \ref{rem hk=1}.

The right analog of the volume given by \eqref{eq:vol eta} in the number field case should be the arithmetic volume
of integral models of Shimura varieties with respect to (the arithmetic Chern classes of) Hermitian metrized bundles. In the case of Siegel moduli space, K\"ohler \cite{Koh} has obtained a proportionality principle, which resembles our Theorem \ref{th:taut hom quot}. However, since the Siegel moduli space is non-proper, it does not have an immediate definition of arithmetic volumes. A conjecture of Maillot and Roessler \cite{MR}
may be viewed as the analogous relation \eqref{rel in C} for certain PEL type Shimura varieties.
In a future work we will pursue the analogous questions for more general Shimura varieties, guided by Theorem \ref{th:taut hom quot}.

\subsection{An analog of the Colmez conjecture}
In \cite{Col1} Colmez conjectured a formula relating the stable Faltings height of an abelian variety with Complex Multiplication (by the ring of integers of a CM field) to the special value at $s=1$ of the logarithmic Artin $L$-function attached to a class function arising from the corresponding CM type. Though the averaged (over all CM type for a fixed CM field) version of the Colmez conjecture has been proved \cite{AGHMP,YZ18}, the general case remains wide open. In \S\ref{s:Col} we obtain a function field analog of the Colmez conjecture. 
Let $\pi: X\to Y$ be a finite Galois \'etale covering with Galois group $\Sig$.
For an $r$-tuple $\s\in\Sig^r= \Sig\times\cdots\times \Sig$, consider the twisted diagonal map: 
\begin{equation*}
  \s= (\s_1,\s_2,\cdots,\s_r): X \to X^r
\end{equation*}
sending $x\in X$ to $(\s_i(x))_{i=1}^r$. We restrict $\Sht^\mu_G$ to the twisted diagonal above:
\begin{equation*}
    \Sht^\mu_{G,\s}:=\Sht^\mu_G\times_{X^r,\s}X.
\end{equation*} 
This may be viewed as an analog of a Shimura variety of Hilbert-Blumenthal type. The element  $\Phi=\sum_{i=1}^r\s_i\in \Qlbar[\Sig]$ in  the group algebra is an analog of the CM type. 
In principle, Theorem  \ref{th:taut hom quot} could allow us to compute the volume of $ \Sht^\mu_{G, \s}$ with respect to any top degree form. In Theorem \ref{th:vol 1-leg PGL}, we treat the case $G=\PGL_n$,  where the volume involves special values of (logarithmic) Artin $L$-functions attached to the CM type. In the simplest case $G=\PGL_2$, the formula takes the following shape (we defer the unexplained notation to \S\ref{s:Col}):
  \begin{eqnarray*}
    \vol(\Sht^\mu_{\PGL_2,\s}, \eta) &=&-\frac{(r+1)!}{3!}q^{3g_X} \z_X(2) \cdot \# \pi_1(G) \cdot  \\ 
   && \left( 2\frac{\z'_X(2)}{\log q\,\z_X(2)}
   +  |\Sig|^2 \frac{L'_{Y, (\Phi\ast \Phi^\vee)^\natural}(2)}{\log q\, L_{Y, (\Phi\ast \Phi^\vee)^\natural}(2)} 
   +(g_Y-1)|\Sig|^2\Big((\Phi\ast\Phi^\vee)(1)-\frac{r}{|\Sig|}\Big)\right).
   \end{eqnarray*}
The recipe of the class function in Colmez conjecture (cf. \cite[\S2]{Col2} for a more explicit formula involving $\Phi\ast\Phi^\vee$) is completely analogous to the one above, except that here we have the special value at $s=2$. 
In a future work we will pursue a similar question for Shimura varieties.
   
\subsection*{Acknowledgments}
We thank Dori Bejleri for asking an inspiring question after the second author talked about formulas for higher derivatives of $L(s,\eta)$ at the University of Maryland in February 2024.

We thank Peter Sarnak for useful comments, and Zeyu Wang for spotting an error in an early draft.

T.F.~was supported by the NSF (grants DMS-2302520 and DMS-2441922), the Simons Foundation, and the Alfred P. Sloan Foundation. Z.Y.~was supported by a Simons Investigator grant. W.Z. ~was supported by NSF DMS-2401548 and a Simons Investigator grant.

\section{Notation}  

\subsection{Cohomology}For an algebraic stack $\cY$ over a field $k$ and a prime number $\ell$ not equal to $\ch(k)$, we follow the formalism of \cite{LiuZheng} for $\ol \Q_\ell$-sheaves on $\cY$. We write $\cD(\cY)$ for the constructible derived category of \'etale $\ol \Q_\ell$-sheaves on $\cY$ in the sense of \emph{loc. cit.}.

When talking about $\ell$-adic cohomology of $\cY$, we always mean \emph{geometric cohomology}, i.e., the cohomology of the base change $\cY_{\ov k}$. We will abbreviate $\rH^*(\cY) := H^*(\cY_{\ol k}; \ol\Q_\ell)$. 

\subsection{The curve} Fix a prime $p$ and finite field $\F_q$ of characteristic $p$. Let $X$ be a smooth, projective, geometrically connected curve over $\F_q$. We set some notation for invariants of $X$. 

\subsubsection{Homology and cohomology of $X$}
Let $\xi\in \cohog{2}{X}(1)$ be the cycle class of a point. Let $[X]\in \homog{2}{X}(-1) \cong \rH^0(X)$ be the fundamental class. We use the notation
\begin{equation*}
\int_{X}: \cohog{2}{X}(1)\to \Qlbar
\end{equation*}
to denote the pairing with $[X]$. The cup product on $\cohog{1}{X}$ gives a symplectic pairing 
\begin{eqnarray*}
\j{-,-}: \cohog{1}{X}\times \cohog{1}{X}&\to& \Qlbar(-1)\\
(\z_{1}, \z_{2})&\mt& \int_{X}\z_{1}\z_{2}. 
\end{eqnarray*}

For $z\in \homog{i}{X}$ we often write its homological degree $i$ as $|z|$; similarly, for $\z\in \cohog{i}{X}$ we write its cohomological degree $i$ as $|\z|$.

We use 
\begin{equation*}
\PD: \homog{*}{X}\isom \cohog{2-*}{X}(1)
\end{equation*}
to denote the Poincar\'e duality isomorphism. In particular, 
\begin{equation}\label{eq:xi}
\PD(1)=\xi \quad \text{and} \quad  \PD([X])=1.
\end{equation}

Writing $\j{-, -}$ also for the evaluation pairing induced by $\rH_i(X) \cong \rH^i(X)^*$, the isomorphism $\PD$ is characterized by the identity
\begin{equation*}
\int_{X}\PD(z)\cdot\z=\j{z,\z}
\end{equation*}
for any $z\in \homog{|z|}{X}$ and $\z\in\cohog{|z|}{X}$. In other words, we have $\j{z, \z} = \j{\PD(z), \z}$, justifying the abuse of notation. 

\subsection{Reductive groups}\label{sssec:reductive-groups-notation}
For a torus $T$, we denote by $\xch(T)$ and $\xcoch(T)$ the character and cocharacter groups of $T$, respectively. For a cocharacter $\lambda \in \xcoch(T)$, we denote by $L_\lambda = C_G(\lambda)$ the corresponding Levi subgroup of $G$, and $P_\lambda$ the attracting parabolic to $L_\lambda$ under the adjoint action of $\G_m$ on $G$ via $\lambda$.

Given a character $\chi$ of $L_\lambda$, the associated line bundle $\cO(\chi)$ on the flag variety $G/P_\lambda$ is $G \times^{P_\lambda} \A^1$ where the action of $P_\lambda$ on $G \times \A^1$ is so that $(gp, t) = (g, \chi(p) t)$. Under this normalization, anti-dominant weights correspond to semiample line bundles on $G/P_\lambda$.

\begin{exam}\label{ex:parabolic-line}
Let $G = \GL_n$, $T$ be the standard diagonal torus, and $\lambda = (1, 0, \ldots, 0)$. Then $P_\lambda$ is the subgroup of matrices of the form 
\[
\begin{pmatrix} 
* & * & \ldots & * \\
0 & * & \ldots & * \\
\vdots & \vdots & \vdots & \vdots \\
0 & * & \ldots & * 
\end{pmatrix} 
\]
and $G/P_\lambda$ identifies with $\PP^{n-1}$. Under our conventions the character $\chi = (1, 0, \ldots, 0) \in \xch(T)$ corresponds to $\cO(-1)$ on $G/P_\lambda$.
\end{exam}

Let $G$ be a split (connected) reductive group over a field $k$. Choose a maximal torus and Borel subgroup $T\subset B\subset G$. This choice induces a notion of dominant coweights in $\xcoch(T)$. Denote the Weyl group $W(G,T)$ by $W$.

Fix a prime $\ell \neq \ch{k}$. Let $V_{T}:=\homog{1}{T,\Qlbar}:=\xcoch(T)_{\Qlbar}(1)$ be the $\Qlbar$ version of the Tate module of $T$. 
Let
\begin{equation*}
R := \Sym(\xch(T)_{\Qlbar}(-1))=\Qlbar[V_{T}]
\end{equation*}
with the grading such that $\xch(T)_{\Qlbar}(-1)$ is concentrated in degree $2$. Then there is a  canonical isomorphism of graded $\Qlbar$-algebras
\begin{equation*}
R\cong \cohog{*}{\BB T}.
\end{equation*}
Under this isomorphism, the subring of $W$-invariants is identified with $\cohog{*}{\BB G}$,
\begin{equation}\label{coho BG}
\cohog{*}{\BB G}\cong R^W =\Qlbar[V_{T}]^{W}.
\end{equation}

\subsection{Graded algebras}\label{sss:aug gr} Let $E$ be a field and  $A$ be an $E$-algebra with an augmentation $\e: A\xr{\e}E$. The \emph{augmentation ideal} is $\ker(\e)$ and its associated $\ker(\e)$-adic filtration is denoted $F_{\aug}^{n}A=\ker(\e)^{n}$. The \emph{associated graded algebra} of $A$ (for its augmentation filtration) is  
\begin{equation*}
\Gr^{\bu}_{\aug}(A)=\op_{n\ge0}\ker(\e)^{n}/\ker(\e)^{n+1}.
\end{equation*}
In case $A$ is $\ZZ_{\ge0}$-graded, and $\e$ is also a graded map, $\Gr^{\bu}_{\aug}(A)$ inherits the $\ZZ_{\ge0}$-grading from $A$, in addition to the natural grading as the associated graded algebra.

The same discussion works more generally in the setting of an augmented algebra object $(A,\e)$ in a symmetric monoidal abelian $E$-linear category.

\section{Characteristic classes under modification}\label{s:char class under mod}

In this section, we prove a general result about the behavior of characteristic classes under ``modifications'' of $G$-bundles on a stack. We fix an algebraically closed base field $k$.

\subsection{Modifications of $G$-bundles}\label{sss:setup mod along div} 
Let $S$ be a stack over $k$ and $\cF_{0}, \cF_{1}$ be two  $G$-torsors over $S$. Let $D\subset S$ be a Cartier divisor. Let 
\begin{equation}\label{modif Gbun}
\ph: \cF_{0}|_{S-D}\isom \cF_{1}|_{S-D}
\end{equation}
be an isomorphism of $G$-bundles. We call $\ph$ a \emph{modification (of the $G$-bundles $\cF_0, \cF_1$) along $D$}. 

For any $V\in\Rep_k(G)$, let $\cF_{i,V}$ be the associated vector bundle $\cF_{i}\times^{G}V$ over $S$. The isomorphism $\ph$ induces an isomorphism  of vector bundles over $S-D$,
\begin{equation*}
\ph_{V}: \cF_{0,V}|_{S-D}\isom \cF_{1,V}|_{S-D}.
\end{equation*}

For a $T$-bundle $\cF$ on $S$ and $\l\in \xcoch(T)$, there is a unique $T$-bundle $\cF(\l D)$ such that for any $\chi\in \xch(T)$, the associated line bundle $\cF(\l D)_{\chi}$ satisfies 
\[
\cF(\l D)_{\chi}  \cong \cF_{\chi}(\j{\chi,\l}D).
\]
This generalizes the construction of twisting a line bundle by a Cartier divisor. 

\begin{defn}\label{defn: modification} Let $\l\in \xcoch(T)$ be a dominant coweight. We say the modification $\ph$ from \eqref{modif Gbun} is {\em of type $\l$} if the following condition holds fppf locally on $S$: there are $T$-reductions $\cF_{0,T}$ and $\cF_{1,T}$ of $\cF_{0}$ and $\cF_{1}$ respectively, such that the map $\ph$ respects the $T$-reductions, and extends to an isomorphism of $T$-bundles over $S$
\begin{equation}\label{phT ext}
\wt \ph_{T}: \cF_{0,T}(\l D)\isom \cF_{1,T}.
\end{equation}
\end{defn}

\begin{exam}\label{ex: GL modif}
For $G=\GL_n$ and a $G$-bundle $\cF$, we consider the rank $n$ vector bundle $\cE:=\cF_V$ associated to the standard representation $V$ of $\GL_n$. Let $\l=(1,0,\cdots, 0)\in \ZZ^n=\xcoch(T)$.  A modification $\ph: \cF_0|_{S-D}\isom \cF_1|_{S-D}$ of $G$-bundles has type $\l$ if the isomorphism $\ph_V: \cE_0|_{S-D}\isom \cE_1|_{S-D}$ (where $\cE_i=\cF_{i,V}$ ) extends to a map of coherent sheaves $\cE_0\to\cE_1$ on $S$ that fits into an exact sequence
\begin{equation}\label{upper 1}
0\to \cE_0 \xr{\ph} \cE_1 \to i_{D*}\cP\to 0
\end{equation}
where $\cP$ is a line bundle on $D$. We call such a modification $\ph$ an ``upper modification of $\cE_0$ of colength $1$ along $D$''. 
\end{exam}

\subsubsection{Canonical parabolic reduction}\label{sssec: canonical parabolic reduction}
Let 
\begin{equation}\label{modif Gbun lambda}
\ph: \cF_{0}|_{S-D}\isom \cF_{1}|_{S-D}
\end{equation}
be a modification along $D$, of type $\lambda$. Let $P_{\l}$ and $P_{-\l}$ be the parabolic subgroups of $G$ defined as the attracting and repelling loci to the Levi $L_{\l}=C_{G}(\l)$, under the adjoint action of $\Gm$ on $G$ via $\l$.

\begin{prop}\label{prop:canonical-parabolic-red} The restriction $\cF_{0}|_{D}$ carries a canonical $P_{\l}$-reduction.
\end{prop}

\begin{proof} By definition, locally for the fppf topology on $S$ we have $T$-reductions $\cF_{0,T}$ and $\cF_{1,T}$ such that $\ph$ respects the $T$-reductions and extends to an isomorphism 
\[
\wt\ph_{T}: \cF_{0,T}(\l D)\isom \cF_{1,T}.
\]

We claim that the $P_{\l}$-reduction $\cP:=(\cF_{0,T}|_{D})\times^{T}P_{\l}$ is independent of choices of the $T$-reductions. To see this, we observe that a $P_\l$-reduction $\cP'$ of $\cF_{0}|_D$ is determined by its adjoint bundle $\Ad(\cP')\subset \Ad(\cF_0)|_D$. Now $\Ad(\cP)\subset \Ad(\cF_0)|_D$ is equal to the image of the vector bundle map over $D$
\begin{equation*}
(\Ad(\cF_{0})\cap \Ad(\cF_1))|_{D}\to \Ad(\cF_{0})|_{D}.
\end{equation*}
Here, $\Ad(\cF_{0})\cap \Ad(\cF_1)$ denotes the preimage of $\Ad(\cF_1)$ under the map $\Ad(\cF_0)\to j_*\Ad(\cF_1)$ induced by $\ph$ (where $j:S-D\incl S$).

This characterization of $\cP$ proves its independence of the choice of the $T$-reduction of $\cF_0$, which allows us to descend it from a fppf cover of $D$, as desired. 

\end{proof}
  
We thus obtain a canonical $P_{\l}$-reduction of $\cF_{0}|_{D}$, which we denote by $\cF_{0}|_{D, P_{\l}}$. Similarly, the restriction $\cF_{1}|_{D}$ carries a canonical $P_{-\l}$-reduction that we denote by $\cF_{1}|_{D, P_{-\l}}$. 

Denote the induced $L_{\l}$-torsor of $\cF_{i}|_{D,P_{\pm \l}}$ by $\cF_{i}|_{D,L_{\l}}$. Then there is a canonical isomorphism of $L_{\l}$-torsors 
\begin{equation}\label{psi D}
\psi_{D}: \cF_{0}|_{D,L_{\l}}(\l D)\cong \cF_{1}|_{D,L_{\l}}
\end{equation}
where, observing that $\lambda$ factors tautologically through $Z(L_\lambda)$, $\cE(\l D)$ denotes the \emph{central twist} of an $L_{\l}$-bundle $\cE$ by $D$. Explicitly, $\cE(\l D)$ can be described as the unique $L_{\l}$-bundle such that for any representation $V$ of $L_{\l}$ where the center $Z(L_{\l})$ acts through a character $\om:Z(L_{\l})\to \Gm$, the associated bundle $\cE(\l D)_{V}=\cE_{V}(\j{\om,\l}D)$, where $\j{\om, \l}\in \ZZ=\End(\Gm)$ is the composition 
\[
\Gm\xr{\l}Z(L_{\l})\xr{\om}\Gm.
\]
This definition is evidently consistent with  Definition \ref{defn: modification}. 

\begin{exam}\label{ex: GL par red}
We continue with Example \ref{ex: GL modif} where $\cE_0\incl \cE_1$ is an upper modification of colength $1$ along $D$. Restricting \eqref{upper 1} to $D$, we get a map 
\begin{equation*}
\ph_V|_{D}: \cE_{0}|_{D}\to \cE_{1}|_{D}
\end{equation*}
of constant rank $n-1$ between rank $n$ bundles over $D$. Denote its image by $\cH\subset \cF_{1,V}|_{D}$, a rank $n-1$ subbundle. This gives the $P_{-\l}$-reduction of $\cF_{1}|_{D}$. Similarly, the $P_{\l}$-reduction of $\cF_{0}|_{D}$ is given by the line subbundle $\ker(\ph|_{D})$, which is isomorphic to $\cP\ot \cO(-D)|_D$ with $\cP$ as in \eqref{upper 1}.
\end{exam}

\subsection{Formulation of the theorem}
Recall from \eqref{coho BG} that characteristic classes of $G$-bundles are given by $W$-invariant polynomial functions on $V_{T}$.  In particular, for $f\in R^{W}=\Qlbar[V_{T}]^{W}$ and a $G$-torsor $\cF$ on $S$ corresponding to $b_{\cF}: S\to \BB G$, we denote its $f$-characteristic class to be
\begin{equation*}
f(\cF):=b_{\cF}^{*}f\in \cohog{*}{S}.
\end{equation*}

Let $W_{\lambda}$ be the Weyl group of $L_\lambda$, viewed as a subgroup of $W$.  A coweight $\l\in \xcoch(T)$ can be viewed as a vector in $V_{T}(-1)$. It thus makes sense to define the partial derivative
\begin{equation*}
\pl_{\l}f\in \Qlbar[V_{T}]^{W_{\lambda}}(-1)=R^{W_{\lambda}}(-1).
\end{equation*}
Similarly we define higher derivatives in the direction of $\l$:
\begin{equation*}
\pl^{n}_{\l}f\in R^{W_{\lambda}}(-n).
\end{equation*}
Since $\cohog{*}{\BB L_{\l}}=\Qlbar[V_{T}]^{W_{\lambda}}$, we may view $\pl^{n}_{\l}f$ as an element in $\cohog{*}{\BB L_{\l}}(-n)$.

\begin{theorem}\label{th:master gen} Consider a modification of $G$-bundles 
$\ph: \cF_{0}|_{S-D}\isom \cF_{1}|_{S-D}$
of type $\l$. Let 
\[
\nu_{D}:=c_{1}(\cO(D))|_{D}\in \cohog{2}{D}(1)
\]
be the Chern class of the normal bundle of $D$, and 
\[
i_{D!}: \cohog{*}{D}(-1)\to \cohog{*+2}{S}
\]
be the Gysin map for the regular embedding $i_{D}: D\incl S$.\footnote{It is induced from the canonical map $\Qlbar[-2](-1)\to i^{!}_{D}\Qlbar$ by taking global sections.} 

Then for any characteristic class $f\in \cohog{*}{\BB G}=R^{W}$, we have 
\begin{equation}\label{diff char classes}
f(\cF_{1})-f(\cF_{0})=i_{D!}\left(\sum_{n\ge1}\frac{1}{n!}(\pl^{n}_{\l}f)(\cF_{0}|_{D, L_{\l}})\cdot \nu_{D}^{n-1}\right)\in \cohog{*}{S}.
\end{equation}
Note that $\pl^{n}_{\l}f$ can be viewed as a characteristic class for $L_{\l}$-torsors (up to a Tate twist), hence $(\pl^{n}_{\l}f)(\cF_{0}|_{D, L_{\l}})$ makes sense as an element in $\cohog{*}{D}(-n)$.
\end{theorem} 


\begin{exam}\label{ex: GL} 
We continue with Example \ref{ex: GL modif} and \ref{ex: GL par red}. Identify $R$ with $\Qlbar[x_1,\cdots, x_n]$, with the action of $S_n=W$ permuting the variables. 
Consider $f:=e_{i}(x_{1},\cdots, x_{n})\in R^W$, the degree $i$ elementary symmetric polynomial in $x_{1},\cdots, x_{n}$, so that $f(\cF)=c_{i}(\cE)$ for rank $n$ vector bundles $\cE$. We have
\begin{equation*}
\pl_{\l}e_{i}(x_{1},\cdots, x_{n})=\pl_{x_{1}}e_{i}(x_{1},\cdots, x_{n})=e_{i-1}(x_{2},\cdots, x_{n}).
\end{equation*}
The higher derivatives $\pl^{>1}_{x_{1}}e_{i}(x_{1},\cdots, x_{n})$ are zero. According to Example \ref{ex: GL par red}, the induced $L_\l=\Gm\times \GL_{n-1}$ bundle from the $P_\l$-reduction of $\cF_0|_D$ is the pair $(\cP\ot \cO(-D)|_D, \cH)$. In this case, Theorem \ref{th:master gen} reads
\begin{equation}\label{Chern class diff}
c_{i}(\cF_{1})-c_{i}(\cF_{0})=i_{D!}c_{i-1}(\cH), \quad 1\le i\le n.
\end{equation}

\end{exam}

\begin{exam}[Pfaffians] Consider the case $G=\SO_{2m}$ and $f\in \cohog{*}{\BB G}=\Qlbar[x_{1},\cdots, x_{m}]^{W}$ is the Pfaffian $f=x_{1}\cdots x_{m}$.  (See \S\ref{ss: H BG SO} for more on the group theoretical description.)

Consider $\l=(1,0,\cdots, 0)$ and an isomorphism $\ph: \cF_{0}|_{S-D}\isom \cF_{1}|_{S-D}$ of $G$-torsors of type $\l$. Let $V$ be the standard representation of $G$. Using $\ph$ to view $\cF_{0,V}, \cF_{1,V}$ as being subsheaves of their identified rationalizations, define the coherent sheaf $\cF^{\flat}_{1/2}: =\cF_{0,V}\cap \cF_{1,V}$. The image of $\cF^{\flat}_{1/2}|_{D}\to \cF_{1,V}|_{D}$ gives a co-isotropic hyperplane $\cH\subset \cF_{1,V}|_{D}$, hence corresponds to an $\SO_{2m-2}$-bundle $\cH/\cH^{\bot}$ over $D$. 

Note that
\begin{equation*}
\pl_{\l}f=\pl_{x_{1}}(x_{1}\cdots x_{m})=x_{2}\cdots x_{m}
\end{equation*}
is the Pfaffian for $\SO_{2m-2}$. The higher derivatives $\pl^{>1}_{x_{1}}(x_{1}\cdots x_{m})=0$. In this case, Theorem \ref{th:master gen} reads
\begin{equation*}
\Pf(\cF_{1})-\Pf(\cF_{0})=i_{D!}\Pf(\cH/\cH^{\bot}).
\end{equation*}
\end{exam}

\subsection{Proof of Theorem \ref{th:master gen}} We prove Theorem \ref{th:master gen} using a series of reductions.

\subsubsection{Product compatibility}\label{sss:prod} We show that if Theorem \ref{th:master gen} holds for two elements $f, g\in R^{W}$, then it holds for their product $fg$.  

Abbreviate $\cH=\cF_{0}|_{D, L_{\l}}$. By assumption, we have
\begin{align}\label{fg der}
fg(\cF_{1})-fg(\cF_{0})&=\left(f(\cF_{0})+\sum_{n\ge1}i_{D!}(\frac{1}{n!}\pl_{\l}^{n}f(\cH)\cdot \nu_{D}^{n-1})\right)\left(g(\cF_{0})+\sum_{n\ge1}i_{D!}(\frac{1}{n!}\pl_{\l}^{n}g(\cH)\cdot \nu_{D}^{n-1})\right) \\
&\hspace{1cm} -f(\cF_{0})g(\cF_{0}). \nonumber
\end{align}
Expanding this expression, we encounter two types of terms:
\begin{enumerate}
\item $f(\cF_{0})i_{D!}(h)$ for some $h\in \cohog{*}{D}$. We use projection formula to rewrite it as
\begin{equation*}
f(\cF_{0})i_{D!}(h)=i_{D!}(i_{D}^{*}f(\cF_{0})\cdot h).
\end{equation*}
Note that $i_{D}^{*}f(\cF_{0})=f(\cF_{0}|_{D})=f(\cH)$, we then have
\begin{equation*}
f(\cF_{0})i_{D!}(h)=i_{D!}(f(\cH)\cdot h).
\end{equation*}
Similarly we have
\begin{equation*}
i_{D!}(h)g(\cF_{0})=i_{D!}( h\cdot g(\cH)).
\end{equation*}

\item $i_{D!}(h_{1}\cdot \nu^{n_{1}-1}_{D})\cdot i_{D!}(h_{2} \cdot \nu^{n_{2}-1}_{D})$, where $h_{1},h_{2}\in \cohog{*}{D}$. Again using the projection formula, and the identity $i_{D}^{*}i_{D!}(-)=(-)\cdot \nu_{D}$, we have
\begin{equation*}
i_{D!}(h_{1}\nu^{n_{1}-1}_{D})\cdot i_{D!}(h_{2}\nu^{n_{2}-1}_{D})=i_{D!}(h_{1}h_{2}\nu^{n_{1}+n_{2}-1}_D).
\end{equation*}

\end{enumerate}
Using these observations, the expansion of \eqref{fg der} can be written as
\begin{equation*}
\sum_{n\ge1}i_{D!}\left(\left(\sum_{n_{1}+n_{2}=n}\frac{1}{n_{1}!}\pl_{\l}^{n_{1}}f(\cH)\frac{1}{n_{2}!}\pl_{\l}^{n_{2}}g(\cH)\right)\cdot \nu_{D}^{n-1}\right)
\end{equation*}
and we recognize the inner summation gives $\frac{1}{n!}\pl_{\l}^{n}(fg)(\cH)$ by the Leibniz rule. This shows that 
\[
fg(\cF_1) - fg (\cF_0) = i_{D!} \left( \sum_{n \geq 1} \frac{1}{n!}\pl_{\l}^{n}(fg)(\cH) \cdot \nu_D^{n-1} \right) ,
\]
as desired. \qed

\begin{cor}\label{c:prod} If Theorem \ref{th:master gen} holds for two reductive groups $G_{1}$ and $G_{2}$, then it holds for $G=G_{1}\times G_{2}$. 
\end{cor}
\begin{proof} Since $\cohog{*}{\BB G}\cong \cohog{*}{\BB G_{1}}\ot \cohog{*}{\BB G_{2}}$, it suffices to treat elements of the form $f_{1}\ot f_{2}=(f_{1}\ot 1) (1\ot f_{2})$, where $f_{i}\in\cohog{*}{\BB G_{i}}$. By the discussion in \S\ref{sss:prod}, it suffices to consider the case of $f_{1}\ot 1$ and $1\ot f_{2}$ separately. The equality \eqref{diff char classes} for $f_{1}\ot 1$ is obtained by pullback from the same equality for the modification of the induced $G_{1}$-bundles for the class $f_{1}$, and similarly for $f_2$. 
\end{proof}

\subsubsection{Reduction: central isogeny}
Let $\th: G\to  G'$ be a central isogeny of connected reductive groups over $k$. We show that if Theorem \ref{th:master gen} holds for $G'$, then it holds for $G$. 

Indeed, if $\ph: \cF_{0}|_{S-D}\isom \cF_{1}|_{S-D}$ is a modification of type $\l$, letting $\cF'_{i}=\cF_{i}\times^{G}G'$, it induces a modification $\ph': \cF'_{0}|_{S-D}\isom \cF'_{1}|_{S-D}$ of $G'$-bundles of type $\l$ (now viewed as a dominant coweight of $T'$, a maximal torus of $G'$ containing $\th(T)$). The pullback map $\th^{*}: \cohog{*}{\BB G'}\to \cohog{*}{\BB G}$ is an isomorphism, so we may view $f\in R^{W}$ also as a class  $f\in \cohog{*}{\BB G'}$, and we have $f(\cF_{i})=f(\cF'_{i})\in \cohog{*}{S}$, and $\pl^{n}_{\l}f(\cF_{i}|_{D,L_{\l}})=\pl^{n}_{\l}f(\cF'_{i}|_{D,L'_{\l}})\in \cohog{*}{D}$. Formula \eqref{diff char classes} now holds for $f$ and the modification $\ph'$, which can then be viewed as the same formula for $f$ and $\ph$.

This reduction allows us to reduce to the case where $G$ is a product of an adjoint group and a torus. In view of Corollary \ref{c:prod}, we further reduce the statement to the case where $G$ is either simple and adjoint, or $G=\Gm $.

\subsubsection{Case $G = \Gm$}\label{sss:master Gm}
In this case, $\cF_{0}$ and $\cF_{1}$ are line bundles over $S$, and $\l\in \xcoch(\Gm)=\ZZ$ is an integer. A modification $\ph: \cF_{0}|_{S-D}\isom \cF_{1}|_{S-D}$ is of type $\l$ if and only if $\ph$ extends to an isomorphism $\cF_{0}(\l D) \cong \cF_{1}$. Now $\cohog{*}{\BB\Gm}=\Qlbar[c_1]$, where $c_1\in \cohog{2}{\BB\Gm}$ is the first Chern class of the universal line bundle on $\BB\Gm$. In view of \S\ref{sss:prod}, it suffices to prove \eqref{diff char classes} for $f=c_1$. Then 
$$f(\cF_{1})-f(\cF_{0})=c_{1}(\cF_{1})-c_{1}(\cF_{0})=c_{1}(\cO(\l D))=i_{D!}(\l)$$
which gives \eqref{diff char classes} since $\pl_{\l}c_1 =\l\in \cohog{0}{\BB\Gm}$ and higher derivatives of $c_1$ vanish.

It remains to deal with the case where $G$ is of adjoint type, for which we need some preparation.

\subsubsection{Wonderful compactification} Let $G$ be a simple adjoint group. Choose a Borel subgroup $B\subset G$ and a maximal torus $T\subset B$. Let $\{\a_{i}\}_{i\in I}$ be the set of simple roots. Let $G\incl \ov G$ be the wonderful compactification \cite{DP}. The $G\times G$-orbits on $\ov G$ are given by intersections of simple normal crossing divisors $\{D_{i}\}_{i\in I}$ indexed by $I$. Each divisor classifies a map $\ov G \rightarrow [\AA^1/\Gm]$, so all of them together give a smooth map 
\begin{equation}\label{pi from wc}
\pi: \ov G\to  [\AA^{I}/\Gm^{I}].
\end{equation}

For each subset $J\subset I$, let $e_{J}\in \AA^{I}$ be the point with $j$-coordinate $1$ if $j\in J$ and $0$ otherwise. Let $P_{J}$ be the standard parabolic subgroup of $G$ whose Levi $L_{J}$ containing $T$ has simple roots $\{\a_{j}\co j\in J\}$. Let $P^{-}_{J}$ be the parabolic subgroup opposite to $P_{J}$ containing $L_{J}$. Then there is a canonical $G\times G$-equivariant isomorphism
\begin{equation}\label{J fiber}
\pi^{-1}(e_{J})\cong (G\times G)/(P_{J}\times_{L_{J}}P^{-}_{J}).
\end{equation}

Let $\l\in \xcoch(T)$ be a dominant coweight. Then it extends to a map $a_{\l}: \AA^{1}\to \AA^{I}$. Base changing \eqref{pi from wc} along $a_{\l}: [\AA^{1}/\Gm]\to [\AA^{I}/\Gm^{I}]$, we get a map
\begin{equation}\label{eq:pi_lambda}
\pi_{\l}: \ov G_{\l}:=\ov G\times_{[\AA^{I}/\Gm^{I}], a_{\l}}[\AA^{1}/\Gm]\to [\AA^{1}/\Gm]
\end{equation}
whose fiber over $1$ is $G$ and fiber over $0$ is $\pi^{-1}(e_{J})$, where $J:=\{i\in I|\j{\a_{i},\l}=0\}$. The Levi subgroup $L_{J}$ is the fixed point subgroup $L_{\l}$ of the adjoint action of $\l$, and the parabolic subgroup $P_{J}=P_{\l}$ (resp. $P^{-}_{J}=P_{-\l}$) is the attracting (resp. repelling) locus of $L_{\l}$ under this action. Hence we can rewrite \eqref{J fiber} as a $G\times G$-equivariant isomorphism
\begin{equation*}
\wt D_{\l}:=\pi^{-1}_{\l}(0)\cong G\times G/(P_{\l}\times_{L_{\l}} P_{-\l}).
\end{equation*}
The preimage $D_\lambda := \pi^{-1}_{\l}(\{0\}/\Gm)=\ov G_{\l}-G \subset \ov G_{\l}$ is a smooth divisor of $\ov G_{\l}$. Let $L^{\flat}_{\l}:=L_{\l}/\Gm$ where $\Gm$ acts by multiplication via $\l:\Gm\to Z(L_{\l})\subset T$ (which is not necessarily injective, so $L^{\flat}_{\l}$ may be a group stack). We can describe the $G\times G$-variety  
$D_{\l}$ as 
\begin{equation}\label{eq: D_lambda}
D_{\l}\cong (G\times G)/(P_{\l}\times_{L^{\flat}_{\l}} P_{-\l}).
\end{equation}

\begin{lemma}\label{l:map to Vin} For a modification of $G$-bundles $\cF_{0}|_{S-D}\isom \cF_{1}|_{S-D}$ as in \S\ref{sss:setup mod along div}, there is a canonical map 
\begin{equation*}
h: S\to [\ov G_{\l}/(G\times G)]
\end{equation*}
with the following additional structures:
\begin{enumerate}
\item The composition $S\to [\ov G_{\l}/(G\times G)]\to \BB G\times \BB G$ is classified by the $G$-bundles $(\cF_{0},\cF_{1})$.
\item The composition $S\to [\ov G_{\l}/(G\times G)]\xr{\pi_{\l}}[\AA^{1}/\Gm]$, which is the datum of a line bundle $\cL$ on $S$ and a section of $\cL^{-1}$, is given by $\cL=\cO(-D)$ and the tautological section $1\in \cL^{-1}=\cO(D)$.
\item In particular, we have
\begin{equation*}
h|_{D}: D\to [D_{\l}/(G\times G)]\cong \BB(P_{\l}\times_{L^{\flat}_{\l}}P_{-\l})
\end{equation*}
which is the same datum of a $P_{\l}$-bundle $\cP_{0}$ on $D$, a $P_{-\l}$-bundle $\cP_{1}$ on $D$ and an isomorphism $\psi: \cP_{0,L_{\l}}\cong \cP_{1,L_{\l}}$ of their induced $L_{\l}^{\flat}$-bundles. Then referring to the canonical parabolic reductions of \S \ref{sssec: canonical parabolic reduction}, we have a canonical isomorphism of $P_{\l}$-bundles $\cP_{0}\cong \cF_{0}|_{D, P_{\l}}$ and a canonical isomorphism of $P_{-\l}$-bundles $\cP_{1}\cong \cF_{1}|_{D, P_{-\l}}$, under which $\psi$ is the isomorphism induced by $\psi_{D}$ in \eqref{psi D} (noting that the central twisting is trivial under our assumption that $G$ is of adjoint type).
\end{enumerate}
\end{lemma}
\begin{proof} The data of $(\cF_{0}, \cF_{1}, \ph)$ gives a map $h_{0}: S-D\to [G/(G\times G)]$. Since $[\ov G_{\l}/(G\times G)]$ is  separated over $\BB(G\times G)$, the extension of $h_{0}$ to $h: S\to [\ov G_{\l}/(G\times G)]$ is unique if it exists. Therefore it suffices to check the existence of $h$ fppf locally on $S$. 

We may thus assume there are $T$-reductions $\cF_{i,T}$ of $\cF_{i}$ such that $\ph$ respects the $T$-reductions and induces an isomorphism $\wt\ph_{T}$ as in \eqref{phT ext}. This is equivalent to a map $h_{0,T}: S-D\to [T/(T\times T)]$. It is well-known that $\ov G$ contains an open subset $\ov G^{\c}$ isomorphic to $(G\times G)/(B\times_{T}B^{-})$, where $B^{-}$ is the opposite Borel to $B$ containing $T$, and the closure $\ov T^{\c}$ of $T$ in $\ov G^{\c}$ can be identified with $\AA^{I}$ via coordinates given by simple roots, so that the projection 
\[
\pi|_{\ov T^{\c}}: \ov T^{\c}\cong \AA^{I}\to [\AA^{I}/\Gm^{I}]
\]
is the tautological projection.  Let $\ov T^{\c}_{\l}\subset \ov G_{\l}$ be the base change of $\ov T^{\c}$ along $a_{\l}: [\AA^{1}/\Gm]\to [\AA^{I}/\Gm^{I}]$. Then $\ov T^{\c}_{\l}$ contains $T$ as an open subset. Our goal is to extend $h_{0,T}$ to a map
\begin{equation}\label{hT}
h_{T}: S\to [\ov T^{\c}_{\l}/(T\times T)]. 
\end{equation}
A map $S\to [\ov T^{\c}/(T\times T)]$ is the same datum as a collection of effective Cartier divisors $\{D_{i}\}_{i\in I}$ on $S$,  two $T$-bundles $\cE_{0}$ and $\cE_{1}$ and maps $\ph_{i}: \cE_{0, \alpha_i}(D_i) \rightarrow \cE_{1, \alpha_i}$ of coherent sheaves on $S$ for each $i\in I$. Such a map has a canonical factorization into $[\ov T^{\c}_{\l}/(T\times T)]$ if we are given a single effective Cartier divisor $D\subset S$ such that $D_{i}=\j{\a_{i},\l}D$ for all $i\in I$. In other words, a map $h_{T}$ as in \eqref{hT} is the same datum as an effective Cartier divisor $D\subset S$, a pair of  $T$-bundles $\cE_{0}$ and $\cE_{1}$ and maps $\ph_{i}: \cE_{0,\a_{i}}(\j{\a_{i}, \l}D)\to \cE_{1,\a_{i}}$ of coherent sheaves on $S$ for each $i\in I$.  Now our $D$, the $T$-bundles $\cE_{0}=\cF_{0,T}, \cE_{1}=\cF_{1,T}$ and $\ph_{i}: \cF_{0,T}(\j{\a_{i}, \l}D)\to \cF_{1,T}$ induced from the $\wt\ph_{T}$ give such a datum, hence inducing a map $h_{T}$ as in \eqref{hT} extending $h_{0,T}$. This proves the existence of $h$. 

The claimed properties of $h$ then follow easily from the construction. 
\end{proof}


\subsubsection{Adjoint case}
Now we show Theorem \ref{th:master gen} for $G$ adjoint by reducing it to the torus case. By Lemma \ref{l:map to Vin}, it suffices to prove the Theorem for the universal case $S=[\ov G_{\l}/(G\times G)]$, $D=[D_{\l}/(G\times G)]\cong \BB(P_{\l}\times_{L_{\l}}P_{-\l})$, and the canonical modification between the pullbacks of the tautological $G$-bundles via the first and second projection $h_{i}: S\to \BB G$, $i=0,1$. Since the map \eqref{eq:pi_lambda} $\pi_{\l}$ is smooth, $D_{\l}$ is a regularly embedded divisor in $S$, hence by purity we have $i_{D}^{!}\Qlbar\cong \Qlbar[-2](1)$. We thus have an excision distinguished triangle
\begin{equation}\label{excision G lam}
R\G_{G\times G}(D_{\l})[-2](-1)\to R\G_{G\times G}(\ov G_{\l})\to R\G_{G\times G}(G)\to 
\end{equation}

From \eqref{eq: D_lambda}, we see that $R\G_{G\times G}(D_{\l})\cong R\G(\BB L_{\l})$ and $R\G_{G\times G}(G)\cong R\G(\BB G)$ are both concentrated in even degrees, hence the long exact sequence attached to the cohomology groups of \eqref{excision G lam} splits into short exact sequences in each even degree $2n$
\begin{equation}\label{shex G}
0\to \upH^{2n-2}_{G\times G}(D_{\l})(-1)\xr{i_{D_{\l}!}} \upH^{2n}_{G\times G}(\ov G_{\l})\to \upH^{2n}_{G\times G}(G)\to 0
\end{equation}
Let $f\in \cohog{2n}{\BB G}$. We need to show that 
\begin{equation}\label{master eqn univ}
h_{1}^{*}f-h_{0}^{*}f=i_{D_{\l}!}\Big(\sum_{i\ge1}\frac{1}{i!}\pl^{i}_{\l}f\cdot \nu_{D_{\l}}^{i-1}\Big)\in \upH^{2n}_{G\times G}(\ov G_{\l}).
\end{equation}

On the other hand, consider the closure $\ov T_{\l}$ of $T\subset G$ in $\ov G_{\l}$, which is stable under the action of $T\times T$. It is well-known that the closure $\ov T$ of $T$ in $\ov G$ is a smooth toric compactification of $T$, and the projection $\pi|_{\ov T}: \ov T\to [\AA^{I}/\Gm^{I}]$ is also smooth. Therefore $\ov T_{\l}$ is also smooth over $[\AA^{1}/\Gm]$. In particular, the divisor $D_{T,\l}=\ov T_{\l}-T$ is regularly embedded into $\ov T_{\l}$. We have a counterpart of the short exact sequence \eqref{shex G} for $T$
\begin{equation}\label{shex T}
0\to \upH^{2n-2}_{T\times T}(D_{T,\l})(-1)\to \upH^{2n}_{T\times T}(\ov T_{\l})\to \upH^{2n}_{T\times T}(T)\to 0.
\end{equation}
Restriction gives a map from \eqref{shex G} to \eqref{shex T} that is easily seen to be injective on both ends, therefore the restriction map $\upH^{*}_{G\times G}(\ov G_{\l})\to \upH^{*}_{T\times T}(\ov T_{\l})$ is also injective. Therefore, to check  \eqref{master eqn univ}, it suffices to check that its image in $\upH^{2n}_{T\times T}(\ov T_{\l})$ holds. This reduces to the case of the tautological modification between the two universal $T$-bundles on $[\ov T_{\l}/(T\times T)]$. Note that the $D_{T,\l}$ is not necessarily connected but has connected components $D_{T,\l}(\l')$ in bijection with coweights $\l'\in \xcoch(T)$ in the $W$-orbit $W\cdot \l$, such that the modification type of the universal $T$-torsors along $D_{T,\l}(\l')$ is $\l'\in\xcoch(T)$.  The torus case of Theorem \ref{th:master gen} has been proved already (by Corollary \ref{c:prod} and \S\ref{sss:master Gm}). Therefore \eqref{master eqn univ} holds. This finishes the proof of Theorem \ref{th:master gen} in the adjoint case, hence in general. \qed

\section{The Atiyah-Bott formula for reductive group schemes}\label{sec: atiyah-bott}

In this section we work over an algebraically closed base field $k=\ov k$. For a smooth projective curve $X/k$ and reductive group scheme $G \rightarrow X$, we prove a formula for the $\ell$-adic cohomology of the associated stack $\Bun_G$ that resembles the famous formula of Atiyah--Bott \cite{AB83} for Betti cohomology when $X$ is replaced by a Riemann surface, and $G$ is split and semisimple. It has since been generalized to various other settings, but we require greater generality (where $G$ may fail to be semisimple, or split) than previously treated in the literature.

\subsection{The classifying stack}

Let $S$ be a scheme over $k$. Let $G\to S$ be a connected reductive group scheme, i.e., a flat relatively affine group scheme all of whose geometric fibers are connected reductive. 

We have the classifying stack $\pi \co \BB G \rightarrow S$ and the direct image complex $\rR \pi_* (\Ql)$ on $S$.  We first show that, roughly speaking, $\BB G$ only depends on the outer class of $G$.

\subsubsection{Canonical quasisplit form}\label{sssec:canonical-quasisplit-form}
Assume that $S$ is connected. Then there is a unique connected reductive group $\GG$ over $k$ such that each $k$-fiber of $G$ is isomorphic to $\GG$. Also fix a pinning of $\GG$, which identifies $\Out(\GG)=\Aut(\GG)/\GG_{\ad}$ with the subgroup of pinned automorphisms of $\GG$. Let $\cP=\un{\Isom}_{S}(\GG, G)$, which is a right $\Aut(\GG)$-torsor over $S$. Let $\cP_{\Out}=\cP/\GG_{\ad}$ be the induced $\Out(\GG)=\Aut(\GG)/\GG_{\ad}$-torsor. We have a canonical isomorphism over $S$,
\begin{equation}\label{BG}
\BB G\cong \cP\times^{\Aut(\GG)}\BB\GG.
\end{equation}

\begin{defn}\label{defn:quasisplit-form} The \emph{canonical quasisplit} form in the inner class of $G$ is
\begin{equation*}
G_{\qs}:=\cP_{\Out}\times^{\Out(\GG)}\GG.
\end{equation*}
\end{defn}

\begin{lemma}\label{l:BG qs}
Assume that $\GG_{\der}$ is of adjoint form, i.e., $\GG_{\der}\to \GG_{\ad}$ is an isomorphism. Then there is a canonical isomorphism over $S$
\begin{equation*}
\BB G\cong \BB G_{\qs}.
\end{equation*}
Here $\Out(\GG)$ acts on $\GG$ (and hence on $\BB\GG$) via pinned automorphisms.
\end{lemma}
\begin{proof}
The assumption $\GG_{\der} \xrightarrow{\sim} \GG_{\ad}$ implies that $\GG=\GG_{\ad}\times \AA$ for some torus $\AA$. The action of $\GG_{\ad}$ on $\BB\GG=\BB\GG_{\ad}\times\BB\AA$ (induced by the conjugation action on $\GG_{\ad}$) is {\em canonically} trivial. Therefore we may rewrite the right side of \eqref{BG} as
\begin{equation*}
\cP_{\Out}\times^{\Out(\GG)}\BB\GG
\end{equation*}
which is $\BB G_{\qs}$.
\end{proof}

\subsubsection{Comparison under central isogenies}  Let $\theta \colon G \rightarrow G'$ be a central isogeny of reductive group schemes over $S$, which induces a map $\BB\th: \BB G\to \BB G'$.
\[
\begin{tikzcd}
\BB G \ar[rr, "\theta"] \ar[dr, "\pi"'] & & \BB G' \ar[dl, "\pi'"] \\
& S
\end{tikzcd}
\]

\begin{lemma}\label{lem: classifying cohomology pullback}
The pullback map along $\BB\th$ 
\begin{equation*}
\rR\pi'_* \Qll{\BB G'} \rightarrow \rR \pi_* \Qll{\BB G}
\end{equation*}
is an isomorphism. 
\end{lemma}
\begin{proof} It suffices to check on stalks, so we may reduce to the case $S=\Spec k$. Let $\Gamma=\ker(\theta)$.
The presentation $\BB G' = [\BB G/ \BB\Gamma]$ exhibits $\BB G$ as a gerbe over $\BB G'$ banded by the group scheme $\Gamma$. Since $\Gamma$ is finite, pullback induces an isomorphism on cohomology with $\Ql$-linear coefficients for any $\Gamma$-gerbe. 
\end{proof}

Let $\pi^{\qs}: \BB G_{\qs}\to S$ be the structure map. 
\begin{prop}\label{p:inv under isog} There is a canonical isomorphism in $D(S)$
\begin{equation*}
\rR \pi_* \Qll{\BB G}\cong \rR \pi^{\qs}_* \Qll{\BB G_{\qs}}.
\end{equation*}
\end{prop}
\begin{proof}
Let $G'=G_{\ad}\times A$, where $A=G/G_{\der}$ is the quotient torus of $G$. We have maps
\begin{equation*}
\BB G\to\BB G'\cong \BB G'_{\qs}\leftarrow \BB G_{\qs}.
\end{equation*}
Here the middle isomorphism is given by Lemma \ref{l:BG qs}. Both $G\to G'$ and $G_{\qs}\to G'_{\qs}$ are central isogenies, hence by Lemma \ref{lem: classifying cohomology pullback} they induce isomorphisms on the direct images of classifying stacks. Composing these isomorphisms gives the desired canonical isomorphism.
\end{proof}

\subsubsection{Canonical splitting of $\rR \pi_* \Qll{\BB G}$}\label{sss:splitting}
Choose a connected component of $\cP_{\Out}$ and denote it by $S'$. It is a torsor for a finite subgroup $\Out(\GG)^{\na}\subset \Out(\GG)$. Let 
$\nu: S'\to S$ be the projection. Then $G_{\qs}$ becomes the constant group $\GG\times S'$ after base change along $\nu$. Therefore, with $R^W = \RGamma(\BB \GG)$ as in \S\ref{sssec:reductive-groups-notation}, we have a canonical isomorphism
\begin{equation}\label{nu pullback coho BG}
\nu^{*}\rR \pi_* \Qll{\BB G}\cong \nu^{*}\rR \pi_* \Qll{\BB G_{\qs}}\cong \RGamma(\BB \GG)\ot \Qll{S'}=R^{W}\ot \Qll{S'}.
\end{equation}
 The descent datum of the left side is given by the natural $\Out(\GG)^{\na}$-action on $R^{W}$. In particular, we conclude that $\rR \pi_* \Qll{\BB G}$  is canonically isomorphic to the direct sum of its cohomology sheaves
\begin{equation}\label{canon split coho BG}
\rR \pi_* \Qll{\BB G}\cong \bigoplus_{n\ge0}\rR^{2n} \pi_* \Qll{\BB G}[-2n].
\end{equation}
Each local system $\rR^{2n} \pi_* \Qll{\BB G}$ is equipped with an isomorphism 
\begin{equation*}
\nu^{*}\rR^{2n} \pi_* \Qll{\BB G}\cong \upH^{2n}(\BB \GG)\ot\Qll{S'}
\end{equation*}
whose descent datum is the natural $\Out(\GG)^{\na}$-action on the right side.

\subsubsection{The Gross motive}\label{sss:motive}

Using notation from \S\ref{sss:aug gr}, we have the augmentation filtration of $R^{W}=\rR\Gamma(\BB\GG)$, and in particular a graded vector space
\begin{equation}\label{def M}
\VV_{\GG}=\Gr^{1}_{\aug}R^{W}=R^{W}_{+}/(R^{W}_{+})^{2}.
\end{equation}
When $\GG$ is almost simple, the multiset of degrees of $\VV_{\GG}$ (i.e., $d$ repeated $\dim \VV_{\GG, 2d}$ times) coincides with the multiset $\{e_{i}+1\}$, where $\{e_{i}\}$ are the exponents of $\GG$.

Under the isomorphism \eqref{nu pullback coho BG}, the augmentation filtration on $R^{W}$ descends to a filtration $F_{\aug}\rR \pi_* \Ql$ on $\rR \pi_* \Ql$ (by direct summands), and $\VV_{\GG}\ot \Qll{S'}$ descends to
\begin{equation*}
\VV_{G}:=\Gr^1_{\aug} \rR \pi_* (\Ql) \in \cD(S).
\end{equation*}
The canonical splitting \eqref{canon split coho BG} induces a splitting of $\VV_{G}$,
\[
\VV_{G} = \bigoplus_{d\in \ZZ_{>0}} \VV_{G,2d}[-2d]
\]
where each $\VV_{G,2d}$ is a local system on $S$.

\begin{remark}In \cite{Gro97}, B.~Gross introduced the motive $\MM_{G}$ of a reductive group $G$ over any base field $K$. In our setup, if $S=\Spec K$, then $\VV_{G}$ as a $\Gal(K^{s}/K)$-module is the $\ell$-adic realization of $\MM_{G_{K}}(-1)$. Comparing with Gross's definition \cite[(1.5)]{Gro97}, our grading on $\VV_{G}$ has been doubled, and each $\VV_{2d}$ is already Tate twisted by $\Qlbar(-d)$. 
Our $\VV_G$ coincides with what Gaitsgory--Lurie denote $M(G)$ in \cite[Remark 6.4.10]{GL14}.
\end{remark}


\begin{exam} Let $\nu: S'\to S$ be an \'etale double cover. Let $G$ be a unitary group over $S$ defined as the isometries of a $S'/S$-Hermitian rank $n$ vector bundle $\cF$ over $S'$. Let $\cL=(\nu_{*}\cO_{S'})^{\s=-1}$ where $\s:S'\to S'$ is the nontrivial involution over $S$. Then we have a canonical isomorphism 
\begin{equation*}
\VV_{G}\cong \bigoplus_{i=1}^{n}\cL^{\ot i}[-2i].
\end{equation*}
\end{exam}

The following is an immediate consequence of Proposition \ref{p:inv under isog}.

\begin{lemma}\label{l:V under isog}
\begin{enumerate}
\item Let $\th: G\to G'$ be a central isogeny. Then $\th^{*}$ induces a canonical isomorphism
\begin{equation*}
\VV_{G'}\cong \VV_{G}.
\end{equation*}

\item Let $A=G/G_{\der}$ and $\th: G\to G_{\ad}\times A$ be the canonical central isogeny. Then $\th$ induces a canonical isomorphism
\begin{equation*}
\VV_{G}\cong \VV_{G_{\ad}}\op \VV_{A}.
\end{equation*}
Moreover, $\VV_{A}$ is concentrated in degree $2$, and $\VV_{G_{\ad}}$ is concentrated in even degrees $\ge4$.
\end{enumerate}
\end{lemma}

\subsubsection{More on central isogenies}\label{ssec: classifying space preliminaries} 

The rest of the discussion in this subsection is only used in the reduction step \S\ref{sss:AB central isog} in the proof of the Atiyah-Bott formula.

Let $\theta \colon G \rightarrow G'$ be a central isogeny of reductive group schemes over a stack $S$ with kernel $\Gamma$.  Since $\Gamma$ is abelian, its classifying stack $\BB \Gamma$ is an abelian group stack over $S$. Furthermore, since $\Gamma$ is central in $G$, the multiplication map $a \co \Gamma \times_S G \rightarrow G$ is a group homomorphism, and induces a map of classifying stacks 
\begin{equation}\label{eq: action map BG}
a \co \BB \Gamma \times_S \BB G \rightarrow \BB G
\end{equation}
which extends to an action of the abelian group stack $\BB\Gamma$ on $\BB G$. This realizes $\BB G'$ as the quotient stack $\BB G' \cong [\BB G/\BB\Gamma]$.

\begin{lemma}\label{lem: Gamma action on BG}
With $a$ the action map from \eqref{eq: action map BG} and $\pr_{1}: \BB G \times_S \BB \Gamma\to \BB G$ the projection, we have $\pr_1^* = a^*$ as maps 
\[
\rR \pi_* \Qll{\BB G}  \rightarrow \rR \Pi_* \Qll{\BB G \times_S \BB \Gamma},
\]
where $\Pi: \BB G \times_S \BB \Gamma\to S$ is the structure map.
\end{lemma}
\begin{proof}
Since $\Gamma$ is finite, pullback along the neutral section $e \co S \rightarrow \BB \Gamma$ induces an isomorphism on cohomology with $\Ql$-linear coefficients. Hence it suffices to check the statement after pullback along the neutral section $\BB G \xrightarrow{\Id \times e} \BB G \times_S \BB \Gamma$. But the maps $a$ and $\pr_1$ agree when composed with the neutral section, so $a^*$ and $\pr_1^*$ obviously agree after such pullback.
\end{proof}

\subsubsection{$G$-characteristic classes} 
For any map of stacks $S' \rightarrow S$, a $G$-torsor on $S'$ is classified by a map $S' \rightarrow \BB G$. Any cohomology class $c \in \rH^*(\BB G)$ thus gives rise to a system of compatible cohomology classes associated to $G$-torsors on $S' \rightarrow S$, which we call \emph{$G$-characteristic classes}. 

In the situation of \S\ref{ssec: classifying space preliminaries}, let $\cL$ be a $\Gamma$-torsor on $S$. Then $\cL$ can be interpreted as a section of $\BB \Gamma$, and composing it with the action map \eqref{eq: action map BG} gives a map $\BB G \rightarrow \BB G$, which we refer to as ``twisting a $G$-torsor by $\cL$''.

\begin{cor}\label{cor: twisting Chern classes}
The action of twisting by $\cL$ has a trivial effect on the $G$-characteristic classes. 
\end{cor}

\begin{proof}
It suffices to check in the universal case, where the base is $\BB \Gamma \times_S \BB G$. In this case, the assertion amounts to Lemma \ref{lem: Gamma action on BG}. 
\end{proof}

\subsection{Formulation of Atiyah-Bott formula}

Let $X$ be a smooth connected and projective curve over $k$. Let $G\to X$ be a connected reductive group scheme.

Let $\Bun_{G}$ be the moduli stack of $G$-bundles over $X$. Since $\Bun_{G}$ classifies sections to the structure map $\pi: \BB G\to X$, it only depends on $\BB G$.

When $G$ is the constant group $\GG\times X$, we also write $\Bun_{G}$ as $\Bun_{\GG}$.

\subsubsection{Components of $\Bun_{G}$} 


Let $\pi_{0}(\Bun_{G})$ be the set of connected components of $\Bun_{G}$. When $G=\GG\times X$ is constant, $\pi_{0}(\Bun_{\GG})$ can be identified with the algebraic fundamental group $\pi_{1}^{\alg}(\GG)$, which in particular has a group structure. For a maximal torus $\TT\subset \GG$, there is a canonical surjection of abelian groups
\begin{equation}\label{eq: pi0}
\xcoch(\TT)\surj \pi_{0}(\Bun_{\GG})
\end{equation}
whose kernel is the coroot lattice. 

For general $G$, we have a local system of abelian groups over $X$
\begin{equation*}
\pi^{\alg}_{1}(G/X):=\cP_{\Out}\times^{\Out(\GG)}\pi_{1}^{\alg}(\GG).
\end{equation*}
Then there is a canonical isomorphism
\begin{equation}\label{eq:bunG_connected_components}
\pi_{0}(\Bun_{G})\cong \homog{0}{X, \pi^{\alg}_{1}(G/X)}.
\end{equation}
This is a reformulation of \cite[Theorem 2]{Hei10}: the right side above is the same as the coinvariants of $\pi^{\alg}_{1}(G/X)$ under the monodromy action of $\pi_{1}(X)$. 

For $\om\in \pi_{0}(\Bun_{G})$ we denote by $\Bun_G^\omega$ the corresponding 
connected component of $\Bun_G$.

\subsubsection{The evaluation map} The universal $G$-bundle is classified by the evaluation map 
\begin{equation}\label{eq: AB evaluation}
\ev\co X \times \Bun_G \rightarrow \BB G.
\end{equation}
Then \eqref{eq: AB evaluation} induces a map  
\begin{equation}\label{eq: AB evaluation 2}
\rR \pi_* \Qll{\BB G}   \rightarrow \rR \Gamma(\Bun_G) \otimes \Qll{X}.
\end{equation}
Tensoring \eqref{eq: AB evaluation 2} with the dualizing complex $\bD_X$ and taking cohomology, we obtain a map 
\[
\RGamma (X; \rR \pi_* \Qll{\BB G}  \otimes \bD_X) \rightarrow  \RGamma(X; \bD_X) \otimes \RGamma(\Bun_G).
\]
Composing this with the trace map $\RGamma(X, \bD_X) \rightarrow \Ql$, we obtain in particular a canonical map 
\begin{equation}\label{eq: AB map complex}
\RGamma(X; \rR \pi_* \Qll{\BB G} \otimes \bD_X) \rightarrow  \RGamma(\Bun_G).
\end{equation}

For $\cK\in \cD(X)$, the \emph{homology} of $X$ with coefficients in $\cK$ is defined\footnote{A priori, the relative homology of a map $f$ should be defined to be the left adjoint $f_{\sh}$ of $f^*$, if it exists. In this case, Verdier duality identifies $f_{\sh}(\cK) \cong \rR f_{*} (\cK \otimes \bD_X)$.} to be
\[
\rH_*(X; \cK) := \rH^*(X; \cK \otimes \bD_X).
\]
In these terms, we can interpret the map obtained by passing to cohomology in \eqref{eq: AB map complex} as
\begin{equation}\label{eq: AB map homology}
\frev_{G}: \rH_{*}(X; \rR \pi_* \Qll{\BB G}) \rightarrow  \rH^{*}(\Bun_G).
\end{equation}
We denote its further restriction to $\Bun_{G}^{\om}$ by $\frev^{\om}_{G}: \rH_{*}(X; \rR \pi_* \Qll{\BB G}) \rightarrow  \rH^{*}(\Bun^{\om}_G)$.

\subsubsection{The constant group case}\label{sss:fz} We make the above discussions more explicit when $G=\GG\times X$ is a constant group scheme. The map \eqref{eq: AB map homology} now reads
\begin{equation*}
\frev_{G}: \homog{*}{X}\ot \cohog{*}{\BB \GG}=\homog{*}{X}\ot R^{W}\to \cohog{*}{\Bun_{G}}.
\end{equation*}
For $z\in \homog{*}{X}$ and $f\in R^{W}$, we write 
\begin{equation}\label{fz}
f^{z}=\frev_{G}(z\ot f)\in \cohog{*}{\Bun_{G}}.
\end{equation}

\subsubsection{Graded evaluation map}
We define  
\begin{equation}\label{eq: hom positive part}
\rH_*(X; \VV_{G})_+  := \bigoplus_{i<2j} \rH_i(X; \VV_{G,2j})[i-2j] \subset \rH_*(X; \VV_{G})
\end{equation}
to be the subspace of $\rH_*(X; \VV_{G})$ with cohomological degree $>0$. For any $\om\in \pi_{0}(\Bun_{G})$, we will construct a canonical ``evaluation'' map
\begin{equation}\label{eq: AB map cohomology 2}
\frev^{\om}_{\aug, G}: \rH_*(X; \VV_{G})_{+} \rightarrow \Gr^1_{\aug} \rH^*(\Bun_G^\omega).
\end{equation}

We first make some reductions. If $\th:G\to G'=G_{\ad}\times A$ (where $A=G/G_{\der}$) is the canonical map, it induces a map $\th_{\Bun}: \Bun^{\om}_{G}\to \Bun^{\om'}_{G'}$. If $\frev^{\om'}_{\aug, G'}$ is defined, then we let $\frev^{\om}_{\aug, G}$ be the composition
\begin{equation*}
\rH_*(X; \VV_{G})_{+}\cong \rH_*(X; \VV_{G'})_{+}\xr{\frev^{\om'}_{\aug, G'}}\Gr^{1}_{\aug}\rH^*(\Bun_{G'}^{\omega'})\xr{\Gr^{1}_{\aug}\th_{\Bun}^{*}}\Gr^{1}_{\aug}\rH^*(\Bun_G^\omega).
\end{equation*}
Here we use Lemma \ref{l:V under isog}, and the last map is the map on $\Gr^{1}_{\aug}$ induced by the ring homomorphism $\th_{\Bun}^{*}$. Using the additivity $\VV_{G'}\cong \VV_{G_{\ad}}\op \VV_{A}$, we thus reduce to consider separately the case $G$ is adjoint (or more generally semisimple) and $G$ is a torus.

\subsubsection{The torus case} Let $G=A$ be a torus. In this case, $\VV_{A}=\rR^{2}\pi_* \Qll{\BB A}[-2]$, which is canonically a summand of $\rR\pi_* \Qll{\BB A}$ by \eqref{canon split coho BG}. We define $\frev^{\om}_{\aug, A}$ to be the composition
\begin{equation*}
\rH_*(X; \VV_{A})_{+}\subset \rH_*(X; \rR \pi_* \Qll{\BB A})_{+}\xr{\frev^{\om}_{A}} \rH^{>0}(\Bun_A^\omega)\surj \Gr^{1}_{\aug}\rH^{*}(\Bun_A^\omega).
\end{equation*}

\subsubsection{The semisimple case} Let $G$ be semisimple. Let $\rR^{>0} \pi_* \Qll{\BB G}=F^{1}_{\aug}(\rR\pi_* \Qll{\BB G})$ be the direct sum of the positive degree summands under the canonical splitting \eqref{canon split coho BG}. By construction we have a canonical map that is a surjection in each degree,
\begin{equation*}
\rR^{>0} \pi_* \Qll{\BB G}\to \VV_{G}.
\end{equation*}
The desired map \eqref{eq: AB map cohomology 2} is characterized by the following lemma.

\begin{lemma}
There is a unique map $\frev^{\om}_{\aug, G}$ making the following diagram commute:
\begin{equation*}
\xymatrix{ \rH_{*}(X; \rR^{>0} \pi_* \Qll{\BB G})_{+} \ar[r] ^-{\frev^{\om}_{G}} \ar@{->>}[d] & \rH^{>0}(\Bun_G^\omega)\ar@{->>}[d]\\
\rH_{*}(X; \VV_{G})_{+} \ar[r]^-{\frev^{\om}_{\aug,G}} &
\Gr^{1}_{\aug}\rH^{*}(\Bun_G^\omega)}
\end{equation*}
Here $\rH_{*}(X; \rR^{>0} \pi_* \Qll{\BB G})_{+}\subset \rH_{*}(X; \rR^{>0} \pi_* \Qll{\BB G})$ is the positive degree part.
\end{lemma}
\begin{proof}
Uniqueness is clear because the vertical maps are surjective. It remains to check that
\begin{equation}\label{F2}
\mbox{$\frev^{\om}_{G}$ sends $\rH_{*}(X; F^{2}_{\aug}\rR \pi_* \Qll{\BB G})_{+}$ to $F^{2}_{\aug}\rH^{*}(\Bun_G^\omega)$.}
\end{equation}
For this we compare $G$ with the constant group situation.

Let $\nu: X'=\cP_{\Out}\to X$ as in \S\ref{sss:splitting}, so that $\nu^{*}\BB G_{\ad}$ is canonically isomorphic to $\BB \GG_{\ad}\times X'$. This induces a canonical map $\ph: \Bun_{G}\to \Bun_{\GG_{\ad}, X'}$ by pulling back $G$-bundles along $\nu$ and pushout along $\BB G \to \BB G_{\ad} \cong \BB\GG_{\ad}$. We have a commutative diagram 
\begin{equation}\label{eq:AB-diag-1}
\xymatrix{X' \times \Bun_G \ar[r]^-{\wt\ev}  \ar[d]^{\id_{X'}\times \ph} & \nu^{*}\BB G \ar[d] \\
X' \times \Bun_{\GG_{\ad},X'}\ar[r]^-{\ev'}  &  \BB \GG_{\ad}
}
\end{equation}
Here $\wt\ev$ is the base change of $\ev: X \times \Bun_G \to \BB G$ along $\nu$. Applying the discussion of the constant group case in \S \ref{sss:fz}, \eqref{eq:AB-diag-1} induces a commutative diagram
\begin{equation*}
\xymatrix{\rH_{*}(X')\ot R^{W} \ar[r] ^-{\frev^{\om'}_{\GG_{\ad},X'}} \ar@{->>}[d]_{\a} & \rH^{>0}(\Bun^{\om'}_{\GG_{\ad},X'})\ar[d]^{\ph^{*}}\\
\rH_{*}(X; \rR \pi_* \Qll{\BB G}) \ar[r] ^-{\frev^{\om}_{G}} & \rH^{>0}(\Bun^{\om}_G)
}
\end{equation*}
Here the map $\a$ is the canonical quotient to $\Out(\GG)$-coinvariants.  Since $\ph^{*}$ respects the augmentation filtrations, and $\a$ maps $\rH_{*}(X')\ot F^{2}_{\aug}R^{W}$ surjectively to $\rH_{*}(X; F^{2}_{\aug}\rR \pi_* \Qll{\BB G})$, it suffices to check that \eqref{F2} holds for the constant group $\GG\times X'$ over $X'$. 

We henceforth rename $X'$ to $X$ and let $G=\GG\times X$, where $\GG$ is semisimple. In this case, \eqref{F2} is equivalent to: for any $f_{1}, f_{2}\in R^{W}_{+}$ and  any $z\in \homog{*}{X}$, we have $(f_{1}f_{2})^{z}\in F^{2}_{\aug}\rH^{*}(\Bun_G^\omega)$. Indeed, let $\D_{*}(z)\in \homog{*}{X\times X}\cong \homog{*}{X}\ot \homog{*}{X}$ be the image of $z$ under the diagonal map for $X$. Write $\D_{*}(z)=\sum_{i} z'_{i}\ot z''_{i}$ for $z'_{i}, z''_{i}\in \homog{*}{X}$. Then
\begin{equation}\label{contraction f1f2}
(f_{1}f_{2})^{z}=\sum_{i}f_{1}^{z'_{i}}f_{2}^{z''_{i}}.
\end{equation}
Since $G$ is semisimple, $f_{1},f_{2}$ have cohomological degrees $\ge4$ while  $z'_{i},z''_{i}$ have homological degrees $\le 2$, so $f_{1}^{z'_{i}}$ and $f_{2}^{z''_{i}}$ both have positive degrees, hence belong to $\cohog{>0}{\Bun_{G}^{\om}}$. Therefore $(f_{1}f_{2})^{z}\in F^{2}_{\aug}\cohog{*}{\Bun_{G}^{\om}}$. This proves \eqref{F2} and finishes the proof.
\end{proof}

Now we can state the Atiyah-Bott formula for the cohomology of $\Bun_{G}$.

\begin{theorem}\label{thm: atiyah-bott formula}
For any $\om\in \pi_{0}(\Bun_{G})$, the $\Qlbar$-algebra map
\begin{equation}\label{eq: atiyah-bott formula}
\Sym_{\Qlbar}^{\bl} (\rH_*(X; \VV_G)_+) \rightarrow \Gr^{\bl}_{\aug} \rH^*(\Bun_G^\omega)
\end{equation}
induced by \eqref{eq: AB map cohomology 2} is an isomorphism.  
\end{theorem}

\subsubsection{The constant group case}\label{sss:concrete AB} When $G=G_{0}\times X$, choose a basis of $\homog{*}{X}$:
\begin{equation*}
z_{0}=1\in \homog{0}{X}, \quad z_{1},\cdots,z_{2g}\in \homog{1}{X} \mbox{ and } z_{2g+1}=[X]\in \homog{2}{X}.
\end{equation*}
Choose also homogeneous free generators $f_{1},\cdots, f_{n}$ of $R^{W}$ as a polynomial ring. Then Theorem \ref{thm: atiyah-bott formula} implies that $\cohog{*}{\Bun_{G}^{\om}}$ is the polynomial ring over $\ol \Ql$ with free generators
\begin{equation}\label{eq:constant-AB-generators}
f_{i}^{z_{j}}|_{\Bun_{G}^{\om}}, \quad 0\le j\le 2g+1, 1\le i\le n, \mbox{ such that $(\deg f_{i},|z_{j}|)\ne (2,2)$.}
\end{equation}
(Here $|z_j|$ is the homological degree of $z_j$.) In particular, we get canonical isomorphisms between the cohomology rings of different components of $\Bun_{G}$, under which the classes with the same name $f^{z}$ (in the constant group case) correspond to each other.

\begin{remark} 
When $X$ and $G$ are defined over $\F_{q}$, $\Bun_G$ is also defined over $\F_q$.  The Frobenius $\Fr$ acts on $\pi_0(\Bun_{G,\ov\F_{q}})$ and induces a map $\Bun_{G,\ol\F_{q}}^\omega \rightarrow \Bun_{G,\ol\F_{q}}^{\Fr(\omega)}$. From the construction of \eqref{eq: atiyah-bott formula}, we obtain a commutative diagram 
\begin{equation}\label{eq: AB formula Frob equivariant}
\begin{tikzcd}
\Sym_{\Qlbar}^{\bl} (\rH_*(X_{\ol \F_{q}}; \VV_G)_+) \ar[r] \ar[d, "\Fr^*"] &  \Gr_{\aug}^{\bl} \rH^*(\Bun_{G,\ov\F_{q}}^{\Fr(\omega)})  \ar[d, "\Fr^*"] \\
\Sym_{\Qlbar}^{\bl} (\rH_*(X_{\ol \F_{q}}; \VV_G)_+) \ar[r]  &  \Gr_{\aug}^{\bl} \rH^*(\Bun_{G,\ol\F_{q}}^\omega) 
\end{tikzcd}
\end{equation}
In particular, if $\Fr(\omega) = \omega$ then the isomorphism \eqref{eq: atiyah-bott formula} is automatically $\Fr$-equivariant. 
\end{remark}

\subsubsection{Prior results}\label{sssec:AB-prior-results} There have been many partial results towards Theorem \ref{thm: atiyah-bott formula} in the literature. The formula is inspired by a formula for closely related moduli spaces of $G$-bundles on a Riemann surface, for split semisimple $G$, due to Atiyah--Bott \cite{AB83}. The analogous formula for $\Bun_G$ on a complex algebraic curve, again for split semisimple $G$, was established by Teleman in \cite{Tel98}. 

Turning towards $X/\F_q$, in the case where $G \rightarrow X$ is a constant semisimple group scheme (i.e., pulled back from a semisimple group $G_0/\F_q$), Theorem \ref{thm: atiyah-bott formula} is proved in \cite{HS10}. In the case where $G \rightarrow X$ has simply connected semisimple generic fiber (which implies that $\Bun_G$ is geometrically connected), it is proved in \cite{Gai19} based on the work of Gaitsgory--Lurie \cite{GL14}; the same result appears in \cite[Theorem 6.1.3]{Ho21}. We will bootstrap from this last result to prove Theorem \ref{thm: atiyah-bott formula} in the general reductive case.

\subsection{Propagation of the Atiyah--Bott formula}

We prove that Theorem \ref{thm: atiyah-bott formula} is stable under certain operations. 
We denote by $\Bun_G^0$ the neutral component (i.e., the connected component of $\Bun_{G}$ containing the trivial $G$-bundle).

\subsubsection{Products}\label{sssec: AB products} Let $G_{1}, G_{2}$ be reductive group schemes over $X$, and $G=G_{1}\times _{X}G_{2}$. Then Theorem \ref{thm: atiyah-bott formula} holds for both $G_1$ and $G_2$, if and only if it holds for $G$. Indeed, we have $\BB G\cong (\BB G_1) \times_X (\BB G_2)$ and $\Bun_{G} \cong \Bun_{G_1} \times \Bun_{G_2}$. Thus both sides of \eqref{eq: atiyah-bott formula} for $G$ can be written canonically as the tensor product of their counterparts for $G_{1}$ and $G_{2}$.

\subsubsection{Central isogenies}\label{sss:AB central isog}
Let $\theta \colon G \rightarrow G'$ be a central isogeny, with kernel $\Gamma$. It induces a map $v: \Bun_G \rightarrow \Bun_{G'}$.

\begin{prop}\label{prop:map-bun-surjective}
The map $v^0\colon\Bun_G^0 \rightarrow \Bun_{G'}^0$ is surjective.
\end{prop}

We record some preliminary Lemmas which will be used in the proof. 

\begin{lemma}\label{lem:torsors-fiber} Let $f \co G_1 \rightarrow G'$ be any homomorphism of groups over $S$ whose kernel $F := \ker(f)$ is central in $G_1$. (In particular, $F$ is abelian, so that the category $\BB F(S)$ of $F$-torsors on $S$ has a natural abelian group structure in groupoids). Then the fiber of the category $\BB G_1(S)$ over a given $G'$-torsor $\cF' \in \BB G'(S)$, is itself a torsor for $\BB F(S)$. 
\end{lemma}

\begin{proof}
Let $\cF_0$ be a $G_1$-torsor equipped with an isomorphism $\tau_0 \co \cF_0 \times^{G_1} G' \cong \cF'$. If $\cF$ is another $G_1$-torsor equipped with an isomorphism $\tau \co \cF \times^{G_1} G' \cong \cF'$, then we claim that the sheaf $\ul{\mrm{Isom}}((\cF_0, \tau_0), (\cF, \tau))$ of isomorphisms from $\cF_0$ to $\cF$ intertwining the identifications $\tau_0$ and $\tau$, is an $F$-torsor. Clearly $F$ acts on it; to check that it is an $F$-torsor, we can work locally, thus reducing to the case where $\cF, \cF_0$ are trivial, where the assertion is clear. 
\end{proof}

\begin{lemma}\label{lem:factorization}
Let $f \co G \rightarrow G'$ be a central isogeny of reductive groups over a curve. Then there is a factorization $G \rightarrow G_1 \rightarrow G'$ such that $\ker(G_1 \rightarrow G')$ is an induced torus.
\end{lemma}

\begin{proof}
Let $K := \ker(f)$, a finite flat group scheme over $X_{\ol k}$ of multiplicative type. Then the Cartier dual $K^D = \ul{\Hom}(K, \GG_m)$ is a finite \'etale group scheme over $X_{\ol k}$. By choosing a surjection onto $K^D$ from a permutation module for $\pi_1(X_{\ol k})$ and then forming Cartier duals, we can embed $K$ in an induced torus $T$. Now we let $G_1 := (G \times_X T)/K$, where $K$ acts diagonally as 
\[
k \cdot (g,t) = (gk, k^{-1} t). 
\]
We then have a map $G_1 \rightarrow G'$ given by $(g, t) \mapsto f(g)$. Clearly the kernel is $(K \times_X T)/K \cong T$. 
\end{proof}

\begin{proof}[Proof of Proposition \ref{prop:map-bun-surjective}] In the proof, we base change the situation to $\ov k$ without changing notation.

Choose a factorization $G \rightarrow G_1 \rightarrow G'$ as in Lemma \ref{lem:factorization}. Let $T := \ker(G_1 \rightarrow G')$, which by assumption is an induced torus. First we claim that the map $\Bun_{G_1} \rightarrow \Bun_{G'}$ is surjective. It suffices to show this on geometric points, so let $\kappa$ be a separably closed field over $k$. Then the map $\Bun_{G_1}(\kappa) \rightarrow \Bun_{G'}(\kappa)$ is identified with 
\[
H^1(X_\kappa; G_1) \rightarrow H^1(X_\kappa; G').
\]
The obstruction classes for this map lie in $H^2(X_\kappa; T)$, which vanishes because $T$ is an induced torus (by Shapiro's Lemma and the vanishing of Brauer groups of curves over separably closed fields \cite[Corollary 5.8]{Gro68}). 
This shows that $\Bun_{G_1} \surj \Bun_{G'}$. Lemma \ref{lem:torsors-fiber} shows that this map is a $\Bun_T$-torsor, so we obtain an isomorphism $\Bun_{G_1}/\Bun_T \xrightarrow{\sim} \Bun_{G'}$. Also, we observe that the map $\Bun_{G_1}^0 \rightarrow \Bun_{G'}^0$ is surjective. Indeed, the exact sequence $\pi_1^{\alg}(T/X) \rightarrow \pi_1^{\alg}(G_1/X) \rightarrow \pi_1^{\alg}(G'/X) \rightarrow 0$ plus \eqref{eq:bunG_connected_components} show that for a geometric point of $\Bun_{G'}^0$, any lift to $\Bun_{G_1}$ can be translated to $\Bun_{G_1}^0$ by the action of $\Bun_T$. 

Let $D := G_1/G$, so that $G \rightarrow G_1 \rightarrow D$ is an exact sequence. Then we have a Cartesian square
\begin{equation}\label{eq:cart-sq-1}
\begin{tikzcd}
\Bun_G \ar[r] \ar[d]  & \Bun_{G_1} \ar[d] \\
\pt \ar[r] & \Bun_{D} 
\end{tikzcd}
\end{equation}
where the bottom horizontal arrow is the map induced by trivial $D$-bundle. Let $\Bun_G^\na\subset \Bun_G$ be the preimage of $\Bun^0_{G_1}$ under $\Bun_G\to \Bun_{G_1}$. Then $\Bun_G^\na$ is a union of connected components of $\Bun_G$. We have a Cartesian diagram
\begin{equation*}
\begin{tikzcd}
\Bun_G^\na \ar[r] \ar[d]  & \Bun_{G_1}^0 \ar[d, "d"] \\
\pt \ar[r] & \Bun_{D}^0 
\end{tikzcd}
\end{equation*}
The map $d$ is equivariant under the translation actions of $\Bun_T^0$. Since $T\to D$ is an isogeny, the induced map $\Bun_T^0\to \Bun_D^0$ is surjective, hence the translation action of $\Bun_T^0$ on $\Bun_D^0$ is  transitively. This shows that $d$ is surjective, and that $\Bun_G^\na\to \Bun^0_{G_1}/\Bun_T^0$ is surjective. We observed earlier that $\Bun_{G_1}^0 \rightarrow \Bun_{G'}^0$ is surjective; since this map factors through $\Bun_{G_1}^0/\Bun_{T}^0$, we deduce that $\Bun_{G_1}^0/\Bun_{T}^0 \surj \Bun_{G'}^0$ is also surjective. Composing this with the surjection $\Bun_G^\na \surj \Bun_{G_1}^0/\Bun_{T}^0$, we conclude that
\begin{equation}\label{Bun na surj}
    \mbox{The map $v^\na: \Bun^\na_G\to \Bun^0_{G'}$ is surjective.}
\end{equation}
In particular, we learn that the map $v' \co v^{-1}(\Bun^0_{G'}) \rightarrow \Bun^0_{G'}$ is surjective. By Lemma \ref{lem:torsors-fiber}, $v'$ is a torsor for the abelian group stack $\Bun_{\Gamma}$. In particular, it is flat. At the same time, we see that $\Bun_{\Gamma}$ is finite by presenting $\Gamma$ as the kernel of an isogeny of tori. Therefore, $v'$ is also finite. Restricting $v'$ to the connected component $\Bun_G^0$, we see that the map $v^0\colon \Bun^0_G\rightarrow\Bun^0_{G'}$ is also finite and flat. This implies that the image of $v^0$ is both open and closed, and clearly non-empty, hence must be everything.

\end{proof}

Return to the central isogeny $\theta \co G \rightarrow G'$ with kernel $\Gamma$. Note that $\Bun_\Gamma$ is an abelian group stack. By the constructions of \S \ref{ssec: classifying space preliminaries}, $\Bun_\Gamma$ acts on $\Bun_G$ such that the map $\Bun_G\to \Bun_{G'}$ induced by $\theta$ is $\Bun_\Gamma$-invariant. Let $\Bun_{\Gamma}^{\na}$ be the subgroup stack of $\Bun_{\Gamma}$ that preserves the neutral component $\Bun^{0}_{G}$, then we have a canonical map
\begin{equation}\label{Bun G 0 to G'}
\Bun_G^0/\Bun^{\na}_{\Gamma}\to \Bun_{G'}^0.
\end{equation}

\begin{cor} 
The map \eqref{Bun G 0 to G'} is an isomorphism. 
\end{cor}

\begin{proof}
Lemma \ref{lem:torsors-fiber} shows that $\Bun_G$ is a $\Bun_\Gamma^{\natural}$-torsor over its image. Then it suffices to see that the map $\Bun_G^0 \rightarrow \Bun_{G'}^0$ is surjective, which is Proposition \ref{prop:map-bun-surjective}.
\end{proof}

We have a corresponding classifying map 
\[
f \co \Bun_{G'}^0 \rightarrow \BB \Bun^{\na}_\Gamma.
\]
Each direct image $\rR^i f_* (\Ql)$ is a local system on $\BB \Bun^{\na}_{\Gamma}$ with stalk $\rH^i(\Bun^{0}_G)$, together with the action of $\pi_0(\Bun^{\na}_{\Gamma})$ coming from the action of $\Bun^{\na}_{\Gamma}$ on $\Bun^{0}_{G}$. 

\begin{lemma}\label{lem: trivial action} Suppose that Theorem \ref{thm: atiyah-bott formula} holds for $G$. Then the action of $\pi_0(\Bun^{\na}_{\Gamma})$ on $\rH^i(\Bun_G^0)$ is trivial. 
\end{lemma}

\begin{proof} According to the Atiyah--Bott formula for $G$, $\rH^*(\Bun^{0}_G)$ is generated by the K\"unneth components of $G$-Chern classes of the universal $G$-bundle on $X \times \Bun_G$. But we saw in Corollary \ref{cor: twisting Chern classes} that twisting by $\Gamma$-torsors has trivial effect on $G$-Chern classes. 
\end{proof}

\begin{exam}
Consider the universal cover $\SL_2  = G \rightarrow G' = \PGL_2$. Then $\Gamma = \mu_2$ and $\Bun_{\Gamma} = \Bun_{\Gamma}^{\na}$ can be identified with the moduli stack of 2-torsion line bundles. The action on $\Bun_{\SL_2}$ is via tensoring with 2-torsion line bundles. Those all have trivial Chern class, so the action on the Atiyah--Bott description is evidently trivial.
\end{exam}

\begin{cor}\label{cor: propagate central isogeny} Suppose that Theorem \ref{thm: atiyah-bott formula} holds for $\Bun_G^0$.
Then $f$ induces an isomorphism 
\[
\rH^*(\Bun_{G'}^0 ) \xrightarrow{\sim} \rH^*(\Bun_G^0),
\]
and Theorem \ref{thm: atiyah-bott formula} holds for $\Bun_{G'}^0$. 
\end{cor}

\begin{proof}From the presentation $\Bun_{G'}^0 = [\Bun_G^0/\Bun^{\na}_{\Gamma}]$ we have a spectral sequence 
\[
\rH^i(\BB\Bun^{\na}_{\Gamma}; \rH^j(\Bun_G^0)) \implies \rH^{i+j}(\Bun_{G'}^0).
\]
Since $\Gamma$ is finite and of multiplicative type, $\Bun^{\na}_{\Gamma}$ is a finite abelian group stack, so the higher cohomology of $\BB \Bun^{\na}_{\Gamma}$ vanishes with rational coefficients. The spectral sequence therefore degenerates. Furthermore, since $\rH^j(\Bun_G^0)$ is the trivial local system on $\Bun^{\na}_{\Gamma}$ by Lemma \ref{lem: trivial action}, we obtain that the pullback map 
\[
\rH^*(\Bun_{G'}^0) \rightarrow \rH^*(\Bun_G^0)
\]
is an isomorphism. Moreover, it is clear from construction and Lemma \ref{lem: classifying cohomology pullback} that this isomorphism is compatible with the formulation of Theorem \ref{thm: atiyah-bott formula}. 
\end{proof}

\subsubsection{Pure inner twisting}\label{sssec:pure-inner-twist}
Let $G \rightarrow S$ be a reductive group scheme and $\cP \in \BB G(S)$. We may view $\cP$ as a $G$-torsor in the \'etale site of $S$. Let $G' = \ul{\Aut}_S(\cP)$ be the (internal) automorphism group of $\cP$ (as a $G$-torsor) in the \'etale site of $S$. Then by general nonsense, $G'$ is a twist of $G$ in the \'etale site of $S$. Furthermore, the \emph{pure inner twisting} construction $\cF \mapsto \ul{\mrm{Isom}}_S^G(\cF, \cP)$ carries $G$-torsors to $G'$-torsors, and induces an equivalence of groupoids
\[
\BB G(S) \xrightarrow{\sim} \BB G'(S).
\]

Now let $\Bun_G^\omega$ be a connected component of $\Bun_G$. Fix a $G$-bundle $\cP \in \Bun_G^\omega (k)$. Letting $G' = \ul{\Aut}_X(\cP)$ be the corresponding pure inner twist of $G$, the pure inner twisting construction $\cF \mapsto \ul{\mrm{Isom}}_S^G(\cF, \cP)$ induces an isomorphism 
\[
\Bun_G \xrightarrow{\sim} \Bun_{G'}
\]
carrying $\cP$ to the trivial $G'$-torsor, hence in particular carrying
\[
\Bun_G^\omega \xrightarrow{\sim} \Bun_{G'}^0.
\]

\begin{lemma}\label{lem: propagate pure inner twist}
Let $\cC$ be a class of connected reductive group schemes over $X$ preserved by pure inner twisting, e.g., $\cC = \{\text{semisimple } G/X\}$ or $\cC = \{\text{connected reductive } G/X\}$. Suppose Theorem \ref{thm: atiyah-bott formula} holds for the neutral components $\Bun_{G}^{0}$ of all groups $G \in \cC$. Then Theorem \ref{thm: atiyah-bott formula} holds for all $G \in \cC$. 
\end{lemma}

\begin{proof}
Pure inner twisting by $\cP$ induces an isomorphism $\BB G \xrightarrow{\sim} \BB G'$ of classifying stacks over $X$, such that the diagram 
\[
\begin{tikzcd}
X \times \Bun_G^\omega \ar[r, "\ev"]  \ar[d, "\wr"] & \BB G \ar[d, "\wr"] \\
X \times \Bun_{G'}^0 \ar[r, "\ev"]  &  \BB G'
\end{tikzcd}
\]
commutes. Then it is clear that Theorem \ref{thm: atiyah-bott formula} for the bottom row implies Theorem \ref{thm: atiyah-bott formula} for the top row. 
\end{proof}

\subsection{The simply connected case}\label{sssec: AB simply connected}

In the case where $G$ is semisimple with simply connected generic fiber, $\Bun_G$ is connected and Theorem \ref{thm: atiyah-bott formula} is already proved in \cite[(0.7)]{Gai19}, and later is reproved by similar methods in \cite[Theorem 6.13]{Ho21}.

\subsection{The semisimple case}

We prove that Theorem \ref{thm: atiyah-bott formula} holds for semisimple groups.

Let $G$ be a semisimple group scheme over $X$. Let $G_{\sc} \rightarrow G$ be the simply connected cover. Then we saw in \S \ref{sssec: AB simply connected} that Theorem \ref{thm: atiyah-bott formula} holds for $\Bun_{G_{\sc}}$. From Corollary \ref{cor: propagate central isogeny} we obtain that Theorem \ref{thm: atiyah-bott formula} holds for $\Bun_G^0$, for all semisimple groups $G$. Then from Lemma \ref{lem: propagate pure inner twist}, we obtain that Theorem \ref{thm: atiyah-bott formula} holds for all semisimple groups. 

\subsection{The torus case} We will prove Theorem \ref{thm: atiyah-bott formula} when $G = T$ is a torus over $X$. Applying Lemma \ref{lem: propagate pure inner twist} to $\cC=\{T\}$, it suffices to prove Theorem \ref{thm: atiyah-bott formula} for the neutral component $\Bun^{0}_{T}$.

\subsubsection{Split tori}
First suppose $T \cong \TT\times X$ is split. Then the result is classical, but let us spell it out. By the additive nature of both sides of \eqref{eq: atiyah-bott formula}, it suffices to treat the case $\TT=\Gm$, in which case $\Bun^{0}_{T}=\Pic^{0}_{X}$. Upon choosing $x_{0}\in X(k)$, we have
\begin{equation*}
\Pic^{0}_{X} \cong \Jac_{X} \times_k \BB \Gm.
\end{equation*}
Here $\Jac_{X}$ is the moduli of degree zero line bundles on $X$ with a trivialization at $x_{0}$. It follows that
\begin{equation}\label{Gr1 aug Pic}
\Gr^{1}_{\aug}\upH^{*}(\Pic^{0}_{X})\cong \upH^{1}(\Jac_{X})[-1]\op \upH^{2}(\BB\Gm)[-2],
\end{equation}
and the natural map
\begin{equation*}
\Sym^{\bu}\Gr^{1}_{\aug}\upH^{*}(\Pic^{0}_{X})\to \Gr^{\bu}_{\aug}\upH^{*}(\Pic^{0}_{X})
\end{equation*}
is an isomorphism. Thus, to prove Theorem \ref{thm: atiyah-bott formula} in this case, it suffices to show that
\begin{equation*}
\frev^{0}_{\aug, \Gm}: (\upH_{*}(X)\ot \VV_{\Gm})_{+}\to \Gr^{1}_{\aug}\upH^{*}(\Pic^{0}_{X})
\end{equation*}
is an isomorphism.

We have $\VV_{\Gm}=\Ql[-2](-1)$, so that
\begin{equation}\label{VGm}
(\upH_{*}(X)\ot \VV_{\Gm})_{+}=\rH_1(X)[-1](-1) \op \rH_{0}(X)[-2](-1).
\end{equation}
Under the isomorphisms \eqref{Gr1 aug Pic} and \eqref{VGm}, the map $\frev^{0}_{\aug, \Gm}$ respects each direct summand. On the first summand, as a map $\rH_1(X)(-1)\to \upH^{1}(\Jac_{X})$, or equivalently, an element in $\upH^{1}(X)\ot \upH^{1}(\Jac_{X})(1)$, it is the K\"unneth component of $c_{1}(\cL^{\univ})$ for the  universal line bundle $\cL^{\univ}$ on $X\times \Jac_{X}$. This shows that $\frev^{0}_{\aug, \Gm}$ is an isomorphism on the first summand. On the second summand, if we compose with restriction to $x_{0}$, the evaluation map becomes the identity map $\upH^{2}(\BB\Gm)[-2](-1)\to \Ql[-2](-1)=\VV_{\Gm}$. This shows $\frev^{0}_{\aug, \Gm}$ is an isomorphism, and completes the proof for split tori.

\subsubsection{Induced tori}\label{sss:induced tori}
Let $\nu: Y\to X$ be a finite \'etale map, $T_{0}$ a torus over $Y$, and $T=\Res_{Y/X}T_{0}$ be the Weil restriction, a torus over $X$. We claim that if Theorem \ref{thm: atiyah-bott formula} holds for $T_{0}$ and $Y$, then it holds for $T$ and $X$. Indeed, $\rH_{*}(Y,\VV_{T_{0}})$ is canonically isomorphic to $\rH_{*}(X,\VV_{T})$, and $\Bun_{T,X}$ is canonically isomorphic to $\Bun_{T_{0},Y}$. 

In particular, by the split tori case already proved, Theorem \ref{thm: atiyah-bott formula} holds for any induced torus $T=\Res_{Y/X}(\TT\times Y)$, where $Y$ is finite \'etale over $X$.

\subsubsection{General tori}
In general, $T$ splits over some finite \'etale Galois $\nu \co  Y \rightarrow X$, so that $\nu^{*}T\cong \TT\times Y$ for some torus $\TT$ over $\Spec k$. Then we have a canonical map $\a: T\to T'=\Res_{Y/X}(\TT\times Y)$ and a norm map $\b: T'\to T$ such that the composition $\b\circ\a=[N]$, where $N=\deg(\nu)$. Therefore one can find a torus $S$ over $X$ together with an isogeny $T'\to T\times_{X} S$. Since Theorem \ref{thm: atiyah-bott formula} holds for $\Bun^{0}_{T'}$ by \S\ref{sss:induced tori}, it also holds for $\Bun_{T\times_{X} S}^{0}$ by Corollary \ref{cor: propagate central isogeny}, hence also for $\Bun^{0}_{T}$ by \S\ref{sssec: AB products}. This finishes the proof in the torus case.

\subsection{The reductive case}
If $G$ is reductive, let $Z^\circ$ be the connected center of $G$ and $G_{\der}$ the derived subgroup of $G$. Then the map $Z^\circ \times_X G_{\der} \rightarrow G$ is a central isogeny. Theorem \ref{thm: atiyah-bott formula} has been established for $Z^\circ$ and $G_{\der}$, hence also for $Z^\circ \times_X G_{\der}$. By Corollary \ref{cor: propagate central isogeny}, Theorem \ref{thm: atiyah-bott formula} holds for $\Bun_G^0$.

At this point we know that Theorem \ref{thm: atiyah-bott formula} holds for the geometric neutral component $\Bun_G^0$ for every reductive $G$. Then by its propagation under pure inner twisting construction (Lemma \ref{lem: propagate pure inner twist}) it holds for all the geometric connected components of $\Bun_G$ for every reductive $G$.

\section{Arithmetic volume of shtukas for split groups}\label{sec:split-arithmetic-volume}

Fix a prime $p$. Henceforth our base field will be $k := \F_q$ where $q$ is a power of $p$. Let $X$ be a smooth, projective, geometrically connected curve over $k$, of genus $g$. 

In this section, we formulate and prove our main formula for the ``arithmetic volume'' of the moduli stack of shtukas for a \emph{split} reductive group scheme over $X$. (The nonsplit case will be treated later, in \S \ref{sec:nonsplit-volume}.) As we only work with constant group schemes over $X$ in this section, we will denote by $G$ a connected reductive group over $\F_{q}$ (which was denoted by $\GG$ in \S\ref{sec: atiyah-bott}); the discussions in \S\ref{sec: atiyah-bott} are applied to the constant group scheme $G\times X$ over $X$.

\subsection{Hecke correspondences}

Below we write $\VV_G$ as $\VV$.
Let $\om\in \pi_{0}(\Bun_{G})$. Theorem \ref{thm: atiyah-bott formula} gives a bigraded isomorphism of algebras
\begin{equation}\label{eq: split AB}
\AB_{\om}: \Sym^{\bl}((\homog{*}{X}\ot \VV)_{+})\to\Gr^{\bl}_{\aug}\cohog{*}{\Bun^{\om}_{G}}.
\end{equation}
We recall the notation $f^{z}\in \cohog{*}{\Bun^{\om}_G}$ from \S\ref{sss:fz} for $f\in R^{W}$ and $z\in \homog{*}{X}$. 

We will apply the general result in \S\ref{s:char class under mod} to the Hecke correspondence for $\Bun_{G}$. Let $\mu \in \xcoch(T)^{+}$ be a dominant coweight. Consider the Hecke correspondence
\begin{equation*}
\xymatrix{ & \Hk^{\mu}_G\ar[dl]_{h_{0}}\ar[dr]^{h_{1}}\ar[rr]^-{p_{X}} & & X\\
\Bun_{G} && \Bun_{G}}
\end{equation*}
that classifies modifications $(x, \cF_{0}\dashrightarrow\cF_{1})$ at one moving point $x\in X$ of type \emph{equal} to $\mu$ (rather than $\le\mu$).\footnote{Our main theorems will treat the case where $\mu$ is minuscule, in which case there is no distinction between these notions.} 

For a connected component $\Bun_G^{\om}$, let ${}^{\om} \Hk^{\mu}_G\subset \Hk^{\mu}_G$ be the preimage of $\Bun_{G}^{\om}$ under $h_{0}$. Then ${}^{\om} \Hk^{\mu}_G$ has a natural correspondence structure
\begin{equation}\label{Hk om}
\xymatrix{& {}^{\om} \Hk^{\mu}_G\ar[rr]^-{p_{X}}\ar[dl]_{h_{0}}\ar[dr]^{h_{1}} & & X\\
\Bun^{\om}_{G} & & \Bun^{\om'}_{G}}
\end{equation}
where $\om'=\om+\ov\mu$, and $\ov\mu$ denotes the image of $\mu$ in $\pi_{0}(\Bun_{\GG})$, the quotient of $\xcoch(\TT)$ by the coroot lattice, cf \eqref{eq: pi0}.

\subsubsection{Tautological classes}\label{sssec:tautological-classes} Consider $X\times \Hk^{\mu}_G$ with projection to the two factors denoted $q_{X}: X\times \Hk^{\mu}_G \to X$ and $p_{\Hk}: X\times \Hk^{\mu}_G \to \Hk^{\mu}_G$. Let $\G(p_{X})\subset X\times \Hk^{\mu}_G$ be the graph of $p_{X}: \Hk^{\mu}_G\to X$. Let $\cF^{\univ}$ be the universal $G$-bundle over $X\times \Bun_{G}$. We have two $G$-bundles on $X\times \Hk^{\mu}_G$ 
\begin{equation*}
\cF_{0}:=(q_{X}\times h_{0})^{*}\cF^{\univ}, \quad \cF_{1}:=(q_{X}\times h_{1})^{*}\cF^{\univ}
\end{equation*}
together with the universal Hecke modification of type $\mu$ along $\G(p_{X})$
\begin{equation}\label{univ modif} 
\ph^{\univ}: \cF_{0}\dashrightarrow \cF_{1}.
\end{equation}
By Proposition \ref{prop:canonical-parabolic-red}, the $G$-torsor $\cF_{0}|_{\G(p_{X})}$ carries a canonical $P_{\mu}$-reduction classified by a map
\begin{equation}\label{Hk ev L}
\ev_{\mu}:\Hk^{\mu}_G\to \BB P_{\mu}.
\end{equation}
We emphasize that the $P_\mu$-reduction comes from $\cF_0|_{\G(p_X)}$, not $\cF_1|_{\G(p_X)}$. 

Restricting to the $\om$-component, which we denote by $\ev^{\om}_{\mu}$, it induces a ring homomorphism
\begin{equation}\label{eq: parabolic embedding}
\ev^{\om,*}_{\mu}: R^{W_{\mu}}\cong \cohog{*}{\BB P_{\mu}}\to \cohog{*}{{}^{\om} \Hk^{\mu}_G}
\end{equation}
which is easily seen to be injective. Combining this with pullback $p_X^*$ along the leg map, we get a ring homomorphism
\begin{equation}\label{taut class Hk}
    \cohog{*}{X}\ot R^{W_\mu}\to \cohog{*}{{}^{\om} \Hk^{\mu}_G}.
\end{equation} 
The image of this map can be regarded as ``tautological classes on $\cohog{*}{{}^{\om} \Hk^{\mu}_G}$''.

\subsubsection{Tautological endomorphisms of cohomology}\label{ssec:split-Gamma}

Now assume that $\mu \in \xcoch(T)$ is minuscule and dominant. Let $D_\mu :=\j{2\rho,\mu} = \dim G/P_{\mu}$.  We consider classes in $\cohog{2D_\mu+2}{\Hk^\mu_G}$ which are the image under \eqref{taut class Hk} of 
\begin{equation*}
    \y+\xi\y', \mbox{ where } \y\in R^{W_\mu}_{2(D_\mu+1)}, \y'\in R^{W_\mu}_{2D_\mu}.
\end{equation*}

Let $\om'=\om+\ov \mu \in \pi_{0}(\Bun_{G})$, where we recall that $\ov \mu$ is the image of $\mu\in \xcoch(T)$ in $\pi_{0}(\Bun_{G})$. Consider the degree-preserving map 
\begin{align}\label{eq: g^yy'}
\G_{\mu}^{\y+\xi\y'}: \cohog{*}{\Bun^{\om'}_{G}}&\to \cohog{*}{\Bun^{\om}_{G}}\\
\th &\mt h_{0*}((\y+\xi\y')\cdot h_{1}^{*}\th).
\end{align}
The pushforward map on cohomology $h_{0*}$ is defined because $h_0$ is smooth and proper (recall that $\mu$ is assumed to be minuscule). The map $\G_{\mu}^{\y+\xi\y'}$ is degree-preserving because the relative dimension of $h_0$ is $D_\mu+1$.

\subsubsection{Variant on compactly supported cohomology}\label{sss:Gamma c}
With the same notation, we define 
\begin{align}\label{Gamma c}
{}_{c}\Gamma^{\eta + \xi \eta'}_\mu: \cohoc{*}{\Bun^{\om}_{G}}&\to \cohoc{*}{\Bun^{\om'}_{G}}\\
\th &\mt h_{1*}((\eta + \xi \eta')\cdot h_0^*\th).
\end{align}

The definition is arranged so that under the Poincar\'e duality pairing
\begin{equation}\label{PD}
\j{-,- }: \cohoc{i}{\Bun^\omega_G}\times \cohog{2\dim \Bun_G-i}{\Bun^{\omega}_{G}}\to \cohoc{2\dim \Bun_G}{\Bun^{\omega}_{G}}\cong \Qlbar(-\dim \Bun_G)
\end{equation}
the map ${}_c\Gamma^{\eta + \xi \eta'}_\mu$ is adjoint to $\Gamma^{\eta + \xi \eta'}_\mu$, i.e.,
\begin{equation}\label{Gamma adj}
\j{{}_{c}\Gamma^{\eta + \xi \eta'}_\mu (\alpha), \beta}=\j{\alpha, \Gamma^{\eta + \xi \eta'}_{\mu}(\beta)}
\end{equation}
for all $\alpha\in \cohoc{i}{\Bun_{G}^\omega}$ and $\beta\in \cohog{2\dim \Bun_G-i}{\Bun^{\omega'}_{G}}$.

\subsection{Arithmetic volume of the moduli of shtukas}\label{ss:vol split} We can now define the moduli spaces of shtukas and the notion of arithmetic volume of their tautological classes.

\subsubsection{Moduli of shtukas}\label{sss:Sht split}
Let $r\in \ZZ_{\ge0}$ and $\mu=(\mu_{1},\cdots, \mu_{r})$ be a sequence of minuscule dominant coweights of $G$ that is {\em admissible} in the sense that 
\begin{equation}\label{eq:modification-condition}
\mu_{1}+\mu_{2}+\cdots +\mu_{r} \text{ lies in the coroot lattice of $G$.}
\end{equation}
In other words, we have $\sum_{i=1}^r\ov\mu_i=0\in \pi_0(\Bun_G)$,  cf. \eqref{eq: pi0}.

Let $\Hk^{\mu}_G$ be the fiber product
\begin{equation*}
\Hk_G^{\mu_{1}}\times_{\Bun_{G}}\Hk_G^{\mu_{2}}\times_{\Bun_{G}}\cdots\times_{\Bun_{G}}\Hk_G^{\mu_{r}}
\end{equation*}
that classifies iterated modifications $(x_{1},\cdots, x_{r}, \cF_{0}\dashrightarrow \cF_{1}\dashrightarrow\cdots \dashrightarrow \cF_{r})$, where $\cF_{i-1}\dashrightarrow \cF_{i}$ is a modification along $x_{i}$ of type $\mu_{i}$. For $0\le i\le r$, let $h_{i}: \Hk^{\mu}_G\to \Bun_{G}$ record $\cF_{i}$. Let ${}^{\om} \Hk^{\mu}_G\subset \Hk^{\mu}_G$ be the preimage of $\Bun_{G}^{\om}$ under $h_{0}$.  

The \emph{moduli stack $\Sht^{\mu}_{G}$ of $G$-shtukas with $r$ legs and modification type $\mu$} is the fiber product 
\begin{equation}\label{def Sht}
\xymatrix{\Sht^{\mu}_{G}\ar[r]\ar[d] & \Hk^{\mu}_G\ar[d]^{(h_{0}, h_{r})}\\
\Bun_{G}\ar[r]^-{(\id,\Fr)} & \Bun_{G}\times \Bun_{G}}
\end{equation}

We decompose $\Sht^{\mu}_{G}$ according to which connected component of $\Bun_{G}$ the first bundle $\cE_{0}$ lies in. This gives a decomposition
\begin{equation*}
\Sht^{\mu}_{G}=\coprod_{\om\in \pi_{0}(\Bun_{G})}{}^{\om}\Sht^{\mu}_{G}
\end{equation*}
so that ${}^{\om}\Sht^{\mu}_{G}$ is the preimage of ${}^{\om} \Hk^{\mu}_G$. Let
\begin{equation}\label{eq:omega_j}
\om_{j}=\om+\ov \mu_{1}+\cdots+\ov \mu_{j}\in \pi_{0}(\Bun_{G}), \quad j=1,2,\cdots, r.
\end{equation}
Then for $(x_{1},\cdots, x_{r}, \cF_{0}\dashrightarrow \cF_{1}\dashrightarrow\cdots \dashrightarrow \cF_{r}\cong {}^{\t}\cF_{0})\in {}^{\om}\Sht^{\mu}_{G}$,  we have $\cF_{j}\in \Bun_{G}^{\om_{j}}$.

\subsubsection{Definition of Arithmetic volume}
For $j = 1, 2, \ldots, r$, let $D_{\mu_j}=\dim (G/P_{\mu_{j}})=\j{2\r, \mu_{j}}$ and choose a pair of classes
\begin{equation*}
\y_{j}\in R^{W_{\mu_{j}}}_{2(D_{\mu_j}+1)}, \quad \y'_{j}\in R^{W_{\mu_{j}}}_{2D_{\mu_j}}.
\end{equation*}
We have a composite morphism 
\begin{equation*}
\Sht^{\mu}_{G}\to \Hk_G^\mu \to \Hk_G^{\mu_{j}}.
\end{equation*}
Composing this further with the map $\Hk_G^{\mu_{j}} \xrightarrow{\ev_{j}} \BB P_{\mu_{j}}$ from \eqref{Hk ev L} (corresponding to the canonical parabolic reduction of $\cF_{j-1}|_{\G(x_j)}$), we obtain the map 
\begin{equation*}
\ev_{j}: \Sht^{\mu}_{G} \rightarrow \Hk_G^{\mu_{j}}  \xr{\ev_{j}} \BB P_{\mu_{j}}.
\end{equation*}
For $1 \leq j \leq r$, let $p_{j}: \Sht^{\mu}_{G}\to X$ be the map recording the $j$th leg. 

We let $\eta$ be the $r$-tuple $(\eta_1 + \xi \eta_1', \ldots, \eta_r + \xi \eta_r')$. Recall the operators ${}_c\G^{\eta_j + \xi \eta'_j}_{\mu_j}$ defined in \eqref{Gamma c}. Write
\[
{}_{c}\Gamma^{ \eta}_{ \mu}  :=  {}_{c}\Gamma^{\eta_r + \xi \eta'_r}_{\mu_r} \circ \ldots \circ {}_{c}\Gamma^{\eta_1+ \xi  \eta'_1}_{\mu_1} \co \rH_c^*(\Bun_G^\omega) \rightarrow \rH_c^*(\Bun_G^{\omega_r}) =   \rH_c^*(\Bun_G^\omega)
\]
for the composition of the ${}_{c}\Gamma^{\eta_j + \xi \eta'_j}_{\mu_j}$; note that the target component $\omega$ agrees with the source component $\omega$ by the assumption \eqref{eq:modification-condition}. 

\begin{defn}\label{def: arith vol}
Let $\mu = (\mu_1, \ldots, \mu_r)$ and $\eta = (\eta_1 + \xi \eta_1', \ldots, \eta_r + \xi \eta_r')$. The \emph{arithmetic volume} of ${}^{\om}\Sht^{\mu}_{G}$ with respect to $\eta$ is the \emph{graded} trace
\begin{equation}\label{gen int over ShtG}
\vol({}^{\omega}\Sht_G^\mu, \eta) := \Tr({}_{c}\Gamma^{ \eta}_{ \mu}\circ\Frob  \mid \rH_c^*(\Bun_G^\omega)).
\end{equation}
Since $\rH_c^*(\Bun_G^\omega)$ is infinite-dimensional, it is not clear that the RHS is well-defined; this will be justified in Proposition \ref{p:trace conv}.
\end{defn}

\begin{remark}
We justify why the definition \eqref{gen int over ShtG} deserves to be viewed as an ``arithmetic volume''. The numerics are arranged so that $\prod_{j=1}^{r}(\ev_{j}^{*}\y_{j}+p_{j}^{*}\xi\cdot \ev_{j}^{*}\y'_{j})$ is a top-degree cohomology class of ${}^\omega\Sht^{\mu}_{G}$. If ${}^\omega\Sht^{\mu}_G$ were proper, then 
\begin{equation}\label{eq:literal-fundamental-class}
\int_{{}^{\om}\Sht^{\mu}_{G}}\prod_{j=1}^{r}(\ev_{j}^{*}\y_{j}+p_{j}^{*}\xi\cdot \ev_{j}^{*}\y'_{j}) 
\end{equation}
would be well-defined, as the evaluation of the fundamental class of ${}^\omega\Sht^{\mu}_G$ on the top-degree cohomology class, and we would take this to be $\vol(\Sht_G^\mu, \eta)$. But ${}^\omega\Sht^{\mu}_G$ is almost never proper, so \eqref{eq:literal-fundamental-class} has no a priori meaning.

 If $\Bun_G^\omega$ were proper, then a suitable version of the Grothendieck--Lefschetz trace formula (e.g., \cite[Proposition 11.8]{FYZ}) would identify \eqref{eq:literal-fundamental-class} with $\Tr(\Frob\circ
{}_{c}\Gamma^{ \eta}_{ \mu}  \mid \rH_c^*(\Bun_G^\omega))$. The latter expression has the potential at least to make sense when $\Bun_G^\omega$ is not proper (there is a convergence question, because $\rH_c^*(\Bun_G^\omega)$ is infinite-dimensional), and we will show that it indeed is always well-defined. 
\end{remark}

Our main theorem (Theorem \ref{th:vol gen}) computes the arithmetic volume in terms of differential operators applied to the $L$-function of the Gross motive of $G$ (cf. \S\ref{sss:motive}). These differential operators are determined by local data, which we describe next.

\subsection{Local operators on characteristic classes}\label{ssec:nabla-operator} We define some ``local'' operators on $R^W = \cohog{*}{\BB G}$. 

\subsubsection{The operator $\nb^{\y}_{\mu}$}\label{sssec:nabla-operator}
Let $\mu \in \xcoch(T)$. Then we have the partial derivative $\pl_\mu$ on $R = \Sym(\xch(T)_{\Ql}(-1))$ which lowers the grading by $2$ and twist by $1$. It carries the subring $R^W$ to $R^{W_{\mu}}$, so we may view it as a derivation 
\begin{equation*}
\pl_{\mu}: R^{W}\to R^{W_{\mu}}[-2](-1).
\end{equation*}

Let $P_\mu \subset G$ be the parabolic subgroup of $G$ corresponding to $\mu$. Pushforward along the proper smooth map $\BB P_\mu \rightarrow \BB G$ induces an $R^{W}$-linear map
\begin{equation}\label{eq:integration-map}
\int_{G/P_{\mu}}: R^{W_{\mu}}=\cohog{*}{\BB P_\mu} \to \cohog{*}{\BB G}[-2D_\mu](-D_\mu)=R^{W}[-2D_\mu](-D_\mu)
\end{equation}
where $D_\mu := \dim (G/P_{\mu}) = \tw{2\rho, \mu}$. 

Let $\y\in R^{W_{\mu}}$ be homogeneous of degree $2(D_\mu+1)$. Consider the map
\begin{eqnarray}\label{eq: nabla lambda eta}
\nb^{\y}_{\mu}&:& R^{W}\to R^{W}\\
&& f\mt \int_{G/P_{\mu}} \y \cdot \pl_{\mu}f.
\end{eqnarray}
Then $\nb^{\y}_{\mu}$ is a degree-preserving derivation.

\subsubsection{The operator $\ov\nb^{\y}_{\mu}$}\label{sssec:ov-nabla}

Since $\nb^{\y}_{\mu}$ is a derivation that carries $R^{W}_{+}$ to $R^{W}_{+}$, it also sends $R^{W}_{+}\cdot R^{W}_{+}$ to itself, hence induces a graded linear endomorphism of $\VV$
\begin{equation*}
\ov\nb^{\y}_{\mu}\in \End^{gr}(\VV).
\end{equation*}
In other words, $\ov\nb^{\y}_{\mu}$ restricts to a linear endomorphism of each $\VV_{d}$. The eigenvalues of $\ov\nb^{\y}_{\mu}$ are the crucial ``structure constants'' for the upcoming formulas for arithmetic volumes. We refer to them as \emph{eigenweights}. 

\begin{exam}\label{exam:eps}
In many examples of interest, $\VV_{2d}$ has dimension at most one, and then the eigenweights can be calculated in the following way. By the assumption, we can choose homogeneous free generators $f_{1},f_{2},\cdots, f_{n}$ of $R^{W}$ (which we know is a polynomial ring) of increasing degrees
\begin{equation*}
2d_{1}< 2d_{2}<\cdots < 2d_{n}.
\end{equation*}
This means $\VV=\bigoplus_{i=1}^{n}\VV_{2d_{i}}$ and $\VV_{2d_{i}}$ is one-dimensional spanned by the image of $f_{i}$. Then, for degree reasons, we necessarily have for each $i=1,\cdots, n$,
\begin{equation*}
\nb_{\mu}^{\y} (f_{i})=\e_{i}(\eta, \mu) f_{i}+\mbox{(polynomial in $f_{1},\cdots, f_{i-1}$)}.
\end{equation*}
The numbers $\{\e_{i}(\eta, \mu)\}_{i=1}^n$ are the eigenweights of $\ov\nb^{\y}_{\mu}$.
\end{exam}

\subsection{Example of calculating eigenweights}\label{ssec:GLn-lambda-minimal} We will calculate the operator $\nb^{\y}_{\mu}$ in the example $G=\GL_{n}$ and $\mu=(1,0,\cdots,0)$, so that $D_{\mu} = n-1$. We identify $R=\Qlbar[x_{1},\cdots, x_{n}]$ (cf. Example \ref{ex: GL}) with the obvious action of the Weyl group $W\cong S_n$ by permutations on the variables, so that 
\[
R^W  = \Qlbar[e_1, \ldots, e_n]
\]
where the $e_i$ are the elementary symmetric polynomials in the $x_i$.

\begin{prop}\label{prop:GLn-minimal-eps}
Let $G = \GL_n$ and $\mu = (1, 0, \ldots, 0)$. Choose $\eta = x_1^n \in R^{W_\mu}$. Then we have
\begin{equation*}
\nb_{\mu}^\eta(e_{i})= (-1)^{n-1} e_{i} \text{ for all } i=1, 2, \ldots, n.
\end{equation*}
In particular, we have
 $\epsilon_i(\eta, \mu) = (-1)^{n-1}$ for each $i = 1, \ldots, n$. 
\end{prop}

\begin{exam}
For $i=1$, we want to calculate 
\[
\nb_{\mu}^\eta(e_{1}) = \int_{G/P_\mu} x_1^n.
\]
An $R^W$-module basis for $R^{W_\mu}$ is given by the powers $x_1^0, \ldots, x_1^{n-1}$, and $\int_{G/P_\mu}$ extracts the coefficients of $(-x_1)^{n-1}$, by our convention that $-x_1$ corresponds to $\cO(1)$ on $G/P_\mu$ (cf. Example \ref{ex:parabolic-line}). We have the characteristic relation
\[
x_1^n - e_1 x_1^{n-1} + \ldots + (-1)^n e_n = 0.
\]
Thus 
\[
\int_{G/P_\mu} x_1^n = \int_{G/P_\mu} (e_1 x_1^{n-1} + \ldots )  = (-1)^{n-1} e_1 .
\]
For $i>1$, the combinatorics become much more complicated to analyze directly in this way, so we will first develop some machinery. 
\end{exam}

\subsubsection{Combinatorial formula for the integration map}
We will first give a general description of the integration map \eqref{eq:integration-map} in more concrete terms, which will be useful for explicit calculations. 

We maintain the notation of \S \ref{sssec:nabla-operator}, so $G$ is an arbitrary split reductive over $\F_q$ and $\mu \in \xcoch(T)$. Consider the $G$-equivariant top Chern class of the tangent bundle of $G/P_{\mu}$
\begin{equation}\label{eq:equivariant-top-Chern}
c_{D_\mu}(T_{G/P_{\mu}})=\frR_{\mu} :=\prod_{\substack{\alpha \in \Phi(G)  \\ \j{\a,\mu} < 0}}\a\in R^{W_{\mu}}
\end{equation}
where $\Phi(G)$ is the set of roots of $G$.

\begin{lemma}\label{lem:w-average-integration}For any $f\in R^{W_{\mu}}$, we have
\begin{equation}\label{Av w f/R}
\int_{G/P_{\mu}}f=\sum_{w\in W/W_{\mu}}w(f/\frR_{\mu}).
\end{equation}
Here the sum is over a set of representatives of $W/W_{\mu}$, which is well-defined since $f/\frR_{\mu}$ is $W_{\mu}$-invariant. (The element $f/\frR_{\mu}$ is understood in the fraction field of $R$, yet the sum above will lie in $R^{W}$.)
\end{lemma}
\begin{proof}
Suppose the Lemma is proved whenever the parabolic subgroup is a Borel. Let $B$ be the standard Borel subgroup contained in $P_\mu$. Let $B_\mu := B \cap L_\mu$, a Borel subgroup of $L_\mu$. Let $\frR^\mu$ be the analogous construction for $B_\mu < L_\mu$ (i.e., the $L_\mu$-equivariant top Chern class of $L_\mu/B_\mu)$, so that 
\[
\frR :=\prod_{\substack{\alpha \in \Phi(G) \\ \a < 0 }} \a  =  \frR_\mu \frR^\mu \in R.
\]
By assumption, we have a commutative diagram
\[
\begin{tikzcd}
\rH_G^*(G/B) \ar[d] \ar[r, equals] & \rH_{L_\mu}^*(L_\mu/B_\mu) \ar[d, twoheadrightarrow] \ar[r,"\sim"] & R \ar[d, twoheadrightarrow, "\sum_{w \in W_\mu} w((\cdot) /\frR^\mu)"] \\
\rH_G^*(G/P_\mu) \ar[d] \ar[r, equals] &  \rH_{L_\mu}^*(\pt) \ar[r, "\sim"] & R^{W_\mu} \ar[d, twoheadrightarrow]  \\
\rH_G^*(\pt) \ar[rr, "\sim"] & & R^W
\end{tikzcd}
\]
where the composite right vertical arrow is $\sum_{w \in W} w((\cdot)/\frR)$ (which is the content of the lemma for $P_{\mu}=B$). The description of the map $R\to R^{W_{\mu}}$ in the above diagram is the content of the lemma for $G=L_{\mu}$ with its Borel $B_{\mu}$. This implies that the second right vertical arrow must be $\sum_{w \in W/W_\mu} w((\cdot)/\frR_\mu)$.

So it suffices to verify the Lemma in the case where $P_\mu = B$ is a Borel subgroup. By computing on simple reflections, which reduces to an $\SL_{2}$ calculation, one checks that $\int_{G/B}$ is sign-equivariant under the $W$-action:
\begin{equation*}
\int_{G/B}wf=(-1)^{\ell(w)}\int_{G/B}f, \quad \forall f\in R.
\end{equation*}
Let $R^{\sgn}$ be the sign-isotypic component of $R$ under the $W$-action. Let
\begin{eqnarray*}
\Av_{\sgn}: R&\to& R^{\sgn}\\
f &\mapsto& \frac{1}{|W|}\sum_{w\in W}(-1)^{\ell(w)}wf
\end{eqnarray*}
be the projector to $R^{\sgn}$. Since $\int_{G/B}$ is sign-equivariant, it has a factorization as a composition
\begin{equation*}
\int_{G/B}: R\xr{\Av_{\sgn}}R^{\sgn}\xr{\io}R^{W}
\end{equation*}
for a unique $R^{W}$-linear map $\io$. It is well-known that $R^{\sgn}$ is a free $R^{W}$-module with basis $\frR$; hence, to determine $\io$, it suffices to compute $\io(\frR)$. 

By the definition of $\int_{G/B}$ as a pushforward, we have
\begin{equation*}
\int_{G/B}\frR=\int_{G/B}c_{\dim G/B}(T_{G/B})=\chi(G/B)=|W|.
\end{equation*}
This implies $\io(\frR)=|W|$, which in turn implies that
\begin{equation*}
\int_{G/B}f=\io(\Av_{\sgn}(f))=|W|\Av_{\sgn}(f)/\frR=\sum_{w\in W} w(f/\frR),
\end{equation*}
proving \eqref{Av w f/R}.
\end{proof}

\subsubsection{Proof of Proposition \ref{prop:GLn-minimal-eps}}\label{sssec:GLn-minimal-eps-proof} Now we return to the setup of Proposition \ref{prop:GLn-minimal-eps}, so $G = \GL_n$ and $\mu = (1,0, \ldots, 0) \in \xcoch(T)$.

We have arranged that $\pl_{\mu}=\pl_{x_{1}}$. We choose $\eta = x_1^n$. In this case, $\BB P_\mu \rightarrow \BB G$ is a $\PP^{n-1}$-bundle with $x_1$ being the first Chern class of the tautological bundle $\cO(-1)$ in our conventions (cf. Example \ref{ex:parabolic-line}).

Each $e_i$ spans the (1-dimensional) eigenspaces $\VV_{2d_i}$, and we want to calculate 
\begin{equation*}
\nb_{\mu}^\eta (e_i)= \int_{G/P_\mu} x_1^n \partial_{x_1} e_i.
\end{equation*}

Let $\wh{e}_i \in R$ be the $i$th elementary symmetric polynomial in $x_2, \ldots, x_n$ (omitting $x_1$). Note that $\partial_{x_1} (e_i) = \wh{e}_{i-1}$, so that we are interested in calculating $\int_{G/P_\mu} x_1^n \cdot \wh{e}_{i-1}$. 

Let $t$ be a formal variable. Then we have 
\begin{equation}\label{eq:elementary-symmetric-factorize}
\prod_{j=2}^n (t-x_j) = \sum_{i=0}^{n-1} (-1)^{i} t^{n-1-i} \wh{e}_i .
\end{equation}
Identifying $W = S_n$ with the permutations of the $\{x_i\}$, the subgroup $W_\mu = S_{n-1}$ is the stabilizer of $x_1$. We claim that 
\begin{equation}\label{eq:gln-lagrange-interpolation}
\sum_{w \in S_n/S_{n-1}} w \left(x_1^n \frac{\prod_{j=2}^n (t-x_j)}{\prod_{j=2}^n (x_1-x_j) } \right) =  t^n -  \prod_{i=1}^n (t-x_i).
\end{equation}
Indeed, the left side is the Lagrange interpolation formula for the unique degree $n-1$ polynomial $P(t)$ such that $P(x_i) = x_i^n$ for each $i=1, \ldots, n$, and the right side visibly has this same property, so they must agree. 

Comparing coefficients of $t^{n-i}$ in \eqref{eq:gln-lagrange-interpolation}, and using \eqref{eq:elementary-symmetric-factorize}, yields
\begin{equation}\label{eq:eigenweight-minimal-example}
\sum_{w \in S_n/S_{n-1}} w \left(\frac{x_1^n \cdot  \wh{e}_{i-1}}{\prod_{j=2}^n (x_1-x_j) } \right) = e_i.
\end{equation}
Note that the denominator $\prod_{j=2}^n (x_1-x_j)$ is $(-1)^{n-1} \frR_\mu$ from \eqref{eq:equivariant-top-Chern}. Invoking Lemma~\ref{lem:w-average-integration}, we find that 
\begin{equation}\label{eq:integration-GLn-colength-one}
\nb_{\mu}^\eta(e_{i}) = \int_{G/P_\mu} x_1^n \wh{e}_{i-1} {=} (-1)^{n-1} \sum_{w \in S_n/S_{n-1}} w \left(\frac{x_1^n \cdot  \wh{e}_{i-1}}{\prod_{j=2}^n (x_1-x_j) } \right) \stackrel{\eqref{eq:eigenweight-minimal-example}}= (-1)^{n-1}e_i,
\end{equation}
as desired. \qed

\subsection{Convergence of trace} 

Since the definition of arithmetic volume involves taking trace on an infinite-dimensional vector space, some analysis is required to see that the trace actually converges to a well-defined number; we undertake this analysis now in a more general context.

\subsubsection{More general arithmetic volumes}

Fix an admissible sequence $\mu=(\mu_1,\cdots, \mu_r)$ of dominant minuscule coweights of $T$. Recall $N=\sum_{j=1}^r(D_{\mu_j}+1)$. For any $\th\in \cohog{2N}{{}^\om\Hk^\mu_G}$, we can define the operator $\G^\th_\mu$ on $\cohog{*}{\Bun^\om_G}$ similarly as in \S\ref{ssec:split-Gamma} as the composition
\begin{equation*}
    \G^\th_\mu:=h_{0*}(h_r^*(-)\cup \th): \cohog{*}{\Bun^\om_G}\to \cohog{*}{\Bun^\om_G}.
\end{equation*}
Similarly we have the version with compact support
\begin{equation*}
    {}_c\G^\th_\mu:=h_{r*}(h_0^*(-)\cup \th): \cohoc{*}{\Bun^\om_G}\to \cohoc{*}{\Bun^\om_G}.
\end{equation*}

\begin{lemma}\label{l:trace cGamma vs Gamma}
    For each $i\in \ZZ$, we have
    \begin{equation*}
        \Tr({}_c\G^\th_\mu\circ \Frob\mid \cohoc{i}{\Bun^\om_G})=q^{\dim \Bun_G}\Tr(\Frob^{-1}\circ\G^\th_\mu\mid\cohog{2\dim \Bun_G-i}{\Bun^\om_G}).
    \end{equation*}
\end{lemma}
\begin{proof}
    By the analog of \eqref{Gamma adj}, under the Poincar\'e duality pairing \eqref{PD}, the endomorphism ${}_{c}\Gamma^{\th}_{\mu}$ of $\cohoc{i}{\Bun^\om_G}$ is adjoint to the endomorphism $\Gamma^{\th}_{\mu}$ of $\cohog{2\dim \Bun_G-i}{\Bun_G^\om}$. Also, the automorphism $\Frob$ on $\cohoc{i}{\Bun^\om_G}$ is adjoint to the automorphism $q^{\dim \Bun_G}\Frob^{-1}$ on $\cohog{2\dim \Bun_G-i}{\Bun^\om_G}$. The conclusion follows.
\end{proof}

\begin{prop}\label{p:trace conv} Fix an embedding of fields $\io: \Qlbar\incl \CC$. For any $\th\in \cohog{2N}{{}^\om\Hk^\mu_G}$, the two series of complex numbers
    \begin{equation}\label{alt trace converge}
        \sum_{i\in \ZZ}(-1)^i\io\left(\Tr(\Frob^{-1}\c \G^\th_\mu|\cohog{i}{\Bun_G^\om})\right) \mbox{ and }\sum_{i\in \ZZ}(-1)^i\io\left(\Tr({}_c\G^\th_\mu\circ\Frob|\cohoc{i}{\Bun_G^\om})\right)
    \end{equation}
    are absolutely convergent. We denote their sums by 
    \begin{equation*}
        \Tr_\io(\Frob^{-1}\c \G^\th_\mu|\cohog{*}{\Bun_G^\om})\mbox{ and }\Tr_\io({}_c\G^\th_\mu\circ\Frob|\cohoc{*}{\Bun_G^\om})
    \end{equation*}
    respectively.
\end{prop}

By Lemma \ref{l:trace cGamma vs Gamma}, it suffices to prove the absolute convergence of the first series in \eqref{alt trace converge}, whose proof occupies the rest of the subsection.


\subsubsection{The Ran grading}\label{sss:Ran gr}

For $m\ge0$ define $\frR_m\cohog{*}{\Bun^{\om}_{G}}\subset \cohog{*}{\Bun_G^\om}$ to be the subspace spanned by elements of the form
\begin{equation*}
f_{1}^{z_{1}}f_{2}^{z_{2}}\cdots f_{s}^{z_{s}} \quad \mbox{such that $\sum_{i}|z_{i}|=m$ and $\deg(f_{i})>|z_{i}|$ for each $i=1,\cdots, s$.}
\end{equation*}
Here the $f_{i}\in R^{W}_{+}$ are homogeneous. We note that $\deg(f_{i})>|z_{i}|$ automatically holds unless $\deg(f_{i})=2=|z_{i}|$, which is not allowed in the description of the Atiyah--Bott formula \S\ref{sss:concrete AB} (for in this case $f_{i}^{z_{i}}\in \Qlbar$). 

From the definition, it is clear that $\frR_m\cohog{*}{\Bun^{\om}_{G}}$ are multiplicative under cup product in that
\begin{equation*}
    \frR_m\frR_l\subset \frR_{m+l}.
\end{equation*}

\begin{lemma}\label{l:Ran}
    We have a decomposition
    \begin{equation}\label{eq:cohog-ran-decomposition}
        \cohog{*}{\Bun_G^\om}\cong \bigoplus_{m\ge0}\frR_m\cohog{*}{\Bun_G^\om}.
    \end{equation}
    We call the resulting $\ZZ_{\ge0}$-grading on $\cohog{*}{\Bun_G}$ the {\em Ran grading} \footnote{The nomenclature comes from the fact that $\cohog{*}{\Bun^\om_G}$ is the factorization homology of $\cohog{*}{\BB G}$, a key observation of \cite{GL14}.}.
\end{lemma}
\begin{proof} 

    Abbreviate $\frR_m\cohog{*}{\Bun_G^\om}$ simply by $\frR_m$. Let $f_1,\cdots, f_n$ be a set of free homogeneous generators of $R^W$. Let $z_0=1 \in \homog{0}{X}$, $z_1,\cdots, z_{2g}\in\homog{1}{X}$, and $z_{2g+1}=[X]\in \homog{2}{X}$ be a homogeneous basis of $\homog{*}{X}$. Then by Atiyah-Bott formula (see \S\ref{sss:concrete AB}), $\cohog{*}{\Bun_G^\om}$ has a $\Qlbar$-basis consisting of
    \begin{equation}\label{monomial basis}
        \prod_{i=1}^n\prod_{j=0}^{2g+1} (f_i^{z_j})^{e_{ij}} \quad \text{ where } \quad 
        e_{ij} \begin{cases}
            \in \ZZ_{\ge0} & \mbox{if $j=0$ or $j=2g+1$},\\
            \in \{0,1\} & \mbox{if $j=1,\cdots,2g$}.
        \end{cases}
    \end{equation}
    Let $\frR'_m$ be the subspace spanned by those monomials of the form \eqref{monomial basis} where 
    \[
    \sum_{i=1}^n\Big( \sum_{j=1}^{2g} e_{ij} + 2e_{i,2g+1} \Big)=m.
    \]
    Clearly $\cohog{*}{\Bun_G^\om}$ is the direct sum of $\frR'_m$. It remains to show that $\frR_m=\frR'_m$. The inclusion $\frR'_m\subset \frR_m$ is clear. Conversely, to show that $\frR_m\subset \frR'_m$, we observe that $\frR'_m$ are also multiplicative, so it suffices to show that
    \begin{equation}\label{fz in R'}
        f^z\in \frR'_{|z|}
    \end{equation}
    for any $f\in R^W$ and homogeneous $z\in \homog{*}{X}$. 
    
    Any $f\in R^W$ can be written as  $f=P(f_1,\cdots, f_n)$ for a unique polynomial $P$ in $n$ variables. We prove \eqref{fz in R'} by induction on the degree of $P$. The case $\deg (P)\le 1$ is clear. Suppose \eqref{fz in R'} holds for all $f$ whose degree in $f_1,\cdots, f_n$ is $<d$. Let $f=P(f_1,\cdots, f_n)$ for $\deg(P)=d$. By linearity, it suffices to treat the case where $P$ is a monomial of degree $d>1$. In this case we write $P=P_1P_2$ where $P_i$ are monomials of degree $<d$. Let $g_1=P_1(f_1,\cdots, f_n)$ and $g_2=P_2(f_1,\cdots, f_n)$. Then by \eqref{contraction f1f2} we have
    \begin{equation}\label{fz decomp}
        f^z=(g_1g_2)^z=\sum_{j}(g_1)^{z'_j}(g_2)^{z''_j}
    \end{equation}
    where $\D_*(z)=\sum_jz'_j\ot z''_j\in \homog{|z|}{X\times X}$. By induction hypothesis $(g_1)^{z'_j}\in \frR'_{|z'_j|}$ and $(g_2)^{z''_j}\in \frR'_{|z''_j|}$. Therefore \eqref{fz decomp} implies $f^z\in \frR'_{|z|}$.  This shows $\frR_m=\frR'_m$, completing the proof of the lemma. 
\end{proof}

For each $m$, define the {\em Ran filtration} $F_\bu\cohog{*}{\Bun_G^\om}$ on $\cohog{*}{\Bun_G^\om}$ by
    \begin{equation}\label{eq:filtrant-ran-decomposition}
F_m\cohog{*}{\Bun_G^\om} \cong \bigoplus_{m'\le m}\frR_{m'}\cohog{*}{\Bun_G^\om}.
\end{equation}

\subsubsection{Monomial basis for $\cohog{*}{\Bun^\om_G}$} As in the proof of Lemma \ref{l:Ran}, we now fix a choice of homogeneous free generators $f_1,\cdots, f_n$ for $R^W$. Also choose a homogeneous basis $z_0=1,z_1,\cdots, z_{2g}, z_{2g+1}=[X]$ for $\homog{*}{X}$ that are also eigenvectors for the Frobenius action on $\homog{*}{X}$.\footnote{We are using here the fact that $\Frob$ acts semisimply on $\homog{1}{X}$; however, without using this fact, the argument can still be made with slight modification towards the end.} Let $\frB$ be the set of elements of the form \eqref{monomial basis}; they form a basis for $\cohog{*}{\Bun_G^\om}$ that we call \emph{the monomial basis}.
We use the same notation $\frB$ for varying components $\om\in \pi_0(\Bun_G)$.


For $\a\in \frB$ in the form \eqref{monomial basis}, we call 
\begin{equation}\label{naive deg}
d:=\sum_{1\le i\le n, \ 0\le j\le 2g+1}e_{ij}
\end{equation}
the {\em naive degree} of $\a\in \frB$.

\subsubsection{Hecke action on tautological classes} 
From here until the end of Lemma \ref{l:bound trunc Gamma}, we will fix a dominant minuscule coweight $\mu$ of $G$, and consider the one-step Hecke stack ${}^{\om} \Hk^{\mu}_G$.  

First we need to prove a technical result about the behavior of tautological classes under Hecke correspondences. We follow the notation of \S \ref{sssec:tautological-classes}. 

\begin{prop}\label{p:master eqn}
For $f\in R^{W}$ and $z\in \homog{|z|}{X}$, we have
\begin{equation*}
h_{1}^{*}(f^{z})-h_{0}^{*}(f^{z})=\PD(z)\pl_{\mu}(f)+\begin{cases}(1-g)\j{z,\xi}\xi\pl_{\mu}^{2}(f) & |z|=2, \\ 
0 & \mbox{otherwise}.\end{cases}
\end{equation*}

\end{prop}
\begin{proof} We apply Theorem \ref{th:master gen} to $S:=X\times {}^{\om} \Hk^{\mu}_G$, the two $G$-bundles $\cF_{0}$ and $\cF_{1}$ with the modification \eqref{univ modif} along the divisor $D=\G(p_{X})$. Note that $D\incl S$ is the pullback of the diagonal $\D_{X}\incl X\times X$ under $(q_{X},p_{X}\c p_{\Hk}): S\to X\times X$. Therefore the normal bundle of $D$ is the pullback $p_{X}^{*}T_{X}$. In particular, its first Chern class is $\nu_{D}=(2-2g)\xi$, and $\nu_{D}^{2}=0$. 

Let $\cH$ be the $L_\mu$-torsor on $\Gamma(p_X)$ obtained from the restriction $\cF_0|_{\Gamma(p_X)}$ (cf. \S\ref{sssec: canonical parabolic reduction}). Since $\nu_D^2  =0 $ as noted above, Theorem \ref{th:master gen} says in this case that 
\begin{equation}\label{fF1 fF0}
f(\cF_{1})-f(\cF_{0})=i_{D!}\left((\pl_{\mu}f)(\cH)+\frac{1}{2}(\pl^{2}_{\mu}f)(\cH)\cdot \nu_{D}\right).
\end{equation}
By the Cartesian square
\begin{equation*}
\begin{tikzcd}[column sep = huge]
D \ar[d, "i_D"] \ar[r] & \Delta_X \ar[d] \\
X \times_k {}^{\om} \Hk^{\mu}_G   \ar[r, "\Id \times p_X"] & X \times X
\end{tikzcd}
\end{equation*}
we have 
\begin{equation*}
i_{D!}(\pl_{\mu}f)(\cH)=i_{D!}i_{D}^{*}(p_{\Hk}^{*}(\pl_{\mu}f)(\cH))=[D]\cdot p_{\Hk}^{*}((\pl_{\mu}f)(\cH))=[\D_{X}]\cdot p_{\Hk}^{*}((\pl_{\mu}f)(\cH)).
\end{equation*}
By our convention \eqref{eq: parabolic embedding}, we are viewing $R^W \inj \rH^*(\Hk_\mu^\om)$ embedded via the classifying map for $\cH$, so we can simply write $\pl_{\mu}f(\cH)$ as $\pl_{\mu}f\in \cohog{*}{{}^{\om} \Hk^{\mu}_G}$, which we will do below. Similarly, using $\nu_D = (2-2g) \xi$ we have
\begin{align*}
i_{D!}\Big(\frac{1}{2}(\pl^{2}_{\mu}f)\nu_{D} \Big) & =i_{D!}i_{D}^{*}(p_{\Hk}^{*}((1-g)\xi\pl^{2}_{\mu}f))=[D]\cdot p_{\Hk}^{*}((1-g)\xi\pl^{2}_{\mu}f)\\
&= (1-g)[\D_{X}]\cdot(1\ot\xi) p_{\Hk}^{*}(\pl^{2}_{\mu}f)=(1-g)(\xi\ot\xi)p_{\Hk}^{*}(\pl^{2}_{\mu}f)
\end{align*}
where in the last step we used that $[\Delta_X] \cdot (1 \otimes \xi) = (\xi \otimes \xi ) \in \cohog{4}{X \times X}(2)$. Plugging these into \eqref{fF1 fF0} we get
\begin{equation*}
f(\cF_{1})-f(\cF_{0})=[\D_{X}]\cdot p_{\Hk}^{*}(\pl_{\mu}f)+(1-g)(\xi\ot\xi)p_{\Hk}^{*}(\pl^{2}_{\mu}f).
\end{equation*}
Now we expand both sides using K\"unneth formula for $S=X\times {}^\omega \Hk^{\mu}_G$, and use $[\D_{X}]=\xi\ot 1-\b+1\ot \xi$; then contracting with $z \in \rH_*(X)$ gives the desired formula.
\end{proof}

\subsubsection{Ran filtration for Hecke stack}\label{sssec:split-filtration}
We also define an increasing filtration on $\cohog{*}{{}^{\om} \Hk^{\mu}_G}$ as follows. Let
\begin{align*}
F_{m}\cohog{*}{{}^{\om} \Hk^{\mu}_G} & =0, \quad m<-2,\\
F_{-2}\cohog{*}{{}^{\om} \Hk^{\mu}_G}& =p_{X}^{*}\cohog{2}{X}\cdot R^{W_{\mu}},\\
F_{-1}\cohog{*}{{}^{\om} \Hk^{\mu}_G}& =p_{X}^{*}\cohog{\ge1}{X}\cdot R^{W_{\mu}},\\
F_{0}\cohog{*}{{}^{\om} \Hk^{\mu}_G}&=p_{X}^{*}\cohog{*}{X}\cdot R^{W_{\mu}}, \quad \mbox{ which is the image of \eqref{taut class Hk}}\\
F_{m}\cohog{*}{{}^{\om} \Hk^{\mu}_G}&=F_{0}\cohog{*}{{}^{\om} \Hk^{\mu}_G}\cdot h_{0}^{*}F_{m}\cohog{*}{\Bun^{\om}_{G}}, \quad m \geq 1.
\end{align*}
Both filtrations $F_\bu\cohog{*}{\Bun_G^\om}$ and $F_\bu\cohog{*}{{}^\om\Hk^\mu_G}$ are also multiplicative, i.e.,
\begin{equation*}
F_{i}\cdot F_{j}\subset F_{i+j}.
\end{equation*}
In particular, $F_{0}\cohog{*}{\Bun^{\om}_{G}}$ (resp. $F_{0}\cohog{*}{{}^{\om} \Hk^{\mu}_G}$) is a subring of $\cohog{*}{\Bun^{\om}_{G}}$ (resp. $\cohog{*}{{}^{\om} \Hk^{\mu}_G}$).

\begin{lemma}\label{l:h0}
The map $h_{1}^{*}$ sends  $F_{m}\cohog{*}{\Bun^{\om'}_{G}}$ to $F_{m}\cohog{*}{{}^{\om} \Hk^{\mu}_G}$ for all $n\in\ZZ$.
\end{lemma}
\begin{proof}
Consider $f^{z}\in F_{|z|}\cohog{*}{\Bun^{\om'}_{G}}$. Then by Proposition \ref{p:master eqn} we have
\begin{equation}\label{diff fz}
h_{1}^{*}(f^{z})-h_{0}^{*}(f^{z})\in p_{X}^{*}\cohog{\ge 2-|z|}{X}\cdot R^{W_{\mu}}=F_{|z|-2}\cohog{*}{{}^{\om} \Hk^{\mu}_G}.
\end{equation}
Since $h_{0}^{*}(f^{z}) \in F_{|z|}\cohog{*}{{}^{\om} \Hk^{\mu}_G}$ by definition, \eqref{diff fz} implies that $h_{1}^{*}(f^{z})\in F_{|z|}\cohog{*}{{}^{\om} \Hk^{\mu}_G}$.

By construction, both filtrations on $\cohog{*}{\Bun^{\om'}_{G}}$ and on $\cohog{*}{{}^{\om} \Hk^{\mu}_G}$ are multiplicative. Since $h_{1}^{*}$ preserves the filtrations on a set of ring generators, it preserves the filtrations on all elements of $\cohog{*}{\Bun^{\om'}_{G}}$.
\end{proof}

Theorem \ref{thm: atiyah-bott formula} gives a canonical identification of the cohomology rings of different components of $\Bun_{G}$. Using this identification, it makes sense to define a map
\begin{eqnarray*}
\D_{\mu}: \cohog{*}{\Bun^{\om}_{G}}&\to& \cohog{*}{{}^{\om} \Hk^{\mu}_G}\\
\th &\mt& h_{1}^{*}\th-h_{0}^{*}\th.
\end{eqnarray*}

\begin{lemma}\label{l:diff h0 h1}
The map $\D_{\mu}$ sends $F_{m}\cohog{*}{\Bun^{\om}_{G}}$ to $F_{m-2}\cohog{*}{{}^{\om} \Hk^{\mu}_G}$ for all $n\in\ZZ$. The induced map on the associated graded 
\begin{equation*}
\Gr^F_\bu\D_{\mu}: \Gr^F_{\bu}\cohog{*}{\Bun^{\om}_{G}}\to \Gr^F_{\bu-2}\cohog{*}{{}^{\om} \Hk^{\mu}_G}
\end{equation*}
is a derivation with respect to the ring homomorphism 
\begin{equation*}
\Gr^F_\bu(h_{0}^{*})=\Gr^F_\bu(h_{1}^{*}): \Gr^F_{\bu}\cohog{*}{\Bun^{\om}_{G}}\cong\Gr^F_{\bu}\cohog{*}{\Bun^{\om'}_{G}} \to \Gr^F_{\bu}\cohog{*}{{}^{\om} \Hk^{\mu}_G}.
\end{equation*}
It is characterized by the property that
\begin{equation*}
(\Gr^F_{|z|}\D_{\mu})(f^{z})=\PD(z)\pl_{\mu}(f)\in \Gr^F_{|z|-2}\cohog{*}{{}^{\om} \Hk^{\mu}_G}.
\end{equation*}
\end{lemma}
\begin{proof}
By definition, a general element in $F_{m}\cohog{*}{\Bun^{\om}_{G}}$ is a linear combination of $f=f_{1}^{z_{1}}f_{2}^{z_{2}}\cdots f_{s}^{z_{s}}$ where $|z_{1}|+\cdots+|z_{s}|\le m$. We may write
\begin{equation}\label{diff prod}
h_{1}^{*}f-h_{0}^{*}f=\sum_{i=1}^{s}\left(\prod_{i'<i}h_{1}^{*}f_{i'}^{z_{i'}}\right)(h_{1}^{*}f_{i}^{z_{i}}-h_{0}^{*}f_{i}^{z_{i}})\left(\prod_{i''>i}h_{0}^{*}f_{i''}^{z_{i''}}\right).
\end{equation}
By Lemma \ref{l:h0} and \eqref{diff fz}, each summand above lies in $F_{\sum_{i'\ne i}|z_{i'}|+|z_{i}|-2}\cohog{*}{{}^{\om} \Hk^{\mu}_G}=F_{m-2}\cohog{*}{{}^{\om} \Hk^{\mu}_G}$. Passing to associated graded, we learn that $h_{0}^{*}f$ has the same image as $h_{1}^{*}f$ in $\Gr^F_{m}\cohog{*}{{}^{\om} \Hk^{\mu}_G}$, hence \eqref{diff prod} tells that $\D_{\mu}$ is a derivation after passing to associated graded. The calculation of $(\Gr^F_{|z|}\D_{\mu})(f^{z})$ follows from Proposition \ref{p:master eqn}.
\end{proof}

\begin{lemma}\label{l:h1 push}
The map $h_{0*}$ carries $F_{m-2}\cohog{*}{{}^{\om} \Hk^{\mu}_G} \rightarrow F_{m}\cohog{*-2D_\mu-2}{\Bun^{\om}_{G}}(-D_\mu-1)$ for all $m\in\Z$.
\end{lemma}
\begin{proof}
The canonical parabolic reduction fits into a commutative diagram
\begin{equation}\label{eq: hecke reduction cartesian}
\begin{tikzcd}
{}^\omega \Hk^\mu_G \ar[r] \ar[d, "p_X \times h_0"'] & X \times \BB P_\mu \ar[d] \\
X \times \Bun_G^{\omega} \ar[r, "{\ev^{\omega}}"]  \ar[d] & X \times \BB G
 \\
\Bun_G^{\omega} 
\end{tikzcd} 
\end{equation}
whose top square is moreover Cartesian, and whose vertical maps are fiber bundles for $G/P_\mu$. Therefore, using the notation $\int_{G/P_{\mu}}: R^{W_{\mu}}\to R^{W}$ introduced in \S\ref{sssec:nabla-operator}, we can rewrite $h_{0*}$ on the image of $\cohog{*}{X} \ot \cohog{*}{\BB P_\mu} = \cohog{*}{X}\ot R^{W_{\mu}}\subset \cohog{*}{{}^{\om} \Hk^{\mu}_G}$ as the composition
\begin{equation}\label{h0 int}
h_{0*}: \cohog{*}{X}\ot R^{W_{\mu}}\xr{\id\ot\int_{G/P_{\mu}}}\cohog{*}{X}\ot R^{W}\xr{\id \ot \ev^{\om, *}} \cohog{*}{X\times \Bun^{\om}_{G}}\xr{\int_{X}}\cohog{*}{\Bun^{\om}_{G}}.
\end{equation}
For $n\le 0$, we have
\begin{equation*}
h_{0*}F_{m}\cohog{*}{{}^{\om} \Hk^{\mu}_G}=h_{0*}(\upH^{\ge -m}(X)\ot R^{W_{\mu}})\subset \int_{X}(\upH^{\ge -m}(X)\cdot R^{W}).
\end{equation*}
Here $R^{W}$ is embedded into $\cohog{*}{\Bun_{G}^{\om}}$ via $\ev^{\om,*}$. For $\z\in \cohog{|\z|}{X}$ and $f\in R^{W}$, $\int_{X}(\z f)$ is a linear combination of $f^{z}$ where $|z|+|\z|=2$. This implies 
\begin{equation*}
\int_{X}(\upH^{\ge -m}(X)\cdot R^{W})\subset F_{m+2}.
\end{equation*}

Now consider the case $m>0$. We abbreviate $F_{m}\cohog{*}{\Bun^{\om}_{G}}$ simply as $F_{m}$.
Using that $h_{0*}$ is linear over $h_{0}^{*}\cohog{*}{\Bun_{G}^{\om}}$, and $h_{0*}F_{0}\cohog{*}{{}^{\om} \Hk^{\mu}_G}\subset F_{2}$ by the preceding paragraph, we conclude that
\begin{align*}
h_{0*}F_{m}\cohog{*}{{}^{\om} \Hk^{\mu}_G}& =h_{0*}(F_{0}\cohog{*}{{}^{\om} \Hk^{\mu}_G}\cdot h_{0}^{*}F_{m}) \\
& =h_{0*}F_{0}\cohog{*}{{}^{\om} \Hk^{\mu}_G}\cdot F_{m}
\subset F_{2}\cdot F_{m}\subset F_{m+2}.
\end{align*}
\end{proof}

\begin{cor}\label{c:Gamma increase Ran by 2} For any $\b\in\cohog{*}{X\times \BB P_\mu}$, $\G^{\b}_\mu$ sends $F_m\cohog{*}{\Bun_G^{\om'}}$ to $F_{\le m+2}\cohog{*}{\Bun_G^{\om}}$, for all $m\in\ZZ$. 
\end{cor}
\begin{proof}
Let $\a\in F_m\cohog{*}{\Bun_G^{\om'}}$. By Lemma \ref{l:h0}, $h_1^*\a\in F_{m}\cohog{*}{^\om\Hk_G^{\mu}}$. By definition, $\b\in F_{0}\cohog{*}{^\om\Hk_G^{\mu}}$, hence $h_1^*(\a)\b\in F_{m}\cohog{*}{^\om\Hk_G^{\mu}}$. Finally, by Lemma \ref{l:h1 push}, $\G^\b_\mu(\a)=h_{0*}(h_1^*(\a)\b)\in F_{\le m+2}\cohog{*}{^\om\Hk_G^{\mu}}$.
\end{proof}

\subsubsection{Truncation of $\G^\b_\mu$} Let $\b\in \cohog{*}{X\times \BB P_\mu}$ be homogeneous. For $i\in \ZZ$ let $\frR_{i}\G^{\b}_\mu: \cohog{*}{\Bun_G^{\om'}}\to \cohog{*}{\Bun_G^{\om}}$ be the Ran degree $i$ part of $\G^\b_\mu$: for $\a\in \frR_m \cohog{*}{\Bun_G^{\om'}}$, $\frR_i\G^\b_\mu(\a)$ is the Ran degree $m+i$ piece of $\G^\b_\mu(\a)$. Define $\frR_{\ge i}\G^{\b}_\mu: \cohog{*}{\Bun_G^{\om'}}\to \cohog{*}{\Bun_G^{\om}}$ to be the sum
\begin{equation*}
\frR_{\ge i}\G^{\b}_\mu=\sum_{i'\ge i}\frR_{i'}\G^{\b}_\mu.
\end{equation*}

The main estimate to prove the convergence of trace is the following.
\begin{lemma}\label{l:bound trunc Gamma}
Fix $i\ge0$. For a monomial basis element $\a\in \frB$, write $\frR_{\ge -i}\G^\b_\mu$ as a linear combination of the monomial basis for $\cohog{*}{\Bun_G^{\om}}$:
\begin{equation*}
\frR_{\ge -i}\G^\b_\mu(\a)=\sum_{\a'\in \frB}c^{\a'}_{\a}\a'
\end{equation*}
for $c^{\a'}_{\a}\in\CC$. Then there is $C_{\b,\mu,i}>0$ depending only on $\b,\mu,i$ (and not on $\a$) such that
\begin{equation*}
\sum_{\a'\in \frB}|c^{\a'}_{\a}|_{\io}\le C_{\b,\mu,i}(d+1)^{i/2+1}, \mbox{ for all $\a\in \frB$ with naive degree $d$}. 
\end{equation*}
Here $|\cdot|_{\io}$ denotes the absolute value under the embedding $\io:\Qlbar\to \CC$.
\end{lemma}
\begin{proof}
Write $\a=\prod_{s=1}^{d}f_{i_s}^{z_{j_s}}$, where $d$ is the naive degree. Let $m=\sum_{s=1}^{d}|z_{j_s}|$ be the Ran degree of $\a$. Let $\D_\mu(f^z)=h_1^*(f^z)-h_0^*(f^z)$. Expanding $h_1^*\a=\prod_{s}h_1^*(f_{i_s}^{z_{j_s}})=\prod_s(h_0^*(f_{i_s}^{z_{j_s}})+\D_\mu(f_{i_s}^{z_{j_s}}))$, we get
\begin{eqnarray}
\notag
\G^\b_\mu(\a)&=&h_{0*}(h_1^*\a \cdot \b)=\sum_{J\subset \{1,2,\cdots, d\}}\pm h_{0*}\left(h_0^*(\prod_{s\notin J}f_{i_s}^{z_{j_s}})\cdot \prod_{s\in J}\D_\mu(f_{i_s}^{z_{j_s}})\b\right)\\
\label{expand Gamma}&=&\sum_{J\subset \{1,2,\cdots, d\}}\pm \left(\prod_{s\notin J}f_{i_s}^{z_{j_s}}\right)h_{0*}(\prod_{s\in J}\D_\mu(f_{i_s}^{z_{j_s}})\b).
\end{eqnarray}
By Lemma \ref{l:diff h0 h1}, $\prod_{s\in J}\D_\mu(f_{i_s}^{z_{j_s}})\in F_{\sum_{s\in J}(|z_{j_s}|-2)}\cohog{*}{{}^\om\Hk^\mu_G}$. Since $\b\in F_{0}\cohog{*}{{}^\om\Hk^\mu_G}$, by Lemma \ref{l:h1 push}, we have
\begin{equation*}
h_{0*}(\prod_{s\in J}\D_\mu(f_{i_s}^{z_{j_s}})\b)\in F_{\le 2+\sum_{s\in J}(|z_{j_s}|-2)}\cohog{*}{{}^\om\Hk^\mu_G}.
\end{equation*}
Therefore, the $J$-summand of \eqref{expand Gamma} lies in $F_{\le m-2|J|+2}\cohog{*}{\Bun_G^{\om}}$. Therefore, in computing $\frR_{\ge -i}\G^\b_\mu(\a)$, we only need to sum over those $J\subset \{1,2,\cdots, d\}$ with $m-2|J|+2\ge m-i$, i.e., $|J|\le i/2+1$. The number of such $J$ is $\le (d+1)^{i/2+1}$. For fixed $J$, writing  $h_{0*}(\prod_{s\in J}\D_\mu(f_{i_s}^{z_{j_s}})\b)$ as a sum of elements in $\frB$, the sum of absolute values of coefficients have an upper bound $C_{\b,\mu,i}$, since there are only finitely many possibilities for $\prod_{s\in J}\D_\mu(f_{i_s}^{z_{j_s}})$ when $|J|$ is bounded by $i$. Therefore for each $|J|\le i/2+1$, the corresponding summand in \eqref{expand Gamma} will contribute at most $C_{\b,\mu,i}$ to the total sum of absolute values of coefficients of $\frR_{\ge -i}\G^\b_\mu(\a)$, giving the desired upper bound.
\end{proof}

\subsubsection{Proof of Proposition \ref{p:trace conv}--first reductions}
The map $(p_1,\cdots, p_r, h_0): {}^\om\Hk^\mu_G\to X^r\times \Bun_G^\om$ is an iterated fibration with fibers $G/P_{\mu_j}$.  Therefore, as a module over $h_0^*\cohog{*}{\Bun_G^\om}$, $\cohog{*}{{}^\om\Hk^\mu_G}$ is generated by elements of the form $\b_1\cdots\b_r$, where $\b_j\in \cohog{*}{X\times \BB P_{\mu_j}}$ is pulled back to ${}^\om\Hk^\mu_G$ via $(p_j,\ev_j): \Hk^\mu_G\to X\times \BB P_{\mu_j}$.
It is therefore sufficient to treat the case where $\th=(h_0^*\a)\b_1\cdots\b_r$, where $\a\in \frB$ is a monomial basis element for $\cohog{*}{\Bun_G^\om}$ and $\b_i\in \cohog{*}{X\times \BB P_{\mu_i}}$. In this case $\G^\th_\mu$ is the composition
\begin{equation*}
        \G^\th_\mu=(\cup \a)\c \G^{\b_1}_{\mu_1}\c\cdots\c \G^{\b_r}_{\mu_r}.
\end{equation*}
Here $\cup\a$ is the endomorphism of $\cohog{*}{\Bun_G^\om}$ given by cup product with $\a$.

Below we base change all $\Qlbar$-vector spaces to $\CC$-vector spaces using $\io$. In particular, $\G^\th_\mu$ can be represented as a matrix with entries in $\CC$ under the monomial basis of $\cohog{*}{\Bun_G^\om}$. We can therefore talk about the diagonal matrix coefficients of $\G^\th_\mu$. 

\begin{lemma}\label{l:bound diag entry} Let $t$ be the Ran degree of $\a$, and let $u=r(t/2+r)$. There exists $C_{\th,\mu}>0$ (depending on $\th$ and $\mu$) such that, when writing $\G^\th_\mu$ as a matrix using the monomial basis $\frB$, its diagonal entry at any $\a'\in \frB$ with naive degree $d\ge 0$ has absolute value bounded above by $C_{\th,\mu} (d+1)^{u}$.
\end{lemma}
\begin{proof} Let $i=t+2(r-1)$. Write each $\G^{\b_j}_{\mu_j}$ as $\frR_{\ge -i}\G^{\b_j}_{\mu_j}+\frR_{<-i}\G^{\b_j}_{\mu_j}$ (where $\frR_{<-i}\G^{\b_j}_{\mu_j}$ is the part of $\G^{\b_j}_{\mu_j}$ that decreases the Ran degree by more than $i$). We expand
\begin{equation*}
\G^\th_\mu=(\cup \a)\c (\frR_{\ge -i}\G^{\b_1}_{\mu_1}+\frR_{< -i}\G^{\b_1}_{\mu_1})\c\cdots\c (\frR_{\ge -i}\G^{\b_r}_{\mu_r}+\frR_{<-i}\G^{\b_r}_{\mu_r})
\end{equation*}
into a sum of compositions of $\frR_{\ge -i}\G^{\b_j}_{\mu_j}$ and $\frR_{< -i}\G^{\b_j}_{\mu_j}$ with $\cup \a$. If $\frR_{< -i}\G^{\b_j}_{\mu_j}$ appears at the $j$th place, then the corresponding term will send $\frR_{m}\cohog{*}{\Bun_G^\om}$ to $\frR_{< m-i+2(r-1)+t}\cohog{*}{\Bun_G^\om}=\frR_{<m}\cohog{*}{\Bun_G^\om}$, because each of the other terms $\frR_{\ge -i}\G^{\b_{j'}}_{\mu_{j'}}$ (there are $\le r-1$ of them) can at most increase the Ran grading by $2$, and $\cup \a$ increases the Ran grading by $m$. Therefore, under the monomial basis, the diagonal entries of $\G^\th_\mu$ are the same as the diagonal entries of 
\begin{equation}\label{trunc Gamma r}
(\cup \a)\c \frR_{\ge -i}\G^{\b_1}_{\mu_1}\c\cdots\c \frR_{\ge -i}\G^{\b_r}_{\mu_r}.
\end{equation}
Now if $\a'\in \frB$ has naive degree $d$, by iterative use of Lemma \ref{l:bound trunc Gamma}, the sum of absolute values of coefficients of $\frR_{\ge -i}\G^{\b_1}_{\mu_1}\c\cdots\c \frR_{\ge -i}\G^{\b_r}_{\mu_r}(\a')$ in terms of $\frB$ is bounded above by 
\begin{equation}\label{poly bound}
\prod_{s=1}^{r}\left(C_{\b_s,\mu_s,i}(d+1+2(s-1))^{i/2+1}\right).
\end{equation}
Since $\a\in \frB$, all matrix entries of $\cup\a$ in terms of the basis $\frB$ are either $0$ or $1$, we conclude that all diagonal entries of \eqref{trunc Gamma r} are bounded by \eqref{poly bound}. Therefore the same is true for the diagonal entries of $\G^\th_\mu$. For a suitable constant $C_{\th,\mu}$ depending only on $\th$ and $\mu$, one can bound \eqref{poly bound} from above by $C_{\th, \mu} (d+1)^{r(i/2+1)}=C(d+1)^u$. 
\end{proof}

\subsubsection{Finish of the proof of Proposition \ref{p:trace conv}}
Now decompose $\cohog{*}{\Bun_G^\om}$ according to the naive degrees of monomial basis elements:
\begin{equation*}
\cohog{*}{\Bun_G^\om}=\bigoplus_{d\ge0}\frN_d,
\end{equation*}
where $\frN_d$ is the span of $\a'\in \frB$ with naive degree $d$. Note that this grading splits the augmentation filtration on $\cohog{*}{\Bun_G^\om}$ introduced in \S\ref{sss:aug fil}. All the generators $f_i^{z_j} \in \frB$ have cohomological degrees between $2$ and $2 d_n$, where $d_n$ is the largest degree of the generators $\{f_j\}$ of $R^W$. Hence a monomial of naive degree $d$ has degree in $[2d, 2d_nd]$. Therefore
\begin{equation*}
\cohog{i}{\Bun_G^\om}\subset \bigoplus_{i/(2d_n)\le d\le i/2}\frN_d.
\end{equation*}
In particular, by Lemma \ref{l:bound diag entry}, the diagonal entries of $\G^\th_\mu$ on $\cohog{i}{\Bun_G^\om}$ are bounded by $C(i/2+1)^{u}$. The Frobenius action is diagonal with respect to the monomial basis, and has eigenvalues with absolute value $q^{i/2}$ on $\cohog{i}{\Bun_G^\om}$. Therefore
\begin{equation}\label{eq:convergence-estimate-last}
|\Tr(\Frob^{-1}\c\G^\th_\mu|\cohog{i}{\Bun_G^\om})|\le C_{\th, \mu} q^{-i/2} (i/2+1)^{u}\dim \cohog{i}{\Bun_G^\om}.
\end{equation}
Now $(i/2+1)^{u}\dim \cohog{i}{\Bun_G^\om}$ has polynomial growth in $i$ while $q^{-i/2}$ decays exponentially, so the summation of \eqref{eq:convergence-estimate-last} over $i$ converges. This finishes the proof of Proposition \ref{p:trace conv}.
\qed

\subsection{Calculation of the arithmetic volume} We will compute the arithmetic volume defined in Definition \ref{def: arith vol}, at least under the following assumption. 

\begin{assumption}[Commutativity of local operators] \label{assump:split-operators-commute} We assume that the operators $\ov\nb^{\y_j}_{\mu_j}$ pairwise commute for $j=1, \ldots, r$. (When $r=0$, we interpret this assumption as being vacuously true.)
\end{assumption}

\begin{exam}\label{ex:assumption-satisfied}
Assumption \ref{assump:split-operators-commute} is satisfied in many cases of interest. For example, under the degree splitting  
\[
\VV \cong \bigoplus_{d} \VV_{2d}
\]
if the nonzero $\VV_{2d}$ all have dimension $1$, then Assumption \ref{assump:split-operators-commute} is automatically satisfied since the operators $\ov\nb^{\y_j}_{\mu_j}$ are degree-preserving. The one-dimensionality of graded pieces of $\VV$ holds for all simple groups except those of type $D_{2m}$.

On the other hand, for $G = \mathrm{PSO}_{4m}$, there are three minuscule coweights. There is a natural choice of $\eta$ coming from the Killing form of $G$ (see \S \ref{sec:more-examples}). In this case, it was not known to the authors whether the $\ol \nabla^{\eta}_{\mu}$ commute; eventually a counterexample was found in \cite{FAI}.
\end{exam}

\subsubsection{The $L$-function} Under Assumption \ref{assump:split-operators-commute}, the endomorphisms $\{\ov\nb^{\y_j}_{\mu_j}\}$ of $\VV$ have a common system of generalized eigenvalues. More precisely, one can find a basis $v_1,\cdots, v_n$ of $\VV$ consisting of homogeneous elements, such that $v_i$ belongs to a generalized eigenspace of  $\ov\nb^{\y_j}_{\mu_j}$ for each $1\le j\le r$. Let $d_i$ be the degree of $v_i$, and let $\ep_i(\y_j, \mu_j)\in \Qlbar$ be the generalized eigenvalue of $\ov\nb^{\y_j}_{\mu_j}$ on $v_i$. The $n$-tuple (of $(r+1)$-tuples)
\begin{equation*}
    \left(d_i, \ep_i(\y_1, \mu_1),\cdots, \ep_i(\y_r, \mu_r)\right), \quad i=1,2,\cdots, n
\end{equation*}
is independent of the choice of the basis $v_i$ up to permutation of the index $i$. 

To formulate the answer for our eventual calculation of $\vol(\Sht_G^\mu, \eta)$, we introduce the following $L$-function in $n$-variables:
\begin{equation}\label{eq: L of G}
\sL_{X, G}(s_{1},\cdots, s_{n}) :=\prod_{i=1}^{n}\z_{X}(s_{i}+d_{i}).
\end{equation}
We also consider the regularized version by removing all factors $(1-q^{-s_{i}})$ (i.e., where $d_{i}=1$) in the denominator:
\begin{equation}\label{eq: L with pole}
\sL_{X, G}^*(s_{1},\cdots, s_{n}) :=\prod_{\substack{1\le i\le n \\ d_{i}\ne1}}\z_{X}(s_{i}+d_{i})\prod_{\substack{1\le i\le n \\ d_{i}=1}}\z^{*}_{X}(s_{i}+1),
\end{equation}

\subsubsection{Invariants of the tautological classes}\label{sssec:deg-of-taut}

Recall $\y, \y'\in R^{W_{\mu}}$ are homogeneous of degree $2(D_\mu+1)$ and $2D_\mu$ respectively.  We will define numerical constants associated to $\eta$ and $\eta'$. 

\begin{defn}\label{def: d eta}
By assumption, $\y\in R^{W_{\mu}}$ has degree $2(D_\mu+1)$, hence $\int_{G/P_{\mu}}\y\in R^{W}$ is of degree $2$. Viewing $R^{W}$ as a subring of $\cohog{*}{X\times \Bun_{G}^{\om}}$ via pullback along the tautological map, we can then apply $\int_{X}(-)$ to define (using notation from \eqref{fz})
\begin{equation*}
d^{\om}_{\mu}(\y):=\left(\int_{G/P_{\mu}}\y\right)^{[X]}=\int_{X}\int_{G/P_{\mu}}\y\in \rH^0(\Bun_G^\omega) \cong \Qlbar.
\end{equation*}
This number $d^{\om}_{\mu}(\y)$ depends on the component $\om$ of $\Bun_G$.  
\end{defn}

\begin{exam}\label{ex:semisimple-degrees}
If $G$ is semisimple, then the degree 2 part of $R^{W}$ vanishes, so we automatically have $d^{\om}_{\mu}(\y)=0$.
\end{exam}

\begin{defn} \label{def: d eta'}
Recall that $\y' \in R^{W_\mu}$ has degree $2 D_\mu = 2 \dim (G/P_\mu)$. Then we define
\begin{equation*}
d_{\mu}(\y'):=\int_{G/P_{\mu}}\y' \in \Qlbar.
\end{equation*}
\end{defn}

\subsubsection{The differential operators} Recall from \eqref{eq:omega_j} that 
$\om_{j}=\om+\ov \mu_{1}+\cdots+\ov \mu_{j}\in \pi_{0}(\Bun_{G})$. The definition is arranged so that ${}_{c}\Gamma^{\eta_j + \xi \eta'_j}_\mu$ goes from $\cohoc{*}{\Bun^{\om_{j-1}}_{G}}$ to $\cohoc{*}{\Bun^{\om_{j}}_{G}}$. 

For $j=1,\cdots, r$ consider the first order differential operator 
\begin{equation}\label{eq: split differential operator}
\frd_{j}:=d^{\om_{j-1}}_{\mu_j}(\y_{j}) + d_{ \mu_j}(\y'_{j})-(\log q)^{-1}\sum_{i=1}^{n}\e_{i}(\y_{j}, \mu_j)\pl_{s_{i}}.
\end{equation}Here we refer to Definitions \ref{def: d eta} and \ref{def: d eta'} for the numerical constants $d^{\om_{j-1}}_{\mu_j}(\y_{j}) $ and $ d_{\mu_j}(\y'_{j})$.

\begin{theorem}\label{th:vol gen} Let $\mu=(\mu_{1},\cdots, \mu_{r})$ be a sequence of minuscule dominant coweights of $G$ satisfying \eqref{eq:modification-condition}. Let $\y_{j}\in \cohog{2D_{\mu_j}+2}{\BB L_{\mu_j}}$ and $\y'_{j}\in \cohog{2D_{\mu_j}}{\BB L_{\mu_j}}$ for $j=1, 2, \ldots, r$, satisfying Assumption \ref{assump:split-operators-commute}. For $\eta = (\eta_1+\xi \eta_1', \ldots, \eta_r+\xi \eta_r')$, we have 
\begin{equation}\label{eq:split-volume-formula}
\vol({}^{\omega}\Sht_G^\mu, \eta) = q^{\dim \Bun_{G}}\left(\prod_{j=1}^{r}\frd_{j}\right)\sL^{*}_{X,G}(s_{1},\cdots, s_{n})\Big|_{s_{1}=s_{2}=\cdots=s_{n}=0}.
\end{equation}
\end{theorem}

\begin{remark}
Even the case $r=0$, where $\mu$ and $\eta$ are empty, is interesting. In this case, the result collapses to the Tamagawa Number Formula proved by \cite{GL14}, at least in the case where $G$ is semisimple and simply connected. Our proof of this special case is the same as theirs, being based on the Atiyah--Bott description. 
\end{remark}

\begin{exam}
A common situation is: $G$ is semisimple, all $\mu_{i}$ are equal to the same minuscule coweight $\mu$, all $\y_{j}$ are equal to a common $\y$, and all $\y'_{j}$ are equal to a common $\y'$. Since $G$ is semisimple, Example \ref{ex:semisimple-degrees} implies that $d_{\mu_j}^{\omega_{j-1}}(\eta_j) = 0$ for all $j$. In this case, Assumption \ref{assump:split-operators-commute} is trivially satisfied (since all the endomorphisms are equal), and Theorem \ref{th:vol gen} can be written as
\begin{equation*}
\vol({}^\om\Sht_G^\mu, \eta)  =q^{\dim \Bun_{G}}
\Big(d_{\mu}(\y')-(\log q)^{-1}\frac{d}{ds}\Big)^r\Big|_{s=0}\left(\prod_{i=1}^{n}\z_{X}(\e_{i}(\y,\mu)s+d_{i})\right).
\end{equation*}
\end{exam}

The remainder of this Section is devoted to the proof of Theorem \ref{th:vol gen}. Our strategy is to compute the ``eigenvalues'' of the individual $\G_{\mu_j}^{\y_j+\xi\y'_j}$, quotation marks because $\G_{\mu_j}^{\y_j+\xi\y'_j}$ does not map $\cohog{*}{\Bun^{\om}_{G}}$ to itself, hence does not have an a priori notion of eigenvalues. However, as discussed in \S \ref{sss:concrete AB}, $\cohog{*}{\Bun^{\om}_{G}}$ and $\cohog{*}{\Bun^{\om'}_{G}}$ can both be identified with the polynomial ring generated by the K\"unneth components of universal characteristic classes using the Atiyah-Bott description.  We can use this identification to view $\G_{\mu}^{\y+\xi\y'}$ as an endomorphism, and thus make sense of its eigenvalues. To do this, we will show that the Ran filtration on $\cohog{*}{\Bun^{\om}_{G}}$ introduced in \eqref{eq:filtrant-ran-decomposition} is stable under $\G_{\mu}^{\y+\xi\y'}$ (in the above sense), and its eigenvalues on the associated graded are easy to calculate.

\subsubsection{Analysis of the associated graded}
Until \S\ref{sssec:split-completion}, we denote by $\mu$ a minuscule coweight of $G$ and consider the one-step Hecke stack ${}^\om\Hk^\mu_G$.  Let $\y, \y'\in R^{W_{\mu}}$ be homogeneous of degree $2(D_\mu+1)$ and $2D_\mu$ respectively. Recall the definitions of $d_\mu^\omega(\eta ')$ and $d_\mu(\eta')$ from \S \ref{sssec:deg-of-taut}. The following is a refinement of Corollary \ref{c:Gamma increase Ran by 2}.

\begin{lemma}\label{lem: split case associated graded endomorphism}
Let $\g^{\y}_{\mu}: \Gr^F_{\bu}\cohog{*}{\Bun^{\om'}_{G}}\to \Gr^F_{\bu}\cohog{*}{\Bun^{\om}_{G}}$ be the derivation characterized by the equation
\begin{equation*}
\g^{\y}_{\mu}(f^{z})=(\nb^{\y}_{\mu}(f))^{z}, \quad z\in \homog{*}{X}, f\in R^{W}
\end{equation*}
where $\nb^{\y}_{\mu}(f)$ was defined in \eqref{eq: nabla lambda eta}. Then for all $m \in \Z$, the operator $\G_{\mu}^{\y + \xi \y'}$ from \eqref{eq: g^yy'} carries $F_{m}\cohog{*}{\Bun^{\om'}_{G}}$ to $F_{m}\cohog{*}{\Bun^{\om}_{G}}$. On the associated graded, it takes the form
\begin{equation*}
\Gr^F_{\bu}\G_{\mu}^{\y + \xi \y'}(\th)=(d^{\om}_{\mu}(\y)+d_{\mu}(\y'))\th+\g^{\y}_{\mu}(\th).
\end{equation*}

\end{lemma}
\begin{proof}
Let $\th\in F_{m}\cohog{*}{\Bun^{\om'}_{G}}$. We have
\begin{eqnarray}\label{Gf}
\G_{\mu}^{\y + \xi \y'}(\th)&=&h_{0*}((\y+\xi\y')(h_{1}^{*}\th-h_{0}^{*}\th)+(\y+\xi\y') h_{0}^{*}\th) \nonumber \\ 
&=&h_{0*}((\y+\xi\y')\D_{\mu}(\th))+(h_{0*}(\y+\xi\y')) \cdot \th.
\end{eqnarray}
To make sense of the second summand, we identify $\th$ with a class in $\cohog{*}{\Bun^\om_G}$ using the Atiyah--Bott presentation. By Lemma \ref{l:diff h0 h1}, we have $\D_{\mu}(\th)\in F_{m-2}\cohog{*}{{}^{\om} \Hk^{\mu}_G}$. Since $\y+\xi \y'\in F_{0}\cohog{*}{{}^{\om} \Hk^{\mu}_G}$, the product $(\y+\xi\y')\D_{\mu}(\th)$ also lies in $F_{m-2}\cohog{*}{{}^{\om} \Hk^{\mu}_G}$. By Lemma \ref{l:h1 push}, the first summand in \eqref{Gf} lies in $F_{m}\cohog{*}{\Bun^{\om}_{G}}$. For the second summand, using \eqref{h0 int} and the definitions of $d^{\om}_{\mu}(\y)$ and $d_{\mu}(\y')$, we get
\begin{equation*}
h_{0*}(\y+\xi\y')=d^{\om}_{\mu}(\y)+d_{\mu}(\y').
\end{equation*}
Hence the second summand in \eqref{Gf} is $(d^{\om}_{\mu}(\y)+d_{\mu}(\y'))\th$, which also lies in $F_{m}\cohog{*}{\Bun^{\om}_{G}}$.

Now we calculate the effect of $\G_{\mu}^{\y + \xi \y'}$ on the associated graded. Since $\xi\y'\in F_{-2}$, Lemma \ref{l:h1 push} implies that $h_{0*} (\xi \eta'\Delta_\mu(\theta)) \in F_{m-2}$. Therefore, we have 
\begin{equation*}
\Gr^F_{m}\G_{\mu}^{\y + \xi \y'}(\th)\equiv h_{0*}(\y\D_{\mu}(\th))+(d^{\om}_{\mu}(\y)+d_{\mu}(\y'))\th\mod F_{m-2}.
\end{equation*}
The map $\th\mt h_{0*}(\y\D_{\mu}(\th))$ on $\Gr^F_n\cohog{*}{\Bun_G^{\om'}}$ is given by 
\begin{equation*}
    \d^\y_\mu(\th):= (\Gr^F_{m-2}h_{0*})(\y\cdot (\Gr^F_n\D_\mu)(\th)).
\end{equation*}
By Lemma \ref{l:diff h0 h1}, $(\Gr^F_\bu\D_\mu)(\th)$ is a derivation while $\Gr^F_{\bu}h_{0*}$ is $\Gr^F_{\bu}\cohog{*}{\Bun^{\om}_{G}}$-linear, hence we see that $\d^{\y}_{\mu}$ is also a derivation. On $f^{z}$ (where $z\in \homog{*}{X}$ and $f\in R^{W}$) we have by Lemma \ref{l:diff h0 h1} and \eqref{h0 int} that 
\begin{eqnarray*}
    \d^{\y}_{\mu}(f^{z}) & = & (\Gr^F_{\bu}h_{0*})(\PD(z)\y\pl_{\mu}f) \\
    &  = & \int_{X}\PD(z)\cdot \left(\int_{G/P_{\mu}}\y\pl_{\mu}f\right)=(\nb_{\mu}^{\y}f)^{z}.
\end{eqnarray*}
Therefore $\d^\y_\mu=\g^\y_\mu$. This finishes the proof.
\end{proof}

\subsubsection{Augmentation filtration}\label{sss:aug fil} The graded ring $\Gr^{F}_{\bu}\cohog{*}{\Bun^{\om}_{G}}$ carries a natural augmentation by projecting to $\Gr^{F}_{0}\cohog{0}{\Bun^{\om}_{G}}=\Qlbar$. Note that this restricts to the natural augmentation on $\Gr^{F}_{0}\cohog{*}{\Bun^{\om}_{G}}\cong R^{W}$. We apply the construction in \S\ref{sss:aug gr} to form the associated graded with respect to the adic filtration given by the augmentation ideal:
\begin{equation*}
\Gr^{\bl}_{\aug}\Gr^{F}_{\bu}\cohog{*}{\Bun^{\om}_{G}}.
\end{equation*}
This is a triply-graded ring, with the three gradings denoted $\bl, \bu$ and $*$. 

\begin{exam} Abbreviating $\Gr^{F}_{i} = \Gr^{F}_{i}\cohog{*}{\Bun^{\om}_{G}}$, we have
\begin{equation*}
\Gr^{1}_{\aug}\Gr^{F}_{i}=\begin{cases}
\Gr^{1}_{\aug}R^{W}=\VV & i=0;\\
\Gr^{F}_{1}/(R^{W}_{+}\Gr^{F}_{1}) & i=1;\\ 
\Gr^{F}_{2}/(R^{W}_{+}\Gr^{F}_{2}+\Gr^{F}_{1}\cdot \Gr^{F}_{1}) & i=2.
\end{cases}
\end{equation*}
\end{exam}

For $z=1\in \homog{0}{X}$, the isomorphism $(-)^{z}: R^{W}\isom \Gr^{F}_{0}\cohog{*}{\Bun^{\om}_{G}}$ induces an isomorphism 
\begin{equation}\label{eq:ex-gr0map}
\homog{0}{X}\ot \VV\isom  \Gr^{1}_{\aug}\Gr^{F}_{0}\cohog{*}{\Bun^{\om}_{G}}=\Gr^{1}_{\aug}R^{W}.
\end{equation}
For $z\in \homog{1}{X}$, it is clear from the definition that the map $(-)^{z}: R^{W}\isom F_{1}\cohog{*}{\Bun^{\om}_{G}}$ is a derivation. In particular, $(f_{1}f_{2})^{z}\in R^{W}_{+}F_{1}\cohog{*}{\Bun^{\om}_{G}}$. This induces a map
\begin{equation}\label{eq:ex-gr1map}
\homog{1}{X}\ot \VV\to \Gr^{1}_{\aug}\Gr^{F}_{1}\cohog{*}{\Bun^{\om}_{G}}.
\end{equation}
For $z=[X]\in \homog{2}{X}$, it is clear from the definition that the map $(-)^{z}: R^{W}\to F_{2}\cohog{*}{\Bun^{\om}_{G}}$ is a derivation modulo $F_{1,+}F_{1,+}$, where $F_{1,+}$ is the positive cohomological degree part of $F_{1}\cohog{*}{\Bun^{\om}_{G}}$. Therefore it induces a map
\begin{equation}\label{eq:ex-gr2map}
\homog{2}{X}\ot \VV_{>2}\to \Gr^{1}_{\aug}\Gr^{F}_{2}\cohog{*}{\Bun^{\om}_{G}}.
\end{equation}

The direct sum of \eqref{eq:ex-gr0map}, \eqref{eq:ex-gr1map}, and \eqref{eq:ex-gr2map} gives a map
\begin{equation*}
(\homog{\bu}{X}\ot \VV)_{+}\to \Gr^{1}_{\aug}\Gr^{F}_{\bu}\cohog{*}{\Bun^{\om}_{G}}.
\end{equation*}
This then induces a map of triply-graded commutative $\Qlbar$-algebras
\begin{equation*}
\Gr_{\aug}\AB^{\om}: \Sym^{\bl}((\homog{\bu}{X}\ot \VV)_{+})\to \Gr^{\bl}_{\aug}\Gr^{F}_{\bu}\cohog{*}{\Bun^{\om}_{G}}.
\end{equation*}
Here the $\bu$-grading and $*$-grading (cohomological) on the left is understood as follows: a monomial $(z_{1}\ot f_{1})(z_{2}\ot f_{2})\cdots (z_{s}\ot f_{s})$, where $z_{i}\in \homog{
|z_{i}|}{X}$ and $f_{i}\in R^{W}$ homogeneous of degree $|f_{i}|$, has $\bu$-degree $\sum|z_{i}|$ and $*$-degree $\sum|f_{i}|-|z_{i}|$.

\begin{lemma}
The map $\Gr_{\aug}\AB^{\om}$ is an isomorphism.
\end{lemma}
\begin{proof} As in \S\ref{sss:concrete AB}, we have a set of free generators $f_{i}^{z_{j}}$ for $\cohog{*}{\Bun^{\om}_{G}}$. Now $f_{i}^{z_{j}}\in F_{|z_{j}|}$ and we denote its image in $\Gr^{F}_{|z_{j}|}$ by $\ov f_{i}^{z_{j}}$. From the description in \S\ref{sss:concrete AB} of $\cohog{*}{\Bun^{\om}_{G}}$ as the polynomial ring on generators \eqref{eq:constant-AB-generators}, it is easy to see that $\Gr^{F}_{\bu}\cohog{*}{\Bun^{\om}_{G}}$ is a polynomial ring with free generators $\ov f_{i}^{z_{j}}$. Then the assertion follows from observing that the $z_j \otimes \ov f_{i}$ (with $|f_{i}|>|z_{j}|$) form a basis for $(\homog{\bu}{X}\ot \VV)_{+}$, mapping to $\ov f_{i}^{z_{j}}$ under $\Gr_{\aug}\AB^\omega$.  
\end{proof}

Using Lemma \ref{lem: split case associated graded endomorphism}, we can now describe the operator $\G_{\mu}^{\y + \xi \y'}$ on $\Gr^{\bl}_{\aug}\Gr^{F}_{\bu}$ explicitly. 

\begin{prop}\label{prop: split action on graded} The operator $\Gr^F_{\bu}\G_{\mu}^{\y + \xi \y'}$ preserves the adic filtration by the augmentation ideals and passes to the associated graded
\begin{equation*}
\Gr^{\bl}_{\aug}\Gr^{F}_{\bu}\G_{\mu}^{\y + \xi \y'}: \Gr^{\bl}_{\aug}\Gr^{F}_{\bu}\cohog{*}{\Bun^{\om'}_{G}}\to \Gr^{\bl}_{\aug}\Gr^{F}_{\bu}\cohog{*}{\Bun^{\om}_{G}}.
\end{equation*}
Let $(\id_{\homog{\bu}{X}}\ot\ov\nb^{\y}_{\mu})_{+}$ be the restriction of $\id_{\homog{\bu}{X}}\ot\ov\nb^{\y}_{\mu}\in \End(\homog{\bu}{X}\ot \VV)$ to $(\homog{\bu}{X}\ot \VV)_{+}$, and $(\id_{\homog{\bu}{X}}\ot\ov\nb^{\y}_{\mu})_{+}^{\der}$ be its unique extension to a derivation on $\Sym^{\bl}((\homog{\bu}{X}\ot \VV)_{+})$. Then under the isomorphisms $\Gr^{\bl}_{\aug}\AB^{\om}$ and $\Gr^{\bl}_{\aug}\AB^{\om'}$, $\Gr^{\bl}_{\aug}\Gr^{F}_{\bu}\G_{\mu}^{\y + \xi \y'}$ is identified with the endomorphism of
\begin{equation*}
\Sym^{\bl}((\homog{\bu}{X}\ot \VV)_{+})
\end{equation*}
given by
\begin{equation*}
\Gr^{\bl}_{\aug}\Gr^{F}_{\bu}\G_{\mu}^{\y + \xi \y'}=(d^{\om}_{\mu}(\y)+d_{\mu}(\y'))\id+(\id_{\homog{\bu}{X}}\ot\ov\nb^{\y}_{\mu})_{+}^{\der}.
\end{equation*}
\end{prop}


\subsubsection{Completion of the proof}\label{sssec:split-completion}
For the sequence $\mu = (\mu_1, \ldots, \mu_r)$ and a fixed $\omega \in \pi_0(\Bun_G)$, set $\omega_j :=\omega + \ol\mu_1 + \ldots + \ol\mu_j$. Abbreviate 
\[
H^{i}:=\cohog{i}{\Bun_{G}^{\omega}}, \quad H^{*}:=\op_{i\in \ZZ}H^{i}, \quad \text{and} \quad H_{c}^{i}:=\cohoc{i}{\Bun_{G}^{\omega}}.
\]
By definition and Proposition \ref{p:trace conv},
\begin{equation}\label{alt trace Hc}
\vol({}^\om\Sht_G^\mu, \eta)  = \sum_{i}(-1)^{i}\Tr(\Frob\c {}_c\Gamma^{\eta + \xi \eta'}_\mu \mid H^{i}_{c}).
\end{equation}
By Lemma \ref{l:trace cGamma vs Gamma}, we have
\begin{equation}\label{alt trace H}
    \vol({}^\om\Sht_G^\mu, \eta)  =q^{\dim\Bun_G} \sum_{i}(-1)^{i}\Tr(\Frob^{-1}\c\Gamma^{\eta + \xi \eta'}_\mu \mid H^{i}),
\end{equation}
where we write
\[
\Gamma^{ \eta + \xi \eta'}_{\mu} = \Gamma^{ \eta_1 + \xi \eta'_1}_{\mu_1}\circ \ldots \circ \Gamma^{ \eta_r + \xi \eta'_r}_{\mu_r}.
\]



By Lemma \ref{lem: split case associated graded endomorphism}, the Ran filtration on $H^*$ is preserved by $\Gamma^{ \eta + \xi \eta'}_{\mu}$. By Proposition \ref{prop: split action on graded}, $\Gr^F_\bu\Gamma^{ \eta + \xi \eta'}_{\mu}$ preserves the augmentation filtration on $\Gr^F_\bu H^*$. Hence the trace \eqref{alt trace H} can be calculated on its associated graded $\Gr^{\bl}_{\aug}\Gr^{F}_{\bu} H^*$. Again by Proposition \ref{prop: split action on graded}, the action of $\Gamma^{ \eta + \xi \eta'}_\mu$ on this double associated graded is (in terms of the Atiyah--Bott description) the composition of the endomorphisms
\begin{equation}\label{eq: split endomorphism factor}
\left[(d^{\om_{j-1}}_{\mu_j}(\y_{j})+d_{\mu_j}(\y'_j))\id+(\id_{\homog{\bu}{X}}\ot\ov\nb^{\y_j}_{\mu_j})_{+}^{\der} \right] 
\end{equation}
for $ j = 1, \ldots, r$. By Assumption \ref{assump:split-operators-commute}, $\VV$ has a basis consisting of homogeneous elements that are simultaneous generalized eigenvectors for the operators $\{\ov\nb^{\y_j}_{\mu_j}\}_{1\le j\le r}$. We can thus choose a decomposition into lines
\[
\VV \cong \bigoplus_{l=1}^n  \LL_{l}
\]
such that $\LL_{l}\subset \VV_{2d_l}$ (so as Frob-module $\LL_l\cong \Qlbar(-d_l)$), and $\LL_l$ is spanned by a generalized eigenvector $v_l$ with eigenvalues $\epsilon_l(\eta_j, \mu_j)$ for $\ov\nb^{\y_j}_{\mu_j}$, for all $j=1,\cdots, r$. 

In the Atiyah--Bott description \eqref{eq: split AB} of $H^*$, we have 
\begin{align}\label{eq:split-coh-end}
\rH^*(\Bun_G^\omega) & \cong \Sym^{\bl} (\rH_{\bu}(X) \otimes \VV)_+ \cong \Sym^{\bl} \Big(\rH_{\bu}(X) \otimes (\oplus_{l=1}^n \LL_{l} )\Big)_+ \\
&\cong \Big(\bigotimes_{l=1}^n \Sym^{\bl} (\rH_{\bu}(X) \otimes \LL_{l}) \Big)_+.
\end{align}
In particular, abbreviating
\[
\Sym^{N_1, \ldots, N_n}(\rH_{\bu}(X) \otimes \VV)_+ := \Big(\Sym^{N_1} (\rH_{\bu}(X) \otimes \LL_1)  \otimes \ldots \otimes \Sym^{N_n} (\rH_{\bu}(X) \otimes \LL_n)\Big)_+,
\]
each $\bl$-graded piece of \eqref{eq:split-coh-end} splits as 
\[
\Sym^N (\rH_{\bu}(X) \otimes \VV)_+ \cong \bigoplus_{N_1 + \ldots + N_n = N} \Sym^{N_1, \ldots, N_n}(\rH_{\bu}(X) \otimes \VV)_+.
\]
The endomorphism $\id_{\homog{\bu}{X}}\ot \ov\nb^{\y_j}_{\mu_j}$ preserves each summand $\Sym^{N_1, \ldots, N_n}(\rH_{\bu}(X) \otimes \VV)_+$, and acts on it with generalized eigenvalues $\sum_{l=1}^n N_l \epsilon_l(\eta_j, \mu_j)$. Thus the contribution of $\Sym^{N_1, \ldots, N_n}(\rH_{\bu}(X) \otimes \VV)_+$ to the trace \eqref{alt trace H} is
\begin{equation}\label{trace on Sym Ni}
    q^{\dim\Bun_G}\prod_{j=1}^r\Big(d^{\om_{j-1}}_{\mu_j}(\y_j)+d_{\mu_j}(\y'_j)+\sum_{l=1}^nN_l\ep_l(\y_j,\mu_j) \Big)\Tr(\Frob^{-1}, \Sym^{N_1, \ldots, N_n}(\rH_{\bu}(X) \otimes \VV)_+).
\end{equation}

Form the multivariate generating series 
\[
\Lambda(t_1, \ldots, t_n) = \sum_{N_1, \ldots, N_n \geq 0} \Tr(\Frob^{-1} \mid \Sym^{N_1, \ldots, N_n}(\rH_{\bu}(X) \otimes \VV)_+ )t_1^{N_1} \ldots t_n^{N_n}.
\]
Note that for $t = q^{-s}$, we have $(-\log q)^{-1} \pl_s =  t \pl_t$. Then the sum of \eqref{trace on Sym Ni} over all $(N_1,\cdots, N_n)$ becomes
\begin{equation*}
    q^{\dim\Bun_G}\left(\prod_{j=1}^r\frd_j\right)\L(q^{-s_1},\cdots, q^{-s_n})|_{s_1=\cdots=s_n=0}.
\end{equation*}
Therefore, it suffices to show that
\[
\Lambda(q^{-s_1}, \ldots, q^{-s_n})  = \sL_{X,G}^*(s_1, \ldots,s_n).
\]
Both sides factor over each variable, reducing this to a one variable statement: 
\[
 \Tr(\Frob^{-1} \mid \Sym^{\bl} (\rH_{\bu}(X) \otimes \LL_{l})_+)  = \begin{cases} \zeta_X(s+d_l) & d_l > 1, \\ 
\zeta_X^*(s+1) & d_l = 1.
\end{cases}
\]

We now consider two cases, based on whether $d_l > 1$. 
\begin{enumerate}
\item If $d_l >1$, then $(\rH_{\bu}(X) \otimes \LL_{l})_+ = \rH_{\bu}(X) \otimes \LL_{l} \cong \rH_{\bu}(X)(-d_l)$. So we have 
\begin{equation}\label{eq: zeta}
\Lambda(t) = \frac{\det(1- t \Frob^{-1}  \mid \rH_1(X) \otimes \LL_{l})}{(1-q^{-d_l}t)(1-q^{-d_l+1}t)}  = \frac{\det(1-q^{-d_l} t \Frob^{-1}   \mid \rH_1(X))}{(1-q^{-d_l}t)(1-q^{-d_l+1}t)}.  
\end{equation}
The action of $\Frob^{-1}$ on $\rH_1(X)$ is adjoint to the action of $\Frob$ on $\rH^1(X)$. Hence \eqref{eq: zeta} can be rewritten as 
\begin{equation}\label{eq: zeta-2}
\frac{\det(1-q^{-d_l} t \Frob   \mid \rH^1(X))}{(1-q^{-d_l}t)(1-q^{-d_l+1}t)}  
\end{equation}
and upon setting $t= q^{-s}$ we obtain $\zeta_X(s+d_l)$, as desired.

\item If $d_l=1$, then we exclude the summand $\rH_2(X, \LL_{l})$ from $(\rH_{\bu}(X) \otimes \LL_{l})_+$. This implies that the denominator in \eqref{eq: zeta-2} should be only $(1-q^{-d_l} t)$, and since $d_l=1$ this effectively means we multiply the entire expression by $(1-t)$. Upon setting $t=q^{-s}$, we therefore obtain $\Lambda(q^{-s}) = \zeta_X^*(s+1)$, as desired. 
\end{enumerate}
This finally completes the proof of Theorem \ref{th:vol gen}. \qed

%

\subsection{An example}\label{ssec:gln-colength-one-arithmetic-volume}

Let $G = \GL_n$ and $\mu^\sharp := (1,0, \ldots, 0)$, $\mu^\flat := (0, 0, \ldots, -1) \in \xcoch(T)$. Then $\dim G/P_{\mu^{\sh}} = \dim G/P_{\mu^{\flat}}  =  n-1$. We identify $R = \Qlbar[x_1, \ldots, x_n]$ and $W  = S_n$ acting by permutation of the $x_i$. Then $W_{\mu^\sharp} \cong S_{n-1}$ is the subgroup fixing $x_1$, and $W_{\mu^\flat} \cong S_{n-1}$ is the subgroup fixing $x_n$.

\subsubsection{Hecke correspondences} The Hecke stack $\Hk^{\mu^\sharp}_G$ is the moduli stack of upper modifications of rank $n$ vector bundles of colength $1$, in the terminology of \cite[Definition 6.5]{FYZ}. This means that for a commutative $k$-algebra $R$, $\Hk^{\mu^\sharp}_G(R)$ parametrizes $x \in X(R)$ along with an injective map 
\begin{equation}\label{eq:hk-sharp-point}
\cE_{0} \inj \cE_1 
\end{equation}
where $\cE_0, \cE_1$ are rank $n$ vector bundles on $X_R$ and $\cE_1/\cE_0$ is a line bundle over the graph of $x$ in $X_R$. 

The Hecke stack $\Hk^{\mu^\flat}_G$ is defined similarly except that \eqref{eq:hk-sharp-point} is replaced by 
\begin{equation}\label{eq:hk-flat-point}
\cE_{0} \hookleftarrow \cE_1 
\end{equation}
with cokernel a line bundle over the graph of $x$ in $X_R$.

The connected components of $\Bun_G$ are indexed by $d \in \Z$. Concretely, $\Bun_G^d$ is the component of $\Bun_G$ parametrizing bundles of degree $d$. For $? \in \{\sharp, \flat\}$, ${}^d \Hk_G^{\mu^?}$ is the open-closed component of $\Hk_G^{\mu^?}$ where $\deg \cF_0 = d$. Let $h_{0}, h_{1}: \Hk^{\mu^?}_G \to \Bun_{G}$ be the maps recording $\cF_{0}$ and $\cF_{1}$, respectively. Then we have a correspondence
\begin{equation*}
\xymatrix{ & {}^{d}\Hk^{\mu}_G\ar[dl]_{h_{0}}\ar[dr]^{h_{1}}\\
\Bun^{d}_{G} & & \Bun^{d\pm1}_{G}}
\end{equation*}
where the sign is $+1$ if $\mu = \mu^\sharp$ and $-1$ if $\mu = \mu^\flat$. 

\subsubsection{Tautological bundles}\label{sssec:gln-example-tautological}
Under the map $\Hk_G^{\mu^\sharp} \rightarrow G/P_{\mu^\sharp} \cong \PP^{n-1}$ from \eqref{Hk ev L}, the pullback of $\cO(-1)$ is the ``tautological line bundle'' $\cP^\sharp$ on $\Hk_G^{\mu^\sharp}$ whose fiber along an $R$-point \eqref{eq:hk-sharp-point} is $\ker(\cF_0|_x \rightarrow \cF_1|_x)$. Similarly, under the map $\Hk_G^{\mu^\flat} \rightarrow G/P_{\mu^\flat} \cong \PP^{n-1}$, the pullback of $\cO(-1)$ is the tautological line bundle $\cP^\flat$ whose fiber along an $R$-point \eqref{eq:hk-flat-point} is $(\cF_0/\cF_1)	$. 

By our conventions (cf. \S \ref{sssec:reductive-groups-notation}), these are arranged so that under the map $R^{W_\mu} \rightarrow \cohog{*}{\Hk^{\mu}_G}$ induced by the canonical parabolic reduction, $c_1(\cP^{\sh})$ agrees with the image of $x_1 \in R^{W_\mu}$ of $\mu = \mu^{\sharp}$, while $c_1(\cP^{\flat})$ agrees with the image of $-x_n$ if $\mu = \mu^{\flat}$. 

\subsubsection{Operators}\label{sssec:gln-example-operators} For $\mu \in \{\mu^\sharp, \mu^\flat\}$, we write 
\[
 [\mu]= 
\begin{cases} \sh & \mu =\mu^\sharp \\
\flat & \mu=\mu^\flat
\end{cases} \hspace{1cm} \text{and} \hspace{1cm}
|\mu| = 
\begin{cases} 1 & \mu =\mu^\sharp \\
-1 & \mu=\mu^\flat
\end{cases} 
\]
For a rank $n$ vector bundle $\cE$ on $X$, we define the map
\begin{eqnarray*}
\G_{\mu}^{\cE}   &:& \cohog{*}{\Bun_{G}^{d \pm 1 }} \to \cohog{*}{\Bun_{G}^{d}} \\
&& \th\mt h_{0*}(h_{1}^{*}\th\cup c_{n}(p_{X}^{*}\cE^{[\mu] *}\ot \cP^{[\mu]})).
\end{eqnarray*}
where the source degree is $d+1$ if $\mu = \mu^\sharp$ and $d-1$ if $\mu=\mu^\flat$. 

\subsubsection{Arithmetic volume}
Consider a sequence of modifications of type $\mu =(\mu_1,\cdots ,\mu_{r})$, where each $\mu_i \in \{\mu^\sharp, \mu^\flat\}$ with exactly half of each type (so in particular $r$ must be even).

Fix a sequence of rank $n$ vector bundles $\cE_1, \ldots, \cE_r$ on $X$. Consider the composition
\begin{equation*}
\G_{\mu}^{\cE}:= \G_{[\mu_{r}]}^{\cE_{r}}\c\cdots\c\G_{[\mu_{2}]}^{\cE_{2}}\c\G_{[\mu_{1}]}^{\cE_{1}} \co \cohog{*}{\Bun^{d}_{G}}\to \cohog{*}{\Bun^{d}_{G}}
\end{equation*}
and its compactly supported version
\begin{equation*}
{}_{c}\G_{\mu}^{\cE} := {}_{c}\G_{[\mu_{r}]}^{\cE_{r}}\c\cdots\c {}_{c}\G_{[\mu_{2}]}^{\cE_{2}}\c {}_{c}\G_{[\mu_{1}]}^{\cE_{1}} \co \cohoc{*}{\Bun^{d}_{G}}\to \cohoc{*}{\Bun^{d}_{G}}.
\end{equation*}

We write
\[
\vol({}^{d}\Sht^{\mu}_{G}, \prod_{j=1}^{r}c_{n}(p_{j}^{*}\cE^{*}_{j}\ot\cP^{\mu_{j}})) = \Tr({}_{c}\G_{\mu}^{\cE}\circ\Frob\mid  \cohoc{*}{\Bun^{d}_{G}})
\]
for the corresponding arithmetic volume. Its value will be formulated in terms of the $L$-functions
\begin{equation}\label{eq:split-L}
L_{X,G}(s) := \sL_{X,G}(s, s, \ldots, s) =  \prod_{i=1}^{n}\z_{X}(s+i),
\end{equation}
and
\begin{equation}\label{eq:split-normalized-L}
L^{*}_{X,G}(s)= \sL^{*}_{X,G}(s, s, \ldots, s) =  (1-q^{-s})L_{X,G}(s).
\end{equation}

For $1 \leq j \leq r$, let $D_j := \deg \cE_j$. We define coefficients $b_0^\mu, \ldots, b_r^\mu$ by writing the following product as a polynomial in $N$:
\begin{equation}\label{def bmuj}
\prod_{j=1}^{r}(d-|\mu_{r}|-\cdots-|\mu_{j}|- |\mu_j| D_{j}+ |\mu_j| N)=b^{\mu}_{r}N^{r}+b^{\mu}_{r-1}N^{r-1}+\cdots+b^{\mu}_{0}.
\end{equation}
Each $b^{\mu}_{i}$ is itself a polynomial in $d$ and $D_1, \ldots, D_r$. 

\begin{theorem}\label{th:main GLn} For any $d\in \Z$, we have
\begin{equation*}
\vol({}^{d}\Sht^{\mu}_{G}, \prod_{j=1}^{r}c_{n}(p_{j}^{*}\cE^{*}_{j}\ot\cP^{[\mu_{j}]}))=q^{n^{2}(g-1)}\sum_{i=0}^{r}b^{\mu}_{i}(-\log q)^{-i}\left(\frac{d}{ds}\right)^{i}\Big|_{s=0}L^{*}_{X,G}(s).
\end{equation*}
\end{theorem}

\begin{proof} As discussed in Example \ref{ex:assumption-satisfied}, Assumption \ref{assump:split-operators-commute} is automatically satisfied because of the choice of group $G = \GL_n$. Hence we may apply Theorem \ref{th:vol gen} in order to calculate the left side. The first step is to rewrite the $\Gamma^{\cE}_{\mu}$ in terms of the $\Gamma_{\mu}^{\eta + \xi \eta'}$ from \S\ref{ssec:split-Gamma}. 

We have
\begin{equation}\label{eq:gln-colength-one-chern-class}
c_n(p_X^* \cE_j^* \otimes \cP^{[\mu_j]}) = c_1(\cP^{[\mu_j]})^n  +  p_X^*(c_1(\cE_j^*)) c_1(\cP^{[\mu_j]})^{n-1} \in \rH^*(\Hk_G^{\mu_j}).
\end{equation}

Recall that canonical parabolic reduction induces a map $\rH^*(\BB P_{\mu_j}) = R^{W_{\mu_j}} \rightarrow \rH^*(\Hk_G^{\mu_j})$. As explained in \S \ref{sssec:gln-example-tautological}, the definition of $\cP^{\sharp/\flat}$ is arranged so that $c_1(\cP^{[\mu_j]})$ agrees with the image of $x_1 \in R^{W_{\mu_j}}$ if $[\mu_j] = \sharp$ and with the image of $-x_n$ if $[\mu_j] = \flat$. 

Let us consider first the case $[\mu_j] = \sharp$. We may then rewrite \eqref{eq:gln-colength-one-chern-class} as 
\[
c_n(p_X^* \cE_j^* \otimes \cP^{[\mu_j]}) = x_1^n - D_j \xi  x_1^{n-1}. 
\]
We thus have $\eta_j = x_1^n$ and $\eta'_j = - D_j x_1^{n-1}$. To apply Theorem \ref{th:vol gen}, we need to calculate the constants $d_{\mu_j}^\omega(\eta)$, $d_{\mu_j}(\eta'_j)$, and the eigenweights $\epsilon_i(\eta_j, \mu_j)$. 
\begin{itemize}

\item Using \eqref{eq:integration-GLn-colength-one} for $i=1$, we find that $d_{\mu}^\omega (\eta) = (-1)^{n-1} e_1^{[X]} = (-1)^{n-1} \omega$, since $\Bun_G^\omega$ parametrizes bundles of degree $\omega$.

\item Since $x_1$ corresponds to $\cO(-1)$ on $G/P_\lambda \cong \PP^{n-1}$ (cf. \S \ref{sssec:reductive-groups-notation}), we see that 
\[
d_{\mu}(\eta') =  -D_j \int_{G/P_\mu} x_1^{n-1} = (-1)^n D_j. 
\]

\item We have $\epsilon_i(\eta, \mu^\sharp) = \epsilon_i(x_1^n, \mu^\sharp) = (-1)^{n-1}$ by Proposition \ref{prop:GLn-minimal-eps}.

\end{itemize}
The operator $\Gamma^{\cE_j}_{\mu_j}$ is applied to the component $\Bun_G^{\omega_j}$ for $\omega_j := d - |\mu_r| - \ldots - |\mu_j|$. Hence for $[\mu_j] = \sharp$, the differential operator \eqref{eq: split differential operator} corresponding to $\Gamma^{\cE}_{\mu_j}$ is 
\begin{equation*}
\frd_{j} = (-1)^{n-1}\Big(\omega_j  -   D_j  - (\log q)^{-1}\sum_{i=1}^{n} \pl_{s_{i}} \Big)  = (-1)^{n-1} \Big(\omega_j - |\mu_j|  D_j + |\mu_j|  (- \log q)^{-1}\sum_{i=1}^{n} \pl_{s_{i}}  \Big).
\end{equation*}

 Next we consider the case $[\mu_j] = \flat$. In this case, \eqref{eq:gln-colength-one-chern-class} becomes
 \[
 c_n(p_X^* \cE_j^* \otimes \cP^{[\mu_j]}) = (-x_n)^n - D_j \xi (-x_n)^{n-1}. 
 \]
We thus have $\eta_j = (-x_n)^n$ and $\eta'_j = - D_j (-x_n)^{n-1}$. We calculate the relevant constants. 
\begin{itemize}

\item Repeating the proof of \eqref{eq:integration-GLn-colength-one} with $x_1$ replaced by $x_n$, we find that $d_{\mu}^\omega (x_n^n) = e_1^{[X]} =  \omega$. Hence $d_\mu^\omega((-x_n)^n) = (-1)^n \omega $.

\item Since $-x_n$ corresponds to $\cO(-1)$ on $G/P_\lambda \cong \PP^{n-1}$ (cf. \S \ref{sssec:reductive-groups-notation}), we see that 
\[
d_{\mu}(\eta') =  -D_j \int_{G/P_\mu} (-x_n)^{n-1} = (-1)^n D_j  .
\]

\item Repeating the proof of Proposition \ref{prop:GLn-minimal-eps} with $x_1$ replaced by $-x_n$, we have $\epsilon_i(\eta, \mu^\flat) = \epsilon_i((-x_n)^n, \mu^\flat) = (-1)^{n-1}$. (Another way to see this is from perspective of \S \ref{sec:more-examples}. Then $\epsilon_i(\Omega, \mu)$ is invariant under the $W = S_n$ action on $\mu$. Also, we claim that replacing $\mu$ by $-\mu$ leaves the eigenweights unchanged. Indeed, we have $\frR_{-\mu} = (-1)^{D_\mu} \frR_\mu$, and the integrand changes from $t_\mu^{D_\mu + 1} \partial_\mu$ to $(-t_\mu)^{D_\mu+1} \partial_{-\mu} = (-1)^{D_\mu} t_\mu^{D_\mu+1}\partial_{\mu}$). 

\end{itemize}
Hence we have in this case that 
\begin{equation*}
\frd_{j} =  (-1)^n \Big(\omega_j +  D_j -  (- \log q)^{-1}\sum_{i=1}^{n} \pl_{s_{i}}  \Big) = (-1)^n \Big(\omega_j - |\mu_j|  D_j + |\mu_j|  (- \log q)^{-1}\sum_{i=1}^{n} \pl_{s_{i}}  \Big). 
\end{equation*} 
Observing that $\dim \Bun_G = n^2(g-1)$, Theorem \ref{th:vol gen} says that 
\begin{align*}
& \vol({}^{d}\Sht^{\mu}_{G}, \prod_{j=1}^{r}c_{n}(p_{j}^{*}\cE^{*}_{j}\ot\cP^{\mu_{j}}) )    \\
& = q^{n^2(g-1)} (-1)^{r/2}  \prod_{j=1}^r  \Big( \omega_j  - |\mu_j|  D_j + |\mu_j| (-\log q)^{-1}\sum_{i=1}^{n} \pl_{s_{i}} \Big) \sL^{*}_{X,G}(s_{1},\cdots, s_{n})\Big|_{s_{1}=s_{2}=\cdots=s_{n}=0}.
\end{align*}
Examining \eqref{def bmuj}, this expands as 
\[
q^{n^2(g-1)}  \sum_{i=0}^r  b_i^\mu (-\log q)^{-i} \Big(\sum_{j=1}^{n} \pl_{s_{j}} \Big)^i  \sL^{*}_{X,G}(s_{1},\cdots, s_{n})\Big|_{s_{1}=s_{2}=\cdots=s_{n}=0}. 
\] 
The result then follows from the elementary observation that for a smooth function of the form $f(s_1, \ldots, s_n) = \prod_{j=1}^n f_j(s_j)$, we have 
\begin{equation}\label{eq:partial-derivative-identity}
\Big(\sum_{j=1}^{n} \pl_{s_{j}}\Big)^i f(s_1, \ldots, s_n)|_{s_1 = \ldots = s_n=s} = \pl_s^i f(s, \ldots, s). 
\end{equation}
This completes the proof.
\end{proof}

\section{Some calculations of eigenweights}\label{sec:more-examples}

In order to apply Theorem \ref{th:vol gen} in examples, we need to explicate the differential operators $\frd_j$ from \eqref{eq: split differential operator}. In practice, the constants appearing there are straightforward to calculate \emph{except} for the eigenweights $\epsilon_i(\eta, \mu)$. 

In this section, we will work out the eigenweights in more examples. The natural $\eta$ of interest arise in the following way (up to sign). Let $\Omega\in R^{W}_{4}$ be a Casimir element corresponding to a fixed nondegenerate $W$-invariant quadratic form on $\xcoch(T)_{\Q}$. The associated bilinear form gives an isomorphism 
\[
\io_{\Omega}: \xcoch(T)_{\Q}\isom \xch(T)_{\Q}.
\]
Let $t_{\mu}:=\pl_{\mu}\Omega \in R^{W_{\mu}}_{2}$, which is easily seen to be equal to $\io_{\Omega}(\mu)\in \xch(T)_{\Q}$. We consider $\y :=t_{\mu}^{D_\mu+1}\in R^{W_{\mu}}_{2(D_\mu+1)}$. This defines $\ov \nb_{\mu}^\eta \in \End(\ol \VV)$ as in \S \ref{sssec:ov-nabla}. Note that the example of \S \ref{ssec:gln-colength-one-arithmetic-volume} was of this form, up to sign normalizations. 

Now let us explain for what families we are able to calculate the eigenweights. 
\begin{itemize}
\item Let $G = \GL_n$. The minuscule coweights are $(1^m, 0^{n-m})$ for $1 \leq m < n$. For $m=1$, we calculated the eigenweights in Proposition \ref{prop:GLn-minimal-eps}. We treat $m=2$ below in Proposition \ref{prop: GL len=2}, where both the answer and proof are substantially more complicated than the case $m=1$. 

\item For $G = \SO_{2m+1}$, there is a unique minuscule coweight, represented by $\mu = (1, 0^{m-1})$, and we compute the eigenweights. 

\item For $G = \mathrm{PSO}_{2n}$, there are three minuscule coweights. The \emph{standard} coweight (corresponding to the standard representation of $\SO_{2n}$) is represented by $\mu = (1, 0^{n-1})$, and we will calculate the corresponding eigenweights. There are also two \emph{spin} coweights, corresponding to the two spin representations of $\Spin_{2n}$.

\item For exceptional groups, since there are only finitely many cases, the eigenweights may be found by a finite algorithm, although we have not carried it out. 
\end{itemize}
The determination of the remaining minuscule coweights in Type A, as well as the spin coweights in Types C and D, proved more challenging. Eventually, a solution was found by an AI agent coded by the first author; it will be explained in the separate paper \cite{FAI}.

\subsection{Special orthogonal groups} We assume $p \neq 2$. Let $G$ be the split special orthogonal group $\SO(V)$ where $\dim_{k}V=n$ and $B$ is the non-degenerate symmetric bilinear pairing on $V$.

\subsubsection{Cohomology of $\BB G$}
\label{ss: H BG SO}
Suppose that $n=2m+1$ is odd. Using the isotropic basis $v_{1},v_{2},\cdots, v_{m}, v_{0}, v_{-m},\cdots, v_{-1}$ such that $B(v_{i},v_{j})=\d_{i,-j}$, we get an identification $\xch(T)\cong \ZZ^{m}$ on which $W=(\ZZ/2)^{m}\rtimes S_{m}$ acts by permuting and changing signs of coordinates. We have an isomorphism
\begin{equation*}
\cohog{*}{\BB G}\cong \Qlbar[x_{1},\cdots, x_{m}]^{(\ZZ/2)^{m}\rtimes S_{m}}=\Qlbar[e_{1}^{(2)},e_{2}^{(2)},\cdots, e_{m}^{(2)}] 
\end{equation*}
Here $x_{i}$ has cohomological degree $2$, and $e_{i}^{(2)} \in \cohog{4i}{\BB G}$ is the $i$-th elementary symmetric polynomial in $x_{1}^{2},\cdots, x_{m}^{2}$.

Next suppose that $n=2m$ is even. Using an isotropic basis $v_{1},\cdots, v_{m},v_{-m},\cdots, v_{-1}$ of $V$, we get an isomorphism $\xch(T)\cong \ZZ^{m}$. The Weyl group $W\subset (\ZZ/2)^{m}\rtimes S_{m}$ is the kernel of the homomorphism $\chi: (\ZZ/2)^{m}\rtimes S_{m}\to \ZZ/2$ that is trivial on $S_{m}$ and nontrivial on each copy of $\ZZ/2$; it acts by permutation and change of sign on the coordinates. We have an isomorphism
\begin{equation}\label{coho BSO even}
\cohog{*}{\BB G}\cong \Qlbar[x_{1},\cdots, x_{m}]^{W}=\Qlbar[e_1^{(2)},e_2^{(2)},\cdots, e_{m-1}^{(2)}, \Pf] 
\end{equation}
where $\Pf=x_{1}x_{2}\cdots x_{m}$ is the Pfaffian. Note that $\Pf^{2}=x_{1}^{2}\cdots x_{m}^{2}$ is the $m$-th elementary symmetric polynomial $e_m^{(2)}$ in $x_{1}^{2},\cdots, x_{m}^{2}$. 

\begin{remark} The class $\Pf$ for $\BB G$ depends on the choice of the particular Lagrangian $\Span\{v_{1},\cdots, v_{m}\}$ in $V$. The space of Lagrangians has two connected components; a choice of a Lagrangian lying in the other component changes $\Pf$ to $-\Pf$. 
\end{remark}

We take the Casimir element $\Omega := \frac{1}{2} \sum_{i=1}^m x_i^2 \in R^W$. 

\subsubsection{Odd orthogonal case} \label{ss:O odd}
Let $\mu=(1,0,\ldots,0) \in \Z^m \cong \xcoch(T)$. Then we have $\pl_{\mu}=\pl_{x_{1}}$ and $t_{\mu}= \partial_{\mu}(\Omega) = x_{1}$. As $D_\mu = \dim G/P_\mu =  2m-1$, we have $\eta = t_\mu^{2m} = x_1^{2m}$. 

\begin{lemma}
Suppose $n=2m+1$ is odd. For $\eta = x_1^{2m}$ and $\mu=(1,0,\ldots,0)$, we have 
\[
\nabla^\eta_\mu(e_i^{(2)}) = -4 e_i^{(2)} \quad \text{ for all } i = 1, 2, \ldots, m.
\]
In particular, the eigenweights are all equal to $-4$. 
\end{lemma}

\begin{proof}
Let $\wh{e}_i^{(2)}$ be the $i$th elementary symmetric polynomial in $x_2^2, \ldots, x_m^2$. Then $\partial_{x_1} e_i^{(2)} = 2 x_1  \wh{e}_{i-1}^{(2)}$. We want to calculate 
\[
\nabla_\mu^\eta(e_i^{(2)}) =  \int_{G/P_\mu} 2 x_1^{2m+1}  \wh{e}_{i-1}^{(2)}.
\]
We will use Lemma \ref{lem:w-average-integration}. In this case, the equivariant Chern class of (the tangent bundle of) $G/P_\mu$ is 
\[
\frR_\mu =  -x_1 \prod_{j=2}^m (x_j-x_1)(-x_j-x_1) = (-1)^m x_1 \prod_{j=2}^m (x_j^2-x_1^2).
\]
Hence Lemma \ref{lem:w-average-integration} says that 
\begin{align}
\int_{G/P_\mu} 2x_1^{2m+1}  \wh{e}_{i-1}^{(2)} &=  2 (-1)^m \sum_{w \in W/W_\mu} w   \left(\frac{x_1^{2m+1} \wh{e}_{i-1}^{(2)}}{ x_1 \prod_{j=2}^m (x_1^2-x_j^2)} \right) \\
&= 2(-1)^m \sum_{w \in W/W_\mu} w  \left( \frac{x_1^{2m} \wh{e}_{i-1}^{(2)}}{ \prod_{j=2}^m (x_1^2-x_j^2)} \right).
\end{align}
The inner sum is the same expression calculated in \S \ref{sssec:GLn-minimal-eps-proof}, except replacing $x_i$ by $x_i^2$ and also allowing substitutions $x_i \mapsto \pm x_i$, so that the answer is doubled compared to \S \ref{sssec:GLn-minimal-eps-proof}. Hence we conclude that 
\[
\nabla_\mu^\eta(e_i^{(2)})  = 4 (-1)^m (-1)^{m-1} e_i^{(2)} = -4 e_i^{(2)},
\]
as desired. 
\end{proof}

\subsubsection{Even orthogonal case} Let $\mu=(1,0,\cdots,0) \in \Z^m \cong \xcoch(T)$. Again we have $\pl_{\mu}=\pl_{x_{1}}$ and $t_{\mu}=x_{1}$. As $N = \dim G/P_\mu =  2m-2$, we have $\eta = t_\mu^{2m-1} = x_1^{2m-1}$.

\begin{lemma}
Suppose $n=2m$ is even. For $\eta = x_1^{2m-1}$ and $\mu=(1,0,\ldots,0)$, we have 
\[
\nabla^\eta_\mu(e_i^{(2)}) = 4 e_i^{(2)} \quad \text{ for all } i = 1, 2, \ldots, m.
\]
In particular, using homogeneous generators $e^{(2)}_{1},e^{(2)}_2, \ldots, e^{(2)}_{m-1}, \Pf$ for $R^{W}$, the eigenweights are
\begin{equation*}
4, 4, \ldots, 4, 2.
\end{equation*}
\end{lemma}

\begin{proof}
First we analyze the Pfaffian. A basis for $R^{W_{\mu}}$ over $R^{W}$ is given by $1, x_1, \ldots, x_1^{2m-2}$ and $\wh{\Pf} = x_2 \ldots x_m$. Since $G/P_\mu$ is embedded as a \emph{quadric} hypersurface under the line bundle corresponding to $(-x_1)$, the map $\int_{G/P_\mu} \co R^{W_{\mu}} \rightarrow R^W$ extracts \emph{twice} the coefficient of $(-x_{1})^{2m-2}$ in this basis. Hence
\[
\nabla_\mu^\eta(\Pf) = \int_{G/P_\mu} x_1^{2m-1} \partial_{x_1}(x_1 x_2 \cdots x_m) = \int_{G/P_\mu} x_1^{2m-2} \Pf = 2 \Pf.
\]

Let $\wh{e}_i^{(2)}$ be the $i$th elementary symmetric polynomial in $x_2^2, \ldots, x_m^2$. Then $\partial_{x_1} e_i^{(2)} = 2 x_1  \wh{e}_{i-1}^{(2)}$. We want to calculate 
\[
\nabla_\mu^\eta(e_i^{(2)}) =  \int_{G/P_\mu} 2 x_1^{2m}  \wh{e}_{i-1}^{(2)}.
\]
We will use Lemma \ref{lem:w-average-integration}. In this case, the equivariant Chern class of (the tangent bundle of) $G/P_\mu$ is 
\[
\frR_\mu = \prod_{j=2}^m (x_j-x_1)(-x_j-x_1) = \prod_{j=2}^m (x_1^2-x_j^2).
\]
Hence Lemma \ref{lem:w-average-integration} says that 
\[
\int_{G/P_\mu} 2x_1^{2m}  \wh{e}_{i-1}^{(2)} = 2 \sum_{w \in W/W_\mu} w  \left(\frac{x_1^{2m} \wh{e}_{i-1}^{(2)}}{\prod_{j=2}^m (x_1^2-x_j^2)} \right).
\]
The inner sum is the same expression calculated in \S \ref{sssec:GLn-minimal-eps-proof}, except replacing $x_i$ by $x_i^2$ and also allowing substitutions $x_i \mapsto \pm x_i$, so that the answer is doubled compared to \S \ref{sssec:GLn-minimal-eps-proof}. Hence we conclude that 
\[
\nabla_\mu^\eta(e_i^{(2)})  = 4 e_i^{(2)},
\]
as desired.\footnote{As a sanity check, note that $\Pf^{2}=e^{(2)}_{m}$. Since $\nb_{\mu}$ is a derivation, we have 
\begin{equation*}
\nb_{\mu}^{\eta}(\Pf^{2})=2 \Pf \nb_{\mu}^{\eta}(\Pf)=4\Pf^{2}.
\end{equation*}
which is consistent with the fact that $\nb_{\mu}(e^{(2)}_{m})=4e^{(2)}_{m}$.}
\end{proof}

\subsection{General linear groups: a case of non-minimal modification} So far we have only computed eigenweights for ``minimal'' modification types $\mu$, where the discrepancy between the bundles being modified is as small as possible. Beyond this case, the eigenweights quickly get very complicated. 

We will illustrate this for $G = \GL_n$, where the minuscule coweights are of the form $\mu=(1^m, 0^{n-m})$ for $1 \leq m < n$. We will calculate the corresponding eigenweights for $m=2$, so we assume $n\geq 3$. We take the Casimir element $\Omega = \frac{1}{2}\sum_{i=1}^n x_i^2$, so $t_\mu = \partial_\mu (\Omega) = x_1+x_2$. Then $\eta = t_{\mu}^{2(n-2)+1}$.

\begin{prop}\label{prop: GL len=2}
For $G = \GL_n$ ($n\geq 3$), $\mu=(1,1,0,\cdots, 0)$ and $\eta$ as above, we have 
\begin{equation}\label{eq:GL-l_2-eps}
\e_{i}(\eta, \mu)=\frac{1}{n}\binom{2n-2}{n-1}-\binom{2n-3}{n-i}+2\binom{2n-3}{n-i-1}-\binom{2n-3}{n-i-2}
\end{equation}
for $1\leq i\leq n$.
\end{prop}

After we first obtained the calculation \eqref{eq:GL-l_2-eps}, we contributed it to an internal benchmark at Google DeepMind, where it was eventually solved by Gemini Deep Think (\emph{IMO Gold Version}). After comparing the two solutions, we actually preferred Gemini's argument and will present it (with our own edits) instead of our original one. 

\subsubsection{Change of basis}
Recall that the \emph{power sums} in $\{x_1, \ldots, x_n\}$ are $p_k = \sum_{j=1}^n x_j^k$. Abbreviate $I$ for the augmentation ideal of $R^W$. 

\begin{lemma}\label{lem:elementary-polynomial-bases}
Modulo $I^2$, the following identities hold for $k \ge 1$:
\begin{enumerate}
    \item $p_k \equiv (-1)^{k-1} k e_k \pmod{I^2}$
    \item $p_k \equiv k h_k \pmod{I^2}$
\end{enumerate}
\end{lemma}
\begin{proof}
(1) We start with Newton's identity:
$$ k e_k = \sum_{i=1}^{k} (-1)^{i-1} e_{k-i} p_i = e_{k-1}p_1 - e_{k-2}p_2 + \cdots + (-1)^{k-1}e_0 p_k. $$
Consider the terms modulo $I^2$. For any term $e_{k-i} p_i$ with $i < k$, both $e_{k-i}$ and $p_i$ are in the ideal $I$. Therefore, their product lies in $I^2$. This means all terms except the last one vanish modulo $I^2$. Hence we have
$$ k e_k \equiv (-1)^{k-1} e_0 p_k  \pmod{I^2} = (-1)^{k-1} p_k \pmod{I^2}. $$ 

(2) The argument is analogous, using another of Newton's identities: 
$$ k h_k = \sum_{i=1}^{k} h_{k-i} p_i = h_{k-1}p_1 + h_{k-2}p_2 + \cdots + h_0 p_k. $$
\end{proof}

Thus, in the quotient space $\VV = I/I^2$, the power sum and elementary polynomial bases are related by a simple scalar multiple. The advantage of working with the $p_i$ is that their partial derivatives enjoy a simple recursive formula:
$$ \partial_\mu p_i = \left(\pderiv{}{x_1} + \pderiv{}{x_2}\right) \sum_{j=1}^n x_j^i = i(x_1^{i-1} + x_2^{i-1}). $$

\subsubsection{Expression in terms of Divided Differences} We prepare to apply Lemma \ref{lem:w-average-integration}. For $\mu = (1, 1, 0, \ldots, 0)$, we have 
\[
\frR_\mu = \prod_{i >2} (x_i-x_1)(x_i-x_2).
\]
Therefore, Lemma \ref{lem:w-average-integration} says that for $f \in R^{W_\mu}$, we have 
\begin{equation}\label{eq:gemini-1} \int_{G/P_\mu} f = \sum_{1 \le a < b \le n}  \frac{w_{a,b} (f)}{\prod_{i \neq a,b} (x_a-x_i)(x_b-x_i)}, 
\end{equation}
where $w_{a,b}$ is the permutation exchanging $(x_1, x_2) \leftrightarrow (x_a, x_b)$ and fixing the other $x_i$. Let $A(z) := \prod_{k=1}^n(z-x_k)$, so that  $A'(x_a) = \prod_{k \ne a}(x_a-x_k)$. Then we can rewrite
\[
\frac{1}{\prod_{i\neq a,b} (x_a-x_i)(x_b-x_i)} = \frac{(x_a-x_b)(x_b-x_a)}{A'(x_a)A'(x_b)} = - \frac{(x_a-x_b)^2}{A'(x_a)A'(x_b)}.
\]
Substituting this into \eqref{eq:gemini-1}, and then applying it to $f_i = \eta \cdot \partial_\mu p_i  = i (x_1+x_2)^{2n-3}(x_1^{i-1}+x_2^{i-1})$, we get:
$$ \nabla^\eta_\mu(p_i) = -i \sum_{a<b} \frac{(x_a+x_b)^{2n-3}(x_a^{i-1}+x_b^{i-1})(x_a-x_b)^2}{A'(x_a)A'(x_b)}. $$
By symmetry, we can rewrite this as a sum over ordered pairs $a \neq b$:
\begin{equation}\label{eq:nablaop1} \nablaop(p_i) = -i \sum_{a=1}^n \frac{x_a^{i-1}}{A'(x_a)} \left( \sum_{b \ne a} \frac{(x_a+x_b)^{2n-3}(x_a-x_b)^2}{A'(x_b)} \right).
\end{equation}

We will now invoke the theory of \emph{divided differences}. For a polynomial $f(z)$, the divided difference $f[x_1, \ldots, x_n]$ is determined recursively by $f[x_i] = f(x_i)$ and 
\[
f[x_r, \ldots, x_s] = \frac{f[x_{r+1}, \ldots, x_s] - f[x_r, \ldots, x_{s-1}]}{x_s - x_r}.
\]
For example, $f[x_1] = f(x_1)$ and $f[x_1, x_2] = \frac{f(x_2)-f(x_1)}{x_2-x_1}$. We may view the divided difference as a rational function in $x_1,\ldots, x_n$, or as a number if specific values $x_1, \ldots, x_n \in \ol \Q$ are chosen. Generally we adopt the former perspective, but the second will occasionally be useful for proofs; we will make clear which perspective is being taken. 	 

\begin{lemma}\label{lem:leading-divided-difference}
Let $f \in \Z[x]$. Then we have an identity of rational functions in $\ol \Q(x_1, \ldots, x_n)$,
\[
f[x_1, \ldots, x_n] = \sum_{j=1}^n \frac{f(x_j)}{A'(x_j)}
\]
\end{lemma}

\begin{proof}This being an identity of rational functions over an infinite integral domain $\Z$, with denominators being products of factors $(x_i-x_j)$, it suffices to prove the statement for all specializations of $(x_1, \ldots, x_n)$ to any $n$-tuple of pairwise distinct elements in $\Z$. So we fix any such specialization and treat the divided difference in $\Q$ for the proof. 
 
Then another interpretation of the divided difference $f[x_1, \ldots, x_n] \in \Q$ is as the leading coefficient of the unique degree $n-1$ polynomial $F(z)$ that passes through $(x_1, f(x_1)), \ldots, (x_n, f(x_n))$. In fact, the divided differences can also be characterized in terms of the ``Newton expansion'' of $F(z)$ as
\[
F(z) = \sum_{m=1}^n f[x_1, \ldots, x_m]  \prod_{i=1}^{m-1} (z-x_i),
\]
which implies the preceding characterization. By the Lagrange interpolation formula, we also have
\[
F(z) = \sum_{j=1}^n f(x_j) \frac{\prod_{i \neq j} (z-x_i)}{\prod_{i \neq j} (x_j-x_i)} = \sum_{j=1}^n f(x_j) \frac{A(z)}{(z-x_j)A'(x_j)}.
\]
From this expression, it is clear that the leading coefficient of $F(z)$ is
\[
\sum_{j=1}^n \frac{f(x_j)}{A'(x_j)},
\]
so this equals $f[x_1, \ldots, x_n]$, as desired. 
\end{proof}

\begin{lemma}\label{lem:complete-homogeneous} Let $P_j(z) = z^j$. Then we have 
\[
P_j[x_1, \ldots, x_n]  = h_{j-n+1}(x_1, \ldots, x_n) \in \Z[x_1, \ldots, x_n][\prod_{r <s } (x_r-x_s)^{-1}],
\]
where $h_{j-n+1}(x_1, \ldots, x_n)$ is the complete homogeneous polynomial in $x_1, \ldots, x_n$ of degree $j-n+1$ (which is $0$ by definition if $j-n+1<0$). 
\end{lemma}

\begin{proof}
Consider the generating function 
\begin{equation}\label{eq:gen-fctn-1}
\sum_{j=0}^\infty h_j(x_1, \ldots, x_n) t^j = \prod_{k=1}^n \frac{1}{ (1-x_k t)}.
\end{equation}
Lemma \ref{lem:leading-divided-difference} implies that 
\[
P_j[x_1, \ldots, x_n] = \sum_{i=1}^n \frac{x_i^j}{\prod_{k \neq i} (x_i-x_k)}.
\]
Consider another generating function. 
\[
\sum_{j \geq 0} P_j[x_1, \ldots, x_n] t^j = \sum_{i=1}^n \frac{1}{\prod_{k \neq i} (x_i-x_k)} \frac{1}{1-x_it} .
\]
We claim that the RHS is the partial fraction decomposition of $t^{n-1}\prod_{i=1}^n \frac{1}{1-x_i t}$. Indeed, upon multiplying by $\prod_{i=1}^n (1-x_i t)$ to clear denominators, we are comparing two polynomials of degree $\leq n-1$ in $t$, which when evaluated at each $t=1/x_i$ gives $1/x_i^{n-1}$. The claim is proved, so we deduce that 
\begin{equation}\label{eq:gen-fctn-2}
\sum_{j \geq 0} P_j[x_1, \ldots, x_n] t^j = t^{n-1}\prod_{i=1}^n \frac{1}{1-x_i t}
\end{equation}
and then the Lemma follows by comparing coefficients of the two generating functions \eqref{eq:gen-fctn-1} and \eqref{eq:gen-fctn-2}.
\end{proof}

For $1 \leq a \leq n$, let $Q_a(z) := (x_a+z)^{2n-3}(x_a-z)^2$. By Lemma \ref{lem:leading-divided-difference}, we have
\[
Q_a[x_1, \ldots, x_n] = \sum_{b=1}^n \frac{Q_a(x_b)}{A'(x_b)}. 
\]
Since $Q_a(x_a) = 0$, we can omit the summand with $b = a$, and we are then left with the identity 
\begin{equation}\label{eq:divided-diff-1}
Q_a[x_1, \ldots, x_n] = \sum_{b \ne a} \frac{(x_a+x_b)^{2n-3}(x_a-x_b)^2}{A'(x_b)}.
\end{equation}
Substituting this back into \eqref{eq:nablaop1}, we find that
\begin{equation}\label{eq:divided-diff-2}
\nablaop(p_i) = -i \sum_{a=1}^n \frac{x_a^{i-1}}{A'(x_a)} Q_a[x_1, \ldots, x_n].
\end{equation}

Define integers $K_k$ by the identity $Q(t) = (1+t)^{2n-3}(1-t)^2 = \sum_k K_k t^k$. Comparing the definitions of $Q(t)$ and $Q_a(z)$, we see that 
\begin{equation}\label{eq:Q_a-expansion}
Q_a(z) = \sum_k K_k x_a^k z^{2n-1-k}.
\end{equation}
By linearity of the formation of divided difference, we have
\[
Q_a[x_1, \ldots, x_n] = \sum_k K_k x_a^k ( z^{2n-1-k}[x_1, \ldots, x_n]) = \sum_k K_k x_a^k h_{n-k} 
\]
where we invoked Lemma \ref{lem:complete-homogeneous} in the last equality. Substituting this back into \eqref{eq:divided-diff-2} gives 
\begin{align}\label{eq:gemini-nablaop}
\nablaop(p_i) & = -i \sum_{a=1}^n \frac{x_a^{i-1}}{A'(x_a)}  \left(\sum_k K_k x_a^k h_{n-k}\right) = -i \sum_k K_k h_{n-k} \left(\sum_a \frac{x_a^{i-1+k}}{A'(x_a)}\right) \nonumber \\
& = -i \sum_k K_k h_{n-k} h_{i+k-n},
\end{align}
where in the last step we applied Lemma \ref{lem:complete-homogeneous} to the inner sum. 

\subsubsection{Reduction modulo $I^2$}
We now consider \eqref{eq:gemini-nablaop} modulo $I^2$. Since $h_j \in I$ for $j\ge 1$, a product of the form $h_a h_b$ lies in $I^2$ unless $a=0$ or $b=0$. With $h_0=1$, the sum simplifies to:
$$ 
\sum_k K_k h_{n-k} h_{i+k-n} \equiv K_n h_0 h_{i+n-n} + K_{n-i} h_{n-(n-i)} h_0 = (K_n + K_{n-i})h_i \pmod{I^2}. $$
Putting this into \eqref{eq:gemini-nablaop}, we have found that 
$$ \nablaop(p_i) \equiv -i(K_n + K_{n-i})h_i \pmod{I^2}. $$
Using the relation $h_i \equiv p_i/i \pmod{I^2}$ from Lemma \ref{lem:elementary-polynomial-bases}, we can simplify this to 
$$ \nablaop(p_i) \equiv -i(K_n + K_{n-i})\frac{p_i}{i} = -(K_n + K_{n-i})p_i \pmod{I^2}. $$

This shows that $\epsilon_i(\eta, \mu) = -(K_n + K_{n-i})$. 

\subsubsection{Explication of coefficients} Now we match this up with the statement of Proposition \ref{prop: GL len=2}. The coefficients $K_i$ come from $Q(t) = (1+t)^{2n-3}(1-t)^2$. By the binomial formula, we have 
$$ K_i = \binom{2n-3}{i} - 2\binom{2n-3}{i-1} + \binom{2n-3}{i-2}. $$
For $i=n$, this specializes to
$$ K_n = \binom{2n-3}{n} - 2\binom{2n-3}{n-1} + \binom{2n-3}{n-2}  = -\frac{1}{n}\binom{2n-2}{n-1}. $$
Hence we have
$$ \epsilon_i = -(K_n+K_{n-i}) = \frac{1}{n}\binom{2n-2}{n-1} - \left( \binom{2n-3}{n-i} - 2\binom{2n-3}{n-i-1} + \binom{2n-3}{n-i-2} \right). $$ 
This completes the proof of Proposition \ref{prop: GL len=2}. \qed  

\section{Arithmetic volume of shtukas for non-split groups}\label{sec:nonsplit-volume}
In this section, we generalize the results of \S \ref{sec:split-arithmetic-volume} on arithmetic volume to non-split reductive group schemes over $X$.

\subsection{Group schemes}\label{sssec:non-split-groups} We define the class of reductive group schemes that we will consider. Let $\nu: X'\to X$ be a finite \'etale Galois cover with Galois group $\G$. Let $\xi'\in \cohog{2}{X'}(1)$ be Poincar\'e dual to a point in each connected component of $X'$.\footnote{We want to allow $X'$ to be geometrically disconnected to allow nonsplitness coming from $k$.}

Let $G_{0}$ be a split connected reductive group over $k$ together with a split maximal torus $T_{0}$ and a Borel subgroup $B_{0}$ containing $T_{0}$. Suppose we are given an injective homomorphism
\begin{equation*}
\t: \G\inj \Aut(G_{0}, B_{0}, T_{0}).
\end{equation*}
Let $G$ be the connected reductive group scheme over $X$ defined as 
\begin{equation*}
G=(\Res_{X'/X}(G_{0}\times X'))^{\G}
\end{equation*}
where $\G$ acts diagonally on $G_{0}$ (via $\t$) and on $X'$. In other words, $G$ is the descent of the constant group $G_{0}\times X'$ to $X$ using $\t$ as the descent datum. 

\begin{remark}\label{rem:beyond-quasisplit} Our setup covers all quasisplit reductive groups over $X$. For the purpose of calculating arithmetic volumes, this essentially already captures the full generality of everywhere reductive group schemes over $X$; let us explain why. If $G'$ is a pure inner twist of $G$, then as already observed in \S \ref{sssec:pure-inner-twist}, $\Bun_G$ is canonically isomorphic to $\Bun_{G'}$. Therefore, there is no additional generality gained by considering pure inner twists. As explained in \S \ref{sssec:canonical-quasisplit-form}, every reductive $G \rightarrow X$ is an inner twist of a quasisplit $\GG \rightarrow X$. In particular, if $G$ is adjoint, then inner twists are automatically pure inner twists. In general, $G$ is isogenous to an adjoint group times a torus, which is quasisplit, and one can reduce the volume computation from the quasisplit case. 
\end{remark}

\subsection{The moduli stack of $G$-bundles} 
For $G \rightarrow X$ as in \S \ref{sssec:non-split-groups}, let $\Bun_G$ be the moduli stack of $G$-bundles on $X$. Its  set of \emph{geometric} connected components is denoted $\pi_{0}(\Bun_{G})$. We will define a $\Gamma$-invariant uniformization
\begin{equation*}
\xcoch(T_{0})\to \pi_{0}(\Bun_{G})
\end{equation*}

A result of Heinloth \cite[Theorem 2]{Hei10} identifies $\pi_{0}(\Bun_{G})$ with the $\G$-coinvariants $\pi_{0}(\Bun_{G_{0, X'}})_{\G}$. For each $\l\in \xcoch(T_{0})$, choose any geometric connected component $X'_{1}\subset X'_{\ol k}$, which induces a surjection $\xcoch(T_{0})\to \pi_{0}(\Bun_{G_{0}, X'_{1}})$. The target is a direct summand of $\pi_{0}(\Bun_{G_{0, X'}})$, and we further project it to the $\G$-coinvariants. The resulting map 
\begin{equation*}
\xcoch(T_{0})\to \pi_{0}(\Bun_{G})
\end{equation*}
is $\G$-invariant,  and is independent of the choice of the geometric connected component of $X'$ (for different choices will be equalized after passing to $\G$-coinvariants, because $\G$ permutes geometric components of $X'$ transitively). We again denote this map by 
\begin{equation*}
\l\mt \ov \l\in \pi_{0}(\Bun_{G}).
\end{equation*}

\subsubsection{Atiyah--Bott formula}\label{sssec:non-split-AB-formula}
We explicate the Atiyah--Bott formula in this case in a manner parallel to \S \ref{sss:fz}.

The universal $G_{0}$-bundle over $X'\times \Bun^{\om}_{G}$ is classified by a map
\begin{equation*}
\ev_{\om}: X'\times \Bun_{G}^{\om} \to \BB G_{0},
\end{equation*}
which then induces a map
\begin{equation*}
\homog{*}{X'}\ot R_{0}^{W_{0}}\to \cohog{*}{\Bun_{G}^{\om}}.
\end{equation*}
This map is invariant under the diagonal action of $\G$ on $\homog{*}{X'}$ and $R_{0}^{W_{0}}$, hence factors through the coinvariants
\begin{equation}\label{ev' nonsplit}
\ev'_{\om, \G}:(\homog{*}{X'}\ot R_{0}^{W_{0}})_{\G}\to \cohog{*}{\Bun_{G}^{\om}}.
\end{equation}
For $f\in R_{0}^{W_{0}}$ and $z\in \homog{*}{X'}$, we denote
\begin{equation*}
(f^{z})_{\G}:=\ev'_{\om, \G}(z\ot f)\in \cohog{*}{\Bun_{G}^{\om}}.
\end{equation*}

\begin{remark}\label{rem:homology-as-coinvariants} In general, for a local system $\LL$ on $X$, we may form its homology
\begin{equation*}
\homog{*}{X,\LL}.
\end{equation*}
If $\nu^{*}\LL$ becomes a geometrically constant local system, i.e., it is pulled back from a local system $\LL_{0}$ on $\Spec k$ with a $\G$-action, then we have a canonical isomorphism
\begin{equation*}
(\homog{*}{X'}\ot \LL_{0})_{\G}\isom \homog{*}{X,\LL}.
\end{equation*}
\end{remark}

Remark \ref{rem:homology-as-coinvariants} applies in particular to $\LL = R^W$, so that \eqref{ev' nonsplit} can be viewed as a map
\begin{equation*}
\ev'_{\om}: \homog{*}{X,R^{W}}\to \cohog{*}{\Bun_{G}^{\om}}.
\end{equation*}
Passing to associated graded induces the isomorphism 
\begin{equation*}
\homog{>0}{X,\VV}\to \Gr^{1}_{\aug}\cohog{*}{\Bun_{G}^{\om}}
\end{equation*}
which then induces the isomorphism of bigraded algebras from Theorem \ref{thm: atiyah-bott formula}:
\begin{equation}\label{eq:AB-nonsplit}
\AB^{\om}: 
\Sym_{\Qlbar}^{\bl} (\rH_*(X; \VV)_+) \rightarrow \Gr^{\bl}_{\aug} \rH^*(\Bun_G^\omega).
\end{equation}

\subsection{Hecke correspondences}\label{sss:Hk qs}
For a dominant\footnote{Dominant with respect to $B_0$ as in \S \ref{sssec:non-split-groups}} coweight $\mu \in \xcoch(T_{0})$, we define a Hecke correspondence 
\begin{equation*}
\xymatrix{& \Hk^{\mu}_G \ar[dl]_{h_{0}}\ar[dr]^{h_{1}}\ar[rr]^-{p_{X'}} & & X'\\
\Bun_{G} & & \Bun_{G}}
\end{equation*}
as follows. For a scheme $S/\F_q$, $\Hk^{\mu}_G(S)$ is the groupoid of $(x',\cF_{0},\cF_{1}, \a)$ where 
\begin{itemize}
    \item $x'\in X'(S)$, with image under $\nu$ denoted $x\in X(S)$.
    \item $\cF_{i}$ are $G$-bundles on $X\times S$, which pull back to $G_{0}$-bundles $\cF'_{i}$ over $X'\times S$. 
    \item $\a: \cF_{0}|_{X\times S-\G_{x}}\isom \cF_{1}|_{X\times S-\G_{x}}$ is an isomorphism of $G$-torsors, such that its pullback to $X'\times S$:
\begin{equation*}
\a': \cF'_{0}|_{X'\times S-\nu_{S}^{-1}(\G_{x})}\isom \cF'_{1}|_{X'\times S-\nu_{S}^{-1}(\G_{x})}
\end{equation*}
has relative position $\mu$ along $\G_{x'}\subset\nu_{S}^{-1}\G_{x}$. Here $\nu_{S}=\nu\times\id_{S}:X'\times S\to X\times S$. Note that
\begin{equation*}
\nu_{S}^{-1}\G_{x}=\coprod_{\g\in\G}\G_{\g(x')}\subset X'\times S
\end{equation*}
is the union of graphs of $\G$-conjugates of $x'$. The relative position requirement along $\G_{x'}$ then implies that the relative position of $\a'$ along $\g(x')$ is automatically $\g(\mu):=\t(\g)(\mu)$, with notation as in \S \ref{sssec:non-split-groups}.  
\end{itemize} 

For a connected component $\Bun_G^{\om}$, let ${}^{\om} \Hk^{\mu}_G\subset \Hk^{\mu}_G$ be the preimage of $\Bun_{G}^{\om}$ under $h_{0}$. Then ${}^{\om} \Hk^{\mu}_G$ is a correspondence
\begin{equation*}
\xymatrix{& {}^{\om} \Hk^{\mu}_G\ar[rr]^-{p_{X'}}\ar[dl]_{h_{0}}\ar[dr]^{h_{1}} & & X' \\
\Bun^{\om}_{G} & & \Bun^{\om'}_{G}}
\end{equation*}
where $\om'=\om+\ov \mu$. We have a canonical isomorphism
\begin{equation*}
\g_{\Hk}: {}^{\om} \Hk^{\mu}_G\isom {}^{\om}\Hk^{\g(\mu)}_G
\end{equation*}
sending $(x',\cF_{0},\cF_{1}, \a)$ to $(\g(x'),\cF_{0},\cF_{1}, \a)$.

\subsubsection{Canonical parabolic reduction}\label{sssec:non-split-parabolic-reduction}
As in the split case, the pullback of the universal $G_{0}$-bundle $(p_{X'}\times h_{0})^{*}\cF^{\univ}$ on ${}^{\om} \Hk^{\mu}_G$ has a canonical $P_{\mu}$-reduction classified by the map
\begin{equation*}
\ev_{\mu}^{\om}: {}^{\om} \Hk^{\mu}_G\to \BB P_{\mu}.
\end{equation*}
These fit into a canonical commutative diagram
\begin{equation}\label{Hk ev diag}
\xymatrix{{}^{\om} \Hk^{\mu}_G\ar[r]^-{\ev^{\om}_{\mu}}\ar[d]^-{\g_{\Hk}} & \BB P_{\mu}\ar[d]^{\t(\g)}\\
{}^{\om}\Hk^{\g(\mu)}_G\ar[r]^-{\ev^{\om}_{\g(\mu)}} & \BB P_{\g(\mu)}}
\end{equation}
We get embeddings
\begin{equation}\label{eq:non-split-tautological-classes}
e_{\g}: \g_{\Hk}^{*} \c \ev_{\g(\mu)}^{\om,*}=\ev_{\mu}^{\om,*} \c \t(\g)^{*}: R^{W_{\g(\mu)}}\to \cohog{*}{{}^{\om} \Hk^{\mu}_G}
\end{equation}

\subsection{Arithmetic volume of moduli of shtukas}\label{sssec:nonsplit-moduli-of-shtukas}
Let $\mu=(\mu_{1},\cdots, \mu_{r})$ be a sequence of dominant coweights of $T_{0}$, such that
\begin{equation}\label{eq:non-split-coroot-condition}
\ov \mu_{1}+\cdots+\ov \mu_{r}=0\in \pi_{0}(\Bun_{G}).
\end{equation}
As in \S\ref{sss:Sht split}, we define the iterated version of the Hecke stack $\Hk^{\mu}_G$ defined in \S\ref{sss:Hk qs}, and define $\Sht^{\mu}_{G}$ by the Cartesian square \eqref{def Sht}. Note that the legs of points of $\Sht_{G}^{\mu}$ are now in $X'$. We still denote by $p_{i}: {}^{\om}\Sht^{\mu}_{G}\to X'$ map recording the $i$th leg for $1\le i\le r$.

\subsubsection{Definition of arithmetic volume}

Let $D_{\mu_j}:=\j{2\r,\mu_j}=\dim G_{0}/P_{\mu_j}$. For each $1\le j\le r$, fix $\y_{j}\in R^{W_{\mu_j}}$ of degree $2(D_{\mu_j}+1)$ and $\y'_{j}\in R^{W_{\mu_j}}$ of degree $2D_{\mu_j}$.

Define ${}_{c}\Gamma^{ \eta_j + \xi' \eta'_j}_{ \mu_j}$ as in \S \ref{ssec:split-Gamma} (recall that $\xi' \in H^2(X')(1)$ is the fundamental class of the covering curve). For $\eta = (\eta_1 + \xi' \eta_1', \ldots, \eta_r + \xi' \eta_r')$, define
\[
{}_{c}\Gamma^{ \eta}_{ \mu} := {}_{c}\Gamma^{ \eta_r + \xi' \eta'_r}_{ \mu_r} \circ \ldots \circ {}_{c}\Gamma^{ \eta_1 + \xi' \eta'_1}_{ \mu_1} \co \rH^*_c(\Bun_G^\omega) \rightarrow \rH^*_c(\Bun_G^\omega). 
\]
The target component $\omega$ agrees with the source component $\omega$ by the assumption \eqref{eq:non-split-coroot-condition}.

\begin{defn}
Let ${}_{c}\Gamma^{ \eta + \xi' \eta'}_{ \mu}$ be as in \S \ref{ssec:split-Gamma}. Let $\mu = (\mu_1, \ldots, \mu_r)$ and $\eta = (\eta_1 + \xi' \eta_1', \ldots, \eta_r + \xi' \eta_r')$. The \emph{arithmetic volume} of ${}^{\om}\Sht^{\mu}_{G}$ with respect to $\eta$ is the \emph{graded} trace
\begin{equation}\label{eq:nonsplit-int-defn}
\vol({}^{\om}\Sht^{\mu}_{G}, \eta) := \Tr(
{}_{c}\Gamma^{ \eta}_{ \mu}\circ\Frob  \mid \rH_c^*(\Bun_G^\omega))
\end{equation}
\end{defn}

\begin{remark} In addition to the classes of the form $\ev_{j}^{*}\y_{j}$, we could pull back classes using the evaluation maps $\ev_{j, \g}: \Sht^{\mu}_{G}\to \Hk^{\mu}_G\xr{\g_{\Hk}}\Hk^{\g(\mu)}_G \xr{\ev_{\g(\mu)}} \BB P_{\g(\mu)}$ for any $\g\in \G$. However, because of the diagram \eqref{Hk ev diag}, one can rewrite such classes as pullbacks via $\ev_{j}^{*}$ again. Thus, as long as we are concerned with linear combinations of  terms pulled back by $\ev_{j, \g}$, the form considered in \eqref{eq:nonsplit-int-defn} does not restrict the generality.
\end{remark}

\begin{remark}
    The convergence of the trace in the definition of $\vol({}^{\om}\Sht^{\mu}_{G}, \eta)$ can be proved in a similar way as Proposition \ref{p:trace conv}. We omit the proof here.
\end{remark}

\subsubsection{Gross motive}
\label{ss: G Motive general}
We follow the notation of \S \ref{sssec:non-split-groups}. The constructions in the split case can be applied to $G_{0}$ together with its split maximal torus $T_{0}$ and Weyl group $W_{0}$ to give $R_{0}, R_{0}^{W_{0}}$ and $\VV' :=\Gr^{1}_{\aug}R_{0}^{W_{0}}$, etc.

Since $G$ is quasisplit as a group scheme over $X$, we have its abstract Cartan $T$ as a torus over $X$. The abstract Weyl group $W$ acting on $T$ is now a finite group scheme over $X$. We can then form $R=\Sym(\xch(T)_{\Qlbar}(-1))$, $R^{W}$ and $\VV=\Gr^{1}_{\aug}R^{W}$, all as local systems over $X$, in the same way as in the split case \S\ref{sss:motive}.  

For example, we view $\VV'$ as a geometrically constant local system on $X'$. It carries a $\G$-action induced from $\t$, which can be viewed as a $\G$-equivariant structure on the constant local system $\VV'$ on $X'$. Then $\VV$ is the descent of $\VV'$ to $X$. The generic stalk of $\VV(1)$ is the $\Qlbar$-realization of the Gross motive $\MM_{G}$, as a $\Gal(F^{s}/F)$-module.

\subsubsection{$L$-function}
Let $\EE$ be a local system on $X$. The $L$-function of $\EE$ is defined to be
\begin{equation*}
L_{X}(s,\EE)=\prod_{i\in \ZZ}\det(1-q^{-s}\Frob|\cohog{i}{X,\EE})^{(-1)^{i-1}}.
\end{equation*}
This is consistent with the usual definition of the $L$-function attached $\EE$ as a $\Gal(F^{s}/F)$-module.

We will apply this to the graded local system $\EE=\VV^{*}$ dual to $\VV$, but with a slight normalization. Recall that $\VV^{*}$ has a negative grading $\op_{d}(\VV_{2d})^{*}[2d]$. Let $\cohog{<0}{X,\VV^{*}}$ be the negative degree part of the cohomology. This removes exactly the summand $\cohog{2}{X, (\VV_{2})^{*}}$, which is zero if $Z(G)$ does not contain a nontrivial split torus. Let
\begin{equation*}
L^{*}_{X}(s,\VV^{*})=\prod_{i<0}\det(1-q^{-s}\Frob|\cohog{i}{X,\VV^{*}})^{(-1)^{i-1}}.
\end{equation*}
The normalization in the definition of $\EE$ removes the pole of $L_{X}(s,\VV^{*})$ at $s=0$.

\begin{exam}
For a local system $\EE$ on $X$, we have $L_X(s, \EE(1))  = L_X(s+1, \EE)$. Hence if $\VV = \bigoplus_i \Ql(-d_i)[-2d_i]$, then $\VV^* = \bigoplus \Ql(d_i)[2d_i]$ and $L_X(s, \VV^*) = \prod_{i} \zeta_X(s+d_i)$. In particular, if $\VV$ arises from a split group $G/X$ as in \S \ref{sec:split-arithmetic-volume}, then $L_X(s, \VV^*)$ agrees with the $L$-function $L_{X,G}(s)$ from \eqref{eq:split-L}, and $L_X^*(s, \VV^*)$ agrees with the normalized $L^*_{X,G}(s)$ from \eqref{eq:split-normalized-L}.
\end{exam}

\subsubsection{Local operators}
Recall the operator $\ov\nb_{\mu_j}^{\y_{j}}$ on $\VV'$ from \S \ref{ssec:nabla-operator}. For an Atiyah--Bott generator $f^{z}\in \homog{*}{X',\VV'}$ as in \S \ref{sssec:non-split-AB-formula}, consider the map
\begin{eqnarray*}
\wt\DD^{\y_{j}}_{\mu_j}: \homog{*}{X',\VV'}
&\to& \homog{*}{X',\VV'}^{\G} \\
f^{z} & \mt &\sum_{\g\in \G} (\ov\nb^{\y_{j}}_{\mu_j}(\g^{*} f))^{\g^{-1}_{*}z} 
\in \homog{*}{X',\VV'}^{\G}.
\end{eqnarray*}
It clearly factors through the coinvariants $\homog{*}{X',\VV'}_{\G}=\homog{*}{X, \VV}$.  Composing $\wt\DD^{\y_{j}}_{\mu_j}$ with the natural map from invariants to coinvariants  $\homog{*}{X',\VV'}^{\G}\to \homog{*}{X',\VV'}_{\G}$, we obtain an endomorphism
\begin{equation*}
\DD^{\y_{j}}_{\mu_j}\in \End^{gr}(\homog{>0}{X,\VV}).
\end{equation*}
It preserves the homological grading and the grading on $\VV$. Taking the adjoint, we obtain graded endomorphisms of $\cohog{*}{X,\VV^{*}}$, and in particular endomorphisms of $\cohog{<0}{X,\VV^{*}}$. We still denote these endomorphisms by $\DD^{\y_{j}}_{\mu_j}$.

\begin{assumption}\label{assump:nonsplit-operators-commute}
We assume that the operators $\DD^{\y_{j}}_{\mu_j}$ pairwise commute for $j=1, \ldots, r$. 
\end{assumption}

\begin{remark}Note that Assumption \ref{assump:nonsplit-operators-commute} is satisfied if the $\ov \nabla^{\eta_j}_{\mu_j}$ pairwise commute. This in turn is satisfied for many examples, as discussed in Example \ref{ex:assumption-satisfied}.
\end{remark}

\subsubsection{The volume formula} We will now introduce a multivariate version of $L^{*}(s,\VV^{*})$, taking into account the eigenvalues of the operators $\DD^{\y_{j}}_{\mu_j}$ on $\cohog{<0}{X,\VV^{*}}$. We decompose 
\begin{equation}\label{eigen coho V}
\cohog{<0}{X,\VV^{*}}=\bigoplus_{\e\in\Qlbar^{r}}\cohog{<0}{X,\VV^{*}}[\e]
\end{equation}
where for $\e=(\e_{1},\cdots, \e_{r})$, $\cohog{<0}{X,\VV^{*}}[\e]$ is the simultaneous generalized eigenspaces under $\DD^{\y_{j}}_{\mu_j}$ with generalized eigenvalue $\e_{j}$, for $1\le j\le r$. Let $\vec{s}=(s_{\e})$, with one variable for each $\e\in \Qlbar^{r}$ with a nontrivial summand in \eqref{eigen coho V}. Let
\begin{equation*}
L^{*}_{X}(\vec{s}, \VV^{*}):=\prod_{\e}\prod_{i<0}\det(1-q^{-s_{\e}}\Frob|\cohog{i}{X,\VV^{*}}[\e])^{(-1)^{i-1}}.
\end{equation*}
For the sequence $\mu = (\mu_1, \ldots, \mu_r)$ and a fixed $\omega \in \pi_0(\Bun_G)$, set $\omega_j :=\omega + \ol\mu_1 + \ldots + \ol\mu_j$. For $1\le j\le r$, consider the first order differential operator 
\begin{equation*}
\frd_{j}=d^{\om_{j-1}}_{\mu_j}(\y_{j})+d_{\mu_j}(\y'_{j})-(\log q)^{-1}\sum_{\e\in \Qlbar^{r}}\e_j \pl_{s_{\e_j}}.
\end{equation*}

\begin{theorem}\label{thm:non-split-volume}
Let $\mu = (\mu_{1},\cdots, \mu_{r})$ be a sequence of minuscule dominant coweights of $G_{0}$ satisfying \eqref{eq:non-split-coroot-condition}. For $j=1, \ldots, r$, let $\y_{j}\in \cohog{2D_{\mu_j}+2}{\BB L_{\mu_j}}$ and $\y'_{j}\in \cohog{2D_{\mu_j}}{\BB L_{\mu_j}}$ satisfying Assumption \ref{assump:nonsplit-operators-commute}. For $\eta = (\eta_1+\xi' \eta_1', \ldots, \eta_r+\xi' \eta_r')$, we have
\begin{equation*}
\vol({}^{\om}\Sht^{\mu}_{G}, \eta) =q^{\dim \Bun_{G}} \Big( \prod_{j=1}^{r}\frd_{j} \Big)L^{*}_{X}(\vec{s}, \VV^{*})\Big|_{\vec{s}=0}.
\end{equation*}
\end{theorem}

\subsection{Proof of Theorem \ref{thm:non-split-volume}}
The proof of Theorem \ref{thm:non-split-volume} follows the same strategy as in the split case. We first need the following generalization of Proposition \ref{p:master eqn} to the quasisplit case. Recall that in \S \ref{sssec:non-split-parabolic-reduction}, specifically \eqref{eq:non-split-tautological-classes}, we defined a map $R_0^{W_{\gamma(\mu)}} \rightarrow \cohog{*}{{}^\omega \Hk_G^\mu}$. 

\begin{prop}\label{p:master eqn qs}
Let $g'$ be the genus of $X'$. For $f\in R_{0}^{W_{0}}$ and $z\in \homog{|z|}{X'}$, we have
\begin{equation*}
h_{1}^{*}((f^{z})_{\G})-h_{0}^{*}((f^{z})_{\G})=\sum_{\g\in \G}\g^{*}\PD(z)e_{\g}(\pl_{\g(\mu)}f)+\begin{cases}(1-g')\j{z,\xi'}\xi'\sum_{\g\in \G}e_{\g}(\pl^{2}_{\g(\mu)}f), & |z|=2,\\
0, & \mbox{otherwise.}
\end{cases}
\end{equation*}

\end{prop}
\begin{proof}
We need a slight extension of Theorem \ref{th:master gen}: the modification between $G_{0}$-bundles $\cF_{0}$ and $\cF_{1}$ on $S$ is along a disjoint union of divisors $D=\coprod_{i\in I} D_{i}$, and it has type $\l_{i}$ along $D_{i}$. Then the same proof of Theorem \ref{th:master gen} shows that for $f\in R_{0}^{W_{0}}$, we have
\begin{equation}\label{diff Chern qs}
f(\cF_{1})-f(\cF_{0})=\sum_{i\in I}i_{D_{i}!}\left(\sum_{n\ge1}\frac{1}{n!}(\pl^{n}_{\l_{i}}f)(\cF_{1}|_{D_{i}, L_{\l_{i}}})\cup \nu_{D_{i}}^{n-1}\right).
\end{equation}
Here again $\nu_{D_{i}}=c_{1}(\cO(D_{i}))|_{D_{i}}\in \cohog{2}{D_{i}}(1)$ is the Chern class of the normal bundle of $D_{i}$.

We apply this to the two pullbacks of the universal bundle $\cF_{i}=(q_{X'}\times  h_{i})^{*}\cF^{\univ}$, $i=0,1$, on $S:=X'\times {}^{\om} \Hk^{\mu}_G$. The modification is along the divisor
\begin{equation*}
D=X'\times_{X} {}^{\om} \Hk^{\mu}_G\to X'\times_{X}X'=\coprod_{\g\in \G}\D_{\g},
\end{equation*}
where $\D_{\g}=\{(\g x', x')|x'\in X'\}\subset X'\times X'$ is the graph of $\g\in \G$. Hence $D=\coprod_{\g\in \G}D_{\g}$ where $D_{\g}$ is the preimage of $\D_{\g}$ in $X' \times_X {}^{\om} \Hk^{\mu}_G$. We identify $D_{\g}$ with ${}^{\om} \Hk^{\mu}_G$ using the projection $p_{\Hk}: S\to {}^{\om} \Hk^{\mu}_G$. Each $D_{\g}$ is the graph of $\g\c p_{X'}: {}^{\om} \Hk^{\mu}_G\to X'$. Let 
\begin{equation*}
i_{\g}: D_{\g}\cong{}^{\om} \Hk^{\mu}_G\xr{(\g\c p_{X'},\id)} X'\times {}^{\om} \Hk^{\mu}_G = S
\end{equation*}
be the inclusion of $D_{\g}$.  The modification type of $\cF_{0}\dashrightarrow \cF_{1}$ along $D_{\g}$ is $\g(\mu)$, giving rise to a $P_{\g(\mu)}$-reduction of $\cF_{1}|_{\G(\g|X')}$ classified by the map 
$$D_{\g}\cong{}^{\om} \Hk^{\mu}_G\xr{\g_{\Hk}}{}^{\om} \Hk^{\g(\mu)}_G \xr{\ev_{\g(\mu)}^{\om}}\BB P_{\g(\mu)}.$$  
Then \eqref{diff Chern qs} implies
\begin{equation*}
(q_{X'}\times h_{1})^{*}f-(q_{X'}\times h_{0})^{*}f=\sum_{\g\in \G}i_{\g!}\Big(\sum_{n\ge1} \frac{1}{n!}e_{\g}(\pl^{n}_{\g(\mu)}f)\cdot \nu_{D_{\g}}^{n-1}\Big).\end{equation*}
The same calculation as in the proof of Proposition \ref{p:master eqn} gives
\begin{equation*}
i_{\g!}\left(\sum_{n\ge1} \frac{1}{n!}e_{\g}(\pl^{n}_{\g(\mu)}f)\cdot \nu_{D_{\g}}^{n-1}\right)=[\D_{\g}]\cdot e_{\g}(\pl_{\g(\mu)}f)+(1-g')(\xi'\ot\xi')e_{\g}(\pl^{2}_{\g(\mu)}f).
\end{equation*}
Note that $[\D_{\g}]=(1\times \g)^{*}[\D_{X'}]$. Extract K\"unneth components by pairing with $z\in \homog{*}{X'}$, we get the desired formula.
\end{proof}

\subsubsection{Completion of the proof}
Passing from $\cohoc{*}{\Bun_{G}}$ to $\cohog{*}{\Bun_{G}}$ via Poincar\'e duality, we may rewrite the arithmetic volume as 
\begin{equation*}
\vol({}^\omega \Sht_G^\mu, \eta) = q^{\dim \Bun_{G}}\Tr\Big(\Frob^{-1}\c \prod_{j=1}^{r}\G^{\y_{j} + \xi' \y'_{j}}_{\mu_j}, \Sym(\cohog{>0}{X,\VV^{*}})\Big)
\end{equation*}

We use the Atiyah--Bott formula to identify $\rH^*(\Bun_G^{\omega_j})$ for the different components $\omega_j$ appearing as the source/target of $\G^{\y_{j} + \xi' \y'_{j}}_{\mu_j}$. Then we compute the eigenvalues of each $\G^{\y_{j} + \xi' \y'_{j}}_{\mu_j}$ for $j=1, \ldots, r$. To do this, we define the Ran filtration on $\cohog{*}{\Bun^{\om}_{G}}$ as in \S \ref{sssec:split-filtration}, and show that it is preserved by $\G_{\mu_j}^{\y_{j} + \xi' \y'_{j}}$. Using Proposition \ref{p:master eqn qs}, we find that the action of $\G_{\mu}^{\y + \xi' \y'}$ on the doubly associated graded
\begin{equation*}
\Gr^{\bl}_{\aug}\Gr^{F}_{\bu}\cohog{*}{\Bun^{\om}_{G}}\cong \Sym^{\bl}(\homog{>0}{X,\VV})
\end{equation*}
is the sum of the scalar operator $d^{\om_{j-1}}_{\mu_j}(\y_{j})+d_{\mu_j}(\y'_{j})$
defined in the same way as in \S \ref{sssec:deg-of-taut}, plus the derivation whose action on $[f^{z}]_{\G}\in \homog{>0}{X,\VV}$ is $\DD^{\y_{j}}_{\mu_j}$.

The rest of the proof is the same as in \S \ref{sssec:split-completion} for the split case. \qed

\subsection{Example: unitary groups}\label{ssec:unitary-groups}
We spell out the statement of Theorem \ref{thm:non-split-volume} in the case of non-split unitary groups of the type occurring in \cite{FYZ, FYZ2}. Indeed, an original motivation of the present work was to address a singular version of the Arithmetic Siegel--Weil formula. 

Let $\nu \co X' \rightarrow X$ be a finite \'etale double cover. Let $\Bun_{\U(n)}$ be as in \cite[\S 6.1]{FYZ}, so that $\Bun_{\U(n)}(S)$ is the groupoid of rank $n$ vector bundles $\cF$ on $X' \times S$ plus a Hermitian structure $h \co \cF \xrightarrow{\sim} \sigma^* \cF^*$. 

Let $\Hk_{\U(n)}^1$ be the Hecke stack of modifications of colength $1$ in the sense of \cite[Definition 6.5]{FYZ}. For a commutative $k$-algebra $R$, $\Hk_{\U(n)}^1(R) $ is the groupoid of
 \begin{equation}\label{eq:HkU-R-point}
 \{x' \in X'(R), \cF_{0}\supset \cF^\flat_{1/2}\subset  \cF_{1} \mid \cF_{0}, \cF_{1}\in \Bun_{\U(n)}(R)\}
\end{equation}
such that $\cF_{0}/\cF^\flat_{1/2}$ (resp. $\cF_{1}/\cF^\flat_{1/2}$) is a line bundle along the graph of $x'$  (resp.  $\sigma(x')$). In the notation of \S \ref{sssec:nonsplit-moduli-of-shtukas}, we have 
 \[
 \Hk_{\U(n)}^1 = \Hk_{\U(n)}^{\mu^\flat} \text{for $\mu^\flat = (0, 0, \ldots, -1)$}.
 \]
  Let $h_{0}, h_{1}: \Hk_{\U(n)}^1 \to \Bun_{\U(n)}$ be the maps recording $\cF_{0}$ and $\cF_{1}$ respectively. Let $p_{X'}: \Hk_{\U(n)}^1 \to X'$ be the map recording the support of $\cF_{0}/\cF_{1/2}^\flat$. Then we have a correspondence
\begin{equation*}
\xymatrix{ & \Hk_{\U(n)}^1 \ar[dl]_{h_{0}}\ar[dr]^{h_{1}}\\
\Bun_{\U(n)} & & \Bun_{\U(n)}
}
\end{equation*}

\subsubsection{Tautological bundles}  Canonical parabolic reduction gives a map $\Hk_{\U(n)}^1 \rightarrow G/P_{\mu^\flat} \cong \PP^{n-1}$. The pullback of $\cO(-1)$ is the ``tautological line bundle'' $\cP$ whose fiber along an $R$-point \eqref{eq:HkU-R-point} is $(\cF_0/\cF_{1/2}^\flat)^*$. 

\begin{remark}
This normalization is chosen to match with \cite{FYZ, FHM25} which takes the ``tautological bundle'' to be $\cF_1/\cF_{1/2}^\flat$, which is isomorphic to $\cF_{1/2}^{\sharp}/\cF_0$ where $\cF_{1/2}^{\sharp} = \sigma^* (\cF_{1/2}^\flat)^*$. This is in turn dual to $\cF_0/\cF_{1/2}^\flat$.  
\end{remark}

\subsubsection{Moduli of shtukas}

As in \cite[\S 2]{FYZ2}\footnote{but note that we are using different conventions on the similitude factor.}, let $\Hk_{\U(n)}^r$ be the iterated fibered product of $\Hk_{\U(n)}^1$ over $\Bun_{\U(n)}$, parametrizing 
\[
(x_1', \ldots, x_r', \cF_0 \dashrightarrow \cF_1 \dashrightarrow \ldots  \dashrightarrow \cF_r \cong \Frob^* \cF_0)
\]
and let $\Sht_{\U(n)}^r$ be the fibered product 
\[
\begin{tikzcd}
    \Sht_{\U(n)}^r \ar[r] \ar[d] & \Hk_{\U(n)}^r \ar[d, "{(h_0, h_r)}"] \\
    \Bun_{\U(n)} \ar[r, "{(\Id, \Frob)}"] & \Bun_{\U(n)} \times \Bun_{\U(n)}
\end{tikzcd}
\]

\subsubsection{The $L$-function}
In this case, the $L$-function $L_X(s, \EE)$ specializes to 
$$
L_{X,\U(n)}(s)=\prod_{i=1}^n L(s+i,\chi_{X'/X}^{i})
$$
where $\chi_{X'/X}$ is the quadratic character corresponding to $X'$. The multivariable version is 
\[
\sL_{X, \U(n)}(s_1, \ldots, s_n) = \prod_{i=1}^n L(s_i+i,\chi_{X'/X}^{i}).
\]

\subsubsection{Arithmetic Volume} Let $\cE$ be a rank $n$ vector bundle on $X'$. We define 
\begin{align*}
{}_c \Gamma_{1}^{\cE} \co \cohoc{*}{\Bun_G} & \rightarrow \cohoc{*}{\Bun_G} \\
 \theta & \mapsto h_{1*} (h_0^* \theta \cup c_n(p_{X'}^* \cE^* \otimes \cP))
\end{align*}
analogously to \S \ref{sssec:gln-example-operators}. We define ${}_c \Gamma_{r}^{\cE} = ({}_c \Gamma_{1}^{\cE})^{\circ r}$. Finally, we define the arithmetic volume 
\[
\vol(\Sht^{r}_{\U(n)}, \prod_{i=1}^{r}c_{n}(p_{i}^{*}\cE^{*}\ot\cP))  := \Tr({}_c \Gamma_{r}^{\cE}\circ\Frob, \cohoc{*}{\Bun_G})
\]
where the trace is taken in the graded sense. 

\begin{theorem}\label{th: vol U} 
Let $\cE$ be a rank $n$ vector bundle on $X'$ of degree $D$. If $r$ is even, then we have 
\begin{equation*}
\vol(\Sht^{r}_{\U(n)}, \prod_{i=1}^{r}c_{n}(p_{i}^{*}\cE^{*}\ot\cP))=2 \frac{q^{n^{2}(g-1)}}{(\log q)^{r}}\left(\frac{d}{ds}\right)^{r}\Big|_{s=0}(q^{-Ds}L_{X,\U(n)}(2s)).
\end{equation*}
\end{theorem}

\begin{remark}
In the case $n=1$, Theorem \ref{th: vol U} recovers \cite[Theorem 10.2]{FYZ2} for $\frL = \omega_X^{-1}$. 
\end{remark}

\begin{proof}
There are two connected components $\Bun_{\U(n)} = \Bun_{\U(n)}^{\omega_0} \sqcup \Bun_{\U(n)}^{\omega_1}$, inducing a decomposition
\[
\Sht_{\U(n)}^r = {}^{\omega_0} \Sht_{\U(n)}^r \sqcup {}^{\omega_1} \Sht_{\U(n)}^r.
\]

We will show more precisely that 
\[
\vol({}^{\omega} \Sht_{\U(n)}^r,  \prod_{i=1}^{r}c_{n}(p_{i}^{*}\cE^{*}\ot\cP))= \frac{q^{n^{2}(g-1)}}{(\log q)^{r}}\left(\frac{d}{ds}\right)^{r}\Big|_{s=0}(q^{-Ds}L_{X,\U(n)}(2s)).
\]
for each $\omega \in \{\omega_0, \omega_1\}$. 

As discussed in Example \ref{ex:assumption-satisfied}, Assumption \ref{assump:nonsplit-operators-commute} is automatically satisfied. Hence we may apply Theorem \ref{thm:non-split-volume} in order to calculate the left side. 

We compute (e.g., by the splitting principle) 
\begin{equation*}
c_n(p_{X'}^* \cE^{*} \otimes \cP) = c_1(\cP)^n  +  p_{X'}^*c_1(\cE^{*}) c_1(\cP)^{n-1}.
\end{equation*}
For the pullback map $\ol \Q_\ell[x_1, \ldots, x_n]^{S_{n-1} \times S_1} = R_0^{W_{\mu^\flat}} \rightarrow \rH^*(\Hk_G^\mu)$ induced by canonical parabolic induction, our conventions are arranged so that $c_1(\cP)$ agrees with the image of $-x_n$ (cf. \S \ref{sssec:gln-example-tautological}). Hence we have
\[
c_n(p_{X'}^* \cE^* \otimes \cP) = (-x_n)^n - D \xi  (-x_n)^{n-1} = (-1)^n(x_n^n + D \xi' x_n^{n-1} ) .
\]
Hence we find that ${}_c\Gamma^{\cE}_{1} = {}_c\Gamma^{\eta}_{\mu^\flat}$ for $\eta = (-x_n)^n$ and $\eta' = - D (-x_n)^{n-1}$. To apply Theorem \ref{th:vol gen}, we need to calculate the constants $d_{\mu^\flat}^\omega(\eta)$, $d_{\mu^\flat}(\eta')$, and $\epsilon_i(\eta, \mu^\flat)$. 
\begin{itemize}

\item Using \eqref{eq:integration-GLn-colength-one} for $i=1$, we find that $d_{\mu^\flat}^\omega (\eta) =- e_1^{[X]} = 0$.

\item In this situation, $d_{\mu^\flat}(\eta')$ was already computed in the proof of Theorem \ref{th:main GLn}. There we saw that $ d_{\mu^\flat}(\eta') = -  \int_{G/P_{\mu^\flat}} (-x_n)^{n-1} D = (-1)^n D$.

\item We have $\nu^* \U(n) = \GL_n$, and $R_0^{W_0} = \ol \Q[e_1, \ldots, e_n]$. The eigenweights of $\ol \nabla^\eta_{\mu^\flat}$ were already computed in the proof of Theorem \ref{th:main GLn}, where we saw that they were all equal to $(-1)^{n-1}$. 

We need to compute the eigenvalues of $\DD_{\mu^\flat}^\eta$. The generator of $\Gamma \cong \Z/2\Z$ takes $e_i \mapsto (-1)^i e_i$. Write $\rH^*(X')  = \rH^*(X')^{(+1)} \oplus \rH^*(X')^{(-1)}$ for the eigenspace decomposition under $\Gamma$. We have $\VV'= \bigoplus_{i=1}^n \VV'_{2i}$ where $\Gamma$ acts on $\VV'_{2i}$ by the $i$th power of the sign character. Hence $\rH_*(X', (\VV')^*_{2i})^{(+)} = \rH_*(X')^{(-1)^{i}} (i)$, on which the operator $\sum_{\gamma \in \Gamma} (\gamma^* f)^{\gamma_*^{-1} z}$ acts as multiplication by $2$. Therefore the eigenvalues of $\DD_{\mu^\flat}^\eta$ are all equal to $(-1)^{n-1} 2$, and  
\[
\rH^{<0}(X; \VV^*)[\epsilon] = \bigoplus_{i \text{ even}}  \rH^{<0}(X')^{(+1)}(i)[2i] \oplus \bigoplus_{i \text{ odd}} \rH^{<0}(X')^{(-1)}(i)[2i].
\]
\end{itemize}
Hence we conclude that 
\begin{equation*}
\frd_{j} = (-1)^{n-1}(-D + 2(\log q)^{-1}\sum_{i=1}^{n} \pl_{s_{i}}).
\end{equation*}
Since $r$ is even, we may ignore the sign when composing $r$ such operators. Observing that $\dim \Bun_G = n^2(g-1)$, Theorem \ref{th:vol gen} says that 
\begin{align*}
& \vol({}^{\omega_i}\Sht^{r}_{\U(n)}, \prod_{j=1}^{r}c_{n}(p_{j}^{*}\cE^{*} \ot\cP))     \\
& = q^{n^2(g-1)}   \prod_{j=1}^r  \Big( -D+ 2(\log q)^{-1}\sum_{i=1}^{n} \pl_{s_{i}} \Big) \sL^{*}_{X,\U(n)}(s_{1},\cdots, s_{n})\Big|_{s_{1}=s_{2}=\cdots=s_{n}=0}.
\end{align*}
This agrees with the claimed formula by \eqref{eq:partial-derivative-identity}.

\end{proof}

\section{The phantom tautological ring}

As in the study of other well-known moduli spaces, it is natural to consider the subring of the cohomology ring of $\Sht^{\mu}_G$ generated by tautological classes. In this section,  we construct a ring $C^{\mu}_G$ by generators and relations that maps to $\cohog{*}{\Sht^{\mu}_G}$, with image consisting of tautological classes. We call $C^{\mu}_G$ the ``phantom tautological ring" for $\Sht^{\mu}_G$, because it maps to $\cohog{*}{\Sht^{\mu}_G}$ but the map is generally not injective. We also interpret the volume calculation in Section \ref{ss:vol split} as a linear functional on $C^{\mu}_G$.

\subsection{The phantom tautological ring}\label{ssec:taut-ring} We assume until \S \ref{ssec:phantom-reductive} that $G$ is split and semisimple. We keep the setup as in \S \ref{sss:Sht split}. 


Recall the following maps from $\Sht^\mu_G$: for $0\le i\le r$ the map
\begin{equation*}
    h_i: \Sht^\mu_G\to \Bun_G
\end{equation*}
records $\cF_i$; for $1\le i\le r$ the map
\begin{equation*}
    p_i: \Sht^\mu_G\to X
\end{equation*}
records the $i$th leg. For the modification at the $i$th leg $\cF_{i-1}\dashrightarrow \cF_i$ of type $\mu_i$, $\cF_{i-1}$ carries a canonical reduction to $P_{\mu_i}$ at $x_i$. This gives map
\begin{equation*}
    \ev_{i}: \Sht^\mu_G\to \Hk^\mu_G \to \BB P_{\mu_i}.
\end{equation*}

Pulling back along $p_i$ and $\ev_i$ gives a ring homomorphism
\begin{equation}\label{taut hom}
    \wt\rho:\wt C^\mu_G:=\bigotimes_{i=1}^r(\cohog{*}{X}\ot R^{W_{\mu_i}})\to \cohog{*}{\Hk^\mu_G} \to \cohog{*}{\Sht^\mu_G}.
\end{equation}
Abusing notation, we will use the notation $\wt \rho$ both for the pullback to $\cohog{*}{\Hk^\mu_G}$, and the pullback to $\cohog{*}{\Sht^\mu_G}$. If we equip $\wt C^\mu_G$ with the Frobenius action induced from each tensor factor, $\wt\r$ is Frobenius equivariant. 

\begin{defn}
    For each $\om\in \pi_0(\Bun_G)$, the {\em tautological ring} for ${}^\om\Sht^\mu_G$ is the image of the composition
    \begin{equation*}
        \wt C^\mu_G\xr{\wt\rho}\cohog{*}{\Sht^\mu_G}\to \cohog{*}{{}^\om\Sht^\mu_G}.
    \end{equation*}
\end{defn}

\begin{remark}
    We do not include pullbacks of classes from $\Bun_G$ in the source of the map $\wt\r$. A posteriori, we show in Corollary \ref{c:pullback BunG in taut} that the image of $\wt\r$ contains all classes pulled back from $\Bun_G$ via $h_i$ for all $0\le i\le r$.
\end{remark}

We will show that this map factors through a much smaller quotient ring of the left side.

{\bf Notation:} below, we abbreviate $\phi$ for the Frobenius pullback endomorphism on the cohomology of stacks over $\FF_q$.



\subsubsection{Classes on $X\times X$}

Let $\D_{X}\subset X\times X$ be the diagonal. For any integer $d\ne 0,1$ we define the following classes in $\cohog{2}{X\times X}$
\begin{eqnarray*}
    \Xi_d=\left(\frac{1}{q^d\phi^{-1}-1}\ot \id\right)[\D_X]=\left(\id\ot\frac{1}{q^{d-1}\phi-1}\right)[\D_X].
\end{eqnarray*}
The equality of the two expressions follows from the fact that $\phi$ and $q\phi^{-1}$ are adjoint under the Poincar\'e duality pairing on $\cohog{*}{X}$.

Let $\{\z_j\}_{1\le j\le 2g}$ be a basis for $\cohog{1}{X}$ consisting of eigenvectors under $\phi$ with eigenvalues $\a_j$, and let $\{\z^j\}$ be the basis of $\cohog{1}{X}$ such that $\int_X\z_i\z^j=\d_{ij}$. In terms of these bases we have $\D_X = 1 \otimes \xi  - \sum_{j=1}^{2g} \z_j \z^j + \xi \otimes 1$, so we can rewrite $\Xi_d$ as
\begin{eqnarray*}
    \Xi_d=(q^d-1)^{-1}1\ot \xi+\sum_{j=1}^{2g}(1-\a_jq^{d-1})^{-1}\z_j\ot \z^j+(q^{d-1}-1)^{-1}\xi\ot 1.
\end{eqnarray*}

\begin{lemma}\label{l:diag Xid} 
For any integer $d\ne0,1$ we have
\begin{eqnarray*}
\D_X^*\Xi_d=-\frac{\z'_X(d)}{\log(q)\z_X(d)}\xi.
\end{eqnarray*}
\end{lemma}
\begin{proof}
We have
\begin{eqnarray*}
    \D_X^*\Xi_d&=&\xi\int_{X\times X}[\D_X]\cdot \Xi_d=\xi\int_{X\times X}[\D_X]\cdot \left(\frac{1}{q^d\phi^{-1}-1}\ot\id\right)[\D_X]\\
    &=&\Tr((q^d\phi^{-1}-1)^{-1}|\cohog{*}{X})\xi
\end{eqnarray*}
Thus it suffices to show the equality of rational functions in $q^s$
\begin{equation*}
    -\frac{d}{ds}\log\z_X(s)=\log(q)\Tr((q^s\phi^{-1}-1)^{-1}|\cohog{*}{X}).
\end{equation*}
Now note $\z_X(s)=\det(1-q^{-s}\phi|\cohog{*}{X})^{-1}$ (alternating product of determinant of graded pieces). Thus it suffices to show that for any finite-dimensional graded vector space $V=\op V_i$ over $\CC$ and a graded automorphism $\phi$ of $V$, we have an equality in $\CC(q^s)$
\begin{equation}\label{log det}
    \frac{d}{ds}\log\det(1-q^{-s}\phi|V)=\log(q)\Tr((q^s\phi^{-1}-1)^{-1}|V).
\end{equation}
Both sides are additive in $(V,\phi)$ in short exact sequences, therefore we reduce to the case $\dim V=1$. Changing the parity of grading results in a negative sign on both sides, so we may assume $V=V_0$. In this case, $\phi$ acts on $V$ through a scalar $\a\in \CC^\times$. The equality \eqref{log det} becomes
\begin{equation*}
    \frac{d}{ds}\log(1-\a q^{-s})=\frac{\log(q)}{\a^{-1}q^s-1},
\end{equation*}
which is a direct calculation.


\end{proof}


\subsubsection{An ideal in $\wt C^\mu_G$}
In the ring $\wt C^\mu_G$ from \eqref{taut hom}, for $\z\in \cohog{*}{X}$, let $[\z]_i$ be the class $\z$ put at the $i$th factor of $\cohog{*}{X}$. Similarly, for $1\le i<i'\le r$ and $\Xi\in \cohog{*}{X\times X}$ we have $[\Xi]_{i,i'}$ put at the $i$th and $i'$th factors. For $f\in R^{W_{\mu_i}}$, let $[f]_{i}$ be $f$ put at the factor $ R^{W_{\mu_i}}$.

Let $I^\mu_G\subset \wt C^\mu_G$ be the ideal generated by the following elements for varying homogeneous $f\in R^W$ of degree $2d>2$ and $i=1,2,\cdots, r$
\begin{eqnarray}\label{rel in C}
    D_i(f)&:=&[f]_{i}+\frac{\z'_X(d)}{\log(q)\z_X(d)}[\xi]_{i}\cdot [\pl_{\mu_{i}}f]_{i}\\
\notag    &-&\sum_{i'=i+1}^{r}\left([\Xi_d]_{i, i'}\cdot [\pl_{\mu_{i'}}f]_{i'}+\frac{1-g
    }{q^{d-1}-1}[\xi]_{i}\ot [\xi]_{i'}\cdot [\pl_{\mu_{i'}}^2f]_{i'}\right)\\
\notag    &+&\sum_{i'=1}^{i-1}\left([\Xi_{1-d}]_{i', i}\cdot [\pl_{\mu_{i'}}f]_{i'}+\frac{1-g
    }{q^{1-d}-1}[\xi]_{i'}\ot [\xi]_{i}\cdot [\pl_{\mu_{i'}}^2f]_{i'}\right).
\end{eqnarray}
Note that the above element is an eigenvector under Frobenius with eigenvalue $q^d$. In particular, $I^\mu_G$ is stable under Frobenius.

\begin{defn}\label{def:taut}
    The {\em phantom tautological ring} for $\Sht^\mu_G$ is the quotient ring
    $$C^\mu_G:=\wt C^\mu_G/I^\mu_G$$
    equipped with the action of $\Frob$.
\end{defn}

To write $D_i(f)$ in more manageable way,  we introduce the following element in $\wt C^\mu_G$, for $f\in R^W$
\begin{equation*}
    \d_i(f)=[\pl_{\mu_i}f]_i+\frac{1}{2}[c_1(TX)]_i[\pl^2_{\mu_i}f]_i=[\pl_{\mu_i}f]_i+(1-g)[\xi]_i[\pl^2_{\mu_i}f]_i.
\end{equation*}
We have introduced the notation $[-]_{i,i'}$ for $1\le i<i'\le r$. For $i=i'$, define $[-]_{i,i}$ to be the composition
\begin{equation*}
    [-]_{i,i}: \cohog{*}{X\times X}\xr{\D^*}\cohog{*}{X}\xr{[-]_i} \wt C^\mu_G.
\end{equation*}
Using this notation and Lemma \ref{l:diag Xid}, we have
\begin{equation*}
    [\Xi_d]_{i,i}=[\D_X^*\Xi_d]_i=-\frac{\z'_X(d)}{\log(q)\z_X(d)}[\xi]_{i}.
\end{equation*}
Using these notations, we can rewrite $D_i(f)$ as
\begin{equation}\label{rel in C better}
    D_i(f)=[f]_i-\sum_{i'=i}^{r}[\Xi_d]_{i,i'}\d_{i'}(f)+\sum_{i'=1}^{i-1}[\Xi_{1-d}]_{i',i}\d_{i'}(f).
\end{equation}



\begin{theorem}\label{th:taut hom quot}
    The homomorphism \eqref{taut hom} factors through the phantom tautological quotient $C^\mu_G$
    \begin{equation}\label{taut hom quot}
        \r: C^\mu_G\to \cohog{*}{\Sht^\mu_G}.
    \end{equation}
\end{theorem}

Restricting to each component of $\Sht^\mu_G$, we get a ring homomorphism from $C^\mu_G$ to the tautological ring of ${}^\om\Sht^\mu_G$.

\begin{remark}
    We will see from Proposition \ref{p:taut ring free over HX} that as a \emph{module} over $\cohog{*}{X^r}$ (but not as a ring), $C^\mu_G$ is isomorphic to $\bigotimes_{i=1}^{r}(\cohog{*}{X\times G/P_{\mu_i}})$. The map $\r$ is an analog of the same-named map for Hermitian locally symmetric spaces in \eqref{Herm rho}. 
\end{remark}

\begin{remark}
    The map $\r$ is not injective: as we will see from Corollary \ref{c:taut ring top deg}, the top non-vanishing degree of $C^\mu_G$ is in degree $2r+2\sum_{i=1}^r\dim G/P_{\mu_i}=2\dim \Sht^\mu_G$. However, since $\Sht^\mu_G$ is not proper, the top degree cohomology vanishes. However, we have reasons to believe that $C^\mu_G$ is the ``correct" tautological ring for $\Sht^\mu_G$, supplying phantom cohomology classes that become zero in $\Sht^\mu_G$ due to non-properness. This point of view can be used to give a meaning to the arithmetic volume calculation -- see \S\ref{sss:vol functional} and \S\ref{sss:taut duality}. 
\end{remark}

\begin{cor}\label{c:taut hom gen}

    Let $\ov \y_{r}$ be a geometric generic point of $X^r$. The restriction of the homomorphism $\r$ to the geometric generic fiber $\Sht^\mu_{G,\ov\y_r}$  vanishes on $[f]_i$ for all $f\in R^W_+$ and $1\leq i\leq r$. 
    
    
    In particular, $\r$ induces a homomorphism $\r_{\ov\eta_{r}}$:
    \begin{equation}\label{taut hom gen}
        \r_{\ov\eta_{r}}: \bigotimes_{i=1}^r\cohog{*}{G/P_{\mu_i}}\to \cohog{*}{\Sht^\mu_{G,\ov\eta_{r}}}.
    \end{equation}
\end{cor}
\begin{proof}
The restriction to a geometric generic point of every term other than $[f]_i$ in the relation \eqref{rel in C} obviously vanishes.
\end{proof}

\subsection{Proof of Theorem \ref{th:taut hom quot}}

We need to check that elements of the form \eqref{rel in C} are sent to zero under \eqref{taut hom}. 
Now we will abuse the notation to denote by $\d_i(f)$ its image in $\cohog{*}{\Sht^\mu_G}$, i.e.,
\begin{equation*}
    \d_i(f)=\ev_i^*(\pl_{\mu_i}f)+\frac{1}{2}p_i^*(c_1(TX))\ev_i^*(\pl^2_{\mu_i}f).
\end{equation*}
For $1\le i\le i'\le r$, we also use $[-]_{i,i'}$ to denote the pullback along $(p_i, p_{i'}): \Sht^\mu_G\to X\times X$
\begin{equation*}
    [-]_{i,i'}: \cohog{*}{X\times X}\to \cohog{*}{\Sht^\mu_G}.
\end{equation*}
Thus $[\Xi]_{i,i'}\in \cohog{*}{\Sht^\mu_G}$ is the image of $[\Xi]_{i,i'}\in \wt C^\mu_G$ under $\wt\rho$. For $i=i'$, $[-]_{i,i}=p_i^*\D_X^*(-)$. 

Thus we need to check for $f\in R^W$ of degree $2d>2$ and $1\le i\le r$ that
\begin{eqnarray}\label{evif goal}   \ev_{i}^*f=\sum_{i'=i}^{r}[\Xi_d]_{i,i'} \d_{i'}(f)-\sum_{i'=1}^{i-1}[\Xi_{1-d}]_{i',i} \d_{i'}(f).
\end{eqnarray}

\subsubsection{The case $i=1$}
We first check \eqref{evif goal} for $i=1$, in which case it reads
\begin{equation}\label{evif goal 1}
    \ev_{1}^*f=\sum_{i'=1}^{r}[\Xi_d]_{1,i'} \d_{i'}(f).
\end{equation}

We use the same notation $h_i, p_i$ and $\ev_i$ to denote the counterparts of $h_i,p_i$ and $\ev_i$ as maps from $\Hk^\mu_G$. Recall that for $\cF_\bu\in \Hk^\mu_G$, $\ev_i(\cF_\bu)$ is the canonical $P_{\mu_i}$-reduction of $\cF_{i-1}$. This gives a commutative diagram for $1\le i\le r$
\begin{equation*}   
\xymatrix{\Hk^\mu_G\ar[r]^-{\ev_{i}}\ar[d]_{(p_i,h_{i-1})} & \BB P_{\mu_i} \ar[d]\\
    X\times \Bun_G\ar[r]^-{\ev} & \BB G}
\end{equation*}
The right vertical map is induced by the inclusion $P_{\mu_i}\incl G$. Therefore for $f\in R^W=\cohog{*}{\BB G}$, we have
\begin{equation*}
    (p_i, h_{i-1})^*\ev^*f=\ev_{i}^*f.
\end{equation*}

By Proposition \ref{p:master eqn}, we have for any $f\in R^W$, $z\in \upH_*(X)$ and $1\le i\le r$
\begin{equation*}
    h_i^*(f^z)-h_{i-1}^*(f^z)=p_i^*\PD(z)\ev_{i}^*(\pl_{\mu_i}f)+(1-g)\j{z,\xi}p_i^*\xi\cdot  \ev_{i}^*(\pl^2_{\mu_i}f)=p_i^*\PD(z)\cdot \d_i(f).
\end{equation*}
Adding these up for $i=1,\cdots, r$ we get
\begin{equation}\label{diff hr h0}
    h_r^*(f^z)-h_{0}^*(f^z)=\sum_{i=1}^{r}p_i^*\PD(z)\cdot \d_i(f).
\end{equation}
Now pulling back to $\Sht^\mu_G$ and using that $h_r=\Frob_{\Bun_G}\c h_0$, we get
\begin{equation}\label{h0fz pre}
    (\phi-1)(h_0^*(f^z))=h_0^{*}(\phi(f^z)-f^z)=\sum_{i=1}^{r}p_i^{*}\PD(z)\cdot \d_i(f).
\end{equation}
Therefore
\begin{equation}\label{h0fz}
    h_0^*(f^z)=(\phi-1)^{-1}\sum_{i=1}^{r}p_i^{*}\PD(z)\cdot \d_i(f).
\end{equation}
Using that $\phi$ acts on $\d_i(f)$ by $q^{d-1}$ we get
\begin{equation*}
    h_0^*(f^z)=\sum_{i=1}^{r}p_i^{*}\left((q^{d-1}\phi-1)^{-1}\PD(z)\right)\cdot \d_i(f).
\end{equation*}
Choose a basis $\{z_j\}$ of $\homog{*}{X}$ with dual basis $\{\z^j\}$ of $\cohog{*}{X}$, we get
\begin{equation}\label{ev1f almost}
\ev_1^*f=\sum_j\z^jh_0^*(f^{z_j})=\sum_{i=1}^{r}\left[\sum_{j}\z^j\ot(q^{d-1}\phi-1)^{-1}\PD(z_j)\right]_{1,i}\d_i(f).
\end{equation}
Observe that 
\begin{equation*}
    \sum_{j}\z^j\ot(q^{d-1}\phi-1)^{-1}\PD(z_j)=\left(\id\ot \frac{1}{q^{d-1}\phi-1}\right)[\D_X]=\Xi_d.
\end{equation*}
Thus \eqref{ev1f almost} implies \eqref{evif goal 1}.

\subsubsection{Partial Frobenius and general $i$}
To deduce the general case from $i=1$ case, we consider partial Frobenius on $\Sht^\mu_G$. Let $\mu'=(\mu_2,\cdots, \mu_r,\mu_1)$ be obtained from $\mu$ by a cyclic permutation. We have the partial Frobenius map
\begin{equation*}
    \frP: \Sht^\mu_G\to \Sht^{\mu'}_G
\end{equation*}
sending $(x_1,\cdots, x_r, \cF_0,\cdots, \cF_r={}^{\t}\cF_0)$ to $(x_2,\cdots, x_r,\Frob(x_1), \cF_1,\cdots, \cF_r, \cF_{r+1}={}^{\t}\cF_1)$. We introduce the following  notation: for $i'>r$ we define $\mu_{i'}=\mu_{i'-r}$, and
\begin{align*}
    p_{i'} & =\Frob_X\c p_{i'-r}: \Sht^\mu_G\to X,\\
    h_{i'} & =\Frob\c h_{i'-r}: \Sht^\mu_G\to \Bun_G,\\
    \ev_{i'} & =\Frob\c \ev_{i'-r}: \Sht^\mu_G\to \BB P_{\mu_{i'}}.
\end{align*}
Using these notations we can define $\d_{i'}(f)$ and $[-]_{i,i'}$ when $i'>r$.
Therefore, for $r<i'<i+r$, we have
\begin{equation}\label{i'>r}
    \d_{i'}(f)=\phi(\d_{i'-r}(f)),\quad [\a\ot \b]_{i,i'}=[\phi(\b)\ot \a]_{i'-r,i}, \quad \a,\b\in \cohog{*}{X}.
\end{equation}
From the definition it is easy to see equalities/canonical isomorphisms
\begin{eqnarray}\label{pf}
    p_i\c\frP=p_{i+1}, \quad h_i\c\frP\cong h_{i+1}, \quad \ev_i\c\frP\cong \ev_{i+1}.
\end{eqnarray}
Thus for $i\le r<i'<i+r$
\begin{equation}\label{frP add one}
    \d_{i+1}(f)=\frP^*\d_i(f), \quad [-]_{i+1,i'+1}=\frP^*[-]_{i,i'}.
\end{equation}

Applying partial Frobenius $\frP^{*}$ to \eqref{evif goal 1} $(i-1)$ times, using \eqref{frP add one}, we get for any $i\ge1$
\begin{eqnarray}\label{evif pre}
\ev_i^*f=\sum_{i'=i}^{r+i-1}[\Xi_d]_{i,i'} \d_{i'}(f).
\end{eqnarray}
Using \eqref{i'>r} we rewrite the terms involving $i'>r$ using $i'-r$, at the cost of an extra $\phi$, we get for $r<i'<i+r$
\begin{eqnarray*}
    [\Xi_d]_{i,i'}\d_{i'}(f)&=&\left[(\id\ot \frac{1}{q^{d-1}\phi-1})\D_X\right]_{i,i'} \d_{i'}(f)=\left[(\frac{q^{d-1}\phi}{q^{d-1}\phi-1}\ot\id)\D_X\right]_{i'-r,i}\d_{i'-r}(f)\\
    &=&\left[(\frac{1}{1-q^{1-d}\phi^{-1}}\ot\id)\D_X\right]_{i'-r,i}\d_{i'-r}(f)=-[\Xi_{1-d}]_{i'-r,i}\d_{i'-r}(f).
\end{eqnarray*}
Plugging into \eqref{evif pre} we get \eqref{evif goal}. This finishes the proof of Theorem \ref{th:taut hom quot}.
\qed

\begin{cor}\label{c:pullback BunG in taut}
    The image of the ring homomorphism \eqref{taut hom} contains the images of the pullback maps $h_i^*: \cohog{*}{\Bun_G}\to \cohog{*}{\Sht^\mu_G}$ for $0\le i\le r$.
\end{cor}
\begin{proof}
    For $i=0$ the statement follows from \eqref{h0fz}. For general $i$, applying $\frP^*$ to both sides of \eqref{h0fz} for $i$ times, we get
    \begin{equation*}
        h_{i}^*(f^z)=\sum_{i'=i+1}^{i+r}p_{i'}^*\left((q^{d-1}\phi-1)^{-1}\PD(z)\right)\cdot \d_{i'}(f).
    \end{equation*}
    Using \eqref{i'>r} we rewrite the terms involving $i'>r$ using $i'-r$, and get
    \begin{equation*}
        h_{i}^*(f^z)=\sum_{i'=i+1}^{r}p_{i'}^*\left(\frac{1}{q^{d-1}\phi-1}\PD(z)\right)\cdot \d_{i'}(f)+\sum_{i'=1}^{i}p_{i'}^*\left(\frac{q^{d-1}\phi}{q^{d-1}\phi-1}\PD(z)\right)\cdot \d_{i'}(f).
    \end{equation*}
The right side is visibly in the image of \eqref{taut hom}.
\end{proof}

\subsection{Structure of the phantom tautological ring}

We show that $C^\mu_G$ is a flat deformation of $\otimes_{i=1}^r \cohog{*}{G/P_{\mu_i}}$ over the Artinian ring $\cohog{*}{X^r}$. We will equip $C^\mu_G$ with a volume form that is compatible with our ad-hoc definition of $\int_{\Sht_G}$, and show that $C^\mu_G$ satisfies Poincar\'e duality under the volume form. 

\begin{prop}\label{p:taut ring free over HX}
    The phantom tautological ring $C^\mu_G$ is a free module over $\cohog{*}{X^r}$, with a canonical $\Frob^*$-equivariant isomorphism
    \begin{equation}\label{taut ring red}
        C^\mu_G/\cohog{>0}{X^r}\cdot C^\mu_G\cong \otimes_{i=1}^r \cohog{*}{G/P_{\mu_i}}.
    \end{equation}
\end{prop}
\begin{proof}
    We first check \eqref{taut ring red}. The relations in $I^\mu_G$ modulo $\cohog{>0}{X^r}$ become $[f]_i=0$ for all $1\le i\le r$ and $f\in R^W_{+}$. Therefore we get a canonical isomorphism
    \begin{equation}\label{taut mod HX}
        C^\mu_G/\cohog{>0}{X^r}\cdot C^\mu_G\cong \otimes_{i=1}^r R^{W_{\mu_i}}/(R^W_+)\cong\otimes_{i=1}^r \cohog{*}{G/P_{\mu_i}}.
    \end{equation}
    
    Now we show that $C^\mu_G$ is free over $\cohog{*}{X^r}$. By Lemma \ref{l:Difg} below, if we choose a set of homogeneous generators $f_1,\cdots, f_n$ for $R^W$, then $I^\mu_G$ is generated by    \begin{equation}\label{reg seq}
    \{D_i(f_j)\}_{1\le i\le r, 1\le j\le n}.
    \end{equation}

    We claim that the collection of elements \eqref{reg seq} form a regular sequence in $\wt C^\mu_G$ in any order. Indeed, the Krull dimension of $\wt C^\mu_G$ is $nr$ and there are $nr$ elements in \eqref{reg seq}. The ideal they generate is $I^\mu_G$ and the quotient $C^\mu_G=\wt C^\mu_G/I^\mu_G$ is Artinian because it is finite over the Artinian ring $\cohog{*}{X^r}$ by \eqref{taut mod HX}. Therefore \eqref{reg seq} form a regular sequence in $\wt C^\mu_G$.

    Let $P_t(V)$ denote the Hilbert polynomial of a graded vector space $V$, and $P_t(Y)$ denote the Poincar\'e polynomial of $\cohog{*}{Y}$ for a stack $Y$. Let $2d_i$ be the degree of $f_i$. Since \eqref{reg seq} is a regular sequence in $\wt C^\mu_G$, by using the Koszul complex we see that the Hilbert polynomial of $C^\mu_G$ is 
    \begin{equation*}
        P_t(C^\mu_G)=P_t(\wt C^\mu_G)\prod_{j=1}^n(1-t^{2d_j})^r.
    \end{equation*}
    Using that $P_t(\BB G)=\prod_{j=1}^n(1-t^{2d_j})^{-1}$, we can write
    \begin{equation*}
        P_t(C^\mu_G)=\prod_{i=1}^r\left(P_t(\BB P_{\mu_i})P_t(X)P_t(\BB G)^{-1}\right).
    \end{equation*}
    Since $\cohog{*}{\BB P_{\mu_i}}$ is free over $\cohog{*}{\BB G}$, we have $P_t(\BB P_{\mu_i})/P_t(\BB G)=P_t(G/P_{\mu_i})$, hence
    \begin{equation*}
        P_t(C^\mu_G)=\prod_{i=1}^rP_t(X\times G/P_{\mu_i}).
    \end{equation*}
    In particular, taking $t=1$ we get the total dimension
    \begin{equation}\label{dim taut}
        \dim C^\mu_G=\dim \cohog{*}{X^r}\prod_{i=1}^r\dim \cohog{*}{G/P_{\mu_i}}.
    \end{equation}
    By \eqref{taut mod HX} and Nakayama's lemma, $C^\mu_G$ is a $\cohog{*}{X^r}$-module generated by at most $\prod_{i=1}^r\dim \cohog{*}{G/P_{\mu_i}}$ elements. In view of the dimension equality \eqref{dim taut}, $C^\mu_G$ has to be free over $\cohog{*}{X^r}$ with rank $\prod_{i=1}^r\dim \cohog{*}{G/P_{\mu_i}}$.
\end{proof}

Below are technical lemmas needed in the proof of Proposition \ref{p:taut ring free over HX}. Recall the elements $D_i(f)$ (for $f\in R^W$ and $1\le i\le r$) from \eqref{rel in C} or \eqref{rel in C better}.

\begin{lemma}\label{l:Difg}
For any two homogeneous elements $f,g\in R^W$, and $1\le i\le r$, we have
\begin{equation*}
    D_i(fg)\in (D_1(f),\cdots, D_r(f), D_1(g),\cdots, D_r(g))\subset \wt C^\mu_G
\end{equation*}
where the right side means the ideal generated by the listed elements in $\wt C^\mu_G$.
\end{lemma}
\begin{proof}
    Let $I_{f,g}$ be the ideal generated by $\{D_1(f),\cdots, D_r(f), D_1(g),\cdots, D_r(g)\}$ in $\wt C^\mu_G$. We first prove the statement for $i=1$. To show $D_1(fg)\in I_{f,g}$, we need to show
    \begin{equation}\label{fg1}
        [fg]_1\equiv \sum_{i=1}^r[\Xi_{d+e}]_{1,i}\d_i(fg)\mod I_{f,g}.
    \end{equation}
    Note that
    \begin{equation}\label{d fg}
        \d_i(fg)=\d_i(f)[g]_i+[f]_i\d_i(g)+[c_1(TX)]_i\d_i(f)\d_i(g).
    \end{equation}
    Let $d$ and $e$ be the degrees of $f$ and $g$. Using \eqref{d fg}, \eqref{fg1} is equivalent to
    \begin{equation*}
    [f]_1[g]_1\equiv\sum_{i=1}^r[\Xi_{d+e}]_{1,i}\left(\d_i(f)[g]_i+[f]_i\d_i(g)+[c_1(TX)]_i\d_i(f)\d_i(g)\right)\mod I_{f,g}.
    \end{equation*}
    Now we use the definition of $D_i(f)$ to replace $[f]_i$ above by expressions involving $D_i(f)$ and $\d_{i'}(f)$, it suffices to show
    \begin{eqnarray*}
        &&\sum_{i=1}^r[\Xi_d]_{1,i}\d_i(f)\sum_{i=1}^r[\Xi_e]_{1,i}\d_i(g)\\
        &\equiv& \sum_{i=1}^r[\Xi_{d+e}]_{1,i}\d_i(f)\left(\sum_{i'\ge i}[\Xi_e]_{i,i'}\d_{i'}(g)-\sum_{i'< i}[\Xi_{1-e}]_{i',i}\d_{i'}(g)\right)\\
        &+&\sum_{i=1}^r[\Xi_{d+e}]_{1,i}\left(\sum_{i'\ge i}[\Xi_d]_{i,i'}\d_{i'}(f)-\sum_{i'< i}[\Xi_{1-d}]_{i',i}\d_{i'}(f)\right)\d_i(g)\\
        &+&\sum_{i=1}^r[\Xi_{d+e}]_{1,i}[c_1(TX)]_i\d_i(f)\d_i(g)\mod I_{f,g}.
    \end{eqnarray*}
    Comparing coefficients of $\d_i(f)\d_{i'}(g)$ on both sides, it reduces to showing
    \begin{eqnarray}\label{Xi prod}
        [\Xi_d]_{1,i}[\Xi_e]_{1,i'}=\begin{cases}
            [\Xi_{d+e}]_{1,i}[\Xi_{e}]_{i,i'}-[\Xi_{d+e}]_{1,i'}[\Xi_{1-d}]_{i,i'} & i<i'\\
            -[\Xi_{d+e}]_{1,i}[\Xi_{1-e}]_{i',i}+[\Xi_{d+e}]_{1,i'}[\Xi_d]_{i',i}& i>i'\\
            [\Xi_{d+e}]_{1,i}\left([\Xi_e]_{i,i}+[\Xi_d]_{i,i}+[c_1(TX)]_i\right) & i=i'.
        \end{cases}
    \end{eqnarray}
    Let us prove the above identity in the case $i<i'$. Write $x=q^{d-1}$ and $y=q^{e-1}$, the identity can be written as
    \begin{eqnarray}\label{prod D phi}
        &&\left(1\ot \frac{1}{\phi x-1}\right)[\D]_{1,i}\left(1\ot \frac{1}{\phi y-1}\right)[\D]_{1,i'}\\
\notag        &=& \left(1\ot \frac{1}{q\phi xy-1}\right)[\D]_{1,i}\left(1\ot \frac{1}{\phi y-1}\right)[\D]_{i,i'}\\
\notag        &+&\left(1\ot \frac{1}{q\phi xy-1}\right)[\D]_{1,i'}\left(\frac{\phi x}{\phi x-1}\ot1\right)[\D]_{i,i'}.
    \end{eqnarray}
    Expand both sides in geometric series in $x$ and $y$, the coefficient of $x^ay^b$ on the left side is 
    \begin{equation}\label{left xy}
        (1\ot \phi^a)[\D]_{1,i}(1\ot \phi^b)[\D]_{1,i'}=(1\ot \phi^a\ot \phi^b)[\D]_{1,i,i'}.
    \end{equation}
    Here we use that $[\D]_{1,i}[\D]_{1,i'}$ is the class of the small diagonal in $X\times X\times X$ (the $1,i,i'$-factors), which we denote by $[\D]_{1,i,i'}$. The coefficient of $x^ay^b$ on the right side of \eqref{prod D phi} is
    \begin{equation}\label{right xy}
        \begin{cases}
            (1\ot q^b\phi^b)[\D]_{1,i'}(\phi^{a-b}\ot1)[\D]_{i,i'}& a>b\\
            (1\ot q^a\phi^a)[\D]_{1,i}(1\ot \phi^{b-a})[\D]_{i,i'}  & a\le b.
        \end{cases}
    \end{equation}
    When $a>b$, we have
    \begin{equation*}
        (1\ot q^b\phi^b)[\D]_{1,i'}(\phi^{a-b}\ot1)[\D]_{i,i'}=q^b(1\ot \phi^b)[\D]_{1,i'}(\phi^{a-b}\ot1)[\D]_{i,i'}=(1\ot \phi^b)[\D]_{1,i'}(\phi^{a-b}\ot1)(\phi\ot \phi)^b[\D]_{i,i'}.
    \end{equation*}
    Here we use that $(\phi\ot\phi)[\D]_{i,i'}=\phi[\D]_{i,i'}=q[\D]_{i,i'}$, because $[\D]$ is a divisor class. Therefore
    \begin{equation*}
        (1\ot q^b\phi^b)[\D]_{1,i'}(\phi^{a-b}\ot1)[\D]_{i,i'}=(1\ot \phi^b)[\D]_{1,i'}(\phi^a\ot \phi^{b})[\D]_{i,i'}=(1\ot \phi^a\ot \phi^b)[\D]_{1,i,i'},
    \end{equation*}
    which is equal to \eqref{left xy}.
    Here we use that $\phi^a$ is a ring endomorphism of $\cohog{*}{X}$. This shows that \eqref{right xy} and \eqref{left xy} are equal when $a>b$. The proof of their equality when $a\le b$ is similar. Thus we have verified the $i<i'$ case of \eqref{Xi prod}.

    The case $i>i'$ of \eqref{Xi prod} is proved similarly. 

    Finally let us prove the $i=i'$ case of \eqref{Xi prod}. When $i=1=i'$ both sides are zero. We therefore assume $i=i'>1$. Then the equality \eqref{Xi prod} becomes
    \begin{equation}\label{ii coeff}
        \left(1\ot \frac{1}{\phi x-1}\right)[\D]_{1,i}\left(1\ot \frac{1}{\phi y-1}\right)[\D]_{1,i}=\left(1\ot \frac{1}{q\phi xy-1}\right)[\D]_{1,i}\cdot \Tr\left(\frac{1}{\phi x-1}+\frac{1}{\phi y-1}+1\Big|\cohog{*}{X}\right)[\xi]_i,
    \end{equation}
    Since $\phi$ and $q\phi^{-1}$ are adjoint under the Poincar\'e duality pairing on $\cohog{*}{X}$, we have
    \begin{equation*}
    \left(1\ot \frac{1}{\phi y-1}\right)[\D]_{1,i}=\left(\frac{1}{q\phi^{-1}y-1}\ot 1\right)[\D]_{1,i}.
    \end{equation*}
    Therefore the left side of \eqref{ii coeff} is
    \begin{equation*}
    \left(1\ot \frac{1}{\phi x-1}\right)[\D]_{1,i}\left(\frac{1}{q\phi^{-1}y-1}\ot 1\right)[\D]_{1,i}=\Tr\left(\frac{1}{q\phi^{-1}y-1}\cdot\frac{1}{\phi x-1}\Big|\cohog{*}{X}\right)[\xi\ot\xi]_{1,i}.
    \end{equation*}
    We also have
    \begin{equation*}
    \left(1\ot \frac{1}{q\phi xy-1}\right)[\D]_{1,i}\cdot [\xi]_i=\frac{1}{qxy-1}[\xi\ot\xi]_{1,i}.
    \end{equation*}
    Therefore \eqref{ii coeff} is equivalent to the equality
    \begin{equation}\label{trace frac phi}
    \Tr\left(\frac{1}{q\phi^{-1}y-1}\cdot\frac{1}{\phi x-1}\right)=\frac{1}{qxy-1}\Tr\left(\frac{1}{\phi x-1}+\frac{1}{\phi y-1}+1\right).
    \end{equation}
    where $\Tr$ means alternating trace on $\cohog{*}{X}$. Again we may change $\frac{1}{\phi y-1}$ on the right side above by its adjoint $\frac{1}{q\phi^{-1}y-1}$, and \eqref{trace frac phi} follows from the equality of endomorphisms of $\cohog{*}{X}$ before taking trace
    \begin{equation*}
    \frac{1}{q\phi^{-1}y-1}\cdot\frac{1}{\phi x-1}=\frac{1}{qxy-1}\left(\frac{1}{\phi x-1}+\frac{1}{q\phi^{-1} y-1}+1\right).
    \end{equation*}
    This finishes the proof of the lemma for $i=1$.

    Now consider the case of general $i$. Let $\mu'=(\mu_2,\mu_3,\cdots, \mu_r,\mu_1)$ (so that $\mu'_i=\mu_{i+1}$, with the subscripts understood mod $r$). We define a partial Frobenius
    \begin{equation}\label{partial Frob on C}
        \frP^*: \wt C^{\mu'}_G\to \wt C^{\mu}_G
    \end{equation}
    that maps $\a_1\ot\a_2\ot\cdots\ot \a_r$ to $\phi(\a_r)\ot\a_1\ot\cdots\ot \a_{r-1}$, where $\a_i\in \cohog{*}{X\times G/P_{\mu'_{i}}}$. It is easy to check that $D_2(f)=\frP^*(D_1(f))\in \wt C^\mu_G$, where $D_1(f)$ is viewed as an element of $\wt C^{\mu'}_G$. We have proved that $D_1(fg)\in I'_{f,g}:=(D_1(f),\cdots, D_r(f), D_1(g),\cdots, D_r(g))\subset \wt C^{\mu'}_G$. Applying $\frP^*$, we see that
    \begin{equation*}
        D_2(fg)=\frP^*D_1(fg)\in \frP^* I'_{f,g}=I_{f,g}\subset \wt C^\mu_G.
    \end{equation*}
    Repeating this argument we see that $D_i(fg)\in I_{f,g}$ for all $i=2,\cdots, r$. This finishes the proof of the lemma.
\end{proof}

\begin{cor}\label{c:taut ring top deg} Recall that $D_{\mu_i}=\dim G/P_{\mu_i}$, and write $N:=\dim \Sht^\mu_G=\sum_{i=1}^r(D_{\mu_i}+1)$. Then the top non-vanishing degree of $C^\mu_G$ is $2N$, and $\dim(C^\mu_G)_{2N}=1$. There is a canonical isomorphism
\begin{equation}\label{C top deg}
    \bigotimes_{i=1}^r(\cohog{2}{X}\ot \cohog{2D_{\mu_i}}{G/P_{\mu_i}})\isom (C^\mu_G)_{2N}
\end{equation}
sending $\ot(\xi\ot \y'_i)$ (where $\y'_i\in \cohog{2D_{\mu_i}}{G/P_{\mu_i}}$) to the image of $\ot(\xi\ot \wt\y'_i)\in \wt C^\mu_G$ for any lifting $\wt\y'_i\in R^{W_{\mu_i}}$ of $\y'_i$.
\end{cor}

\subsection{The volume functional}\label{sss:vol functional}
We consider two linear functionals on $(C^\mu_G)_{2N}$.
\subsubsection{} By the description of the top degree component of $C^\mu_G$ in \eqref{C top deg}, it carries a standard linear functional given by pairing with the fundamental class of $\prod_{i=1}^r (X\times G/P_{\mu_i})$
\begin{equation*}
   \int_{\prod_{i=1}^r (X\times G/P_{\mu_i})}: (C^\mu_G)_{2N}\cong \bigotimes_{i=1}^r(\cohog{2}{X}\ot \cohog{2D_{\mu_i}}{G/P_{\mu_i}})\to \Qlbar(-N).
\end{equation*}
This is an isomorphism.

\subsubsection{} Pick an embedding $\io: \Qlbar\incl \CC$. By the convergence result proved in Proposition \ref{p:trace conv}, it makes sense to define a $\CC$-valued linear map
\begin{eqnarray}\label{def vol on C}
    \vol({}^\omega \Sht_G^\mu, -)&:& (\wt C^\mu_G)_{2N}\to \CC\\
    &&\th \mt \Tr_\io(\phi^{-1}\c \G^{\wt\r(\th)}_\mu|\cohog{*}{\Bun_G^\om}).
\end{eqnarray}
Later in Proposition \ref{p:vol on taut}, we will show that $\vol({}^\omega \Sht_G^\mu, -)$ takes values in $\Qlbar$ and is independent of the choice of $\io$.

We will first prove:
\begin{prop}\label{p:vol factor thru C}
    The linear functional $\vol({}^\omega \Sht_G^\mu, -)$ on $(\wt C^\mu_G)_{2N}$ factors through the quotient $(C^\mu_G)_{2N}$.
\end{prop}

The proof will be given after some preparations.

\begin{lemma}\label{l:trAB}
    Let $V=\op_{n\in \ZZ}V_n$ and $W=\op_{n\in \ZZ} W_n$ be two graded $\CC$-vector spaces such that $\dim V_n$ and $\dim W_n$ are finite for all $n$ and zero for all $n\ll0$. Let $A:V\to W$ be a linear map of degree $d$, and $B: W\to V$ be a linear map of degree $-d$, where $d\in \ZZ$. 
    \begin{enumerate}
        \item The series $\sum_{n} \Tr(BA|V_n)$ is absolutely convergent if and only if the series $\sum_{n} \Tr(AB|W_n)$ is absolutely convergent.
        \item When the condition in part (1) is satisfied, we have
        \begin{equation*}
            \sum_{n}(-1)^n\Tr(BA|V_n)=(-1)^d\sum_{n}(-1)^n\Tr(AB|W_n).
        \end{equation*}
    \end{enumerate}
\end{lemma}
\begin{proof}
    For any $n\in \ZZ$ we have maps $A|_{V_n}: V_n\to W_{n+d}$ and $B|_{W_{n+d}}: W_{n+d}\to V_n$. Therefore
    \begin{equation*}
        \Tr(BA|V_n)=\Tr(AB|W_{n+d}).
    \end{equation*}
    The rest of the statements follow easily.
\end{proof}

Let $\nu :=(\mu_2,\cdots, \mu_r)$. Note that
\begin{equation}\label{Hk mu as fiber prod}
    \Hk^\mu_G=\Hk^{\mu_1}_G\times_{\Bun_G}\Hk^\nu_G
\end{equation}
Let $h_i^\nu: \Hk^\nu_G\to \Bun_G$ ($1\le i\le r$) be the projections; let $h^{\mu_1}_0, h^{\mu_1}_1: \Hk^{\mu_1}_G\to\Bun_G$ be the two projections. The fiber product in \eqref{Hk mu as fiber prod} uses $h^{\mu_1}_1$ and $h^\nu_1$.

Suppose $\th_1\in \cohog{*}{\Hk^{\mu_1}_G}$ and $\th'\in \cohog{*}{\Hk^\nu_G}$. We denote by $\th_1\th'\in \cohog{*}{\Hk^\mu_G}$ the cup product of the pullbacks of $\th_1$ and $\th'$ to $\Hk^\mu_G$ via the two projections from $\Hk^\mu_G$ using \eqref{Hk mu as fiber prod}.

Let $\mu'=(\mu_2,\mu_3,\cdots, \mu_r,\mu_1)$ and $\om'=\om+\ol \mu_1$. We have
\begin{equation}\label{Hk mu' as fiber prod}
    \Hk^{\mu'}_G=\Hk^{\nu}_G\times_{\Bun_G}\Hk^{\mu_1}_G
\end{equation}
using the maps $h_r^\nu: \Hk^{\nu}_G\to \Bun_G$ and $h^{\mu_1}_0: \Hk^{\mu_1}_G\to\Bun_G$. Similarly we view $\th'\th_1$ as a cohomology class on $\Hk^{\mu'}_G$ via the cup product of the pullbacks along the projections in \eqref{Hk mu' as fiber prod}.

\begin{lemma}\label{l:partial Frob and trace} Let $\th_1\in \cohog{*}{\Hk^{\mu_1}_G}$ and $\th'\in \cohog{*}{\Hk^\nu_G}$ be homogeneous elements so that $\th_1\th'$ has degree $2N$. Then we have
\begin{equation*}
\Tr(\phi^{-1}\c\G^{\phi(\th_1)\th'}_{\mu}|\cohog{*}{\Bun^\om_G})=(-1)^{|\th_1|}q^{D_{\mu_1}+1}\Tr(\phi^{-1}\c\G^{\th'\th_1}_{\mu'}|\cohog{*}{\Bun^\om_G}).
\end{equation*}
\end{lemma}
\begin{proof}
    Below all traces mean traces of endomorphisms of $\cohog{*}{\Bun^\om_G}$.
    Observe that $\G^{\th_1\th'}_\mu=\G^{\th_1}_{\mu_1}\c\G^{\th'}_{\nu}$, therefore by the absolute convergence of trace proved in Proposition \ref{p:trace conv} together with Lemma \ref{l:trAB},
    \begin{equation*}
        \Tr(\phi^{-1}\c\G^{\phi(\th_1)\th'}_\mu)=\Tr(\phi^{-1}\c\G^{\phi(\th_1)}_{\mu_1}\c\G^{\th'}_{\nu})=(-1)^{|\th_1|}\Tr(\G^{\th'}_{\nu}\c\phi^{-1}\c\G^{\phi(\th_1)}_{\mu_1}).
    \end{equation*}
    On the other hand, the same reasoning gives
    \begin{equation*}
        \Tr(\phi^{-1}\c\G^{\th'\th_1}_{\mu'})=\Tr(\phi^{-1}\c\G^{\th'}_\nu\c\G^{\th_1}_{\mu_1})=\Tr(\G^{\th'}_\nu\c\G^{\th_1}_{\mu_1}\c\phi^{-1}).
    \end{equation*}
    Thus it suffices to show
    \begin{equation*}
        \phi^{-1}\c\G^{\phi(\th_1)}_{\mu_1}=q^{D_{\mu_1}+1}\G^{\th_1}_{\mu_1}\c\phi^{-1}\in \End(\cohog{*}{\Bun_G}).
    \end{equation*}
    Equivalently, we need to show
    \begin{equation}\label{Gamma phi exchange}
        \G^{\phi(\th_1)}_{\mu_1}\c\phi=q^{D_{\mu_1}+1}\phi\c\G^{\th_1}_{\mu_1}.
    \end{equation}
    Writing the two projections $h^{\mu_1}_0, h^{\mu_1}_1:\Hk^{\mu_1}_G\to \Bun_G$ simply as $h_0$ and $h_1$, we have for any $\a\in \cohog{*}{\Bun^\om_G}$
    \begin{equation*}
        \G^{\phi(\th_1)}_{\mu_1}\c\phi(\a)=h_{0*}(h_1^*(\phi(\a))\cup \phi(\th_1))=h_{0*}\phi(h_1^*\a\cup\th_1).
    \end{equation*}
    On the other hand,
    \begin{equation*}
        \phi\c\G^{\th_1}_{\mu_1}(\a)=\phi  h_{0*}(h_1^*\a\cup\th_1).
    \end{equation*}
    Therefore \eqref{Gamma phi exchange} follows from the identity 
    \begin{equation*}
        h_{0*}\phi=q^{\dim h_0}\phi  h_{0*}: \cohog{*}{\Hk^{\mu_1}_G}\to \cohog{*}{\Bun_G}
    \end{equation*}
    where both sides are adjoint to $h_0^*\phi=\phi h_0^*$ with respect to Poincar\'e duality. Here we use  
    the fact that $h_0$ is a smooth projective fiber bundle of relative dimension $D_{\mu_1}+1$.
\end{proof}

Recall the map $\frP^*: \wt C^{\mu'}_G\to \wt C^\mu_G$ from \eqref{partial Frob on C}. 
We have the following corollary of Lemma \ref{l:partial Frob and trace}.

\begin{cor}\label{c:vol pullback partial Frob}
    For any $\th\in (\wt C^{\mu'}_G)_{2N}$, we have
\begin{equation*}
\vol({}^\omega \Sht_G^\mu, \frP^*\th) =q^{D_{\mu_1}+1} \vol({}^{\omega'} \Sht^{\mu'}_G, \th).
\end{equation*}
\end{cor}

\begin{remark}Formally, the above formula is consistent with the fact that the partial Frobenius $\frP: {}^\om\Sht^\mu_G\to {}^{\om'}\Sht^{\mu'}_G$ has degree $q^{D_{\mu_1}+1}$.
\end{remark}

\subsubsection{Proof of Proposition \ref{p:vol factor thru C}}
    We need to show that, for any $f\in R^W_{2d}$, $\th\in (\wt C^\mu_G)_{2N-2d}$ and $1\le i\le r$, the (graded) trace vanishes,
    \begin{equation}\label{tr Dif van}
        \Tr(\phi^{-1}\c\G^{D_i(f)\th}_\mu|\cohog{*}{\Bun^\om_G})=0.
    \end{equation}
    We first prove \eqref{tr Dif van} in the case $i=1$. For $z\in \homog{|z|}{X}$, write
    \begin{equation*}
        \g_1(f^z):=h_0^*(\phi(f^z)-f^z)-\sum_{i=1}^r p_i^*\PD(z)\d_i(f).
    \end{equation*}
    By the deduction from \eqref{h0fz pre} to \eqref{ev1f almost}, $D_i(f)$ belongs to the ideal in $\cohog{*}{{}^\om\Hk^\mu_G}$ generated by $\g_1(f^z)$ for all $z$, therefore it suffices to show that for any
    $\th'\in \cohog{2N-2d+|z|}{{}^\om\Hk^\mu_G}$, we have
    \begin{equation}\label{van tr gamma1}
        \Tr(\phi^{-1}\c\G^{\g_1(f^z)\th'}_\mu|\cohog{*}{\Bun^\om_G})=0.
    \end{equation}
    View ${}^\om\Hk^\mu_G$ as a self-correspondence of $\Bun_G$ with the maps $\Fr\c h_0$ and $h_r$:
    \begin{equation*}
        \xymatrix{& \cH={}^\om\Hk^\mu_G\ar[dl]_{\Fr\c h_0} \ar[dr]^{h_r}\\
        \Bun^\om_G && \Bun^\om_G}
    \end{equation*}
    Then $\phi^{-1}\c\G^\th_\mu$ is the composition
    \begin{equation*}
        \phi^{-1}\c\G^\th_\mu: \cohog{*}{\Bun^\om_G}\xr{h_r^*}\cohog{*}{\cH}\xr{\cup\th}\cohog{*}{\cH}\xr{(\Fr\c h_0)_*}\cohog{*}{\Bun^\om_G}.
    \end{equation*}
    From this we see that $\phi^{-1}\c\G_\mu^{\th'h_r^*(f^z)}$ is the composition
    \begin{equation}\label{cup then Gamma}
        \cohog{j}{\Bun^\om_G}\xr{\cup f^z}\cohog{j+2d-|z|}{\Bun^\om_G}\xr{\phi^{-1}\c\G^{\th'}_\mu}\cohog{j}{\Bun^\om_G},
    \end{equation}
    and $\phi^{-1}\c \G_\mu^{h_0^*(\phi(f^z))\th'}$ is the composition
    \begin{equation}\label{Gamma then cup}
        \cohog{j}{\Bun^\om_G}\xr{\phi^{-1}\c\G^{\th'}_\mu}\cohog{j-2d+|z|}{\Bun^\om_G}\xr{\cup f^z}\cohog{j}{\Bun^\om_G}.
    \end{equation}
    The two maps in \eqref{cup then Gamma} and \eqref{Gamma then cup} differ only by the order of composition. By Lemma \ref{l:trAB} and the absolute convergence of trace proved in Proposition \ref{p:trace conv}, we have
    \begin{equation*}
        \Tr(\phi^{-1}\c \G_\mu^{\th'h_r^*(f^z)})=(-1)^{|z|}\Tr(\phi^{-1}\c \G_\mu^{h_0^*(\phi(f^z))\th'}),
    \end{equation*}
    or
    \begin{equation}\label{van trace h0hr}
        \Tr(\phi^{-1}\c \G^{(h_0^*(\phi(f^z))-h_r^*(f^z))\th'}_\mu)=0.
    \end{equation}
    By the identity \eqref{diff hr h0} that holds in $\cohog{*}{\Hk^\mu_G}$, we have
    \begin{equation*}
        h_0^*(\phi(f^z))-h_r^*(f^z)=\g_1(f^z).
    \end{equation*}
    Therefore \eqref{van trace h0hr} implies \eqref{van tr gamma1}. This proves \eqref{tr Dif van} in the case $\th=D_1(f)\th'$. 
    
    Now we prove \eqref{tr Dif van} by induction on $i$. The case $i=1$ has been proved. Suppose $i>1$ and \eqref{tr Dif van} holds for $D_{i-1}(f)\th$.  It is straightforward to check that $D_i(f)\th=\frP^*(D_{i-1}(f)\th')$, where $\th'\in \wt C^{\mu'}_G$ is characterized by $\frP^*\th'=\th$. By Corollary \ref{c:vol pullback partial Frob}, we have
    \begin{equation*}
        \vol({}^\om\Sht^\mu_G, D_{i}(f)\th)= \vol({}^\om\Sht^\mu_G, \frP^*(D_{i-1}(f)\th'))=q^{D_{\mu_1}+1} \vol({}^{\om'}\Sht^{\mu'}_G, D_{i-1}(f)\th')=0.
    \end{equation*}
    This finishes the induction step and the proof is complete.
    
\qed



The following result is the analogue of Hirzebruch proportionality (see diagram \eqref{Hirz prop}) for the moduli stack of shtukas.
Proposition \ref{p:vol factor thru C} justifies also denoting by $\vol({}^\omega \Sht_G^\mu, -)$ the linear functional on the quotient $(C^\mu_G)_{2N}$ of $(\wt C^\mu_G)_{2N}$.

\begin{prop}\label{p:vol on taut}

    As linear functionals on $(C^\mu_G)_{2N}$, we have
    \begin{equation*}
        \vol({}^\om\Sht^\mu_G, -)=q^{\dim \Bun_G}\prod_{j=1}^n\z_X(d_j)\cdot \int_{\prod_{i=1}^r (X\times G/P_{\mu_i})}(-).
    \end{equation*}
\end{prop}
\begin{proof}
    Applying Theorem \ref{th:vol gen} to the integrand $\ot_{i=1}^r(\xi\ot \y'_i)\in \ot_{i=1}^r(\cohog{2}{X}\ot \cohog{2D_{\mu_i}}{G/P_{\mu_i}})\cong (C^\mu_G)_{2N}$, we get
    \begin{equation*}
        \vol({}^\om\Sht^\mu_G, \ot_{i=1}^r(\xi\ot \y'_i))=q^{\dim\Bun_G}\left(\prod_{i=1}^r\int_{G/P_{\mu_i}}\y'_i\right)\sL_{X,G}(0,0,\cdots,0).
    \end{equation*}
    Note here Assumption \ref{assump:split-operators-commute} is trivially satisfied because all $\y_i$ are zero. Now $\sL_{X,G}(0,0,\cdots,0)=\prod_{i=1}^n\z_X(d_i)$ by definition. The claim follows.
\end{proof}

\begin{remark}
From this Proposition, using the defining relations for $C^\mu_G$, one should in principle be able to recover Theorem \ref{th:vol gen}. However, the intricacy of the multiplicative structure of the ring  $C^\mu_G$ seems to make this approach difficult to implement.
\end{remark}

\subsubsection{Poincar\'e duality on $C^\mu_G$}\label{sss:taut duality}
Using the volume functional, we can define a bilinear pairing on elements of $C^\mu_G$ of complementary degrees:
\begin{equation*}
    \j{-,-}: (C^\mu_G)_{i}\times (C^\mu_G)_{2N-i}\to \Qlbar(-N)
\end{equation*}
defined as
\begin{equation*}
\j{\a,\b}= \vol({}^\om\Sht^\mu_G, \a\b).
\end{equation*}
By Proposition \ref{p:vol on taut}, $\vol({}^\om\Sht^\mu_G, -)$ is independent of the choice of $\om$.

The next Proposition says that the phantom tautological ring restores Poincar\'e duality that is missing on ${}^\om\Sht^\mu_G$ due to non-properness.

\begin{prop}\label{p:taut duality}
    The pairing $\j{-,-}$ is perfect. In particular, $C^\mu_G$ has the structure of a Frobenius algebra.
\end{prop}
\begin{proof} Let $\a, \a'\in \cohog{*}{X^r}$ be of complementary degree; let $\b,\b'\in \ot_{i=1}^r \cohog{*}{G/P_{\mu_i}}$ be of complementary degree. Let $\wt \b, \wt \b'$ be arbitrary liftings of $\b$ and $\b'$ to $C^\mu_G$, viewing $\cohog{*}{G/P_{\mu_i}}$ as a quotient of $C^\mu_G$ by \eqref{taut ring red}. Then from the construction of the pairing on $C^\mu_G$ we have
\begin{equation}\label{pairing C tensor}
\j{\a\ot\wt\b, \a'\ot \wt\b'}=\j{\a,\a'}_{X^r}\j{\b,\b'}_{\prod (G/P_{\mu_i})}.
\end{equation}
Here $\j{-,-}_{Y}$ means the Poincar\'e duality pairing on the smooth projective variety $Y$. Now let $\a$ run over a basis $\{\a_i\}$ of $\cohog{*}{X^r}$ and let $\a'$ run over the dual basis $\{\a^i\}$ under the pairing $\j{-,-}_{X^r}$; similarly let $\b$ run over a basis $\{\b_j\}$ of $\cohog{*}{\prod(G/P_{\mu_i})}$ and let $\b'$ run over the dual basis $\{\b^j\}$.  Take arbitrary liftings $\wt\b_j$ and $\wt\b^j$ in $C^\mu_G$. The freeness assertion in Proposition \ref{p:taut ring free over HX} implies that both  $\{\a_i\ot\wt\b_j\}$ and $\{\a^i\ot\wt\b^j\}$ are bases for $C^\mu_G$; \eqref{pairing C tensor} shows that they are dual bases under $\j{-,-}$. In particular, the pairing $\j{-,-}$ is perfect.
\end{proof}

\subsection{The reductive case}\label{ssec:phantom-reductive} Here we extend the results on the phantom tautological ring to the case where $G$ is split reductive over $k$. We otherwise keep the same setup as in \S\ref{ssec:taut-ring}. The ring $\wt C^\mu_G$ is defined in exactly the same way as in \eqref{taut hom}. For the ideal $I^\mu_G$, we make the following modification. Let
\begin{eqnarray*}
    \Xi^*_1:=\left(\id\ot \frac{1}{\phi-1}\right)([\D_X]-\xi\ot 1)\in \cohog{2}{X\times X},\\
    \Xi^*_0:=\left(\id\ot \frac{1}{q^{-1}\phi-1}\right)([\D_X]-1\ot \xi)\in \cohog{2}{X\times X}.
\end{eqnarray*}
For any $f\in R^W$ of degree $2$, viewed as an element in $\xch(T)^W_{\Qlbar}$, the definition of $\d_i(f)$ boils down to
\begin{equation*}
    \d_i(f)=\j{\mu_i,f}\in \Qlbar.
\end{equation*}
Define
\begin{equation*}
    D^*_i(f)=[f]_i-\sum_{i'= i}^r\j{\mu_{i'},f}[\Xi^*_1]_{i,i'}+\sum_{i'=1}^{i-1}\j{\mu_{i'},f}[\Xi^*_{0}]_{i',i}-\j{\om+\mu_1+\cdots+\mu_{i-1},f}[\xi]_i, \quad \mbox{when }\deg(f)=2.
\end{equation*}
Here $\j{\om, f}$ makes sense because $\om$ is well-defined up to the coroot lattice, on which $\j{-,f}$ vanishes.

We then let $I^\mu_G$ be the ideal of $\wt  C^\mu_G$ generated by $D^*_i(f)$ for all $f\in R^W_2$, $1\le i\le r$, and by $D_i(f)$ for all homogeneous $f\in R^W$ of degree $>2$ and $1\le i\le r$. Finally, define $C^\mu_G=\wt C^\mu_G/I^\mu_G$.

\begin{exam}
    Consider the case $G=\Gm$. In this case, $C^\mu_G\cong \cohog{*}{X^r}$. The map $\rho: C^\mu_G\to \cohog{*}{{}^{\om}\Sht^\mu_{\Gm}}$ is identified with the pullback by the leg map ${}^{\om}\Sht^\mu_{\Gm}\to X^r$. Note that in this case $\r$ is injective, and its image is exactly the $\Pic^0_X(k)$-invariant part of $\cohog{*}{{}^{\om}\Sht^\mu_{\Gm}}$.
\end{exam}

All previous results in this section hold for split reductive groups $G$ with the above definition of $C^\mu_G$. The proofs are mostly the same, except that an extra calculation is needed alongside Lemma \ref{l:Difg}.

\begin{lemma} Recall that $\mu$ is admissible in the sense that $\mu_1+\cdots+\mu_r$ lies in the coroot lattice. Let  $f\in R^W_2$ and $g\in R^W_+$ be homogeneous, and $1\le i\le r$.
\begin{enumerate}
    \item If $\deg(g)>2$, then $D_i(fg)\in (D_1^*(f),\cdots, D_r^*(f), D_1(g),\cdots, D_r(g))\subset \wt C^\mu_G$.
    \item If $\deg(g)=2$, then $D_i(fg)\in (D_1^*(f),\cdots, D_r^*(f), D^*_1(g),\cdots, D^*_r(g))\subset \wt C^\mu_G$.
\end{enumerate}
\end{lemma}
\begin{proof} Using partial Frobenius it suffices to prove the statements for $i=1$. Denote the ideals in question by $I_{f,g}$. We prove only part (1), and the proof of part (2) is similar. Let $\deg(g)=2e$. Expanding $[f]_1[g]_1$ and $[fg]_1$ modulo $I_{f,g}$ as in the proof of Lemma \ref{l:Difg}, we write
    \begin{eqnarray}\label{fg1 deg 2}
        [f]_1[g]_1-[fg]_1\equiv \sum_{i,i'=1}^r(a^*_{i,i'}-b^*_{i,i'})\j{\mu_i,f}\d_{i'}(g)+\sum_{i=1}^rc_i\j{\om,f}\d_i(g)\mod I_{f,g}.
    \end{eqnarray}
    Here
    \begin{eqnarray*}
        a^*_{i,i'}&=&[\Xi^*_1]_{1,i}[\Xi_e]_{1,i'},\\
        b^*_{i,i'}&=&\begin{cases}
            [\Xi_{e+1}]_{1,i}[\Xi_e]_{i,i'}-[\Xi_{e+1}]_{1,i'}[\Xi^*_0]_{i,i'}+[\Xi_{e+1}]_{1,i'}[\xi]_{i'}, & i<i'\\
            

            -[\Xi_{e+1}]_{1,i}[\Xi_{1-e}]_{i',i}+[\Xi_{e+1}]_{1,i'}[\Xi_1^*]_{i',i}, & i>i'\\
            

            [\Xi_{e+1}]_{1,i}([\Xi_{e}]_{i,i}+[\Xi_1^*]_{i,i}+[c_1(TX)]_i), & i=i'
        \end{cases}\\
        c_i&=&[\Xi_{e}]_{1,i}[\xi]_1-[\Xi_{e+1}]_{1,i}[\xi]_i.
    \end{eqnarray*}
    Direct calculation shows $c_i=0$. It remains to calculate $a^*_{i,i'}-b^*_{i,i'}$. As in the proof of Lemma \ref{l:Difg} we let $y=q^{e-1}$. Let $a_{i,i'}$ denote the left side of \eqref{Xi prod}, and $b_{i,i'}$ denote the right side of \eqref{Xi prod}, both viewed as elements in $\Qlbar(x,y)\ot \cohog{4}{X^3}$. We also introduce the two-variable version of $a^*_{i,i'}$ and $b^*_{i,i'}$ by replacing $\Xi^*_1$ with $\left(1\ot \frac{1}{x\phi-1}\right)([\D]-\xi\ot 1)$, $\Xi^*_0$ with $\left(1\ot \frac{1}{q^{-1}x\phi-1}\right)([\D]-1\ot \xi)$, and $\Xi_{e+1}$ with $\left(1\ot \frac{1}{qxy\phi-1}\right)[\D]$. We still denote these by $a^*_{i,i'}$ and $b^*_{i,i'}$. To calculate $a^*_{i,i'}-b^*_{i,i'}$, we calculate the differences $a_{i,i'}-a^*_{i,i'}$ and $b_{i,i'}-b^*_{i,i'}$ and find
    \begin{eqnarray*}
        a_{i,i'}-a^*_{i,i'}&=&\left(1\ot \frac{1}{x\phi-1}\right)[\xi\ot1]_{1,i}\left(1\ot\frac{1}{y\phi-1}\right)[\D]_{1,i'}=\frac{1}{(x-1)(qy-1)}[\xi]_1[\xi]_{i'},\\
        b_{i,i'}-b^*_{i,i'}&=&\frac{1}{(x-1)(qxy-1)}[\xi]_1[\xi]_{i'}.
    \end{eqnarray*}
    In the proof of Lemma \ref{l:Difg} we showed that $a_{i,i'}=b_{i,i'}$, therefore
    \begin{equation*}
        a_{i,i'}^*-b_{i,i'}^*=\left(\frac{1}{(x-1)(qxy-1)}-\frac{1}{(x-1)(qy-1)}\right)[\xi]_1[\xi]_{i'}=\frac{-qy}{(qy-1)(qxy-1)}[\xi]_1[\xi]_{i'}.
    \end{equation*}
    In particular, $a_{i,i'}^*-b_{i,i'}^*$ is independent of $i$ (and equal to $\frac{-qy}{(qy-1)^2}[\xi]_1[\xi]_{i'}$ if we plug in $x=1$). Plugging into \eqref{fg1 deg 2}, using that $\j{\sum_i\mu_i,f}=0$, we conclude that $[f]_1[g]_1\equiv [fg]_1$ modulo $I_{f,g}$.
\end{proof}

\begin{cor}
    Let $\ph: G\to G'$ be a surjection whose kernel is central. Let $\mu=(\mu_1,\cdots, \mu_r)$ be an admissible sequence of minuscule coweights of $G$, also viewed as an admissible sequence of minuscule coweights of $G'$. Then $\ph$ induces an isomorphism $C^\mu_{G'}\isom C^\mu_G$.
\end{cor}
\begin{proof}
    Both $C^\mu_G$ and $C^\mu_{G'}$ are free over $\cohog{*}{X^r}$ by the reductive version of Proposition \ref{p:taut ring free over HX}. It therefore suffices to check that $\ph$ induces an isomorphism $C^\mu_{G'}/\cohog{>0}{X^r}C^\mu_{G'}\isom C^\mu_{G}/\cohog{>0}{X^r}C^\mu_{G}$. However both sides are canonically identified with $\ot\cohog{*}{G'/P'_{\mu_i}}\cong \ot\cohog{*}{G/P_{\mu_i}}$ again by Proposition \ref{p:taut ring free over HX}.
\end{proof}

\section{The Colmez conjecture over function fields} \label{s:Col}

In this section, we consider a function field analog of the Colmez conjecture \cite{Col1}.
Let $\pi: X\to Y$ be a finite \'etale covering. For simplicity, we assume that it is Galois with Galois group $\Sig$.
Consider a map 
\begin{equation*}
  \s= (\s_1,\s_2,\ldots,\s_r): X \to X^r
\end{equation*}
sending $x\in X$ to $(\s_i(x))_{i=1}^r$, where $\s\in\Sig^r= \Sig\times\cdots\times \Sig$ is an $r$-tuple. 
We restrict $\Sht^\mu_G$ to the locus using the above map associated to $\s\in\Sig^r$:
\begin{equation}\label{eq:sht-G-sigma}
    \Sht^\mu_{G,\s}:=\Sht^\mu_G\times_{X^r,\s}X.
\end{equation} In other words, all the legs are conjugate relative to $X\to Y$. Let $p: \Sht^\mu_{G,\s}\to X$ be the leg map (the second projection of \eqref{eq:sht-G-sigma}). Pulling back along $p$ and $\ev_i$ gives a ring homomorphism
\begin{equation}\label{taut hom mix}
    \wt\rho_{\s}:\wt C^\mu_{G,\s}:=\cohog{*}{X} \ot \bigotimes_{i=1}^r R^{W_{\mu_i}}\to \cohog{*}{\Sht^\mu_{G, \s}}.
\end{equation}
Let
\begin{equation*}
    C^\mu_{G,\s}=C^\mu_G\ot_{\cohog{*}{X^r}}\cohog{*}{X}
\end{equation*}
where $\cohog{*}{X}$ is viewed as a $\cohog{*}{X^r}$-algebra via $\s^*$. 
Then $\r$ induces a canonical homomorphism
\begin{equation*}
    \r_\s: C^\mu_{G,\s}\to \cohog{*}{\Sht^\mu_{G,\s}}.
\end{equation*}


\subsection{The one leg case}
Let $G$ be split semisimple. Write $\xi$ for the class $1\ot \xi\in C^\mu_{G,\s}$. Note that inside $C^\mu_{G,\s}$, for any $f\in R^W$ of degree $2d\ge4$, Theorem \ref{th:taut hom quot} implies that  
\begin{equation}\label{rel C sig}
    [f]_i=\xi\ot \sum_{i'=1}^r c_{i,i'}(d)[\pl_{\mu_{i'}}f]_{i'}.
\end{equation}
Here the constants $c_{i,i'}(d)$ are defined using
\begin{equation}\label{eq: c(d)}
    c_{i,i'}(d)\xi=\begin{cases}
        (\s_i,\s_{i'})^*\Xi_d, & i'\ge i,\\
        -(\s_{i'}, \s_{i})^*\Xi_{1-d}, & i'<i.
    \end{cases}
\end{equation}
In particular, since $\xi^2 = 0$, we see that for any two $f,g\in R^W_+$ and $1\le i,i'\le r$,
\begin{equation}\label{fg0}
    [f]_i[g]_{i'}=0\in C^\mu_{G,\s}.
\end{equation}

We have a canonical isomorphism between the top degree of $C^\mu_{G,\s}$ and $\cohog{2}{X}\ot \left(\bigotimes_{i=1}^r\cohog{2D_{\mu_i}}{G/P_{\mu_i}}\right)$, hence a canonical linear form $\int_{X\times \prod G/P_{\mu_i}}$ on $C^\mu_{G,\s}$ that is nonzero only on the top degree. We define 
\begin{equation*}   \vol({}^\om\Sht^\mu_{G,\s}, -) :=q^{\dim\Bun_G}\prod_{j=1}^n\z_X(d_j)\int_{X\times \prod G/P_{\mu_i}}(-): C^\mu_{G,\s}\to \Qlbar.
\end{equation*}

\begin{remark}
    To see the structure of $C^{\mu}_{G,\s}$ more clearly, we focus on its even part $(C^{\mu}_{G,\s})^\even$, which is a free $\Qlbar[\xi]/(\xi^2)$-module with mod $\xi$ reduction $A^\mu_G:=\ot_{i=1}^r\cohog{*}{G/P_{\mu_i}}$.

    We can form the quotient $B^{\mu}_{G}$ of $\otimes_{i=1}^rR^{W_{\mu_i}}$ by imposing the relations \eqref{fg0}. Then $B^{\mu}_{G}$ is a square zero extension of $A^\mu_G$ that fits into an exact sequence
    \begin{equation*}
        0\to \bigoplus_{i=1}^r\VV\ot A^\mu_G\to B^\mu_{G}\to A^\mu_G\to 0.
    \end{equation*}
    For $f\in \VV$ we denote by $[f]_i\in B^\mu_G$ the element $f\ot 1$ in the $i$th direct summand of the first term of the above sequence. Thus $(C^{\mu}_{G,\s})^\even$ can be identified with the pushout of $B^\mu_G$ along the $A^\mu_G$-linear homomorphism $\bigoplus_{i=1}^r\VV\ot A^\mu_G\to \xi\ot A^\mu_G$ sending $[f]_i$ to $\xi\ot \sum_{i'=1}^r c_{i,i'}(d)[\pl_{\mu_{i'}}f]_{i'}$, if $f\in R^W$ is homogeneous of degree $2d$.
\end{remark}

\begin{remark}
    Let $d>D_{\mu_i}$. Since $R^{W_{\mu_i}}/(R^W_+)\cong \cohog{*}{G/P_{\mu_i}}$ has top degree $2D_{\mu_i}$, the homogeneous piece $R^{W_{\mu_i}}_{2d}$ lies in the ideal $(R^W_+)\subset R^{W_{\mu_i}}$. Using that $R^{W_{\mu_i}}$ is a free $R^W$-module, we have a canonical isomorphism
    \begin{equation*}
        (R^W_+R^{W_{\mu_i}})/(R^W_+R^{W_{\mu_i}})^2\cong \VV\ot \cohog{*}{G/P_{\mu_i}}.
    \end{equation*}
    In particular for $d>D_{\mu_i}$ there is a canonical map
    \begin{equation*}
        \pi_d: R^{W_{\mu_i}}_{2d}\subset (R^W_+R^{W_{\mu_i}})\surj(R^W_+R^{W_{\mu_i}})/(R^W_+R^{W_{\mu_i}})^2\cong \VV\ot \cohog{*}{G/P_{\mu_i}}.
    \end{equation*}
\end{remark}


\begin{lemma}\label{l:eta cases}
    Let $\y_i\in \ot R^{W_{\mu_i}}$ be homogeneous of degree $2n_i$ for $1\le i\le r$ and $\y=\y_1\ot\y_2\ot \cdots\ot \y_r\in \wt C^\mu_G$. Assume $\sum_{i=1}^rn_i=1+\sum_{i=1}^rD_{\mu_i}$.
    \begin{enumerate}
        \item If two or more $i$ satisfy $n_i>D_{\mu_i}$, then $\y$ has zero image in $C^\mu_{G,\s}$.
        \item If two or more $i$ satisfy $n_i<D_{\mu_i}$, then $\y$ has zero image in $C^\mu_{G,\s}$.
        \item If there is a unique $1\le i\le r$ such that $n_i>D_{\mu_i}$, and there is a unique $1\le i'\le r$ such that $n_{i'}<D_{\mu_{i'}}$, then we write $\pi_{n_i}(\y_i)\in \VV\ot \cohog{*}{G/P_{\mu_i}}$ in the form $f_{\y_i}\ot \y'_i+\cdots$ where $f_{\y_i}\in \VV_{2n_i-2D_{\mu_i}}$ and $\y'_i\in \cohog{2D_{\mu_i}}{G/P_{\mu_i}}$, and the terms in $\cdots$ involve only non-top degrees of $\cohog{2D_{\mu_i}}{G/P_{\mu_i}}$. Then in $C^\mu_{G,\s}$ we have
        \begin{equation*}
            \y\equiv c_{i,i'}(n_i-D_{\mu_i})\xi\ot\y'_i\ot (\pl_{\mu_{i'}}f_{\y_i}\cdot \y_{i'})\ot(\ot_{j\ne i,i'}\y_j) \in C^\mu_{G,\s}.
        \end{equation*}
        \item If there is a unique $1\le i\le r$ such that $n_i>D_{\mu_i}$, and all other $i'$ satisfies $n_{i'}=D_{\mu_{i'}}$ (so that $n_i=D_{\mu_i}+1$), then we write $\pi_{n_i}(\y_i)\in \VV\ot \cohog{*}{G/P_{\mu_i}}$ in the form $\sum_{j=1}^nf_j\ot \y^{(j)}_{i}$ using a homogeneous basis $\{f_j\}$ (of degree $d_j$) of $\VV$, where $\y^{(j)}_{i}\in \cohog{*}{G/P_{\mu_i}}$. Then
        \begin{equation*}
            \y\equiv \xi\ot\sum_{j=1}^nc_{i,i}(d_j)  (\pl_{\mu_{i}}f_{j}\cdot \y^{(j)}_{i})\ot(\ot_{i'\ne i}\y_{i'}) \in C^\mu_{G,\s}.
        \end{equation*}
        
    \end{enumerate}
\end{lemma}
\begin{proof}
    (1) Since $R^{W_{\mu_i}}/(R^W_+)\cong \cohog{*}{G/P_{\mu_i}}$ has top degree $2D_{\mu_i}$, and element $\y_i\in R^{W_{\mu_i}}$ of degree $>2D_{\mu_i}$ lies in the ideal $(R^W_+)$, hence its image in $C^\mu_{G,\s}$ is divisible by $\xi$. Since $\xi^2=0$, if two $\y_i$ has degree $>2D_{\mu_i}$, the image of $\y$ in $C^\mu_{G,\s}$ is zero.

    (2)(3)(4) By (1) we may assume there is a unique $i$ such that $n_i>D_{\mu_i}$. Write $\pi_{n_i}(\y_i)\in \VV\ot \cohog{*}{G/P_{\mu_i}}$ as $\sum_j f_j\ot \y^{(j)}_i$ using a homogeneous basis $\{f_j\}$ (of degree $d_j$) of $\VV$. Using the relation \eqref{rel C sig} for $[f_j]_i$ in $C^\mu_{G,\s}$, we have
    \begin{equation}\label{eta expansion}
        \y\equiv \xi\ot \sum_{j=1}^n\sum_{i'=1}^rc_{i,i'}(d_j)\y_i^{(j)}\ot [\pl_{\mu_{i'}}f_j]_{i'}\ot(\ot_{i''\ne i}\y_{i''})\in C^\mu_{G,\s}.
    \end{equation}
    We analyze the summand according as $i=i'$ or $i\ne i'$.
    \begin{itemize}
        \item Suppose the summand corresponding to $i=i'$ is nonzero in $C^\mu_{G,\s}$. Now $\sum_{j=1}^n c_{i,i}(d_j)\y_i^{(j)}\ot \pl_{\mu_{i}}f_j\in R^{W_{\mu_i}}$ has degree $n_i-1$. If $n_i-1>D_{\mu_i}$, then this element lies in the ideal $(R^W_+)$ hence its image in $C^\mu_{G,\s}$ is divisible by $\xi$, making the right side of \eqref{eta expansion} zero. Therefore we must have $n_i=D_{\mu_i}+1$. In this case $\sum_{i'\ne i}n_{i'}=\sum_{i'\ne i}D_{\mu_{i'}}$. By part (1), $\y$ is zero in $C^\mu_{G,\s}$ unless $n_{i'}=D_{\mu_{i'}}$ for all $i'\ne i$. This is exactly the situation described in part (4). To show part (4), we only need to check that the $i'\ne i$ summands on the right of \eqref{eta expansion} vanish in $C^\mu_{G,\s}$. Indeed, for $i'\ne i$, the $i'$ factor of that summand has degree $d_j-1+D_{\mu_{i'}}>D_{\mu_{i'}}$ ($d_j>1$ since $G$ is assumed to be semisimple). This proves part (4).

        \item If for some $i'\ne i$ and some $j$, the summand $c_{i,i'}(d_j)\y_i^{(j)}\ot [\pl_{\mu_{i'}}f_j]_{i'}\ot(\ot_{i''\ne i}\y_{i''})$ is nonzero in $C^\mu_{G,\s}$, then the same degree analysis above forces that $\y_i^{(j)}$ has degree $D_{\mu_i}$, $f_j$ thus has degree $d_j=n_i-D_{\mu_i}$, $\y_{i'}$ has degree $n_{i'}=D_{\mu_{i'}}-d_j+1$, and all other $\y_{i''}$ have degree $D_{\mu_{i''}}$. This is exactly the situation described in  part (3).
    \end{itemize}
    When two or more $i$ has $n_i<D_{\mu_i}$ we are in neither the situation of (3) nor (4), and the above analysis shows that $\y$ is zero in $C^\mu_{G,\s}$.
\end{proof}

\subsection{The case $G=\PGL_n$} We now apply the above lemma to a particular example. Let $G=\PGL_n$. For $R$ it is more convenient to identify it with the $R$ for $\SL_n$, i.e., $R=\Qlbar[x_1,\cdots, x_n]/(x_1+\cdots+x_n)$ with $W=S_n$ acting by permuting variables. Let $f_2,\cdots, f_n\in R^W$ be images of elementary symmetric polynomials in $R^W$. 

We consider the case where each $\mu_i$ is either $\mu_+=(1,0,\cdots, 0)$ or $\mu_-=(0,0,\cdots, 0,-1)$. Note that $G/P_{\mu_i}$ is isomorphic to $\PP^{n-1}$. Let $t_i\in R^{W_{\mu_i}}_2=\upH^{2}_G(G/P_{\mu_i})$ be $1/n$ of the $G$-equivariant Chern class of 
$\cO_{\PP^{n-1}}(n)$. When $\mu_i=\mu_+$, $G/P_{\mu_i}$ classifies lines, and $t_i=-x_1$ (as $x_1$ is the Chern class of $\cO(-1)$); when $\mu_i=\mu_-$, $G/P_{\mu_i}$ classifies hyperplanes, and $t_i=x_n$. We have $\int_{G/P_{\mu_i}}t_i^{n-1}=1$.

Note that $\dim \Sht^\mu_{G,\s}=(n-1)r+1$. Let
\begin{equation*}
\y=(t_1+t_2+\cdots+t_r)^{(n-1)r+1}\in C^\mu_{G,\s}.
\end{equation*}

\begin{prop}\label{prop: vol PGL}
The volume $\vol({}^\om\Sht^\mu_{G,\s},\y)$ is equal to $q^{(n^2-1)(g-1)}\z_X(2)\cdots\z_X(n)$ multiplied by the sum of the following quantities:
\begin{eqnarray*}
    &&-\sum_{i=1}^r\sum_{j=2}^n c_{i,i}(j)\binom{(n-1)r+1}{n,n-1,\cdots, n-1} \\
    && -\sum_{1\le i\ne i'\le r}\sum_{j=2}^{n}(-1)^{j\nu(i,i')}c_{i,i'}(j)\binom{(n-1)r+1}{n+j-1,n-j,n-1,\cdots, n-1}.
\end{eqnarray*}
Here
\begin{equation*}
    \nu(i,i')=\begin{cases}
        0 & \mbox{ if } \mu_i=\mu_{i'}\\
        1 & \mbox{ if } \mu_i\ne\mu_{i'}.
    \end{cases}
\end{equation*}
\end{prop}

\begin{proof}

Expand $\y$ in monomials $t_1^{n_1}\ot \cdots \ot t_r^{n_r}$ with total degree $(n-1)r+1$.  By Lemma \ref{l:eta cases}, this monomial has nonzero image in $C^\mu_{G,\s}$ only in the following cases:
\begin{itemize}
    \item There is exactly one $i$ such that $n_i=n$, and the rest $i'$ satisfies $n_{i'}=n-1$. When $\mu_i=\mu_+$ we have
    \begin{equation*}
        t_i^n=-(f_2 t_i^{n-2}+f_3t_i^{n-3}+\cdots+f_n).
    \end{equation*}
    The right side may be considered as the image of $t_i^n$ in $\VV\ot \cohog{*}{G/P_{\mu_i}}$.
    By Lemma \ref{l:eta cases}(4), we have\footnote{The astute reader may notice a subtle issue here: since we identified $R^W$ with $\cohog{*}{\BB\SL_n}$ rather than $\cohog{*}{\BB \PGL_n}$, $\mu_+=(1,0,\cdots, 0)\in \ZZ^{n}/\D(\ZZ)$ should be identified with the rational coweight $(\frac{n-1}{n}, -\frac{1}{n},\cdots, -\frac{1}{n})$ of $\SL_n$, hence $\pl_{\mu_+}f=\frac{n-1}{n}\pl_{x_1}f-\frac{1}{n}\pl_{x_2}f-\cdots-\frac{1}{n}\pl_{x_n}f$. However, because $\pl_{\mu_+}f\equiv\pl_{x_1}f\mod (R^W_+)$ and we only care about the image of $\pl_{\mu_+}f_j$ in $\cohog{*}{G/P_{\mu_i}}$ for the this calculation, we may replace $\pl_{\mu_+}f$ with $\pl_{x_1}f$ in the calculation.}
    \begin{equation*}
        t_1^{n_1}\ot \cdots \ot t_r^{n_r}=t_i(t_1\cdots t_r)^{n-1}\equiv \xi\ot\sum_{j=2}^nc_{i,i}(j)[-\pl_{\mu_+}f_j \cdot t_i^{n-j}]_{i}\ot (\ot_{i'\ne i} t_{i'}^{n-1})\in C^\mu_{G,\s}.
    \end{equation*}
    
    We have
    $\pl_{\mu_+}f_j \cdot t_i^{n-j}=\pl_{x_1}f_j\cdot t_i^{n-j}=c_{j-1}(Q_{n-1})t_i^{n-j}\in \cohog{2n-2}{\PP^{n-1}}$, where $Q_{n-1}\cong \cO^n/\cO(-1)$ is the universal quotient bundle. The relation $c(Q_{n-1})(1-t_i)=1$ in $\cohog{*}{\PP^{n-1}}$ implies $c_{j-1}(Q_{n-1})=t_i^{j-1}$ for $2\le j\le n$. Therefore $\pl_{x_1}f_j \cdot t_i^{n-j}=t_i^{n-1}\in \cohog{2n-2}{\PP^{n-1}}$. We conclude that when $\mu_i=\mu_+$
    \begin{equation}\label{t monomial case 1}
        t_1^{n_1}\ot \cdots\ot t_r^{n_r}\equiv -\sum_{j=2}^nc_{i,i}(j)\cdot \xi\ot t_1^{n-1}\ot\cdots \ot t_r^{n-1}\in C^\mu_{G,\s}.
    \end{equation}
    When $\mu_i=\mu_-$ we have
    \begin{equation*}
        t_i^n=-f_2 t_i^{n-2}+f_3t_i^{n-3}+\cdots+(-1)^{n-1}f_n.
    \end{equation*}
    Similar calculation gives the same formula as \eqref{t monomial case 1}.

    \item There is exactly one $i$ such that $n_i\ge n$, and exactly one $i'$ such that $n_{i'}<n-1$, and the rest of $n_{i''}$ are equal to $n-1$. Note that $n_i-n+n_{i'}=n-1$. When $\mu_{i}=\mu_+$ we have $\pi_{n_i}(t_i^{n_i})=-f_{n_i-n+1}t_i^{n-1}+\mbox{lower powers of } t_i\in \VV\ot \cohog{*}{\PP^{n-1}}$. Lemma \ref{l:eta cases}(3) implies
    \begin{equation*}
        t_1^{n_1}\ot \cdots \ot t_r^{n_r}\equiv -c_{i,i'}(n_i-n+1)\xi\ot [\pl_{\mu_{i'}}f_{n_i-n+1}\cdot t_{i'}^{n_{i'}}]_{i'}\ot (\ot_{i''\ne i'}t_{i''}^{n-1})\in C^\mu_{G,\s}.
    \end{equation*}
    When $\mu_{i'}=\mu_+$ we have $\pl_{\mu_+}f_{n_i-n+1}\cdot t_{i'}^{n_{i'}}=\pl_{x_1}f_{n_i-n+1}\cdot t_{i'}^{n_{i'}}=c_{n_i-n}(Q_{n-1})t_{i'}^{n_{i'}}=t_{i'}^{n-1}$. Thus when $\mu_{i}=\mu_{i'}=\mu_+$ we have
    \begin{equation*}
        t_1^{n_1}\ot \cdots \ot t_r^{n_r}\equiv -c_{i,i'}(n_i-n+1)\xi\ot t_1^{n-1}\ot\cdots \ot t_r^{n-1}\in C^\mu_{G,\s}.
    \end{equation*}    
    When $\mu_{i'}=\mu_-$ we have $\pl_{\mu_-}f_{n_i-n+1}\cdot t_{i'}^{n_{i'}}=-\pl_{x_n}f_{n_i-n+1}\cdot t_{i'}^{n_{i'}}=-c_{n_i-n}(S_{n-1})t_{i'}^{n_{i'}}$ where $S_{n-1}=\ker(\cO^n\to \cO(1))$ is the universal hyperplane bundle over $\PP^{n-1}$. The relation $c(S_{n-1})(1+t_{i'})=1\in \cohog{*}{\PP^{n-1}}$ implies $c_{n_i-n}(S_{n-1})=(-1)^{n_i-n}t_{i'}^{n_i-n}$. Thus in this case $\pl_{\mu_-}f_{n_i-n+1}\cdot t_{i'}^{n_{i'}}=(-1)^{n_i-n+1}t_{i'}^{n-1}$. Thus when $\mu_i=\mu_+$ and $\mu_{i'}=\mu_-$ we have
    \begin{equation*}
        t_1^{n_1}\ot \cdots \ot t_r^{n_r}\equiv (-1)^{n_i-n}c_{i,i'}(n_i-n+1)\xi\ot t_1^{n-1}\ot\cdots \ot t_r^{n-1}\in C^\mu_{G,\s}.
    \end{equation*}  
    Similar calculations show:
    when $\mu_i=\mu_-$ and $\mu_{i'}=\mu_+$
    \begin{equation*}
        t_1^{n_1}\ot \cdots \ot t_r^{n_r}\equiv (-1)^{n_i-n}c_{i,i'}(n_i-n+1)\xi\ot t_1^{n-1}\ot\cdots \ot t_r^{n-1}\in C^\mu_{G,\s}.
    \end{equation*}  
    When $\mu_i=\mu_-$ and $\mu_{i'}=\mu_-$
    \begin{equation*}
        t_1^{n_1}\ot \cdots \ot t_r^{n_r}\equiv -c_{i,i'}(n_i-n+1)\xi\ot t_1^{n-1}\ot\cdots \ot t_r^{n-1}\in C^\mu_{G,\s}.
    \end{equation*} 
\end{itemize}

Summarizing the contribution of all monomials, we get that the image of $\y$ in  $C^\mu_{G,\s}$ is  $\xi\ot t_1^{n-1}\ot\cdots\ot t_r^{n-1}$ multiplied by the constants in the statement of the proposition. This completes the proof.


\end{proof}

\subsubsection{Connection to Artin L-values}
We now rewrite the result in Proposition \ref{prop: vol PGL} in terms of special values of Artin L-functions.
For an irreducible representation $\rho$ of $\Sig$, we have an Artin $L$-function defined by
\begin{eqnarray*}
L_{Y,\rho}(s)=\det(1-q^{-s}\phi| \cohog{*}{Y,\LL_\rho} )  = \prod_i \det(1-q^{-s}\phi| \cohog{i}{Y,\LL_\rho} )^{(-1)^{i+1}} ,
\end{eqnarray*}
where $\LL_\rho$ is the local system on $Y$ corresponding to $\rho$. 

We first relate the constants \eqref{eq: c(d)} to such L-values. We have the following generalization of Lemma \ref{l:diag Xid}.
\begin{lemma}\label{l: X s} 
Let $\s_X:X\to X\times X$ be the graph\footnote{To be clear, we mean $\s_X(x)=(x,\s(x))$.} For any integer $d\ne0,1$ we have 
\begin{eqnarray*}
\s_X ^*\Xi_d=-\sum_{\rho} \chi_{\rho^\vee}(\s)\frac{L'_{Y,\rho}(d)}{\log(q)L_{Y,\rho}(d)}\xi \in \cohog{2}{X}
\end{eqnarray*}
where the sum runs over all irreducible representations $\rho$ of $\Sig$, and $ \chi_{\rho^\vee}$ denotes the character of $\rho^\vee$, the dual representation to $\rho$.
\end{lemma}
\begin{proof}
We have
\begin{eqnarray*}
    \s_X ^*\Xi_d&=&\xi\int_{X\times X}[\s_X(X)]\cdot \Xi_d=\xi\int_{X\times X}[\s_X(X)]\cdot \left(\frac{1}{q^d\phi^{-1}-1}\ot\id\right)[\D_X]\\
    &=&\Tr(\s^* (q^d\phi^{-1}-1)^{-1}|\cohog{*}{X})\xi.
\end{eqnarray*}
Here we note that $\s$ and $\phi$ commute as endomorphisms of $\cohog{*}{X}$.

Decompose the local system $\pi_\ast\Qlbar$ on $Y$, as representations of $\Sig$, 
$$
\pi_\ast\Qlbar=\bigoplus_{\rho} \rho^\vee\boxtimes\LL_\rho,
$$ where the sum runs over all irreducible representations $\rho$ of $\Sig$.
Accordingly, we have
$$
\cohog{*}{X}=\bigoplus_{\rho}\rho^\vee\boxtimes \cohog{*}{Y,\LL_\rho} 
$$
as $\Sig\times \j{\phi}$-modules.
We then have
\begin{eqnarray*}
\Tr(\s^*(q^d\phi^{-1}-1)^{-1}|\cohog{*}{X})&=&\sum_{\rho} \chi_{\rho^\vee}(\s)\Tr((q^d\phi^{-1}-1)^{-1}|\cohog{*}{Y,\LL_\rho}) \\
&=&-\sum_{\rho} \chi_{\rho^\vee}(\s)\frac{L'_{Y,\rho}(d)}{\log(q)L_{Y,\rho}(d)}
\end{eqnarray*}
as desired.

\end{proof}


\begin{defn}
We define the logarithmic Artin L-function associated to a class function $\varphi$ on $\Sigma$. Let $\varphi =\sum_\rho a_\rho \chi_\rho$ be any class function, where $a_\rho\in\Qlbar$ and $\rho$ are irreducible representations of $\Sig$. Then we define
$$\frac{L'_{Y, \varphi}(s)}{L_{Y, \varphi}(s)}:=\sum_{\rho} a_\rho \frac{L'_{Y, \rho}(s)}{L_{Y, \rho}(s)}.
$$
\end{defn}

\subsubsection{Colmez's Conjecture} We have an involution on $\Qlbar[\Sig]$ induced by the inversion on $\Sigma$: for any element $\varphi \in\Qlbar[\Sig]$, define  $\varphi ^\vee(g)= \varphi (g^{-1}).$
Define the convolution of two functions
$$
(\varphi\ast \varphi')(g):=\int_{\Sig} \varphi(gh^{-1})\varphi'(h)\,dh.
$$
Here the Haar measure on $\Sig$ is chosen such that $\vol(\Sig)=1$.

For $j\in\Z$, we define 
$$
\Phi_{j}=\sum_{i=1}^r {\rm sign}(\mu_i)^j \s_i\in\Qlbar[\Sig], \quad \text{ where } {\rm sign}(\mu_i) = \begin{cases} +1 & \mu_i = \mu_+, \\ -1 & \mu_i = \mu_-. \end{cases}
$$ 
In particular, $\Phi_{j}$ depends only on the parity of $j$, and if $j$ is even, we have 
$$
\Phi_{j}=\Phi=\sum_{i=1}^r  \s_i.
$$

For any $\varphi\in \Qlbar[\Sig]$, let $\varphi^\natural$ 
be the projection to the space of class functions on $\Sig$. In terms of the characters we have 
$$
\varphi^\natural=\sum_\rho \langle\varphi,\chi_{\rho^\vee} \rangle\chi_\rho.
$$ 
Here the pairing $\langle\varphi,\varphi' \rangle$ is bilinear and we have $\langle\chi_{\rho},\chi_{\rho^\vee} \rangle=1$.

\begin{theorem}\label{th:vol 1-leg PGL}
The volume  $\vol({}^\om\Sht^\mu_{\PGL_n,\s}, \eta) $ is equal to $q^{\dim\Bun_{\PGL_n}} \prod_{d=2}^n\z_X(d) $ multiplied by
    \begin{eqnarray*}
   &-&\sum_{j=2}^n\left(\binom{(n-1)r+1}{n,n-1,\cdots, n-1}-\binom{(n-1)r+1}{n+j-1,n-j,n-1,\cdots, n-1}\right) \frac{\z'_X(j)}{\log q\,\z_X(j)}
 \\ 
      &-&|\Sig|^2\sum_{j=2}^n  \binom{(n-1)r+1}{n+j-1,n-j,n-1,\cdots, n-1} \frac{L'_{Y,(\Phi_j\ast\Phi_j^\vee)^\natural}(j)}{\log q\,L_{Y,(\Phi_j\ast\Phi_j^\vee)^\natural}(j)} \\
   &-&(g_Y-1) |\Sigma|^2 \sum_{j=2}^n  \binom{(n-1)r+1}{n+j-1,n-j,n-1,\cdots, n-1} ((\Phi_j\ast\Phi_j^\vee)(1)-r/|\Sigma|) .
   \end{eqnarray*}
\end{theorem}
\begin{remark}
    In the number field case, Colmez conjecture \cite{Col1} relates the stable Faltings height of an abelian variety with CM by the ring of integers of a CM field to the special value at $s=1$ (or equivalently at $s=0$ by functional equation) of the logarithmic Artin $L$-function attached to a class function arising from the corresponding CM type. The recipe of the class function in Colmez conjecture (cf. \cite[\S2]{Col2} for a more explicit formula involving $\Phi\ast\Phi^\vee$) is completely analogous to the one above, except that here we have the special value at $s=j\geq 2$. 
\end{remark}
\begin{proof}
We apply Lemma \ref{l: X s} to get
\begin{eqnarray*}
    c_{i,i}(j) &=&-\sum_{\rho} \chi_{\rho^\vee}(1)\frac{L'_{Y,\rho}(j)}{\log(q)L_{Y,\rho}(j)}.
    \end{eqnarray*}
We note that, by the factorization $\z_X(s)=\prod_\rho L_{Y,\rho}(s)^{\dim\rho}$,
$$\sum_{\rho} \chi_{\rho^\vee}(1)\frac{L'_{Y,\rho}(j)}{L_{Y,\rho}(j)}=\frac{\z'_X(j)}{\z_X(j)}.
$$
Similarly, we have
\begin{eqnarray*}
      c_{i,i'}(j) &=&-\sum_{\rho} \chi_{\rho^\vee}(\s_i^{-1}\s_{i'})\frac{L'_{Y,\rho}(j)}{\log(q)L_{Y,\rho}(j)}  
\end{eqnarray*}
when $i'>i$, and
\begin{eqnarray*}c_{i,i'}(j) &=&\sum_{\rho} \chi_{\rho^\vee}(\s_{i'}^{-1}\s_{i})\frac{L'_{Y,\rho}(1-j)}{\log(q)L_{Y,\rho}(1-j)} 
\end{eqnarray*}
when $i'<i$.

By the functional equation
$$L_{Y,\rho}(s)\ep(\rho,1/2) q^{s(2g_Y-2)\dim\rho}=L_{Y,\rho^\vee}(1-s)
$$
we have the following relation between logarithmic derivatives of $L_{Y,\rho}(s)$ at $d$ and $1-d$,
by 
$$
\frac{L'_{Y,\rho}(d)}{\log(q)L_{Y,\rho}(d)}+\frac{L'_{Y,\rho^\vee}(1-d)}{\log(q)L_{Y,\rho^\vee}(1-d)}=-(2g_Y-2)\dim\rho.
$$
We may rewrite the case  $i'<i$ as
\begin{eqnarray*}
c_{i,i'}(j)  =-\sum_{\rho} \chi_{\rho^\vee}(\s_{i}^{-1}\s_{i'})\frac{L'_{Y,\rho}(j)}{\log q\,L_{Y,\rho}(j)}-(2g_Y-2)\sum_{\rho} \chi_{\rho^\vee}(\s_{i'}^{-1}\s_{i}) \dim\rho.
\end{eqnarray*}

Then the displayed expression in Proposition \ref{prop: vol PGL} is the sum 
\begin{eqnarray*}
&-&\sum_{j=2}^n\left(\binom{(n-1)r+1}{n,n-1,\cdots, n-1}-\binom{(n-1)r+1}{n+j-1,n-j,n-1,\cdots, n-1}\right) \frac{\z'_X(j)}{\log q\,\z_X(j)}\\
   &-&\sum_{j=2}^n  \binom{(n-1)r+1}{n+j-1,n-j,n-1,\cdots, n-1}\sum_{\rho}\frac{L'_{Y,\rho}(j)}{\log q \,L_{Y,\rho}(j)} \sum_{(i,i')} \chi_{\rho^\vee}(\s_i^{-1}\s_{i'}) (-1)^{j\nu(i,i')}  \\
   &-&
      \sum_{j=2}^n  \binom{(n-1)r+1}{n+j-1,n-j,n-1,\cdots, n-1} \sum_{\rho} \sum_{(i, i'), i> i'}\chi_{\rho^\vee}(\s_{i'}^{-1}\s_{i})  (-1)^{j\nu(i,i')}  \dim\rho.
\end{eqnarray*}
Now we note that
$$
\sum_{(i,i')} \chi_{\rho^\vee}(\s_i^{-1}\s_{i'}) (-1)^{j\nu(i,i')}=|\Sigma|^2\langle{\Phi_j\ast\Phi_j^\vee, \chi_{\rho^\vee} \rangle}
$$
and
\begin{eqnarray*}
\sum_{\rho} \sum_{(i, i'), i> i'}\chi_{\rho^\vee}(\s_{i'}^{-1}\s_{i}) (-1)^{j\nu(i,i')}  \dim\rho
&=&\sum_{\rho} \sum_{(i, i'), i> i'}\chi_{\rho^\vee}(\s_{i'}^{-1}\s_{i}) (-1)^{j\nu(i,i')}  \chi_{\rho}(1)\\
&=&|\Sigma|\cdot\sum_{ i> i', \s_i=\s_{i'}} (-1)^{j\nu(i,i')} \\
&=&|\Sigma|(|\Sigma|(\Phi_j\ast\Phi_j^\vee)(1)-r)/2.
\end{eqnarray*}
These identities together complete the proof.

\end{proof}


\bibliographystyle{amsalpha}
\bibliography{Bibliography}

\providecommand{\bysame}{\leavevmode\hbox to3em{\hrulefill}\thinspace}
\providecommand{\MR}{\relax\ifhmode\unskip\space\fi MR }
\providecommand{\MRhref}[2]{%
  \href{http://www.ams.org/mathscinet-getitem?mr=#1}{#2}
}
\providecommand{\href}[2]{#2}
\begin{thebibliography}{AGHMP18}

\bibitem[AB83]{AB83}
M.~F. Atiyah and R.~Bott, \emph{The {Y}ang-{M}ills equations over {R}iemann
  surfaces}, Philos. Trans. Roy. Soc. London Ser. A \textbf{308} (1983),
  no.~1505, 523--615.

\bibitem[AGHMP18]{AGHMP}
Fabrizio Andreatta, Eyal~Z. Goren, Benjamin Howard, and Keerthi Madapusi~Pera,
  \emph{Faltings heights of abelian varieties with complex multiplication},
  Ann. of Math. (2) \textbf{187} (2018), no.~2, 391--531.

\bibitem[BD09]{BD09}
Kai Behrend and Ajneet Dhillon, \emph{Connected components of moduli stacks of
  torsors via {T}amagawa numbers}, Canad. J. Math. \textbf{61} (2009), no.~1,
  3--28.

\bibitem[BH]{BH}
Jan~Hendrik Bruinier and Benjamin Howard, \emph{Arithmetic volumes of unitary
  {S}himura varieties}, to appear in Memoirs of the AMS.

\bibitem[Col93]{Col1}
Pierre Colmez, \emph{P\'{e}riodes des vari\'{e}t\'{e}s ab\'{e}liennes \`a
  multiplication complexe}, Ann. of Math. (2) \textbf{138} (1993), no.~3,
  625--683.

\bibitem[Col98]{Col2}
\bysame, \emph{Sur la hauteur de {F}altings des vari\'{e}t\'{e}s ab\'{e}liennes
  \`a multiplication complexe}, Compositio Math. \textbf{111} (1998), no.~3,
  359--368.

\bibitem[DCP83]{DP}
C.~De~Concini and C.~Procesi, \emph{Complete symmetric varieties}, Invariant
  theory ({M}ontecatini, 1982), Lecture Notes in Math., vol. 996, Springer,
  Berlin, 1983, pp.~1--44.

\bibitem[EH17]{EH}
H\'{e}l\`ene Esnault and Michael Harris, \emph{Chern classes of automorphic
  vector bundles}, Pure Appl. Math. Q. \textbf{13} (2017), no.~2, 193--213.

\bibitem[Fen26]{FAI}
Tony Feng, \emph{Eigenweights for {A}rithmetic {H}irzebruch {P}roportionality},
  2026.

\bibitem[FHM25]{FHM25}
Tony Feng, Benjamin Howard, and Mikayel Mkrtchyan, \emph{{Higher Siegel--Weil
  formula for unitary groups II: corank one terms}}, 2025.

\bibitem[FYZ24]{FYZ}
Tony Feng, Zhiwei Yun, and Wei Zhang, \emph{Higher {S}iegel--{W}eil formula for
  unitary groups: the non-singular terms}, Invent. Math. \textbf{235} (2024),
  no.~2, 569--668.

\bibitem[FYZ25]{FYZ2}
\bysame, \emph{Higher theta series for unitary groups over function fields},
  Ann. Sci. \'{E}c. Norm. Sup\'{e}r. (4) \textbf{58} (2025), no.~2, 275--388.

\bibitem[Gai19]{Gai19}
Dennis Gaitsgory, \emph{The {A}tiyah-{B}ott formula for the cohomology of the
  moduli space of bundles on a curve}, 2019.

\bibitem[GL14]{GL14}
Dennis Gaitsgory and Jacob Lurie, \emph{Weil's {C}onjecture for {F}unction
  {F}ields}, 2014.

\bibitem[Gro68]{Gro68}
Alexander Grothendieck, \emph{Le groupe de {B}rauer. {III}. {E}xemples et
  compl\'{e}ments}, Dix expos\'{e}s sur la cohomologie des sch\'{e}mas, Adv.
  Stud. Pure Math., vol.~3, North-Holland, Amsterdam, 1968, pp.~88--188.

\bibitem[Gro97]{Gro97}
Benedict~H. Gross, \emph{On the motive of a reductive group}, Invent. Math.
  \textbf{130} (1997), no.~2, 287--313.

\bibitem[Guo25]{Guo}
Ziqi Guo, \emph{{Modular Heights of Unitary Shimura Varieties}}, 2025.

\bibitem[H\"14]{Hor}
Fritz H\"{o}rmann, \emph{The geometric and arithmetic volume of {S}himura
  varieties of orthogonal type}, CRM Monograph Series, vol.~35, American
  Mathematical Society, Providence, RI, 2014.

\bibitem[Hei10]{Hei10}
Jochen Heinloth, \emph{Uniformization of {$\mathscr{G}$}-bundles}, Math. Ann.
  \textbf{347} (2010), no.~3, 499--528.

\bibitem[Hir58]{Hir}
Friedrich Hirzebruch, \emph{Automorphe {F}ormen und der {S}atz von
  {R}iemann-{R}och}, Symposium internacional de topolog\'{\i}a algebraica
  {I}nternational symposium on algebraic topology, Universidad Nacional
  Aut\'{o}noma de M\'{e}xico and UNESCO, M\'{e}xico, 1958, pp.~129--144.

\bibitem[Ho21]{Ho21}
Quoc~P. Ho, \emph{The {A}tiyah-{B}ott formula and connectivity in chiral
  {K}oszul duality}, Adv. Math. \textbf{392} (2021), Paper No. 107992, 71.

\bibitem[HS10]{HS10}
Jochen Heinloth and Alexander H.~W. Schmitt, \emph{The cohomology rings of
  moduli stacks of principal bundles over curves}, Doc. Math. \textbf{15}
  (2010), 423--488.

\bibitem[K\"05]{Koh}
Kai K\"{o}hler, \emph{A {H}irzebruch proportionality principle in {A}rakelov
  geometry}, Number fields and function fields---two parallel worlds, Progr.
  Math., vol. 239, Birkh\"{a}user Boston, Boston, MA, 2005, pp.~237--268.

\bibitem[LZ17]{LiuZheng}
Yifeng Liu and Weizhe Zheng, \emph{Enhanced six operations and base change
  theorem for higher artin stacks}, 2017.

\bibitem[MR02]{MR}
Vincent Maillot and Damien Roessler, \emph{Conjectures sur les d\'{e}riv\'{e}es
  logarithmiques des fonctions {$L$} d'{A}rtin aux entiers n\'{e}gatifs}, Math.
  Res. Lett. \textbf{9} (2002), no.~5-6, 715--724.

\bibitem[Mum77]{Mum}
D.~Mumford, \emph{Hirzebruch's proportionality theorem in the noncompact case},
  Invent. Math. \textbf{42} (1977), 239--272.

\bibitem[Tel98]{Tel98}
Constantin Teleman, \emph{Borel-{W}eil-{B}ott theory on the moduli stack of
  {$G$}-bundles over a curve}, Invent. Math. \textbf{134} (1998), no.~1, 1--57.

\bibitem[vdG99]{vdG}
Gerard van~der Geer, \emph{Cycles on the moduli space of abelian varieties},
  Moduli of curves and abelian varieties, Aspects Math., E33, Friedr. Vieweg,
  Braunschweig, 1999, pp.~65--89.

\bibitem[WZ23]{WZ}
Torsten Wedhorn and Paul Ziegler, \emph{Tautological rings of {S}himura
  varieties and cycle classes of {E}kedahl-{O}ort strata}, Algebra Number
  Theory \textbf{17} (2023), no.~4, 923--980.

\bibitem[YZ18]{YZ18}
Xinyi Yuan and Shou-Wu Zhang, \emph{On the averaged {C}olmez conjecture}, Ann.
  of Math. (2) \textbf{187} (2018), no.~2, 533--638.

\end{thebibliography}

\appendix

\end{document}